\documentclass[10pt, reqno]{amsart}
\usepackage{amssymb}
\usepackage{graphicx}
\usepackage{xcolor} 
\usepackage{tensor}

\usepackage{enumitem}

\usepackage[english]{babel}
\usepackage[utf8]{inputenc}
\usepackage[colorinlistoftodos]{todonotes}

\usepackage{amsfonts}
\usepackage{amssymb}
\usepackage{amsmath}
\usepackage{amsthm}

\usepackage{graphics}
\usepackage{pdfsync}

\usepackage{graphicx}
\usepackage{amssymb}
\usepackage{epstopdf}
\usepackage{hyperref}
\usepackage{color}
\usepackage{tikz}

\newtheorem{theorem}{Theorem}[section]

\newtheorem{lemma}[theorem]{Lemma}
\newtheorem{proposition}[theorem]{Proposition}
\theoremstyle{definition}
\newtheorem{definition}[theorem]{Definition}

\theoremstyle{remark}
\newtheorem{remark}[theorem]{Remark}
\numberwithin{equation}{section}

\newcommand{\nrm}[1]{\Vert#1\Vert}
\newcommand{\abs}[1]{\vert#1\vert}
\newcommand{\brk}[1]{\langle#1\rangle}
\newcommand{\set}[1]{\{#1\}}

\newcommand{\dist}{\mathrm{dist}}

\newcommand{\supp}{{\mathrm{supp}} \, }

\renewcommand{\Im}{\mathrm{Im}}

\newcommand{\aleq}{\lesssim}
\newcommand{\ageq}{\gtrsim}

\newcommand{\lap}{\Dlt}

\newcommand{\ud}{\mathrm{d}}
\newcommand{\rd}{\partial}
\newcommand{\nb}{\nabla}

\newcommand{\bb}{\Big}

\newcommand{\0}{\emptyset}

\newcommand{\alp}{\alpha}
\newcommand{\bt}{\beta}
\newcommand{\gmm}{\gamma}

\newcommand{\dlt}{\delta}
\newcommand{\Dlt}{\Delta}
\newcommand{\eps}{\epsilon}

\newcommand{\lmb}{\lambda}

\newcommand{\sgm}{\sigma}

\newcommand{\tht}{\theta}

\newcommand{\omg}{\omega}

\newcommand{\zt}{\zeta}

\newcommand{\bfj}{{\bf j}}
\newcommand{\bfk}{{\bf k}}

\newcommand{\bfp}{{\bf p}}

\newcommand{\bfA}{{\bf A}}

\newcommand{\bfD}{{\bf D}}

\newcommand{\bfI}{{\bf I}}

\newcommand{\bbC}{\mathbb C}

\newcommand{\bbR}{\mathbb R}
\newcommand{\bbS}{\mathbb S}

\newcommand{\bbZ}{\mathbb Z}

\newcommand{\calA}{\mathcal A}
\newcommand{\calB}{\mathcal B}
\newcommand{\calC}{\mathcal C}

\newcommand{\calE}{\mathcal E}

\newcommand{\calH}{\mathcal H}

\newcommand{\calJ}{\mathcal J}

\newcommand{\calL}{\mathcal L}
\newcommand{\calM}{\mathcal M}
\newcommand{\calN}{\mathcal N}

\newcommand{\calP}{\mathcal P}

\newcommand{\calR}{\mathcal R}
\newcommand{\calS}{\mathcal S}
\newcommand{\calT}{\mathcal T}

\newcommand{\sbeq}{\subseteq}

\newcommand{\pfstep}[1]{\vskip.5em \noindent {\bf #1.}}

\newcommand{\Id}{\bfI}
\newcommand{\covD}{\bfD}				
\newcommand{\mR}{\mathcal{R}}			
\newcommand{\Diff}{\pi}					
\newcommand{\NM}{\calM}				
\newcommand{\ND}{\calN}				
\newcommand{\NR}{\widetilde{\ND}}			
\newcommand{\NF}{\calN}			
\newcommand{\ANF}{N}			
\newcommand{\BL}{L}				

\newcommand{\med}{\mathrm{med}}

\newcommand{\Vol}{\mathrm{Vol} }

\newcommand{\mb}{\mathbb}
\newcommand{\be}{\begin{equation}}
\newcommand{\ee}{ \end{equation}}

\newcommand{\ep}{\varepsilon}

\newcommand{\al}{\alpha}

\newcommand{\dd}{\,\mathrm{d}}

\newcommand{\vn}[1]{\|#1\|}
\newcommand{\vm}[1]{\left|#1\right|}

\newcommand{\lpr}{ \left( }
\newcommand{\rpr}{ \right) }

\newcommand{\lpp}{\left[}
\newcommand{\rpp}{\right]}

\newcommand{\la}{\Delta}

\newcommand{\Hc}{\dot{H}^1\times L^2}
\newcommand{\Hcr}{\dot{H}^{1/2}}
\newcommand{\hw}{(i\partial_t+\vm{D})}
\newcommand{\hwm}{(i\partial_t-\vm{D})}
\newcommand{\phw}{\hw_A^p}
\newcommand{\pt}{\partial}

\usepackage{mathtools}
\newcommand{\defeq}{\vcentcolon=}

\begin{document}

\title[Small critical data GWP for Maxwell--Dirac]{Global well-posedness of high dimensional Maxwell--Dirac for small critical data}
\author{Cristian Gavrus}%
\address{Department of Mathematics, UC Berkeley, Berkeley, CA, 94720}%
\email{cristian@berkeley.edu}%

\author{Sung-Jin Oh}%
\address{Department of Mathematics, UC Berkeley, Berkeley, CA, 94720}%
\email{sjoh@math.berkeley.edu}%

\thanks{The authors thank Daniel Tataru for many fruitful conversations. C. Gavrus was supported in part by the NSF grant DMS-1266182. S.-J. Oh is a Miller Research Fellow, and acknowledges support from the Miller Institute. Part of this work was carried out during the trimester program 'Harmonic Analysis and PDEs' at the Hausdorff Institute of Mathematics in Bonn.}%

\begin{abstract}
In this paper, we prove global well-posedness of the massless Maxwell-Dirac equation in Coulomb gauge on $\bbR^{1+d}$ $(d \geq 4)$ for data with small scale-critical Sobolev norm, as well as modified scattering of the solutions.
Main components of our proof are A) uncovering null structure of Maxwell-Dirac in the Coulomb gauge, and B) proving solvability of the underlying covariant Dirac equation. A key step for achieving both is to exploit (and justify) a deep analogy between Maxwell-Dirac and Maxwell-Klein-Gordon (for which an analogous result was proved earlier by Krieger-Sterbenz-Tataru \cite{KST}), which says that the most difficult part of Maxwell-Dirac takes essentially the same form as Maxwell-Klein-Gordon.
\end{abstract}
\maketitle

\setcounter{tocdepth}{1}
\tableofcontents
\section{Introduction} \label{sec:intro}

Let $\bbR^{1+d}$ be the $(d+1)$-dimensional Minkowski space with the metric 
\begin{equation*}
\eta = \mathrm{diag} (-1, +1, \ldots, +1)
\end{equation*}
in the rectilinear coordinates $(x^{0}, x^{1}, \ldots, x^{d})$. Associated to the Minkowski metric $\eta$ are the \emph{gamma matrices}, which are $N \times N$ complex-valued matrices $\gmm^{\mu}$ ($\mu=0, 1, \ldots, d$) satisfying the anti-commutation relations
\begin{equation} \label{eq:gmmRel}
	\frac{1}{2} (\gmm^{\mu} \gmm^{\nu} + \gmm^{\nu} \gmm^{\mu}) = - \eta^{\mu \nu} \, \Id,
\end{equation}
where $\Id$ is the $N \times N$ identity matrix, and also the conjugation relations
\begin{equation} \label{eq:gmmRel:conj}
	(\gmm^{\mu})^{\dagger} = \gmm^{0} \gmm^{\mu} \gmm^{0}.
\end{equation}
On $\bbR^{1+d}$, the rank of the gamma matrices $\gmm^{\mu}$ in the standard representation is $N = 2^{\lfloor \frac{d+1}{2} \rfloor}$ \cite[Appendix~E]{dW}. A \emph{spinor field} $\psi$ is a function on $\bbR^{1+d}$ (more generally, open subsets of $\bbR^{1+d}$) that takes values in $\bbC^{N}$, on which $\gmm^{\mu}$ acts as multiplication. 
Given a real-valued 1-form $A_{\mu}$ (connection 1-form), we introduce the \emph{gauge covariant derivative} on spinors
\begin{equation*}
	\covD_{\mu} \psi := \rd_{\mu} \psi + i A_{\mu} \psi,
\end{equation*}
(which acts componentwisely on $\psi$) and the associated \emph{curvature 2-form} 
\begin{equation*}
	F_{\mu \nu} := (\ud A)_{\mu \nu} = \rd_{\mu} A_{\nu} - \rd_{\nu} A_{\mu}.
\end{equation*}

The Maxwell--Dirac system is a relativistic Lagrangian field theory describing the interaction between a connection 1-form $A_{\mu}$, representing an electromagnetic potential, and a spinor field $\psi$, modeling a charged fermionic field (e.g., an electron). Its action (i.e., the space-time integral of the Lagrangian) takes the form
\begin{equation*}
	\calS[A_{\mu}, \psi] = \iint_{\bbR^{1+d}} - \frac{1}{4} F_{\mu \nu} F^{\mu \nu} + i \brk{\gmm^{\mu} \covD_{\mu} \psi, \gmm^{0} \psi} - m \brk{\psi, \psi} \, \ud t \ud x.
\end{equation*}
Here $\brk{\psi^{1}, \psi^{2}} := (\psi^{2})^{\dagger} \psi^{1}$ is the usual inner product on $\bbC^{N}$, where $\psi^{\dagger}$ denotes the hermitian transpose. Furthermore, we use the standard convention of raising and lowering indices using the Minkowski metric $\eta$, and the Einstein summation convention of summing repeated upper and lower indices. The Euler--Lagrange equations for $\calS[A_{\mu}, \psi]$ take the form
\begin{equation} \label{eq:MD} \tag{MD}
\left\{
\begin{aligned}
	\rd^{\nu} F_{\mu \nu} =& - \brk{\psi, \alp_{\mu} \psi}, \\
	i \alp^{\mu} \covD_{\mu} \psi =& m \bt \psi.
\end{aligned}
\right.
\end{equation}
where $\alp^{\mu} = \gmm^{0} \gmm^{\mu}$ and $\bt = \gmm^{0}$. Henceforth, we will refer to \eqref{eq:MD} as the \emph{Maxwell--Dirac equations}. 



A basic feature of \eqref{eq:MD} is invariance under gauge transformations. That is, given any solution $(A, \psi)$ of \eqref{eq:MD} and a real-valued function $\chi$ (gauge transformation) on $I \times \bbR^{d}$, the gauge transform $(\tilde{A}, \tilde{\psi}) = (A - \ud \chi, e^{i \chi} \psi)$ of $(A, \psi)$ is also a solution to \eqref{eq:MD}. In order to make \eqref{eq:MD} a (formally) well-posed system, we need to remove the ambiguity arising from this invariance. To this end, we impose in our paper the (global) \emph{Coulomb gauge} condition, which reads
\begin{equation} \label{eq:coulomb-0}
	\mathrm{div}_{x} A = \sum_{j=1}^{d} \rd_{j} A_{j} = 0.
\end{equation}

In this paper, we show global well-posedness and scattering for \emph{massless} \eqref{eq:MD} (i.e., $m = 0$) on the Minkowski space $\bbR^{1+d}$ with $d \geq 4$ under the Coulomb gauge condition, for initial data which are small in the scale-critical Sobolev space. When restricted to the massless case, \eqref{eq:MD} is invariant under the scaling $(\lmb > 0)$
\begin{equation*}
	(A_{\mu}, \psi) \mapsto \bb( \lmb^{-1} A( \lmb^{-1}t, \lmb^{-1} x), \lmb^{-\frac{3}{2}} \psi(\lmb^{-1} t, \lmb^{-1} x) \bb).
\end{equation*}
For the sake of concreteness, we focus on the case $d = 4$, which is the most difficult. 
\begin{theorem}[Critical small data global well-posedness and scattering on $\bbR^{1+4}$] \label{thm:main}
Consider \eqref{eq:MD} on $\bbR^{1+4}$ with $m = 0$. There exists a universal constant $\eps_{\ast} > 0$ such that the following statements hold.
\begin{enumerate} [leftmargin=*]
\item Let $(\psi(0), A_{j}(0), \rd_{t} A_{j}(0))$ be a smooth initial data set satisfying the Coulomb condition \eqref{eq:coulomb-0} and the smallness condition
\begin{equation} \label{eq:main:smalldata}
	\nrm{\psi(0)}_{\dot{H}^{1/2}(\bbR^{4})} 
	+ \sup_{j =1, \ldots, 4} \nrm{(A_{j}, \rd_{t} A_{j})(0)}_{\dot{H}^{1} \times L^{2}(\bbR^{4})} < \eps_{\ast}.
\end{equation}
Then there exists a unique global smooth solution $(\psi, A)$ to the system \eqref{eq:MD} under the Coulomb gauge condition \eqref{eq:coulomb-0} on $\bbR^{1+4}$ with these data. 
\item For any $T > 0$, the data-to-solution map $(\psi, A_{j}, \rd_{t} A_{j})(0) \mapsto (\psi, A_{j}, \rd_{t} A_{j})$ extends continuously to
\begin{align*}
	\dot{H}^{1/2} \times \dot{H}^{1} \times L^{2} (\bbR^{4}) \cap \set{\hbox{\eqref{eq:main:smalldata} holds}} \to  C([0, T]; \dot{H}^{1/2} \times \dot{H}^{1} \times L^{2} (\bbR^{4})).
\end{align*}
The same statement holds on the interval $[-T, 0]$.
\item For each sign $\pm$, the solution $(\psi, A)$ exhibits \emph{modified scattering} as $t \to \pm \infty$, in the sense that there exist a solution $(\psi^{\pm \infty}, A^{\pm \infty}_{j})$ to the linear system
\begin{equation*}
\left\{
\begin{aligned}
	\Box A_{j}^{\pm \infty} =& 0, \\
	\alp^{\mu} \covD^{B}_{\mu} \psi^{\pm \infty} =& 0, \\
\end{aligned}
\right.
\end{equation*}
such that
\begin{equation*}
	\nrm{(\psi - \psi^{\pm \infty}_{j})(t)}_{\dot{H}^{1/2}(\bbR^{4})}
	+ \nrm{(A_{j} - A^{\pm \infty}_{j})[t]}_{\dot{H}^{1} \times L^{2}(\bbR^{4})} 
	 \to 0 \quad \hbox{ as } t \to \pm \infty.
\end{equation*}
Here, $B_{0} = 0$ and $B_{j}$ can be taken to be either the solution $A^{free}$ to $\Box A^{free} = 0$ with data $A^{free}_{j}[0] = A_{j}[0]$, or $B_{j} = A^{\pm \infty}_{j}$. 
\end{enumerate}
\end{theorem}

In the general case $d \geq 4$, the same theorem holds with the spaces $\dot{H}^{1/2}(\bbR^{4})$ and $\dot{H}^{1} \times L^{2}(\bbR^{4})$ are replaced by and $\dot{H}^{\frac{d-3}{2}}(\bbR^{d})$ and $\dot{H}^{\frac{d-2}{2}} \times \dot{H}^{\frac{d-4}{2}}(\bbR^{d})$, respectively. We refer to Remarks~\ref{rem:hi-d-1}, \ref{rem:hi-d-2}, \ref{rem:hi-d-3}, \ref{rem:hi-d-4} and \ref{rem:hi-d-5} for the necessary modifications in the argument when $d \geq 5$.

\begin{remark} 
Although the theorem is stated only for Coulomb initial data sets, it may be applied to an arbitrary smooth initial data set satisfying the smallness condition \eqref{eq:main:smalldata} by performing a gauge transform.
Indeed, given an arbitrary connection 1-form $A_{j}(0)$ on $\bbR^{d}$, there exists a gauge transformation $\chi \in \dot{H}^{\frac{d}{2}} \cap \dot{W}^{1,d} \cap BMO(\bbR^{d})$ such that the gauge transform $\tilde{A}(0) = A(0) - \ud \chi$ obeys the Coulomb gauge condition \eqref{eq:coulomb}. Moreover, the small data condition \eqref{eq:main:smalldata} is preserved up to multiplication by a universal constant for $\eps_{\ast}$ small enough. Such a gauge transformation can be found by solving the Poisson equation $\lap \chi = \mathrm{div}_{x} A_{j}(0)$.
\end{remark}

We remark that our method do not apply to the case of nonzero mass $m \neq 0$, although the observations made in this paper suggest that it would likely follow from a corresponding result for the massive Maxwell--Klein--Gordon equations; see `Parallelism with Maxwell--Klein--Gordon' in Section~\ref{subsec:main-idea}. The physically interesting case of $d = 3$, with or without mass, remains open.

\subsection{Previous work}
A brief survey of previous results on \eqref{eq:MD} and related equations is in order.
After early work on local well-posedness of \eqref{eq:MD} on $\bbR^{1+3}$ by Gross \cite{Gr} and Bournaveas \cite{Bou}, D'Ancona--Foschi--Selberg \cite{DFS2} established local well-posedness of \eqref{eq:MD} on $\bbR^{1+3}$ in the Lorenz gauge $\rd^{\mu} A_{\mu} = 0$ for data $\psi(0) \in H^{\eps}, A_{\mu}[0] \in H^{1/2 + \eps} \times H^{-1/2+\eps}$, which is almost optimal. In the course of their proof, a deep system null structure of \eqref{eq:MD} in the Lorenz gauge was uncovered. Although we work in a different gauge, our work develop upon many ideas from \cite{DFS2}. D'Ancona--Selberg \cite{DS} extended this approach to \eqref{eq:MD} on $\bbR^{1+2}$ and proved global well-posedness in the charge class. 

Regarding \eqref{eq:MD} on $\bbR^{1+3}$, we also mention \cite{CBC, Geo, FST, Psa} on global well-posedness for small, smooth and localized data, \cite{BMS, NM1} on the non-relativistic limit and \cite{NM2} on unconditional uniqueness at regularity $\psi \in C_{t} H^{1/2}, \ A_{x}[\cdot] \in C_{t} (H^{1} \times L^{2})$ in the Coulomb gauge.

A scalar counterpart of \eqref{eq:MD} is the \emph{Maxwell--Klein--Gordon equations} (MKG). An analogue of Theorem~\ref{thm:main} for (MKG) was proved by Krieger--Sterbenz--Tataru \cite{KST}. As we will explain in Section~\ref{subsec:main-idea}, \cite{KST} may be regarded as one of the direct predecessors of the present work. In the energy critical case $d =4$, global well-posedness of (MKG) for arbitrary finite energy data was recently established by the second author and Tataru \cite{OT2, OT3, OT1}, and independently by Krieger--L\"uhrmann \cite{KL}. In contrast, although \eqref{eq:MD} is also energy critical on $\bbR^{1+4}$, the energy for \eqref{eq:MD} is \emph{not} coercive; whether our Theorem~\ref{thm:main} may be extended to the large data case is therefore unclear.

Finally, we note that optimal small data global well-posedness was proved recently for the \emph{cubic Dirac equation} in $\bbR^{1+2}$ and $\bbR^{1+3}$ by Bejenaru--Herr \cite{BH1, BH2} (massive) and Bournaveas--Candy \cite{BC} (massless). This equation features a spinorial null structure similar to what is considered in this work.

%
%

\subsection{Main ideas} \label{subsec:main-idea}
We now provide an outline of the main ideas of this paper.

\subsubsection*{Null structure of \eqref{eq:MD} in Coulomb gauge} 
Null structure arises in equations from mathematical physics which exhibit covariance properties. It manifests through the vanishing of resonant components of the nonlinearities of such equations, and its presence is fundamental in obtaining well-posedness at critical regularity. An important component of our proof is uncovering the null structure of \eqref{eq:MD} in the Coulomb gauge (MD-CG), which involves both \emph{classical} (i.e., scalar) and \emph{spinorial} null forms.

A classical null form for scalar inputs refers to a linear combination of 
\begin{align*}
Q_{\al \beta}(\phi,\psi) &=\pt_{\al} \phi \pt_{\beta} \psi- \pt_{\beta}  \phi \pt_{\al} \psi, \qquad 0 \leq \al < \beta \leq d \\
Q_{0}(\phi,\psi) &=\pt_{\al} \phi \cdot \pt^{\al} \psi. 
\end{align*} 
These null forms initially arose in the study of global-in-time behavior of nonlinear wave equations with small, smooth and localized data \cite{Kla}. Remarkably, in the work \cite{KlainermanMacHedon} of Klainerman and Machedon, it was realized that the same structure is essential for establishing low regularity well-posedness as well. 

Among the first applications of this idea was the proof of global well-posedness at energy regularity of the Maxwell--Klein--Gordon equations on $\bbR^{1+3}$ \cite{klainerman1994}, which is a scalar analogue of \eqref{eq:MD}. A key observation in \cite{klainerman1994} was that quadratic nonlinearities of Maxwell--Klein--Gordon in the Coulomb gauge (MKG-CG) consist of null forms of the type $Q_{\al \beta}$. Furthermore, in the proof of essentially optimal local well-posedness of MKG-CG in $\bbR^{1+3}$ by Machedon and Sterbenz \cite{MachedonSterbenz}, a secondary trilinear null structure involving $Q_{0}$ was identified in MKG-CG after one iteration. As we explain further below, both of these structures arise in MD-CG, and play an important role in our problem.
 

Another type of null structure that arise in our work is spinorial null forms. 
These are bilinear forms with the symbol
$$ \Pi_{\pm}(\xi) \Pi_{\mp}(\eta), $$
which were first uncovered by D'Ancona, Foschi, Selberg for the Dirac--Klein--Gordon system in \cite{DFS1}. These authors further investigated the spinorial null forms in the study of the Maxwell--Dirac equation on $\bbR^{1+3}$ in the Lorenz gauge (in \cite{DFS2}; see also \cite{DS}). 
In the work of Bejenaru--Herr \cite{BH1, BH2} and Bournaveas--Candy \cite{BC}, these null forms were used in the proof of global well-posedness of the cubic Dirac equation for small critical data. 

A more detailed exposition of the null structure of MD-CG is given in Section~\ref{subsec:MD-nf} below. 
At this point we simply note that the null structure alone is \emph{insufficient} to close the proof of Theorem~\ref{thm:main} due to the presence of \emph{nonperturbative} nonlinearity, which is the next topic of discussion.

\subsubsection*{Presence of nonperturbative nonlinearity}
As in many previous works on low regularity well-posedness, we take a paradifferential approach in treating the nonlinear terms, exploiting the fact that the high-high to low interactions are weaker and that terms where the derivative falls on low frequencies are weaker as well. 

From this point of view, the worst interaction in the Dirac part of MD-CG occurs in the frequency-localized components
$$ \sum_{k} \alp^{\mu} P_{<k-C} A_{\mu} P_k \psi. $$
At the critical Sobolev \footnote{It is worthwhile to note that the issue of nonperturbative nonlinearity does \emph{not} arise if the initial data have $\ell^{1}$ summability in dyadic frequencies (i.e., they belong to the critical $\ell^{1}$-Besov space); see \cite{Tat, Ster}.} regularity this term is \emph{nonperturbative}, in the sense that even after utilizing all available null structure, it cannot be treated with multilinear estimates for the usual wave and Dirac equations. Instead, following the model set in the work of Rodnianski--Tao \cite{RT} and Krieger--Sterbenz--Tataru \cite{KST} on MKG-CG, this term must be viewed as a part of the underlying linear operator, and we must prove its solvability in appropriate function spaces. In fact, we directly establish solvability of the \emph{covariant Dirac operator} $\alp^{\mu} \covD_{\mu}$; see Proposition~\ref{prop:parasys-dirac} below. We note that this is the reason why MD-CG exhibits modified scattering, as opposed to scattering to a free Dirac field, in Theorem~\ref{thm:main}.

The presence of a nonperturbative term is characteristic of geometric wave equations with derivative nonlinearity, whose examples include wave maps, Maxwell--Klein--Gordon, Yang--Mills. This point distinguishes our problem from other nonlinear Dirac equations for which critical well-posedness was previously proved \cite{BH1, BH2, BC}.

\subsubsection*{Parallelism with Maxwell--Klein--Gordon}
In proving solvability of the covariant Dirac operator, as well as uncovering the null structure of MD-CG, we exploit a deep parallelism between the Maxwell--Dirac and the Maxwell--Klein--Gordon equations. On one hand, it provides a clear guiding principle that we hope would be useful in the future study of other Dirac equations. On the other hand, it allows us to borrow some key bounds directly from the Maxwell--Klein--Gordon case \cite{KST}, which simplifies our proof.

Historically, the Dirac equation emerged in an attempt to take the `square root' of the Klein--Gordon equation in order to obtain an equation that is first order in time. Thus `squaring' the Dirac component of the system leads to an equation that looks like the Klein--Gordon part of MKG. Unfortunately, as noted in \cite{DFS2}, this idea seems to be of limited use, since squaring the Dirac equation destroys most of the spinorial null structure. 

An alternative, more fruitful approach was put forth in \cite{DFS2}, which we follow in this paper. The idea is to first take the spatial Fourier transform and diagonalize the Dirac operator $\alp^{\mu} \rd_{\mu}$, decomposing the spinor as $ \psi=\psi_{+}+\psi_{-} $ where $ \psi_{\pm} $ obey appropriate half-wave equations. Splitting $\psi$ in the nonlinearity $\alp^{\mu} A_{\mu} \psi$ into $\psi_{+} + \psi_{-}$ as well, we can divide the equation into two parts: the \emph{scalar part}, which consists of contribution of $\psi_{\pm}$ without multiplication by $\alp^{\mu}$, and the remaining \emph{spinorial part}. A similar decomposition can be performed for the nonlinearity of the Maxwell equations.

One of the key observations of this paper is that the spinorial part enjoys a more favorable null structure compared to the scalar part. In particular, it is entirely perturbative, and furthermore the secondary null structure \`a la Machedon--Sterbenz \cite{MachedonSterbenz} is unnecessary. We refer to Remark~\ref{rem:spin-nf} for a more detailed explanation, after proper notation is set up.

For the remaining scalar part, we observe that its structure closely parallels that of MKG-CG; see Remark~\ref{rem:nonlin} for the detailed statement. As a consequence of this parallelism, we show that MD-CG exhibits nearly identical secondary null structure as MKG-CG uncovered in \cite{MachedonSterbenz}; see Section~\ref{subsec:tri-tri}. Furthermore, the microlocal parametrix construction in \cite{KST} can be borrowed as a black box to establish key estimates in the proof of solvability of the covariant Dirac equation, which handles the nonperturbative nonlinearity; see Section~\ref{sec:para}.

%
%

\subsubsection*{Parametrix construction for paradifferential covariant wave equation}
We end this subsection with a brief discussion on the parametrix construction in \cite{RT} and \cite{KST} for the paradifferential covariant wave equation in the context of MKG-CG. As explained above, it provides the basis for our proof of solvability of the covariant Dirac equation.

A key breakthrough of Rodnianski and Tao \cite{RT} was proving Strichartz estimates for the covariant wave equation by introducing a microlocal parametrix construction, motivated by the gauge covariance of $\Box_{A} = (\partial_{\alpha}+i A_{\alpha})(\partial^{\alpha}+i A^{\alpha})$ under gauge transforms $  \phi'=e^{i \Psi} \phi, \  A'=A- \nabla \Psi $, i.e., $e^{-i \Psi} \Box_{A'} (e^{i \Psi} \phi)= \Box_A \phi$.
The idea was to approximately conjugate (or renormalize) the modified d'Alembertian $ \Box+ 2i A_{<k-c} \cdot \nabla_x P_k $ to $ \Box $ by means of a carefully constructed pseudodifferential gauge transformation
$$ \Box_A^p \approx e^{i \Psi_{\pm}}(t,x,D) \Box e^{-i \Psi_{\pm}} (D,s,y), $$
where $e^{i \Psi_{\pm}}(t, x, D)$ and $e^{-i \Psi_{\pm}} (D,s,y)$ refer to the left- and right-quantization, respectively.
These Strichartz estimates were sufficient to prove global regularity of the Maxwell-Klein-Gordon equation at small critical Sobolev data in dimensions $ d \geq 6 $.

In \cite{KST}, Krieger, Sterbenz and Tataru further advanced the parametrix idea, showing that it interacts well with the function spaces needed to estimate the nonlinearity of the Maxwell--Klein--Gordon equation in dimension $ d=4 $. In particular, the resulting solution obeys similar bounds as ordinary waves, thus yielding control of an $X^{s, b}$-type norm, null-frame norms and square summed frequency and angular-localized Strichartz norms. A short technical summary of the construction in \cite{KST} can be found in Section~\ref{subsec:rn-box} below.

\section{Preliminaries}
\subsection{Notation and conventions} \label{subsec:notation}
We reserve the letter $C$ to denote a positive constant which may vary from expression to expression.
We write $A \aleq B$ and $A = B + O(1)$ for $A \leq CB$ and $\abs{A - B} \leq C$, respectively. We also introduce the shorthand $A \simeq B$ for $A \aleq B$ and $B \aleq A$. We use a subscript to indicate dependence of the implicit constant, e.g., $A \aleq_{\dlt} B$ if $C = C_{\dlt}$ depends on $\dlt$.

Given $\calC, \calC' \subseteq \bbR^{d}$, we use the notation $- \calC = \set{ - \xi : \xi \in \calC}$ and $\calC + \calC' = \set{\xi + \eta : \xi \in \calC, \ \eta \in \calC'}$. Moreover, we define the \emph{angular distance} between $\calC$ and $\calC'$ as
\begin{equation*}
	\abs{\angle(\calC, \calC')} := \inf \set{\abs{\angle(\xi, \eta)} : \xi \in \calC, \ \eta \in \calC'}.
\end{equation*}

We denote by $\calL$ a translation-invariant bilinear operator on $\bbR^{d}$ whose kernel has bounded mass, i.e.,
\begin{equation*} \label{transinvop}
	\calL(f, g) (x) = \int K(x-y_{1}, x - y_{2})  f(y_{1}) g(y_{2}) \, \ud y_{1} \ud y_{2}
\end{equation*}
where $K$ is a measure on $\bbR^{d} \times \bbR^{d}$ with bounded mass. As a consequence, $\calL(f, g)$ obeys a H\"older-type inequality
\begin{equation} \label{eq:L-atom}
	\nrm{\calL(f, g)}_{L^{p}} \aleq \nrm{f}_{L^{q_{1}}} \nrm{g}_{L^{q_{2}}}
\end{equation}
for any exponents $1 \leq p, q_{1}, q_{2} \leq \infty$ such that $p^{-1} = q_{1}^{-1} + q_{2}^{-2}$. 

\subsection{Frequency projections}
Let $ \chi $ be a smooth non-negative bump function supported on $ [2^{-2},2^2] $ which satisfies the partition of unity property
$$ \sum_{k \in \mathbb{Z}} \chi \lpr \frac{\vm{\xi}}{2^k} \rpr=1 $$
for $ \xi \neq 0 $. We define the Littlewood-Paley operators $ P_k $ by 
$$ \widehat{P_k f} (\xi)=\chi \lpr \frac{\vm{\xi}}{2^k} \rpr \hat{f} (\xi). $$
The modulation operators $ Q_j, Q_j^{\pm} $ are defined by 
$$  \mathcal{F} (Q_j f) (\tau,\xi)= \chi \lpr \frac{\vm{\vm{\tau}-\vm{\xi}} }{2^j} \rpr  \mathcal{F}f (\tau,\xi), \quad  \mathcal{F} (Q_j^{\pm}f) (\tau,\xi)= \chi \lpr \frac{\vm{\pm \tau-\vm{\xi}} }{2^j} \rpr  \mathcal{F}f (\tau,\xi). $$ 
for $ j \in \mathbb{Z} $, where $ \mathcal{F} $ denotes the space-time Fourier transform.

Given  $ \ell \leq 0 $ we consider a collection of directions $ \omega $ on the unit sphere which is maximally $ 2^\ell $-separated. To each $ \omega $ we associate a smooth cutoff function $ m_{\omega} $ supported on a cap $ \subset \bbS^{d-1} $ of radius $ \simeq 2^\ell $ around $ \omega $, with the property that $ \sum_{\omega}  m_{\omega}=1 $. We define $ P_\ell^{\omega} $ to be the spatial Fourier multiplier with symbol $ m_{\omega}(\xi/\vm{\xi}) $. In a similar vein, we consider rectangular boxes $ \mathcal{C}_{k'}(\ell') $ of dimensions $ 2^{k'} \times (2^{k'+\ell'})^{d-1} $, where the $ 2^{k'} $ side lies in the radial direction, which cover $\bbR^{d}$ and have finite overlap with each other. We then define $P_{\mathcal{C}_{k'}(\ell')}$ to be the associated smooth spatial frequency localization to $\calC_{k'}(\ell')$. For convenience, we choose the blocks so that $P_{k} P_{\ell}^{\omg} = P_{\calC_{k}(\ell)}$.



We will often abbreviate $f_{k} = P_{k} f$. We will sometimes use the operators  $ \tilde{P}_k,  \tilde{Q}_{j/<j},  \tilde{P}^{\omg}_{\ell} $ with symbols given by bump functions which equal $ 1 $ on the support of the multipliers $ P_k, Q_{j/<j} $ and $ P^{\omg}_{\ell} $ respectively and which are adapted to an enlargement of these supports. Thus,
$$ \tilde{P}_k P_k=P_k, \qquad \tilde{Q}_{j/<j} Q_{j/<j}=Q_{j/<j}, \qquad \tilde{P}^{\omg}_{\ell} P^{\omg}_{\ell}=P^{\omg}_{\ell}. $$ 

Given a sign $s \in \set{+, -}$, define $T_{s}$ as
\begin{equation*}
	\widetilde{T_{+} f} (\tau, \xi) = 1_{\set{\tau > 0}} \tilde{f}(\tau, \xi), \quad
	\widetilde{T_{-} f} (\tau, \xi) = 1_{\set{\tau \leq 0}} \tilde{f}(\tau, \xi).
\end{equation*}
For all $j$, we have $Q_{j / <j} T_{s} = Q_{j / <j}^{s} T_{s}$. Moreover, for $j \leq k - 3$, we have 
\begin{equation*}
	P_{k} Q_{j / <j} T_{s} = P_{k} Q_{j / <j}^{s}, \quad
	P_{k} Q_{j / <j} = \sum_{s \in \set{+, -}} P_{k} Q_{j / <j}^{s}.
\end{equation*}

We now discuss boundedness of the frequency projections. Following the terminology in \cite{Tao2}, we say that a spacetime Fourier multiplier is \emph{disposable} if its (distributional) convolution kernel is a measure with mass $O(1)$; clearly, a disposable multiplier is bounded on any translation-invariant Banach spaces (e.g., $L^{q} L^{r}$). 

For any $k \in \bbZ$, $\ell \leq 0$, angular sector $\omg$ of size $\simeq 2^{\ell}$ and a rectangular box $\calC$, the following operators are disposable:
\begin{equation*}
	P_{k}, \quad P_{k} P_{\ell}^{\omg}, \quad P_{\calC}.
\end{equation*}
For any $Q_{j / <j}^{\Box} \in \set{Q^{s}_{j}, Q^{s}_{<j}, Q_{j}, Q_{<j}}$ with $j \in \bbZ$, the operator $P_{k} Q_{j / <j}^{\Box}$ is disposable if $j \geq k - C$ \cite[Lemma~3]{Tao2}. In general, we have
\begin{equation} \label{eq:q-disp}
	\nrm{P_{k} Q_{j / <j}^{\Box} f}_{L^{q} L^{r}} \aleq 2^{d (k - j)_{+}}\nrm{f}_{L^{q} L^{r}} \qquad (1 \leq q, r \leq \infty)
\end{equation}
In the case $r = 2$, we have an unconditional estimate \cite[Lemma~4]{Tao2}:
\begin{equation} \label{eq:q-LqL2}
	\nrm{P_{k} Q_{j / <j}^{\Box} f}_{L^{q} L^{2}} \aleq \nrm{f}_{L^{q} L^{2}} \qquad (1 \leq q \leq \infty).
\end{equation}
For $j \geq k + 2\ell - C$, the operator $P_{k} P_{\ell}^{\omg} Q_{j / <j}^{\Box}$ is disposable \cite[Lemma~6]{Tao2}.

\subsection{Dyadic function spaces} \label{subsec:dyadic-fs}
Many function spaces we use will be defined \emph{dyadically}, i.e., the norm\footnote{Throughout this paper, we abuse the terminology and refer to semi-norms as simply norms.}  of $f$ will be some summation of \emph{dyadic norms} of $P_{k} f = f_{k}$. Formally, given a sequence of norms $(X_{k})_{k \in \bbZ}$, $1 \leq p \leq \infty$ and $\sgm \in \bbR$, we denote by $\ell^{p} X^{\sgm}$ the norm
\begin{equation*}
\nrm{f}_{\ell^{p} X^{\sgm}} = \bb(\sum_{k} (2^{\sgm k}\nrm{P_{k} f}_{X_{k}})^p \bb)^{1/p},
\end{equation*}
with the usual modification when $p = \infty$. 

An important class of examples is the $L^{2}$-Sobolev spaces on $\bbR^{d}$, i.e., 
\begin{equation*}
	\nrm{f}_{\dot{H}^{\sgm}} = \bb( \sum_{k} (2^{\sgm k} \nrm{P_{k} f}_{L^{2}} )^{2} \bb)^{1/2}.
\end{equation*}
Other examples include $X^{s, b}_{1}$ and $X^{s, b}_{\infty}$ spaces on $\bbR^{1+d}$, which are logarithmic of refinements of the usual $X^{s, b}$ space. Their dyadic norms (independent of $k$) are
\begin{equation} \label{xnorms-1}
\nrm{f}_{X^{b}_{1}} = \sum_{j \in \bbZ} 2^{b j} \nrm{Q_{j} F}_{L^{2} L^{2}}, \qquad
\nrm{f}_{X^{b}_{\infty}} = \sum_{j \in \bbZ} 2^{b j} \nrm{Q_{j} F}_{L^{2} L^{2}}.
\end{equation}
Then we define the $X^{s, b}_{1} = \ell^{2} (X^{b}_{1})^{s}$ and $X^{s, b}_{\infty} = \ell^{2} (X^{b}_{\infty})^{s}$, i.e.,
\begin{equation} \label{xnorms-2}
	\nrm{f}_{X^{s, b}_{1}} = \bb( \sum_{k} (2^{s k} \nrm{f}_{X^{b}_{1}})^{2} \bb)^{1/2}, \qquad
	\nrm{f}_{X^{s, b}_{\infty}} = \bb( \sum_{k} (2^{s k} \nrm{f}_{X^{b}_{\infty}})^{2} \bb)^{1/2}.
\end{equation}
Our main function spaces (cf. Section~\ref{sec:fs}) will be dyadically defined as well.

\subsection{Frequency envelopes} \label{subsec:freqenv}
We borrow from \cite{Tao2} the notion of \emph{frequency envelopes}, which is a convenient means to keep track of dyadic frequency profiles. Given $\dlt > 0$, we say that a sequence $c = (c_{k})_{k \in \bbZ}$ of positive numbers is a \emph{$\dlt$-admissible frequency envelope} if there exists $C_{c} > 0$ such that for every $k, k' \in \bbZ$, we have
\begin{equation*}
	\abs{c_{k} / c_{k'}} \leq C_{c} \, 2^{\dlt \abs{k - k'}}.
\end{equation*}
Given a sequence $(X_{k})_{k \in \bbZ}$ of dyadic norms, we define the $X_{c}$-norm as
\begin{equation*}
	\nrm{f}_{X_{c}} = \sup_{k \in \bbZ} c_{k}^{-1}\nrm{P_{k} f}_{X_{k}}.
\end{equation*}

Dyadically defined norms (cf. Section~\ref{subsec:dyadic-fs}) are controlled in terms of $c$ and $\nrm{f}_{X_{c}}$ in the obvious manner:
\begin{equation*}
	\nrm{f}_{\ell^{p} X^{\sgm}} \leq \bb( \sum_{k} (2^{\sgm k} c_{k})^{p} \bb)^{1/p} \nrm{f}_{X_{c}}.
\end{equation*}
In the converse direction, we say that $c$ is a frequency envelope for $\nrm{f}_{\ell^{p} X^{0}}$ if
$$ \vn{f}_{\ell^{p} X^{0}} \simeq \bb( \sum_{k} c_{k}^{p} \bb)^{1/p}, \qquad \vn{P_k f}_{X_k} \leq c_k. $$ 
Given any $f \in \ell^{p} X^{0}$, we can construct a $\dlt$-admissible frequency envelope $c$ for $\nrm{f}_{\ell^{p} X^{0}}$ by defining
\begin{equation} \label{eq:fe-construct-0}
c_k=\sum_{k'} 2^{-\delta \vm{k-k'}} \vn{P_{k'} f}_{X_{k'}}. 
\end{equation}
By Young's inequality, this frequency envelope inherits any additional $\ell^{p'} X^{\sgm}$ regularity of $f$ for $1 \leq p' \leq \infty$ and $\sgm \in (-\dlt, \dlt)$, i.e.,
\begin{equation*}
	\nrm{2^{\sgm k} c_{k}}_{\ell^{p'}} \aleq \nrm{f}_{\ell^{p'} X^{\sgm}}.
\end{equation*}




We conclude this subsection with a discussion on simple operations on frequency envelopes.
Given a $\dlt$-admissible frequency envelope $c \in \ell^{p}$ $(1 \leq p \leq \infty)$, we may construct a new frequency envelope $\tilde{c}$ by taking $\tilde{c}_{k} = (\sum_{k' < k} c_{k'}^{p})^{1/p}$. For any $\ell \geq 0$, we see (by shifting indices) that
\begin{equation*}
	\abs{\tilde{c}_{k+\ell} / \tilde{c}_{k}}
	 \leq C_{c} 2^{\dlt \ell}, \qquad
	\abs{\tilde{c}_{k-\ell} / \tilde{c}_{k}}
	 \leq C_{c} 2^{\dlt \ell}.
\end{equation*}
In other words, $\tilde{c}$ is also $\dlt$-admissible. 

For $\dlt'$- and $\dlt$-admissible frequency envelopes $b$ and $c$, we denote by $bc = (b_{k} c_{k})_{k \in \bbZ}$ the product frequency envelope, which is clearly $(\dlt + \dlt')$-admissible. 

By Cauchy-Schwarz inequality, note that the frequency envelopes $(\sum_{k' < k} b_{k'} c_{k'})_{k \in \bbZ}$ is dominated by $((\sum_{k' < k} b_{k'}^{2})^{1/2} (\sum_{k' < k} c_{k'}^{2})^{1/2})_{k \in \bbZ}$, i.e.,
\begin{equation*}
	\sum_{k' < k} b_{k'} c_{k'} \leq \bb( \sum_{k' < k} b_{k'}^{2} \bb)^{1/2} \bb( \sum_{k' < k} c_{k'}^{2} \bb)^{1/2}.
\end{equation*}
In particular, if $b, c \in \ell^{2}$, then $bc \in \ell^{1}$.

\subsection{Small parameters}
We use two global small parameters $0 < \dlt_{1} < \dlt_{0}$, which are fixed from right to left. The number $\dlt_{0}$ denotes the exponent in the dyadic estimates proved in Sections~\ref{sec:bi} and \ref{sec:tri}. The number $\dlt_{1}$ is the admissibility constant for frequency envelopes, which is chosen to be much smaller than $\dlt_{0}$ (say $< \frac{1}{1000} \dlt_{0}$).

\section{Function spaces} \label{sec:fs}
In this section, we introduce the main function spaces we use to prove Theorem~\ref{thm:main}.
We build on the spaces defined in \cite{KST}, which in turn build on the prior work \cite{Tat, Tao2} on wave maps.
\subsection{Spaces $S_{k}$, $N_{k}$ and $Y_{k}$}
Let $d \geq 4$ and $\sgm \in \bbR$. We start with the main function spaces for the Maxwell component of \eqref{eq:MD}, which is essentially borrowed from \cite{KST}. In all cases, the scale-critical exponent is $\sgm = \frac{d-2}{2}$.

To define the nonlinearity space $N^{\sgm-1}$, we first let
\begin{equation*}
	N_{k} = L^{1} L^{2} + X^{0,-1/2}_{1},
\end{equation*}
and define $N^{\sgm-1}$ dyadically as
$$ \vn{F}_{N^{\sgm-1}}^2 \defeq \sum_{k} (2^{(\sgm-1) k} \vn{P_k F}_{N_k})^2. $$

Next, we introduce the solution space $S^{\sgm}$. We first define
\begin{align*}
	\nrm{f}_{S^{\mathrm{str}}_{k}}
	= & \sup_{(q, r) : \frac{1}{q} + \frac{d-1}{2 r} \leq \frac{d-1}{4}} 2^{(2 - \frac{1}{q} - \frac{d}{r}) k} \nrm{f}_{L^{q} L^{r}} \\
	\nrm{f}_{S^{\mathrm{box}}_{k}}^{2}
	= &\sup_{\ell \leq 0} \sum_{\omg} \nrm{P_{\ell}^{\omg} Q_{< k + 2 \ell} f}_{S^{\mathrm{box}}_{k}(\ell)}^{2}\\
	\nrm{f}_{S^{\mathrm{box}}_{k}(\ell)}^{2} 
	= & \nrm{f}_{S^{\mathrm{str}}_{k}}^{2} + \sup_{\substack{k' \leq k, \ \ell' \leq 0 \\ k + 2 \ell \leq k' + \ell' \leq k+ \ell}} \sum_{\calC_{k'}(\ell')} 2^{-(d-2) k'} 2^{-(d-3) \ell'} 2^{- k} \nrm{P_{\calC_{k'}(\ell')} f}_{L^{2} L^{\infty}}^{2}.
\end{align*}
When $d = 4$, let 
\begin{equation*}
	S_{k} = S^{\mathrm{str}}_{k} \cap X^{0, 1/2}_{\infty} \cap S^{\mathrm{ang}}_{k} \cap S^{\mathrm{box}}_{k}.
\end{equation*}
where $ S^{\mathrm{ang}}_{k}$ is as in \cite[Eqs.~(6)--(8)]{KST}:
\begin{align*}
\vn{f}_{S_k^{ang}}^2=& \sup_{l<0} \sum_{\omega} \vn{P_l^{\omega} Q_{k+2l} f}_{S_k^{\omega}(l)}^2,  \\
\vn{f}_{S_k^{\omega}(l)}^2= & \vn{f}_{S_k^{str}}^2+ 2^{-2k} \vn{f}_{NE}^2+2^{-3k} \sum_{\pm} \vn{T_{\pm} f}_{PW_{\omega}^{\mp}(l)}^2+ \\
 & + \sup_{\substack{k' \leq k, \ \ell' \leq 0 \\ k + 2 \ell \leq k' + \ell' \leq k+ \ell}}  \sum_{\calC_{k'}(\ell')} \bb( \nrm{P_{\calC_{k'}(\ell')} f}_{S_k^{str}}^{2}  + 2^{-2k}  \nrm{P_{\calC_{k'}(\ell')} f}_{NE}^{2} \\
 & + 2^{-2k'-k} \nrm{P_{\calC_{k'}(\ell')} f}_{L^{2} L^{\infty}}^{2}+ 2^{-3(k'+l')} \sum_{\pm} \nrm{T_{\pm} P_{\calC_{k'}(\ell')} f}_{PW_{\omega}^{\mp}(l) }^{2}  \bb)
\end{align*}
Here, the $NE$ and $PW_{\omg}^{\mp}(\ell)$ are the \emph{null frame spaces} \cite{Tat, Tao2} given by
\begin{align*}
	\nrm{f}_{PW^{\mp}_{\omg}(\ell)}
	= &\inf_{f = \int f^{\omg'}} \int_{\abs{\omg - \omg'} \leq 2^{\ell}} \nrm{f^{\omg'}}_{L^{2}_{\pm \omg'} L^{\infty}_{(\pm \omg')^{\perp}}} \, \ud \omg', \\
	\nrm{f}_{NE}
	= & \sup_{\omg} \nrm{\! \not \! \nb_{\omg} \phi}_{L^{\infty}_{\omg} L^{2}_{\omg^{\perp}}},
\end{align*}
where the $L^{q}_{\omg}$ norm is with respect to the variable $\ell^{\pm} = t \pm \omg \cdot x$, the $L^{r}_{\omg^{\perp}}$ norm is defined on each $\set{\ell^{\pm}_{\omg} = const}$, and $\displaystyle{\! \not \!  \nb_{\omg}}$ denotes derivatives tangent to $\set{\ell^{\pm}_{\omg} = const}$.
The null frame spaces allow one to exploit transversality in frequency space, and play an important role in the proof of the multilinear null form estimate; see \cite[Eqs.~(136)--(138)]{KST} and Proposition~\ref{prop:tri-MKG} below. These norms are used as follows. If we have the angular separation condition
\be \label{angle:caps}
\angle(\pm \calC_{k'}(\ell'), \calC'_{k'}(\ell') ) \simeq 2^{\tilde{\ell}}  \gg 2^{\ell'+k'-\min(k_1,k_2)} \ee
then by using the Minkowski and H\" older $ L^{\infty}_{\omg'} L^{2}_{\omg'^{\perp}} \times L^{2}_{\omg'} L^{\infty}_{\omg'^{\perp}} \to L^2_{t,x}  $ inequalities, one obtains
\be \label{bilL2est}
\vn{  P_{\calC_{k'}(\ell')} \phi_{k_1} \cdot P_{\calC'_{k'}(\ell')}  \psi_{k_2} }_{L^{2} L^{2}} \lesssim 2^{- \tilde{\ell}} 2^{-k_2}\vn{ P_{\calC_{k'}(\ell')} \phi_{k_1}}_{PW_{\omega}^{\mp}(l) } \vn{P_{\calC'_{k'}(\ell')}  \psi_{k_2}}_{NE}
\ee
and the latter two norms are controlled through the $ S_k^{\omega}(l) $ norms. This strategy is motivated by the inequality \eqref{L2:freesol} for free solutions as well as by the wave packet analysis in \ref{spaces:discussion}.

\begin{remark} \label{sboxrk} We note here that the $ S_k^{box} $ component was not used in \cite{KST}, but the factor of $ 2^{-\ell'/2} $ in this norm compared to $ S_k^{ang} $ was actually obtained there in Subsection 11.3 for the main parametrix estimate, which we review below (Theorem \ref{covariantparametrix}). As opposed to the Maxwell-Klein-Gordon case, it turns out that this angular gain is essential here in order to estimate the nonlinear terms of the Maxwell-Dirac system.
\end{remark}

In the higher dimensional case $d \geq 5$, we simply define
\begin{equation*}
	S_{k} = S^{\mathrm{str}}_{k} \cap X^{0, 1/2}_{\infty} \cap S^{\mathrm{box}}_{k}.
\end{equation*}
\begin{remark} \label{rem:fs-hi-d}
In the case $d \geq 5$, note the omission of the component $S^{\mathrm{ang}}_{k}$, which contain the null frame spaces.
In $d= 4$, $S^{\mathrm{ang}}_{k}$ is necessary for the proof of a multilinear null form estimate (Proposition~\ref{prop:tri-MKG}), but in higher dimensions this estimate is unnecessary; see Remark~\ref{rem:hi-d-4} below.
\end{remark}

The spaces follow the following inclusions:
\be X^{0, \frac{1}{2}}_{1} \subset S_k \subset N_k^{*}.   \label{spaceembed} \ee

The main iteration space $S^{\sgm}$ for the hyperbolic part $A_{x}$ of the Maxwell equations consists of an $S^{\sgm}$ norm and an extra term $\nrm{\Box A}_{L^{2} \dot{H}^{\sgm - \frac{3}{2}}}$ for high modulation:
\begin{equation*}
	\nrm{A}_{S^{\sgm}}^{2} = \sum_{k} \bb( 2^{2 \sgm k} \nrm{P_{k} A}_{S_{k}}^{2} + 2^{(2\sgm-3)k} \nrm{\Box P_{k} A}_{L^{2} L^{2}}^{2} \bb).
\end{equation*}
In \cite{KST}, it was proved that
\begin{equation} \label{eq:en-box}
	\nrm{A}_{S^{\sgm}} \aleq \nrm{(A, \rd_{t} A)(0)}_{\dot{H}^{\sgm} \times \dot{H}^{\sgm-1}} + \nrm{\Box A}_{N^{\sgm-1}}.
\end{equation}

The dyadic function space $Y_{k}$ for the elliptic variable $A_{0}$ is defined as
\begin{equation*}
	\nrm{B}_{Y_{k}}^{2} = 2^{k} \nrm{B}_{L^{2} L^{2}}^{2} + 2^{-k} \nrm{\rd_{t} B}_{L^{2} L^{2}}^{2}.
\end{equation*}
Then we define
\begin{equation*}
	\nrm{B}_{Y^{\sgm}}^{2} = \sum_{k} 2^{2 \sgm k} \nrm{P_{k} B}_{Y_{k}}^{2}.
\end{equation*}

\subsection{Spaces $S_{\pm, k}$, $N_{\pm, k}$ and $\tilde{Z}_{\pm, k}$}
Let $d \geq 4$ and $\sgm \in \bbR$; in what follows, the scale-critical exponent is $\sgm = \frac{d-3}{2}$. For the Dirac equation, we need to define analogous spaces adapted to each characteristic cone $\set{\tau = \pm \abs{\xi}}$. Let 
\begin{align*} 
	N_{k}^{+} = & L^{1} L^{2} + X_{+, 1}^{0, -1/2}, 
	& N_{k}^{-} = & L^{1} L^{2} + X_{-, 1}^{0, -1/2}, \\
	S_{k}^{+} = & S^{\mathrm{str}}_{k} \cap X^{0, 1/2}_{+, \infty}\cap Q^{+}_{<k-3} S_{k} , 
	&S_{ k}^{-} = & S^{\mathrm{str}}_{k} \cap X^{0, 1/2}_{-, \infty}\cap Q^{-}_{<k-3} S_{k}.
\end{align*}
The $ X_{\pm, 1}^{0, -1/2} $ and $ X^{0, 1/2}_{\pm, \infty}$  norms are defined by \eqref{xnorms-1}--\eqref{xnorms-2}, with $ Q_j $ replaced by $ Q^{\pm}_j $.
Note that $S_{k} = S_{k}^{+} + S_{k}^{-}$, $N_{k} = N_{k}^{+} \cap N_{k}^{-} $ and 
\begin{equation*}
	Q_{<k-1}^{s} N_{k}^s \subseteq N_{k}, \qquad Q_{<k + O(1)}^{s} S_{k} \subseteq S_{k}^s.
\end{equation*}
We define
\begin{align*}
	\nrm{\psi}_{S^{\sgm}_{\pm}}^{2} =& \sum_{k} \bb(2^{2 \sgm k} \nrm{P_{k} \psi}_{S_{k}^{\pm}}^{2} + 2^{(2 \sgm - 2) k} \nrm{(i \partial_t \pm \vm{D}) P_{k} \psi}_{L^{2} L^{2}}^{2} \bb), \\
	\vn{F}_{N_{\pm}^\sgm}^2 =& \sum_{k} 2^{2 \sgm k} \vn{P_k F}_{N_k^{\pm}}^2.
\end{align*}
An analogue of \eqref{eq:en-box} holds for $S^{\sgm}_{\pm}$, $N^{\sgm}_{\pm}$ and $i \rd_{t} \pm \abs{D}$. For our purpose, however, we need an extension which is valid for a paradifferential covariant half-wave operator; see Theorem~\ref{thm:paradiff} below. 

%
%
%
%
%
%
%
%
%


The $S^{\sgm}_{\pm}$ norm must be augmented with an $L^{1} L^{\infty}$ control for high modulations. To this end, consider the dyadic norm
\begin{equation} 
	\nrm{\psi}_{\tilde{Z}^{\pm}_{k}}
	= 2^{-2k} \nrm{(i \rd_{t} \pm \abs{D}) \psi}_{L^{1} L^{\infty}} \ 
\end{equation}
and the corresponding $\ell^{2}$-summed norm, given by
\begin{equation*}
	\nrm{\psi}_{\tilde{Z}_{\pm}^{\sgm}}^{2} = \sum_{k} 2^{2\sgm k} \nrm{\psi_{k}}_{\tilde{Z}_{\pm, k}}^{2}.
\end{equation*}
Define also
\begin{equation*}
	\nrm{F}_{G_{k}} = 2^{-2k} \nrm{F}_{L^{1} L^{\infty}}, \quad
	\nrm{F}_{G^{\sgm}}^{2} = \sum_{k} 2^{2 \sgm k} \nrm{P_{k} F}_{G_{k}}.
\end{equation*}
For $\psi$ localized at frequency $\set{\abs{\xi} \simeq 2^{k}}$ and $s \in  \set{+, -}$, we have
\begin{equation*}
	\nrm{\psi}_{\tilde{Z}^{\pm}_{k}} \aleq \nrm{(i \rd_{t} \pm \abs{D}) \psi}_{G_{k}}.
\end{equation*}
The main iteration space $\tilde{S}^{\sgm}_{s}$ for the $s$-components of $\psi$ $(s \in \set{+, -})$ is defined as 
\begin{equation} 
	\nrm{\psi}_{\tilde{S}^{\sgm}_{s}}^{2} = \nrm{\psi}_{S^{\sgm}_{s}}^{2} + \nrm{\psi}_{\tilde{Z}_{s}^{\sgm}}^{2}. 
\end{equation}

\begin{remark}
Notice the following simple inequalities:
\be \vn{P_k F}_{N_k^s} \lesssim \vn{P_k F}_{N_s}. \ee 
If the functions $ f_{k'} $ have Fourier support in the regions $ \{ \vm{\xi} \simeq 2^{k'} \} $ and 
$ f=\sum_{k'} f_{k'} $
then
\be \vn{P_k f}_{N_s^0} \lesssim \sum_{k'=k+O(1)} \vn{f_{k'}}_{N_{k'}^s}  \ee
\be \vn{P_k f}_{S_s^\sgm} \lesssim 2^{\sgm k} \sum_{k'=k+O(1)} \lpr \vn{f_{k'}}_{S_{k'}^s}+ \vn{(i \partial_t +s \vm{D}) f_{k'}}_{L^2  \dot{H}^{-1/2}}  \rpr .\ee
\end{remark}

\begin{lemma} \label{lemmadivsymb}
Suppose $ f $ is localized at frequency $ \{ \vm{\xi} \simeq 2^k \} $ and $ s \in \{ +,- \} $.
\begin{enumerate}[leftmargin=*]
\item If $ f $ is localized at $ Q^{s} $-modulation $ \lesssim 2^k $ then
\be \label{simpleembedding} \vn{f}_{L^{2} L^{2}} \lesssim 2^{\frac{k}{2}} \vn{f}_{N_{k}^{s}}. \ee
\item If $ f $ is localized at $ Q^{s} $-modulation $ \gtrsim 2^k $ and $ u $ is defined by
\be \label{operdiv} \mathcal{F} u (\tau,\xi) = \frac{1}{\tau- s\vm{\xi}} \mathcal{F} f (\tau,\xi) \ee
then 
\begin{align}
 \vn{u}_{S_k^s} & \lesssim \vn{f}_{\dot{H}^{-1/2}}  \label{embdivsymb} \\
 \vn{u}_{L^{\infty}L^2} & \lesssim  \vn{f}_{N_k^s} \label{embeasy} 
\end{align}  
\end{enumerate}
\end{lemma}
\begin{proof}
In view of the low modulation, \eqref{simpleembedding} follows by duality from the embedding \eqref{spaceembed}. Similarly, \eqref{embdivsymb} follows from the inequalities
$$ \vn{u}_{S_k^s} \lesssim \vn{u}_{X_{s,1}^{0,\frac{1}{2}}} \lesssim \vn{f}_{X^{0,-\frac{1}{2}}_{s,1}} \lesssim \vn{f}_{\dot{H}^{-1/2}}. $$
Now we prove \eqref{embeasy}.  Since $ N_{k}^{s} $ is an atomic space we consider two cases. First, if $ f $ is an $ X^{0,-1/2}_{s,1} $-atom then we write
$$ u(t)=\int e^{i t \rho} e^{i s t \vm{D}} \phi_{\rho} \, \ud \rho $$
where $ \phi_{\rho} $ satisfies  
$$  \widehat{ \phi_{\rho}}  (\xi)=  \mathcal{F} u (\rho+s \vm{\xi},\xi), \qquad \int \vn{ \phi_{\rho}}_{L^2} \, \ud \rho \lesssim \vn{f}_{X^{0,-1/2}_{s,1}}. $$
If $ f $ is an $ L^1L^2 $-atom we write $ u $ as a superposition of truncated homogenous waves 
$$ u(t)=\int e^{i (t-t') s \vm{D} } f(t) 1_{t>t'} \dd t'. $$
In both cases \eqref{embeasy} follows from the basic inequality for free waves
\begin{equation*}
\vn{e^{i st \vm{D}} \phi}_{L^{\infty} L^2} \lesssim \vn{\phi}_{L^2}. \qedhere 
\end{equation*}
\end{proof}

\subsection{Time interval localized norms}
In a few places in the paper (in particular, Section~\ref{sec:iter}), we need to consider time interval localization of the function spaces. Given an interval $I \subseteq \bbR$ and a distribution $f$ on $I \times \bbR^{d}$, we define\footnote{We use the convention $\inf \emptyset = \infty$.}
\begin{equation*}
	\nrm{f}_{X[I]} = \inf \set{\nrm{\tilde{f}}_{X} : \tilde{f} \in X, \ \tilde{f} = f \hbox{ on } I},
\end{equation*}
where $X$ may denote any norm, e.g., $S^{r}$, $N^{r}$, $\tilde{S}_{s}^{r}$ or $N_{s}^{r}$. 

Let $f \in N^{r}[I]$. Up to equivalent norms, we may take $\tilde{f}$ above in $N^{r}$ to be simply the extension by zero outside $I$. Moreover, for $f \in N^{r}$, we have 
\begin{equation} \label{eq:N-fungibility}
\lim_{T \to 0} \nrm{f}_{N^{r}[0, T]} = 0, \qquad \lim_{T \to \infty} \nrm{f}_{N^{r}[T, \infty)} = 0.
\end{equation}
Similar properties holds for $N^{r}_{s}$. These statements are justified by the following lemma, whose proof can be read off from \cite[Proposition~3.3]{OT2}.
\begin{lemma} \label{lem:int-loc}
Let $f \in N^{r}$ ($r \in \bbR$). For any interval $I \subseteq \bbR$, denote by $1_{I}(t)$ its characteristic function. Then we have $\nrm{1_{I}(t) f}_{N^{r}} \aleq \nrm{f}_{N^{r}}$. Moreover, we have $\lim_{T \to 0+} \nrm{1_{[0, T]}(t) f}_{N^{r}} = 0$ and $\lim_{T \to \infty} \nrm{1_{[T, \infty]}(t) f}_{N^{r}} = 0$. 

Same statements hold with $N^{r}$ replaced by $N_{s}^{r}$ $(s \in \set{+, -}, \ r \in \bbR)$.
\end{lemma}

\subsection{Spaces $Z_{k}$ and $Z_{ell, k}$}
We now introduce $L^{1} L^{\infty}$-type auxiliary norms for the Maxwell components $A_{0}, A_{x}$ of \eqref{eq:MD}, which will be used in our proof of the trilinear estimate (Proposition~\ref{prop:trilinear}) in Section~\ref{sec:tri}. Let $C_{1} > 0$ be a constant to be fixed later in \eqref{eq:C012}. Let
\begin{align}
	\nrm{A}_{Z_{k}}
	:=  & \sup_{\ell < \frac{1}{2} C_{1}} \, 2^{-\frac{d-2}{2} k} 2^{\frac{1}{2} \ell} \bb( \sum_{\omg} \nrm{P^{\omg}_{\ell_{-}} Q_{k + 2 \ell}A}_{L^{1} L^{\infty}}^{2} \bb)^{\frac{1}{2}}, 	\label{eq:Z-def} \\
	\nrm{B}_{Z_{ell, k}}
	:=  & \sup_{\ell < \frac{1}{2} C_{1}} 2^{-\frac{d-2}{2} k} 2^{- \frac{1}{2} \ell} \bb( \sum_{\omg} \nrm{P^{\omg}_{\ell_{-}} Q_{k + 2 \ell}B}_{L^{1} L^{\infty}}^{2} \bb)^{\frac{1}{2}},	\label{eq:Z-ell-def}
\end{align}
where $\ell$ runs over the half integers $\frac{1}{2} \bbZ$. Note that $Z_{k}$ and $Z_{ell, k}$ scale like $L^{\infty} L^{2}$. For $\sgm \in \bbR$, we define the $\ell^{1}$-summed norms
\begin{equation*}
	\nrm{A}_{Z^{\sgm}} = \sum_{k} 2^{\sgm k} \nrm{P_{k} A}_{Z_{k}}, \quad
	\nrm{B}_{Z^{\sgm}_{ell}} = \sum_{k} 2^{\sgm k} \nrm{P_{k} B}_{Z_{ell, k}}.
\end{equation*}
Moreover, we define the norms $\Box Z^{\sgm}$ and $\lap Z^{\sgm}_{ell}$ so that $\nrm{A}_{Z^{\sgm}} = \nrm{\Box A}_{\Box Z^{\sgm}}$ and $\nrm{B}_{Z^{\sgm}_{ell}} = \nrm{\lap B}_{\lap Z^{\sgm}_{ell}}$. As before, the scale-critical exponent is $\sgm = \frac{d-2}{2}$.

The following lemma will be useful for estimating the norms we just defined.
\begin{lemma} \label{lem:Z-bdd}
For $F$ with frequency support in $\set{\abs{\xi} \simeq 2^{k}}$, we have 
\begin{align}
	\nrm{F}_{\Box Z^{\frac{d-2}{2}}}
	\aleq &   \sup_{\ell < \frac{1}{2} C_{1}} \, 2^{-2k} 2^{-\frac{3}{2} \ell} \bb( \sum_{\omg} \nrm{P^{\omg}_{\ell_{-}} Q_{k + 2 \ell} P_k F}_{L^{1} L^{\infty}}^{2} \bb)^{\frac{1}{2}}, \label{eq:Z-L1Linfty}	\\
	\nrm{F}_{\Box Z^{\frac{d-2}{2}}}
	\aleq  &  \nrm{Q_{< k + C_{1}} P_k F}_{L^{1} \dot{H}^{\frac{d-4}{2}}}, \label{eq:Z-L1L2} \\
	\nrm{F}_{\lap Z_{ell}^{\frac{d-2}{2}}}
	\aleq &  \sup_{\ell < \frac{1}{2} C_{1}} \, 2^{-2k} 2^{-\frac{1}{2} \ell} \bb( \sum_{\omg} \nrm{P^{\omg}_{\ell_{-}} Q_{k + 2 \ell} P_k F}_{L^{1} L^{\infty}}^{2} \bb)^{\frac{1}{2}}, \label{eq:Zell-L1Linfty}	\\
	\nrm{F}_{\lap Z_{ell}^{\frac{d-2}{2}}}
	\aleq  &  \nrm{Q_{< k + C_{1}} P_k F}_{L^{1} \dot{H}^{\frac{d-4}{2}}}, \label{eq:Zell-L1L2} \\
	\nrm{F}_{\lap Z_{ell}^{\frac{d-2}{2}}}
	\aleq  & \sup_{\ell < \frac{1}{2} C_{1}}2^{\ell} \nrm{Q_{k+2\ell} P_k F}_{L^{1} \dot{H}^{\frac{d-4}{2}}}. \label{eq:Zell-L1L2-l}
\end{align}

\end{lemma}
\begin{proof}
To prove \eqref{eq:Z-L1Linfty}, note that the symbol of the operator
$ (2^{2k+2 \ell} / \Box)  \tilde{P}^{\omg}_{\ell_{-}} \tilde{Q}_{k + 2 \ell} \tilde{P}_k $
obeys the same bump function estimates as the symbol of $ P^{\omg}_{\ell_{-}} Q_{k + 2 \ell} P_k  $ on the 
rectangular region of size $  (2^{k+\ell})^{d-1} \times 2^{k+2 \ell} \times 2^k $ where it is supported. Thus, this operator is disposable. 
Similarly, the operator $ (2^{2k}/ \la) \tilde{P}_k $ is disposable, which implies \eqref{eq:Zell-L1Linfty}. 
The bound \eqref{eq:Z-L1L2} [resp. \eqref{eq:Zell-L1L2}, \eqref{eq:Zell-L1L2-l}] follows from \eqref{eq:Z-L1Linfty} [resp. \eqref{eq:Zell-L1Linfty}] by applying Bernstein's inequality and using the orthogonality property of the sectors associated to $ (P^{\omg}_{\ell_{-}})_{\omega} $. We note that the proof of \eqref{eq:Z-L1L2}, \eqref{eq:Zell-L1L2} are sharp only in $d =4$.
\end{proof}

\subsection{Motivation of the norms} \label{spaces:discussion}

We end this section with a discussion about the choice of norms in the definition of the $ S_k $ space. For solutions $ \phi $ of the free wave equation $ \Box \phi=0 $ we have $ \vn{\phi}_{S_k} \simeq \vn{\phi[0]}_{L^2 \times \dot{H}^{-1}} $.  
The $ X^{0,1/2}_{\infty} $ space provides control of $ L^{2} L^{2} $ norms that are useful with components of high modulation. 

Additionally, one looks for norms that are both useful in proving bilinear estimates and which are controlled for free wave solutions. In fact, by expressing arbitrary functions $ \phi $ as superpositions of free waves, one can obtain boundedness of $ \vn{\phi}_{S_k} $ in terms of $ \vn{\Box \phi}_{N_k} $. An example of this argument appears in Lemma  \ref{lemmadivsymb}. The $ S_k^{str} $ component corresponds to well-known Strichartz estimates. 

Regarding $ S_k^{ang} $ and $ S_k^{box} $, the $ \ell^2 $ summation in $ P_l^{\omega} $ and $ P_{\calC_{k'}(\ell')} $ is inherited from the initial data. The square summed $ L^2 L^{\infty} $ norms play a particularly important role in the estimates. To motivate the choice of dyadic exponents, let us check that these exponents are sharp. We claim that an inequality
\be \label{example:ineq}
\vn{P_{\calC_{k'}(\ell')} P_k e^{i t \vm{D}} u_0}_{L^2 L^{\infty}} \lesssim C_{k,k',\ell'} \vn{u_0}_{L^2_{x}}
\ee
can be true (uniformly in $ k,k',\ell' $)  only for $ C_{k,k',\ell'}^2 \geq 2^{(d-2)k'} 2^{(d-3) \ell'} 2^{k} $, and is optimal when the latter is an equality. 

We consider the following version of the Knapp example: let $ u(t,x) $ be a solution to $ \Box u=0 $ with Fourier support in $ S=\{ \tau=\vm{\xi} \simeq 2^k, \ \xi \in \calC_{k'}(\ell') \} $ such that for any $\vm{t} \leq T \defeq \frac{1}{C} 2^k 2^{-2(k'+\ell')} $ one has $ \vm{u(t,x)} \simeq 1 $ for $ x $ in a rectangle of sides $ \simeq 2^{-k'} \times (2^{-k'-\ell'})^{d-1} $, dual to $ \calC_{k'}(\ell') $. The uncertainty principle suggests that $ u(t,\cdot) $ becomes dispersed after $ \vm{t} \gg T $ because the smallest rectangular box encompassing $ S $ has sides $ \simeq T^{-1} \times 2^{k'} \times (2^{k'+\ell'})^{d-1} $ (where $ T^{-1} $ and $ 2^{k'} $ are measured in the null directions). In fact, for 
$$ \calC_{k'}(\ell')=\calC \defeq \{ \vm{\xi_1} \simeq 2^k, \ \vm{\xi_1- \xi_1^0} \ll 2^{k'} , \ \vm{\xi_i} \ll 2^{k'+\ell'}, \ i=2,d  \} $$
one can define 
$$ u(t,x) = \text{vol}(\calC)^{-1} \int_{\calC} e^{i x \cdot \xi} e^{i t \vm{\xi}} \dd \xi
$$
and check that $ \vm{u(t,x)} \simeq 1 $ for $ \vm{t} \lesssim T $,  $ \vm{x_1+t} \lesssim 2^{-k'} $, $ \vm{x_i} \lesssim 2^{-k'-\ell'} $. 

Plugging this example into \eqref{example:ineq} gives $ T^{\frac{1}{2}} \lesssim C_{k,k',\ell'}  \text{vol}(\calC)^{-\frac{1}{2}} $, which provides the optimal choice of $ C_{k,k',\ell'} $ in the definition of $ S_k $.

Similar arguments apply to the norms $ PW $ and $ NE $ which are used for $ d=4 $. For instance, plugging the same $ u(t,x) $ in the inequality
$$  \vn{P_{\calC_{k'}(\ell')} e^{i t \vm{D}} u_0}_{L^2_{\omega} L^{\infty}_{\omega^{\perp}}} \lesssim \tilde{C}_{k',\ell'} \vn{u_0}_{L^2_{x}}, \qquad \text{for}  \ \omega=e_1 
$$ 
gives $ (2^{-k'})^{\frac{1}{2}} \lesssim \tilde{C}_{k',\ell'}  \text{vol}(\calC)^{-\frac{1}{2}} $, thus $ \tilde{C}_{k',\ell'} = 2^{\frac{3}{2}(k'+\ell')} $ is the optimal choice for $ PW $. 

Finally, concerning \eqref{bilL2est} we note that when $ \phi_{k_1},  \psi_{k_2} $ are free solutions with Fourier support in $ \calC_{k'}(\ell'), \calC'_{k'}(\ell') $, obeying \eqref{angle:caps}, one has
\be \label{L2:freesol}
\vn{ \phi_{k_1} \psi_{k_2} }_{L^{2} L^{2}} \lesssim 2^{- \tilde{\ell}} 2^{\frac{3}{2}(k'+\ell')} \vn{\phi_{k_1}[0]}_{L^2 \times \dot{H}^{-1}} \vn{\psi_{k_2}[0]}_{L^2 \times \dot{H}^{-1}}.
\ee
While this inequality can be proved by convolution estimates in Fourier space, \eqref{bilL2est} serves as a more general substitute for \eqref{L2:freesol} which applies to arbitrary functions.

\section{Decomposition of the nonlinearity}
In this section we describe the structure of the Maxwell-Dirac equation under the Coulomb gauge condition.

\subsection{Diagonalization of the Dirac equation} 
Our first goal is to rewrite the Dirac operator $\alp^{\mu} \rd_{\mu}$ in a diagonal form. We follow the approach of D'Ancona, Foschi and Selberg \cite{DFS1, DFS2}.

For $\mu = 0, \ldots, d$, recall the definition
\begin{align*}
\alp^{\mu} = \gmm^{0} \gmm^{\mu}.
\end{align*}
Hence $\alp^{0} = \Id$, whereas $\alp^{j}$ are hermitian matrices satisfying
\begin{equation} \label{eq:alp-rel}
	\frac{1}{2} (\alp^{j} \alp^{k} + \alp^{k} \alp^{j}) = \dlt^{jk} \Id,
\end{equation}
thanks to \eqref{eq:gmmRel} and \eqref{eq:gmmRel:conj}. Note that the Dirac operator $ \alp^{\mu} \rd_{\mu} $ then takes the form
\begin{equation*}
	\alp^{\mu} \rd_{\mu} = -i (i \rd_{t} - \alp^{j} D_{j}).
\end{equation*}
where we use the notation $D_{\mu} = \frac{1}{i} \rd_{\mu}$. To diagonalize the operator $\alp^{j} D_{j}$, whose symbol is $\alp^{j} \xi_{j}$, we introduce the multiplier $\Pi(D)$ with symbol
\begin{equation*}
	\Pi(\xi) := \frac{1}{2} \bb( \Id - \frac{\alp^{j} \xi_{j}}{\abs{\xi}} \bb).
\end{equation*}
Note that $\Pi(\xi)$ obeys the identities
\begin{equation*}
	\Pi(\xi)^{\dagger} = \Pi(\xi), \quad \Pi(\xi)^{2} = \Pi(\xi), \quad \Pi(\xi) \Pi(-\xi) = 0.
\end{equation*}
For each sign $s \in \set{+, -}$, we define the multipliers $\Pi_{s}$ with symbols $\Pi_{s}(\xi) := \Pi(s \xi)$. By the preceding identities, $\Pi_{+}$ and $\Pi_{-}$ form orthogonal projections (i.e., $\Pi_{s}^{\dagger} = \Pi_{s}$, $\Pi_{s}^{2} = \Pi_{s}$ and $\Pi_{+} \Pi_{-} = 0$). Moreover, we have
\begin{align*}
	\Id = \Pi_{+}(\xi) + \Pi_{-}(\xi), \quad 
	- \frac{\alp^{j} \xi_{j}}{\abs{\xi}} = \Pi_{+}(\xi) - \Pi_{-}(\xi)
\end{align*}
Thus the Dirac operator can now be written in the form
\begin{equation} \label{eq:dirac-HW}
	\alp^{\mu} \rd_{\mu} 
	= -i \bb( (i \rd_{t} + \abs{D}) \Pi_{+}(D) +(i \rd_{t} - \abs{D}) \Pi_{-}(D) \bb).
\end{equation}

We now present the key identities for revealing the null structure of \eqref{eq:MD}, which are essentially due to D'Ancona, Foschi and Selberg \cite{DFS1, DFS2}. 
We define the self-adjoint operators $\mR_{\mu}$ as
\begin{equation*}
	\mR_{\mu} := \frac{D_{\mu}}{\abs{D}} \quad \hbox{ for } \mu = 0, \ldots, d.
\end{equation*}
For $\mu = j \in \set{1, \ldots, d}$, the operators $\mR_{j}$ are precisely the (self-adjoint) Riesz transforms on $\bbR^{d}$. 
%

\begin{lemma} \label{lem:commAlpPi}
For each $\mu = j \in \set{1, \ldots, d}$ and sign $s \in \set{+, -}$, we have 
\begin{equation} \label{eq:commAlpPi:j}
	\alp^{j} \Pi_{s} =  -s \mR^{j}  + \Pi_{-s}  \alp^{j} .
\end{equation}
\end{lemma}
\begin{proof}
We compute
\begin{equation*}
	\alp^{j} \Pi_{s}(\xi) - \Pi_{-s}(\xi) \alp^{j}
	= -s \frac{1}{2} \frac{\xi_{k}}{\abs{\xi}} (\alp^{j} \alp^{k} + \alp^{k} \alp^{j})
	= -s \frac{\xi^{j}}{\abs{\xi}}. \qedhere
\end{equation*}
\end{proof}

\begin{remark} \label{rem:commAlpPi:0}
For $\mu = 0$, the analogue of \eqref{eq:commAlpPi:j} is
\begin{equation} \label{eq:commAlpPi:0}
	\alp^{0}  
	= -s \mR^{0}  + s \frac{i \rd_{t} + s \abs{D}}{\abs{D}},
\end{equation}
which can be easily justified.
\end{remark}

The Riesz transform term $\mR^{\mu}$ is \emph{scalar} in the sense that it does not involve multiplication by $\alp^{j}$. Its contribution in \eqref{eq:MD} resembles the Maxwell-Klein-Gordon system; see Section~\ref{subsec:nonlin} for details. Remarkably, the other terms in \eqref{eq:commAlpPi:j} and \eqref{eq:commAlpPi:0} turn out to contribute parts with more favorable structure. Indeed, in the case of \eqref{eq:commAlpPi:0}, the presence of the half-wave operator $i \rd_{t} + s\abs{D}$ (with an appropriate sign $s$) makes this term effectively higher order. In the case of \eqref{eq:commAlpPi:j}, the following lemma can be used to uncover a null structure.
\begin{lemma} \label{lem:spin-nf}
For $z \in \bbC^{N}$, $\xi, \eta \in \bbR^{d}$ and $\tht := \abs{\angle(\xi, \eta)}$, we have
\begin{equation}
	\abs{\Pi(\xi) \Pi(-\eta)}
	\leq C \tht.
\end{equation}
\end{lemma}

\begin{proof} 
Using \eqref{eq:alp-rel} and the definition of $\Pi(\xi)$, we compute
\begin{align*}
	\Pi(\xi) \Pi(-\eta)  
	= & \frac{1}{4} \bb( \Id - \frac{\alp^{j} \xi_{j}}{\abs{\xi}} \bb) \bb( \Id + \frac{\alp^{k} \eta_{k}}{\abs{\eta}} \bb)  
	= \frac{1}{4} \bb( \Id - \frac{\alp^{j} \xi_{j}}{\abs{\xi}} + \frac{\alp^{k} \eta_{k}}{\abs{\eta}} - \frac{\alp^{j} \alp^{k} \xi_{j} \eta_{k}}{\abs{\xi} \abs{\eta} } \bb)  \\
	=& - \frac{\alp^{j}}{4} \bb(\frac{\xi_{j}}{\abs{\xi}} - \frac{\eta_{j}}{\abs{\eta}} \bb) 
		- \frac{\alp^{j} \alp^{k}}{8} \bb( \frac{\xi_{j} \eta_{k} - \xi_{k} \eta_{j}}{\abs{\xi} \abs{\eta}} \bb)
		+ \frac{\Id}{4} \bb( \frac{\abs{\xi} \abs{\eta} - \xi \cdot \eta}{\abs{\xi} \abs{\eta}} \bb).
\end{align*} 
Then the lemma follows. \qedhere
\end{proof}
We remark that the identity \eqref{eq:commAlpPi:0} must be applied judiciously, since $\mR^{0}$ is well-behaved on $\psi_{\pm}$ only when the modulation does not exceed the spatial frequency.

\subsection{Maxwell equations in Coulomb gauge}
Here we describe the Maxwell equations under the Coulomb gauge condition $\rd^{\ell} A_{\ell} = 0$.

Let $J_{\mu}$ be a 1-form (called the \emph{charge-current 1-form}) on $\bbR^{1+4}$ such that $\rd^{\mu} J_{\mu} = 0$. Consider the Maxwell equations
\begin{equation} \label{eq:M}
	\rd^{\mu} F_{\nu \mu} = - J_{\nu}.
\end{equation}
Under the Coulomb gauge condition
\begin{equation} \label{eq:coulomb}
	\rd^{\ell} A_{\ell} = 0,
\end{equation}
the Maxwell equations \eqref{eq:M} reduce to
\begin{equation} \label{eq:M-a}
 \lap A_{0} = J_{0}, \quad
\Box A_{j} = \calP_{j} J_{x}
\end{equation}
where $\lap := \rd^{\ell} \rd_{\ell}$ is the \emph{Laplacian}, $\Box := \rd^{\mu} \rd_{\mu}$ is the \emph{d'Alembertian} and $\calP$ denotes the Leray projection to the divergence-free vector fields, i.e.,
\begin{equation} \label{eq:Pj}
	\calP_{j} J := J_{j} - \lap^{-1} \rd^{\ell} \rd_{j} J_{\ell}.
\end{equation}
Moreover, thanks to $\rd^{\mu} J_{\mu} =0$, we also obtain the following elliptic equation for $\rd_{t} A_{0}$:
\begin{equation}
	\lap(\rd_{t} A_{0}) = \rd^{\ell} J_{\ell}.
\end{equation}

\subsection{Decomposition of the nonlinearity} \label{subsec:nonlin}
We are now ready to describe in detail the nonlinearity of the Maxwell-Dirac equation in the Coulomb gauge (MD-CG).

As explained in the introduction, our overall philosophy is that MD-CG can be split into two parts: The scalar part, which does not involve multiplication by the matrix $\alp^{j}$, and the spinorial part arising from the spinorial nature of the Dirac equation. The latter part turns out to possess a more favorable null structure; in particular, there is no need to perform a paradifferential renormalization, nor to use a secondary null structure. On the other hand, the former part is deeply related to the Maxwell-Klein-Gordon equation in the Coulomb gauge (MKG-CG), whose small Sobolev critical global well-posedness was proved in \cite{KST}. We refer to Remarks~\ref{rem:nonlin} and \ref{rem:spin-nf} for a further discussion after the nonlinearity of MD-CG is completely described.

\subsubsection*{Nonlinearity for Maxwell}
Let $(A, \psi)$ be a solution to MD-CG. The charge-current 1-form $J$ reads
\begin{equation*}
	J^{\mu} = \brk{\gmm^{\mu} \psi, \gmm^{0} \psi} = \brk{\psi, \alp^{\mu} \psi}.
\end{equation*}
where we used \eqref{eq:gmmRel}, \eqref{eq:gmmRel:conj} and the definition of $\alp^{\mu}$ in the second identity.
By \eqref{eq:M-a}, $A_{\mu}$ solves the following equations:
\begin{align}
	\lap A_{0} =& \brk{\psi, \alp_{0} \psi} = - \brk{\psi, \alp^{0} \psi} = - \brk{\psi, \psi}, 	\label{eq:MD-a0}\\
	\Box A_{j} =& \calP_{j} \brk{\psi, \alp_{x} \psi}.			\label{eq:MD-ax}
\end{align}
Moreover, thanks to $\rd^{\mu} J_{\mu} = 0$ (which holds since $\psi$ solves a covariant Dirac equation, see remark \ref{rem:d0a0}), we have
\begin{equation} \label{eq:MD-d0a0}
	\lap (\rd_{t} A_{0}) =  \rd^{\ell} \brk{\psi, \alp_{\ell} \psi}.
\end{equation}

We now introduce bilinear version of the nonlinearities in \eqref{eq:MD-a0}, \eqref{eq:MD-ax} and \eqref{eq:MD-d0a0}, in order to set up an iteration scheme for solving MD-CG. Let $\varphi^{1}, \varphi^{2}$ be any spinor fields. For \eqref{eq:MD-a0}, we introduce
\begin{equation} \label{eq:me-def}
	\NM^{E}(\varphi^{1}, \varphi^{2}) := - \brk{\varphi^{1}, \varphi^{2}}.
\end{equation}
We also define
\begin{equation*}
	\bfA_{0}(\varphi^{1}, \varphi^{2}) := \lap^{-1} \NM^{E}(\varphi^{1}, \varphi^{2}),
\end{equation*}
so that $A_{0} = \bfA_{0}(\psi, \psi)$ for a solution $(A_{\mu}, \psi)$ to MD-CG. 

For \eqref{eq:MD-ax}, we use \eqref{eq:commAlpPi:j} to decompose the nonlinearity as
\begin{equation*}
\calP_{j} \brk{\psi, \alp_{x} \psi} = \sum_{s} \calP_{j} \brk{\psi, \alp_{x} \Pi_{s} \psi} = \sum_{s} \bb( -s \NM_{j}^{R} (\psi, \psi) + \NM^{S}_{j, s}(\psi, \psi) \bb),
\end{equation*}
where
\begin{align}
	\NM_{j}^{R} (\varphi^{1}, \varphi^{2}) :=& \calP_{j}\brk{\varphi^{1}, \mR_{x} \varphi^{2}},				\label{eq:mr-def}\\
	\NM_{j, s}^{S} (\varphi^{1}, \varphi^{2}) := & \calP_{j}\brk{\varphi^{1}, \Pi_{-s} \alp_{x} \varphi^{2}}.		\label{eq:ms-def}
\end{align}
We refer to $\NM_{j}^{R}$ and $\NM^{S}_{j, s}$ as the \emph{scalar} and \emph{spinorial} parts, respectively, of the Maxwell nonlinearity; observe that the scalar part does not involve the matrix $\alp^{j}$.
We also introduce
\begin{align*}
	\bfA_{j} (\varphi^{1}, \varphi^{2}) :=& \Box^{-1} \calP_{j} \brk{\varphi^{1}, \alp_{x} \varphi^{2}}, \\
	\bfA^{R}_{j} (\varphi^{1}, \varphi^{2}) :=& \Box^{-1} \NM^{R}_{j} (\varphi^{1}, \varphi^{2}), \\
	\bfA^{S}_{j, s} (\varphi^{1}, \varphi^{2}) :=& \Box^{-1} \NM^{S}_{j, s} (\varphi^{1}, \varphi^{2})
\end{align*}
where $ \Box^{-1} f$ here denotes the solution $\phi$ to the inhomogeneous wave equation $\Box \phi = f$ with $\phi[0] = 0$. For a solution $(A_{\mu}, \psi)$ to MD-CG, we have
\begin{align*}
A_{j} 
=A_{j}^{free} + \bfA_{j}(\psi, \psi) 
=A_{j}^{free} + \sum_{s} \bb( -s \bfA_{j}^{R}(\psi, \Pi_{s} \psi) + \bfA_{j, s}^{S}(\psi, \Pi_{s} \psi) \bb)
\end{align*}
where $A_{j}^{free}$ is the free wave with data $A_{j}^{free}[0] = A_{j}[0]$.

Finally, corresponding to \eqref{eq:MD-d0a0} we define
\begin{equation} \label{eq:d0me-def}
	\rd_{t} \NM^{E} (\varphi^{1}, \varphi^{2}) := \rd^{\ell} \brk{\varphi^{1}, \alp_{\ell} \varphi^{2}}, 
\end{equation}
so that $\lap(\rd_{t} A_{0}) = \rd_{t} \NM^{E}(\psi, \psi)$ for a solution $(A_{\mu}, \psi)$ to MD-CG.
\begin{remark} \label{rem:d0a0}
The notation $\rd_{t}$ in $\rd_{t} \NM$ is merely formal; the actual $\rd_{t}$ derivative of $\NM^{E}(\varphi^{1}, \varphi^{2})$ agrees with $\rd_{t} \NM^{E}(\varphi^{1}, \varphi^{2})$ only if 
\begin{equation*}
	\rd_{\mu} \brk{\varphi^{1}, \alp^{\mu} \varphi^{2}} = 0.
\end{equation*}
Such an identity holds if, for instance, $\varphi^{1}$ and $\varphi^{2}$ obey a (single) covariant Dirac equation $\alp^{\mu} (\rd_{\mu} + i \tilde{A}_{\mu}) \varphi = 0$ for some connection 1-form $\tilde{A}$, which is not necessarily equal to $A$. We will be careful to ensure that this is the case in our iteration scheme; see Sections~\ref{subsec:main-iter} and \ref{subsec:iter-dirac}. 
\end{remark}

\subsubsection*{Nonlinearity for Dirac}
We now turn to the covariant Dirac equation
\begin{equation} \label{eq:cov-dirac}
	\alp^{\mu} \covD_{\mu} \psi = 0.
\end{equation}
Expanding $\covD_{\mu} = \rd_{\mu} + i A_{\mu}$ and using \eqref{eq:dirac-HW}, we may rewrite the above equation as
\begin{equation} \label{eq:cov-dirac-HW-pre}
	(i \rd_{t} + s \abs{D}) \psi_{s} = \Pi_{s} (\alp^{\mu} A_{\mu} \psi) .
\end{equation}
where $s \in \set{+, -}$ and $\psi_{s}$ is the abbreviation $\psi_{s} := \Pi_{s} \psi$.
In view of the half-wave decomposition, it is natural to expand $\psi = \psi_{+} + \psi_{-}$ on the RHS of \eqref{eq:cov-dirac-HW-pre}.
Using Lemma~\ref{lem:commAlpPi}, as well as the formulae $A_{j} = \calP_{j} A_{x}$ and $\psi_{s} = \Pi_{s} \psi_{s}$, we further decompose each of the nonlinearity $\alp^{\mu} A_{\mu} \psi_{s}$ as
\begin{align*}
	\alp^{\mu} A_{\mu} \psi_{s}
	= & A_{0} \Pi_{s} \psi_{s} + A_{j} \alp^{j} \Pi_{s} \psi_{s} \\
	= & \ND^{E}(A_{0}, \Pi_{s} \psi_{s}) - s \ND^{R}(A_{x}, \psi_{s}) + \ND^{S}_{s}(A_{x}, \psi_{s}),
\end{align*}
where $\ND^{E}$, $\ND^{R}$ and $\ND_{s}^{S}$ are bilinear forms defined as follows: 
\begin{align}
	\ND^{E}(A_{0}, \varphi) := & A_{0} \varphi, 	\label{eq:ne-def} \\
	\ND^{R}(A_{x}, \varphi) := & ( \calP_{j} A_{x})( \mR^{j} \varphi ), \label{eq:nr-def} \\
	\ND^{S}_{s}(A_{x}, \varphi) := & A_{j} \Pi_{-s} (\alp^{j} \varphi). \label{eq:ns-def}
\end{align}
We refer to $\ND^{E}, \ND^{R}$ as the \emph{scalar} part of the Dirac nonlinearity, as it does not involve multiplication by $\alp^{\mu}$. The remainder $\ND^{S}_{s}$ is called the \emph{spinorial} part.


We summarize the result of our decomposition so far as follows.
\begin{lemma} \label{lem:cov-dirac-HW}
Let $\psi$ be a spinor field on $\bbR^{1+d}$ and $A_{\mu}$ be a real-valued 1-form obeying $A_{j} = \calP_{j} A_{x}$. If $\psi$ is a solution\footnote{To be pedantic, one may take the $A, \psi, \psi_{s}$ to satisfy \eqref{eq:cov-dirac} and \eqref{eq:cov-dirac-HW} in the sense of distributions, where $A_{\mu}, \psi, \psi_{s}$ are assumed to be in $L^{2}_{loc}(\bbR^{1+d})$.} to \eqref{eq:cov-dirac}, then each of $\psi_{s} = \Pi_{s} \psi$ $(s \in \set{+, -})$ solves
\begin{equation}\label{eq:cov-dirac-HW}
\begin{aligned} 
& \hskip-2em
	\Pi_{s} (i \rd_{t} + s \abs{D}) \psi_{s} \\
	= & \Pi_{s} \sum_{s'} \bb(  \calN^{E}(A_{0}, \Pi_{s'} \psi_{s'})  
				 - s' \ND^{R}(A_{x}, \psi_{s'})
				+ \ND^{S}_{s'}(A_{x}, \psi_{s'}) \bb).
\end{aligned}
\end{equation}
Conversely, if $(\psi_{+}, \psi_{-})$ is a pair of spinor fields solving \eqref{eq:cov-dirac-HW}, then $\psi := \Pi_{+} \psi_{+} + \Pi_{-} \psi_{-}$ is a solution to \eqref{eq:cov-dirac}.
\end{lemma}
\begin{remark} 
In the converse statement, $\psi_{s}$ need not belong to the image of $\Pi_{s}$, i.e., $\Pi_{s} \psi_{s}$ need not equal $\psi_{s}$ for $s \in \set{+, -}$.
\end{remark}

\begin{proof} 
The direct statement has already been proved. To prove the converse statement, we begin by noticing that
\begin{equation*}
	- s' \ND^{R}(A_{x}, \psi_{s'}) + \ND^{S}_{s'}(A_{x}, \psi_{s'})
	= A_{j} \alp^{j} \Pi_{s'} \psi_{s'}
\end{equation*}
by Lemma~\ref{lem:commAlpPi} and $A_{j} = \calP_{j} A_{x}$. Therefore, \eqref{eq:cov-dirac-HW} implies
 \begin{equation*}
	(i \rd_{t} + s \abs{D}) \Pi_{s} \psi_{s}
	= \Pi_{s} \bb( A_{0} \alp^{0} (\Pi_{+} \psi_{+} + \Pi_{-} \psi_{-}) + A_{j} \alp^{j} (\Pi_{+} \psi_{+} + \Pi_{-} \psi_{-}) \bb).
\end{equation*}
Defining $\psi := \Pi_{+} \psi_{+} + \Pi_{-} \psi_{-}$, adding up the preceding equation for $s \in \set{+, -}$ and using \eqref{eq:dirac-HW}, the desired statement follows.
\end{proof}

As discussed in the introduction, the most difficult interaction is when $A_{0}$ and $A_{x}$ have frequencies lower than $\psi_{s}$.
To isolate this part, we introduce the low-high paradifferential operators 
\begin{align*}
	\pi^{E}[A_{0}] \varphi :=& \sum_{k} \calN^{E}(P_{< k-10} A_{0}, P_{k} \varphi) = \sum_{k} P_{< k-10} A_{0} \, P_{k} \varphi, \\
	\pi^{R}[A_{x}] \varphi :=& \sum_{k} \calN^{R}(P_{< k-10} A_{x}, P_{k} \varphi) = \sum_{k} \calP_{j} P_{< k-10} A_{x} \, \mR^{j} P_{k} \varphi, \\
	\pi_{s}^{S}[A_{x}] \varphi :=& \sum_{k} \calN_{s}^{S}(P_{< k-10} A_{x}, P_{k} \varphi) = \sum_{k} P_{< k-10} A_{j} \, \Pi_{\mp} (\alp^{j} P_{k} \varphi).
\end{align*}
and the remainders $\NR^{E}$, $\NR^{R}$ and $\NR_{s}^{S}$ consisting of
\begin{align*}
	\NR^{E}(A_{0}, \varphi) 
	:= & \sum_{k} \calN^{E}(P_{\geq k-10} A_{0}, P_{k} \varphi)
	= \sum_{k} P_{\geq k-10} A_{0} \, P_{k} \varphi, \\
	\NR^{R}(A_{x}, \varphi) 
	:= & \sum_{k} \calN^{R}(P_{\geq k-10} A_{x}, P_{k} \varphi)  
	= \sum_{k} \calP_{j} P_{\geq k-10} A_{x} \, \mR^{j} P_{k} \varphi, \\
	\NR_{s}^{S}(A_{x}, \varphi) 
	:= & \sum_{k} \calN_{s}^{S}(P_{\geq k-10} A_{x}, P_{k} \varphi)  
	= \sum_{k} P_{\geq k-10} A_{j} \, \Pi_{\mp} (\alp^{j} P_{k} \varphi).
\end{align*}
 We also define the paradifferential covariant half-wave operator by
\begin{equation} \label{eq:paradiff-def}
	(i \rd_{t} + s \abs{D})^{p}_{A^{free}}
	:= (i \rd_{t} + s \abs{D})  + s \sum_{k} \calP_{j} P_{<k-5} A^{free}_{x} \mR^{j} P_{k}.
\end{equation}
so that we have
\begin{equation*}
	(i \rd_{t} + s \abs{D})^{p}_{A^{free}} 
	= (i \rd_{t} + s \abs{D})  + s \Diff^{R}[A_{x}^{free}].
\end{equation*}


\begin{remark}[Parallelism with MKG-CG] \label{rem:nonlin}
We are now ready to exhibit more concretely the parallelism between Maxwell--Klein--Gordon in Coulomb gauge (MKG-CG) and the scalar part of MD-CG.

We start with MD-CG. Applying \eqref{eq:commAlpPi:j}, \eqref{eq:commAlpPi:0} to the equations for $A_{0}$ and keeping only the Riesz transform terms, we get
\begin{equation} \label{eq:MD-R:ell}
	\lap A_{0} = - \sum_{s, s'} s' \brk{\psi_{s}, \calR_{0} \psi_{s'}} + \cdots 
\end{equation}
Furthermore, consider the equations for $A_{x}$ and $\psi$ with the spinorial parts $\bfA^{S}$ and $\calN_{\pm}^{S}$ removed. Using also \eqref{eq:commAlpPi:0} to the term $A_{0} \alp^{0} \psi$ in the Dirac equation and throwing away the second term in \eqref{eq:commAlpPi:0}, we arrive at the equations
\begin{equation} \label{eq:MD-R:wave}
\begin{aligned}
	\Box A_{j} =& - \sum_{s, s'} s' \calP_{j} \brk{\psi_{s}, \mR_{x} \psi_{s'}} + \cdots \\
	(i \rd_{t} + s \abs{D}) \psi_{s} = & - \Pi_{s} \sum_{s'} s' A_{\mu} \mR^{\mu} \psi_{s'}  + \cdots
\end{aligned}
\end{equation}

On the other hand, recall from \cite{KST} that the Maxwell--Klein--Gordon equation in the Coulomb gauge (MKG-CG) takes the form
\begin{equation*} \tag{MKG-CG}
\left\{
\begin{aligned}
	\lap A_{0} =& - \Im(\phi \overline{\covD_{0} \phi}) \\
	\Box A_{j} =& - \calP_{j} \Im(\phi \overline{\covD_{x} \phi}) \\
	\Box \phi =& - 2 i A_{\mu} \rd^{\mu} \phi + i \rd_{0} A_{0} \phi + A_{\mu} A^{\mu} \phi
\end{aligned}
\right.
\end{equation*}
Using the half-wave decomposition $\phi_{s} = \frac{1}{2} (\phi + s \frac{\rd_{t}}{i \abs{D}} \phi)$ $(s \in \set{+, -}$) and keeping only the quadratic nonlinearities (except $\rd_{0} A_{0} \phi$, which is harmless), we arrive at
\begin{equation} \label{eq:MKG-CG}
\begin{aligned}
	\lap A_{0} =& - \sum_{s, s'} \Im(\phi_{s} \overline{\rd_{0} \phi_{s'}}) + \cdots \\
	\Box A_{j} =& - \sum_{s, s'} \calP_{j} \Im(\phi_{s} \overline{\rd_{x} \phi_{s'}}) + \cdots \\
	(i \rd_{t} + s \abs{D}) \phi_{s} =& \frac{s}{\abs{D}} \sum_{s'} i A_{\mu} \rd^{\mu} \phi_{s'} + \cdots
\end{aligned}
\end{equation}
Modulo constant factors and balance of derivatives, observe the similarity between \eqref{eq:MD-R:ell}--\eqref{eq:MD-R:wave} and \eqref{eq:MKG-CG}! This similarity will be exploited below to prove a crucial trilinear null form estimate (Proposition~\ref{prop:trilinear}) and solvability of covariant Dirac equation (Proposition~\ref{prop:parasys-dirac}).
\end{remark}

\section{Statement of the main estimates} \label{sec:main-est}
In this short section, we collect the ingredients needed to prove Theorem~\ref{thm:main}. For the sake of concreteness, we restrict to the case $d = 4$ unless otherwise stated. We use the language of frequency envelopes, which is a convenient way of expressing the weak interaction among different dyadic frequency pieces; see Section~\ref{subsec:freqenv} for the notation and conventions. In what follows, we omit the admissibility constant $\dlt_{1}$ of the frequency envelopes.

For the nonlinearity in the $A_{0}$ and $A_{x}$ equations, we have the following bilinear estimates.

\begin{proposition} \label{prop:a}
For any admissible frequency envelopes $b, c$ and signs $s, s' \in \set{+, -}$, we have
\begin{align} 
\nrm{\NM^{E}(\psi, \varphi)}_{(L^{2} \dot{H}^{-1/2})_{bc}}
	+ \nrm{\rd_{t} \NM^{E}(\psi, \varphi)}_{(L^{2} \dot{H}^{-3/2})_{bc}} 
\aleq&  \nrm{\psi}_{(\tilde{S}^{1/2}_{s})_{b}} \nrm{\varphi}_{(\tilde{S}^{1/2}_{s'})_{c}} .	\label{eq:a0} \\
\nrm{\NM^{R}_{j}(\psi, \varphi)}_{(N \cap L^{2} \dot{H}^{-1/2})_{bc}}
\aleq & \nrm{\psi}_{(\tilde{S}^{1/2}_{s})_{b}} \nrm{\varphi}_{(\tilde{S}^{1/2}_{s'})_{c}}, 		\label{eq:axr}\\
\nrm{\NM^{S}_{j, s'}(\Pi_{s} \psi, \varphi)}_{(N \cap L^{2} \dot{H}^{-1/2})_{bc}}
\aleq & \nrm{\psi}_{(\tilde{S}^{1/2}_{s})_{b}} \nrm{\varphi}_{(\tilde{S}^{1/2}_{s'})_{c}}.		\label{eq:axs}
\end{align}
\end{proposition}

For the nonlinearity in the covariant Dirac equation, we first have the following set of bilinear estimates.
\begin{proposition} \label{prop:dirac}
Let $a$ and $b$ be any admissible frequency envelopes. Then the following statements holds.
\begin{enumerate}[leftmargin=*]
\item (Remainders $\NR^{E}$, $\NR^{R}$ and $\NR^{S}$) For any signs $s, s'$, we have
\begin{align} 
	\nrm{\NR^{E}(B, \psi)}_{(N_{s'}^{1/2})_{ab}} 	
	\aleq & \nrm{B}_{Y^{1}_{a}} \nrm{\psi}_{(\tilde{S}^{1/2}_{s})_{b}}  ,			\label{eq:nre} \\
	\nrm{\NR^{R}(A_{x}, \psi)}_{(N_{s'}^{1/2})_{ab}} 	
	\aleq &  \nrm{A_{x}}_{S^{1}_{a}} \nrm{\psi}_{(\tilde{S}^{1/2}_{s})_{b}} ,			\label{eq:nrr} \\
	\nrm{\Pi_{s'} \NR^{S}_{s}(A_{x}, \psi)}_{(N_{s'}^{1/2})_{ab}} 	
	\aleq &  \nrm{A_{x}}_{S^{1}_{a}} \nrm{\psi}_{(\tilde{S}^{1/2}_{s})_{b}} .			\label{eq:nrs} 
\end{align}

\item (Paradifferential operators $\pi^{E}$ and $\pi^{R}$) For opposite signs $s' = - s$, we have 
\begin{align} 
	\nrm{\Diff^{E} [B] \psi}_{(N_{-s}^{1/2})_{ab}} 
	\aleq & \nrm{B}_{Y^{1}_{a}} \nrm{\psi}_{(\tilde{S}^{1/2}_{s})_{b}}  ,			\label{eq:diffe-opp} \\
	\nrm{\Diff^{R} [A_{x}] \psi}_{(N_{-s}^{1/2})_{ab}} 
	\aleq & \nrm{A_{x}}_{S^{1}_{a}} \nrm{\psi}_{(\tilde{S}^{1/2}_{s})_{b}}  .			\label{eq:diffr-opp} 
\end{align}

\item (Paradifferential operator $\pi^{S}$) For any signs $s, s'$, we have 
\begin{align} 
	\nrm{	\Pi_{s'} \Diff^{S}_{s}[A_{x}] \psi}_{(N_{s'}^{1/2})_{ab}} 
	\aleq & \nrm{A_{x}}_{S^{1}_{a}} \nrm{\psi}_{(\tilde{S}^{1/2}_{s})_{b}}  .			\label{eq:diffs} 
\end{align}

\item (High modulation $L^{2} L^{2}$ bounds) For any sign $s$, we have
\begin{align} 
	\nrm{\calN^{E}(B, \psi)}_{(L^{2} L^{2})_{ab}} 
	\aleq \nrm{B}_{Y^{1}_{a}} \nrm{\psi}_{(S^{1/2}_{s})_{b}}, 	\label{eq:ne-himod} \\
	\nrm{\calN^{R}(A_{x}, \psi)}_{(L^{2} L^{2})_{ab}} 
	\aleq \nrm{A_{x}}_{S^{1}_{a}} \nrm{\psi}_{(S^{1/2}_{s})_{b}}, 	\label{eq:nr-himod} \\
	\nrm{\calN^{S}_{s}(A_{x}, \psi)}_{(L^{2} L^{2})_{ab}} 
	\aleq \nrm{A_{x}}_{S^{1}_{a}} \nrm{\psi}_{(S^{1/2}_{s})_{b}}. 	\label{eq:ns-himod} 
\end{align}
\item ($\tilde{Z}_{s}^{1/2}$ bounds) For any sign $s$, we have
\begin{align}
	\nrm{\ND^{E}(B, \psi)}_{G^{1/2}_{ab}}
	\aleq &  \nrm{B}_{Y^{1}_{a}}  \nrm{\psi}_{(S^{1/2}_{s})_{b}}  ,		\label{eq:ne-z} \\
	\nrm{\NR^{R}(A_{x}, \psi)}_{G^{1/2}_{ab}}
	+ \nrm{\Diff^{R}[A_{x}] \psi}_{G^{1/2}_{ab}}
	\aleq &  \nrm{A_{x}}_{S^{1}_{a}}  \nrm{\psi}_{(S^{1/2}_{s})_{b}}  ,		\label{eq:nr-z}  \\
	\nrm{\ND^{S}_{s}(A_{x}, \psi)}_{G^{1/2}_{ab}}
	\aleq &  \nrm{A_{x}}_{S^{1}_{a}}  \nrm{\psi}_{(S^{1/2}_{s})_{b}}  .		\label{eq:ns-z}
\end{align}
\end{enumerate}
\end{proposition}

By \eqref{eq:nrs}, \eqref{eq:diffs}, \eqref{eq:ns-himod} and \eqref{eq:ns-z}, the spinorial nonlinearity $\ND^{S}_{s'}$ can be handled just with bilinear estimates. 
On the other hand, Proposition~\ref{prop:dirac} leaves open treatment of certain parts of $\ND^{E}$ and $\ND^{R}$, namely $\Diff^{E} [A_{0}] \psi$ and $\Diff^{R} [A_{x}] \psi$. For a solution to MD-CG, recall the decomposition $A_{0} = \bfA_{0}(\psi, \psi)$ and $A_{x} = A_{x}^{free} + \bfA_{x}(\psi, \psi)$.
For the terms $\Diff^{E} [\bfA_{0} (\psi, \psi)] \psi$ and $\Diff^{R} [\bfA_{x} (\psi, \psi)] \psi$, which resemble the MKG-CG nonlinearity (see Remark~\ref{rem:nonlin}), we use the following trilinear estimate. 
\begin{proposition} \label{prop:trilinear}
For any admissible frequency envelopes $b$, $c$ and $d$, let
\begin{equation} \label{eq:trilinear-fe}
	f_{k} = \bb( \sum_{k' < k} c_{k'}^{2} \bb)^{1/2} \bb( \sum_{k' < k} d_{k'}^{2} \bb)^{1/2} b_{k}.
\end{equation}
Then for any signs $s, s_{1}, s_{2} \in \set{+, -}$, we have
\begin{equation} \label{eq:tri-r}
\begin{aligned} 
& \hskip-2em
\nrm{\big( \Diff^{E} [\bfA_{0} (\Pi_{s_{1}} \varphi^{1}, \Pi_{s_{2}} \varphi^{2})] - s \Diff^{R} [\bfA_{x} (\Pi_{s_{1}} \varphi^{1}, \Pi_{s_{2}} \varphi^{2})] \big) \psi}_{(N_{s}^{1/2})_{f}} \\
\aleq & \nrm{\varphi^{1}}_{(\tilde{S}^{1/2}_{s_{1}})_{c}} \nrm{\varphi^{2}}_{(\tilde{S}^{1/2}_{s_{2}})_{d}} \nrm{\psi}_{(\tilde{S}^{1/2}_{s})_{b}} .
\end{aligned}
\end{equation}
\end{proposition}
Propositions~\ref{prop:a} and \ref{prop:dirac} will be proved in Section~\ref{sec:bi}, and Proposition~\ref{prop:trilinear} will be proved in Section~\ref{sec:tri}. 

\begin{remark} \label{rem:fe}
In the proof of the main theorem, the frequency envelopes $a, b, c$ inherit $\ell^{2}$-summability from the initial data; hence the products $ab$ and $bc$ are $\ell^{1}$-summable. The bilinear estimates in Propositions~\ref{prop:a} and \ref{prop:dirac} therefore imply that certain parts of the solution (in particular, $\bfA_{0}$ and $\bfA_{x}$) enjoy $\ell^{1}$-summability of the dyadic norms. As in the case of MKG-CG \cite{KST}, this fact allows us to cleanly separate $A$ into $\bfA$ handled by multilinear estimates (Proposition~\ref{prop:trilinear}) and $A^{free}$ handled by a parametrix construction below (Theorem~\ref{thm:paradiff}).
\end{remark}


The remaining term $\Diff^{R}[A_{x}^{free}] \psi$ cannot be treated perturbatively. The optimal estimate, stated in terms of frequency envelopes, is as follows.
\begin{lemma} \label{lem:diffr-free}
Let $A^{free} = (0, A^{free}_{1}, \ldots, A^{free}_{4})$ be a real-valued 1-form obeying $\Box A^{free} = 0$ and $\rd^{\ell} A^{free}_{\ell} = 0$.
For any admissible frequency envelope $a$ and $b$, let $e_{k} = (\sum_{k' < k} a_{k'}) b_{k}$. Then for any sign $s \in \set{+, -}$, we have
\begin{equation} \label{eq:diffr-free}
	\nrm{\Diff^{R}[A^{free}_{x}] \psi}_{(N^{1/2}_{s})_{e}} \aleq \nrm{A^{free}[0]}_{(\dot{H}^{1} \times L^{2})_{a}} \nrm{\psi}_{(\tilde{S}^{1/2}_{s})_{b}}
\end{equation}
\end{lemma}
This can be seen by choosing specific frequency-localized free solutions $ \Box A^{free}_{k'}=0 $, $ (i \pt_t+s \vm{D} )\psi_k=0 $ such that
$$  \nrm{\Diff^{R}[A^{free}_{k',x}] \psi_k}_{X^{\frac{1}{2},-\frac{1}{2}}_{s,1}} \simeq \nrm{A^{free}_{k'}[0]}_{\dot{H}^{1} \times L^{2}}  \vn{\psi_k(0)}_{\dot{H}^{\frac{1}{2}}}.
$$ 
Since the assumption on the data is only that $A_{x}[0] \in \dot{H}^{1} \times L^{2}$, its frequency envelope, constructed as in \eqref{eq:fe-construct-0}, in general only obeys $a \in \ell^{2}$; thus the frequency envelope $e$ is not well-defined under the hypothesis of Theorem~\ref{thm:main}. 

Instead, $\Diff^{R}[A_{x}^{free}]$ should be treated as a part of the underlying linear operator $(i \rd_{t} + s \abs{D})^{p}_{A^{free}}$ defined in \eqref{eq:paradiff-def}. For this operator, we have the following global solvability theorem.
\begin{theorem} \label{thm:paradiff}
Let $A^{free} = (0, A^{free}_{1}, \ldots, A^{free}_{4})$ be a real-valued 1-form obeying $\Box A^{free} = 0$ and $\rd^{\ell} A^{free}_{\ell} = 0$.
Consider the initial value problem
\begin{equation*}
\left\{
\begin{aligned}
(i \rd_{t} + s \abs{D})^{p}_{A^{free}} \psi = & F, \\
\psi(0) =& \psi_{0}.
\end{aligned}
\right.
\end{equation*}
If $\nrm{A^{free}[0]}_{\dot{H}^{1} \times L^{2}}$ is sufficiently small, then for any $ F \in N_{s}^{1/2}\cap L^{2} L^{2} $ and any $ \psi_{0} \in \Hcr $ there exists a global (in time) solution $ \psi \in S_{s}^{1/2} $. Moreover, for any admissible frequency envelope $c$, we have
\begin{equation} \label{eq:paradiff-en}
	\nrm{\psi}_{(S^{1/2}_{s})_{c}} \aleq \nrm{\psi_{0}}_{\dot{H}^{1/2}_{c}} + \nrm{F}_{(N^{1/2}_{s} \cap L^{2} L^{2})_{c}} \ .
\end{equation}
In particular,
\be \label{spsestim} \vn{\psi}_{S_{s}^{1/2}} \lesssim \vn{\psi_{0}}_{\Hcr}+ \vn{F}_{N_{s}^{1/2}\cap L^{2} L^{2}}. \ee
\end{theorem}
A sketch of proof of Lemma~\ref{lem:diffr-free} will be given in Remark~\ref{rem:diffr-free}.
Theorem~\ref{thm:paradiff} will be established in Section~\ref{sec:para} by adapting the parametrix construction for the paradifferential covariant wave equation from \cite{KST}.

\begin{remark} \label{rem:hi-d-1}
In the case of a general dimension $d \geq 4$, all the estimates above hold with the following substitutions:
\begin{align*}
	& L^{2} \dot{H}^{-3/2} \to L^{2} \dot{H}^{\frac{d-7}{2}}, \quad L^{2} \dot{H}^{-1/2} \to L^{2} \dot{H}^{\frac{d-5}{2}}, \quad L^{2} L^{2} \to L^{2} \dot{H}^{\frac{d-4}{2}}, \\
	& N \to N^{\frac{d-4}{2}}, \quad N_{s}^{1/2} \to N^{\frac{d-3}{2}}, \quad G^{1/2} \to G^{\frac{d-3}{2}}, \\
	& S^{1} \to S^{\frac{d-2}{2}}, \quad Y^{1} \to Y^{\frac{d-2}{2}}, \quad \tilde{S}_{s}^{1/2} \to \tilde{S}_{s}^{\frac{d-3}{2}}.
\end{align*}
See Remarks~\ref{rem:hi-d-3}, \ref{rem:hi-d-4} and \ref{rem:hi-d-5} in Sections~\ref{sec:bi}, \ref{sec:tri} and \ref{sec:para}, respectively.
\end{remark}
\section{Proof of the main theorem} \label{sec:iter}
Assuming the estimates in Section~\ref{sec:main-est}, we now prove Theorem~\ref{thm:main}. The MD-CG system takes the form
 \begin{equation} \tag{MD-CG}
	\left\{
\begin{aligned}
	\alp^{\mu} \covD_{\mu} \psi =& 0 \\
	\Box A_{j} =& \calP_{j} \brk{\psi, \alp_{x} \psi} \\
	 \lap A_{0} =& - \brk{\psi, \psi}
\end{aligned}
	\right.
\end{equation}

\subsection{Subcritical local well-posedness of MD-CG} \label{subsec:subcrit-lwp}
We first state a subcritical local well-posedness result for MD-CG, which will be used in our proof below.

Let $(\psi(0), A_{x}[0])$ be an initial data set for MD-CG, where $A_{x}[t]$ is the shorthand for a pair of spatial 1-forms $(A_{j} (t)\, \ud x^{j} , \rd_{t} A_{j} (t) \, \ud x^{j} )$. In particular, $\rd^{\ell} A_{\ell}(0) = \rd^{\ell} \rd_{t} A_{\ell}(0) = 0$. Given $s, N \in \bbR$, we introduce shorthands $H^{s, N} = \dot{H}^{s} \cap \dot{H}^{N}$ and $\calH^{s, N} = (\dot{H}^{s} \times \dot{H}^{s-1}) \cap (\dot{H}^{N} \times \dot{H}^{N-1})$.
\begin{proposition} \label{prop:subcrit-lwp}
For any initial data $\psi(0) \in H^{1/2, 5/2}$ and $A_{x}[0] \in \calH^{1, 3}$, there exists a unique local solution $(A, \psi)$ to MD-CG with these data in the space $\psi \in C_{t}([0, T]; H^{1/2, 5/2})$ and $\rd_{t,x} A_{x} \in C_{t}([0, T]; \calH^{0, 2})$, where $T > 0$ depends only on $\nrm{\psi(0)}_{H^{1/2, 5/2}}$ and $\nrm{A_{x}[0]}_{\calH^{1, 3}}$. The data-to-solution map in these spaces is Lipschitz continuous. Moreover, if $\psi(0) \in \dot{H}^{1/2+N}$, $A_{x}[0] \in \dot{H}^{1+N} \times \dot{H}^{N}$ for $N \geq 2$, then $\psi \in C_{t}([0, T]; \dot{H}^{1/2+N})$, $\rd_{t,x} A_{x} \in C_{t}([0, T]; \dot{H}^{N})$.
\end{proposition}
We omit the proof, which proceeds by a usual Picard iteration (based on the d'Alembertian $\Box$ and the free Dirac operator $\alp^{\mu} \rd_{\mu}$) using the energy integral method.

\subsection{Main iteration procedure} \label{subsec:main-iter}
In this subsection, we prove the existence and uniqueness statements in Theorem~\ref{thm:main}. We begin with a more precise formulation of these statements.

Let $(\psi(0), A_{x}[0])$ be an initial data set for MD-CG. We say that $c = (c_{k})_{k \in \bbZ}$ is a \emph{frequency envelope for} $(\psi(0), A_{x}[0])$ if 
\begin{equation*}
	\nrm{P_{k} \psi(0)}_{\dot{H}^{1/2}} + \nrm{P_{k} A_{x}[0]}_{\dot{H}^{1} \times L^{2}} \leq c_{k}.
\end{equation*}
Given any initial data $\psi(0) \in \dot{H}^{1/2}$, $A_{x}[0] \in \dot{H}^{1} \times L^{2}$, an admissible frequency envelope for $(\psi(0), A_{x}[0])$ such that $(\sum_{k} c_{k}^{2})^{1/2} \aleq \nrm{\psi(0)}_{\dot{H}^{1/2}} + \nrm{A_{x}[0]}_{\dot{H}^{1} \times L^{2}}$ can be constructed as follows (cf. \eqref{eq:fe-construct-0}):
\begin{equation} \label{eq:fe-construct}
	c_{k} := \sum_{k'} 2^{-\dlt_{1} \abs{k - k'}} \bb( \nrm{P_{k'} \psi(0)}_{\dot{H}^{1/2}} + \nrm{P_{k'} A_{x}[0]}_{\dot{H}^{1} \times L^{2}}  \bb).
\end{equation}

\begin{theorem} \label{thm:main-iter}
There exists a universal constant $\eps_{\ast} > 0$ such that the following statements hold.
\begin{enumerate}[leftmargin=*]
\item For any initial data $\psi(0) \in \dot{H}^{1/2}$, $A_{x}[0] \in \dot{H}^{1} \times L^{2}$ for MD-CG satisfying the smallness condition \eqref{eq:main:smalldata}, there exists a unique global solution $(A, \psi)$ to MD-CG with these data in the space $\Pi_{s} \psi \in \tilde{S}^{1/2}_{s}$, $A_{0} \in Y^{1}$, $A_{j} \in S^{1}$. Given any admissible frequency envelope $c$ for $(\psi(0), A_{x}[0])$, we have
\begin{equation} \label{eq:main-fe}
	\sup_{s \in \set{+, -}} \nrm{\Pi_{s} \psi}_{(\tilde{S}^{1/2}_{s})_{c}} + \nrm{A_{x} - A_{x}^{free}}_{(S^{1})_{c^{2}}} + \nrm{A_{0}}_{Y^{1}_{c^{2}}}
	\aleq 1.	
\end{equation}
\item Let $(A', \psi')$ be another solution to MD-CG such that $\Pi_{s} \psi' \in \tilde{S}^{1/2}_{s}$, $A_{0}' \in Y^{1}$, $A_{j} \in S^{1}$ and the data $\psi'(0), A'_{x}[0]$ satisfies \eqref{eq:main:smalldata}. Assume also that $(\psi - \psi')(0) \in \dot{H}^{1/2-\dlt_{2}}$ and $(A_{x} - A_{x}')[0] \in \dot{H}^{1-\dlt_{2}} \times \dot{H}^{-\dlt_{2}}$ for some $\dlt_{2} \in (0, \dlt_{1})$. Then we have
\begin{equation} \label{eq:weak-lip}
\begin{aligned}
& \hskip-2em
	\sup_{s \in \set{+, -}} \nrm{\Pi_{s} (\psi - \psi')}_{\tilde{S}^{1/2-\dlt_{2}}} +  \nrm{A_{x} - A_{x}'}_{S^{1-\dlt_{2}}} + \nrm{A_{0} - A_{0}'}_{Y^{1-\dlt_{2}}}   \\
	\aleq & \nrm{(\psi - \psi')(0)}_{\dot{H}^{1/2-\dlt_{2}}} + \nrm{(A_{x} - A_{x}')[0]}_{\dot{H}^{1-\dlt_{2}} \times \dot{H}^{-\dlt_{2}}}.
\end{aligned}
\end{equation}
\item If $\psi(0) \in \dot{H}^{1/2+N}$, $A_{x}[0] \in \dot{H}^{1+N} \times \dot{H}^{N}$ $(N \geq 0)$, then $\psi \in C_{t}(\bbR; \dot{H}^{1/2+N})$, $\rd_{t, x} A_{x} \in C_{t}(\bbR; \dot{H}^{N})$. In particular, if the data $(\psi(0), A_{x}[0])$ are smooth, then so is the solution $(A, \psi)$.
\end{enumerate}
\end{theorem}

Theorem~\ref{thm:main-iter} is proved by a Picard-type iteration argument as in \cite{KST}. The presence of a non-perturbative interaction with $A^{free}$ precludes the usual Picard iteration procedure based on the free Dirac operator. Instead, we rely on the following solvability result for the covariant Dirac equation which, in particular, contains the contribution of $A^{free}$.
\begin{proposition} \label{prop:parasys-dirac}
There exists a universal constant $\eps_{\ast \ast} > 0$ such that the following holds.
Let $I \subseteq \bbR$ be a time interval containing $0$. Given spinor fields $\psi_{0} \in \dot{H}^{1/2}$ on $\bbR^{4}$ and $F$ on $I \times \bbR^{4}$ such that $\Pi_{s} F \in N_{s}^{1/2} \cap L^{2} L^{2} \cap G^{1/2}[I]$ $(s \in \set{+, -})$, consider the covariant Dirac equation
\begin{equation} \label{eq:parasys-dirac}
	\left\{
\begin{aligned}
	\alp^{\mu} \covD_{\mu}^{A} \psi =& F \quad \hbox{ on } I \\
	\psi(0) =& \psi_{0},
\end{aligned}
	\right.
\end{equation}
where the potential $A = A_{\mu} \ud x^{\mu}$ is given by
\begin{align*}
	A_{0} = \bfA_{0}(\psi', \psi'), \quad
	A_{j} = A^{free}_{j} + \bfA_{j}(\psi', \psi') \quad \hbox{ on } I
\end{align*}
for some free wave $A^{free}_{j} \in C_{t} \dot{H}^{1} \cap \dot{C}^{1}_{t} L^{2}$ $(j=1, \ldots, 4)$ and a spinor field $\psi'$ satisfying $\Pi_{s} \psi' \in \tilde{S}^{1/2}_{s}[I]$ and $\rd_{\mu} \brk{\psi', \alp^{\mu} \psi'} = 0$. If
\begin{equation} \label{eq:parasys-dirac-hyp}
\sup_{s \in \set{+, -}} \nrm{\Pi_{s} \psi'}_{\tilde{S}^{1/2}_{s}[I]} + \sup_{j \in \set{1, \ldots, 4}} \nrm{A_{j}^{free}[0]}_{\dot{H}^{1} \times L^{2}} \leq \eps_{\ast \ast},
\end{equation}
then there exists a unique solution $\psi$ to \eqref{eq:parasys-dirac} on $I \times \bbR^{4}$ such that $\Pi_{s} \psi \in \tilde{S}^{1/2}_{s}[I]$ for $s \in \set{+, -}$. For any admissible frequency envelope $c$, we have
\begin{equation}
	\nrm{\Pi_{s} \psi}_{(\tilde{S}^{1/2}_{s}[I])_{c}} \aleq \nrm{\Pi_{s} \psi_{0}}_{\dot{H}^{1/2}_{c}} + \nrm{\Pi_{s} F}_{(N^{1/2}_{s} \cap L^{2} L^{2} \cap G^{1/2} [I])_{c}}.
\end{equation}
The implicit constants are independent of $ I $.	

%
\end{proposition}
We defer the proof of Proposition~\ref{prop:parasys-dirac} until Section~\ref{subsec:iter-dirac}, which is accomplished by a separate iteration argument. Here we assume the validity of Proposition~\ref{prop:parasys-dirac} and sketch the proof of Theorem~\ref{thm:main-iter}.

\subsubsection*{Step 1: Existence and frequency envelope bound}
We first prove Statement~(1) of Theorem~\ref{thm:main-iter} except uniqueness, which is proved in the next step. We proceed by a Picard-type iteration, where the iterates are constructed recursively as follows. For the zeroth iterate, we take the trivial pair $(A^{0}, \psi^{0}) = 0$. Then for any $n \geq 0$, we first define
\begin{equation*}
	A_{0}^{n+1} = \bfA_{0}(\psi^{n}, \psi^{n}), \quad
	A_{j}^{n+1} = A_{j}^{free} + \bfA_{j}(\psi^{n}, \psi^{n}),
\end{equation*}
where $A_{j}^{free}$ denotes the free wave development of $A_{j}[0] = (A_{j}, \rd_{t} A_{j})(0)$. Next, we define $\psi^{n+1}$ by solving the covariant Dirac equation
\begin{equation*}
	\alp^{\mu} \covD_{\mu}^{A^{n+1}} \psi^{n+1} = 0, \quad \psi^{n+1}(0) = \psi(0).
\end{equation*}
In order to construct $\psi^{n+1}$, we wish to apply Proposition~\ref{prop:parasys-dirac} with $A = A^{n+1}$, or equivalently, $\psi' = \psi^{n}$ and $A^{free}_{j}[0]=A_{j}[0]$. When $n = 0$ we have $\psi^{0} = 0$, so the hypothesis of Proposition~\ref{prop:parasys-dirac} is verified simply by recalling \eqref{eq:main:smalldata} and taking $\eps_{\ast} \leq \eps_{\ast \ast}$. For $n \geq 1$, we make the induction hypothesis
\begin{equation} \label{eq:main-iter-ind}
	\sup_{s \in \set{+, -}} \nrm{\Pi_{s}(\psi^{m} - \psi^{m-1})}_{\tilde{S}^{1/2}_{s}} \leq (C_{\ast} \eps_{\ast})^{m} \quad \hbox{ for all } 1 \leq m \leq n.
\end{equation}
for some universal constant $C_{\ast} > 0$. Recalling \eqref{eq:main:smalldata}, summing up \eqref{eq:main-iter-ind} for $1 \leq m \leq n$ and taking $\eps_{\ast}$ sufficiently small compared to $\eps_{\ast \ast}$ (independent of $n$), we may ensure that the hypothesis \eqref{eq:parasys-dirac-hyp} of Proposition~\ref{prop:parasys-dirac} holds. Moreover, since $\psi^{n}$ obeys a covariant Dirac equation, the condition $\rd_{\mu} \brk{\psi^{n}, \alp^{\mu} \psi^{n}} = 0$ is satisfied by Remark~\ref{rem:d0a0}. 

With an appropriate choice of $C_{\ast}$ and $\eps_{\ast}$, we claim that the $(n+1)$-th iterate $(A^{n+1}, \psi^{n+1})$ has the following properties:
\begin{equation} \label{eq:main-iter-fe} 
	\sup_{s \in \set{+, -}} \nrm{\Pi_{s} \psi^{n+1}}_{(\tilde{S}^{1/2}_{s})_{c}} + \nrm{A^{n+1}_{x} - A_{x}^{free}}_{(S^{1})_{c^{2}}} + \nrm{A^{n+1}_{0}}_{Y^{1}_{c^{2}}}
	\aleq  1,	
\end{equation}
\begin{equation} \label{eq:main-iter-diff}
\begin{aligned}
	\sup_{s \in \set{+, -}} \nrm{\Pi_{s}(\psi^{n+1} - \psi^{n})}_{\tilde{S}^{1/2}_{s}} + \nrm{A^{n+1}_{x} - A^{n}_{x}}_{S^{1}} 
	+ \nrm{A^{n+1}_{0} - A^{n}_{0}}_{Y^{1}}
	\leq (C_{\ast} \eps_{\ast})^{n+1}. 
\end{aligned}\end{equation}
Assuming these, the proof of existence and \eqref{eq:main-fe} may be concluded as follows. Note that \eqref{eq:main-iter-diff} ensures that the induction hypothesis \eqref{eq:main-iter-ind} remains valid up to $m = n+1$. Moreover, these estimates immediately imply convergence of $(A^{n}, \psi^{n})$ in the topology $\Pi_{s} \psi^{n} \in \tilde{S}_{s}^{1/2}$, $A_{j} \in S^{1}$ and $A_{0} \in Y^{1}$ to a solution $(A, \psi)$ to MD-CG; furthermore, the solution obeys the frequency envelope bound \eqref{eq:main-fe}.

It only remains to establish \eqref{eq:main-iter-fe} and \eqref{eq:main-iter-diff}; we start with \eqref{eq:main-iter-fe}. Decomposing $\lap A_{0}^{n+1}$, $\lap \rd_{t} A_{0}^{n+1}$ and $\Box A_{x}^{n+1}$ as in Section~\ref{subsec:nonlin} and applying Proposition~\ref{prop:a}, the proof of \eqref{eq:main-iter-fe} is reduced to establishing
\begin{equation} \label{eq:main-iter-fe-key}
	\sup_{s \in \set{+, -}} \nrm{\Pi_{s} \psi^{m}}_{(\tilde{S}^{1/2}_{s})_{c}} \aleq 1 \quad \hbox{ for } m = 1, \ldots, n+1.
\end{equation}
Choosing $\eps_{\ast}$ sufficiently small and summing up the induction hypothesis \eqref{eq:main-iter-ind}, we obtain
\begin{equation} \label{eq:main-iter-ind-1}
	\sup_{s \in \set{+, -}} \nrm{\Pi_{s} \psi^{m}}_{\tilde{S}^{1/2}_{s}} + \nrm{A_{x}[0]}_{\dot{H}^{1} \times L^{2}} \leq C \eps_{\ast} \leq \eps_{\ast \ast} \quad \hbox{ for } m = 0, \ldots, n.
\end{equation}
This bound allows us to apply Proposition~\ref{prop:parasys-dirac}, which implies \eqref{eq:main-iter-fe-key} as desired.

Next, we turn to \eqref{eq:main-iter-diff}. For any $\mu \in \set{0, 1, \ldots, 4}$, we may write
\begin{equation*}
	A_{\mu}^{n+1} - A_{\mu}^{n} = \bfA_{\mu}(\dlt \psi^{n}, \psi^{n}) + \bfA_{\mu}(\psi^{n-1}, \dlt \psi^{n}),
\end{equation*}
where we have used the shorthand $\dlt \psi^{n} = \psi^{n} - \psi^{n-1}$. Decomposing $\lap \bfA_{0} = \NM_{0}$, $\lap \rd_{t} \bfA_{0} = \rd_{t} \NM_{0}$ and $\Box \bfA_{x} = \NM_{x}$ as in Section~\ref{subsec:nonlin} and applying\footnote{Proposition~\ref{prop:a} is stated in terms of admissible frequency envelopes. Constructing frequency envelopes as in \eqref{eq:fe-construct-0}, Proposition~\ref{prop:a} easily implies the non-frequency envelope version, which we use here. The same remark applies to the application of estimates in Propositions~\ref{prop:dirac} and \ref{prop:trilinear} below.} Proposition~\ref{prop:a}, we obtain
\begin{align*}
& \hskip-2em
	\nrm{A_{0}^{n+1} - A_{0}^{n}}_{L^{2} \dot{H}^{3/2}}
	+ \nrm{\rd_{t} A_{0}^{n+1} - \rd_{t} A_{0}^{n}}_{L^{2} \dot{H}^{1/2}} 
	+ \nrm{A_{x}^{n+1} - A_{x}^{n}}_{S^{1}} \\
\aleq &\sup_{s, s' \in \set{+, -}} \bb( \nrm{\Pi_{s} \psi^{n}}_{\tilde{S}^{1/2}_{s}} + \nrm{\Pi_{s} \psi^{n-1}}_{\tilde{S}^{1/2}_{s}} \bb) \nrm{\Pi_{s'} \dlt \psi^{n}}_{\tilde{S}^{1/2}_{s}}
\end{align*}
By \eqref{eq:main-iter-ind} and \eqref{eq:main-iter-fe} for $\psi^{n}$ and $\psi^{n-1}$, it follows that
\begin{equation*}
	\nrm{A_{0}^{n+1} - A_{0}^{n}}_{Y^{1}}
	+ \nrm{A_{x}^{n+1} - A_{x}^{n}}_{S^{1}} 
	\aleq \eps_{\ast} (C_{\ast} \eps_{\ast})^{n}
\end{equation*}
which is acceptable by choosing $C_{\ast}$ larger than the implicit (universal) constant. 

We now estimate the $\tilde{S}^{1/2}_{s}$ norm of $\dlt \psi^{n+1} = \psi^{n+1} - \psi^{n}$. We begin by computing
\begin{align*}
	\alp^{\mu} \covD^{A^{n}}_{\mu} \dlt \psi^{n+1}
	=& - i \alp^{\mu} (A^{n+1}_{\mu} - A^{n}_{\mu}) \psi^{n+1} \\
	=& - i \alp^{\mu} \bb( \bfA_{\mu}(\dlt \psi^{n}, \psi^{n}) + \bfA_{\mu}(\psi^{n-1}, \dlt \psi^{n}) \bb) \psi^{n+1}.
\end{align*}
By symmetry, it suffices to consider only the contribution of $\bfA_{\mu}(\dlt \psi^{n-1}, \psi^{n})$.
Using the shorthand $\psi^{n+1}_{s} = \Pi_{s} \psi^{n+1}$, we expand
\begin{align}
& \hskip-2em
\Pi_{s} \bb( \alp^{\mu}  \bfA_{\mu}(\dlt \psi^{n}, \psi^{n}) \psi^{n+1} \bb) \notag \\
	= & \Pi_{s} \bb( \pi^{E}[\bfA_{0}(\dlt \psi^{n}, \psi^{n})] \psi_{s}^{n+1} -s \pi^{R}[\bfA_{x}(\dlt \psi^{n}, \psi^{n})] \psi_{s}^{n+1} \bb) \label{eq:main-iter-dltpsi-1} \\
	& + \Pi_{s} \NR^{E}(\bfA_{0}(\dlt \psi^{n}, \psi^{n}), \Pi_{s} \psi_{s}^{n+1}) -s \Pi_{s} \NR^{R}(\bfA_{x}(\dlt \psi^{n}, \psi^{n}), \psi_{s}^{n+1}) \label{eq:main-iter-dltpsi-2} \\
	& + \Pi_{s} \ND^{E}(\bfA_{0}(\dlt \psi^{n}, \psi^{n}), \Pi_{-s} \psi_{-s}^{n+1}) + s \Pi_{s} \ND^{R}(\bfA_{x}(\dlt \psi^{n}, \psi^{n}), \psi_{-s}^{n+1}) \label{eq:main-iter-dltpsi-3} \\
	& + \Pi_{s} \ND^{S}_{s}(\bfA_{x}(\dlt \psi^{n}, \psi^{n}), \psi_{s}^{n+1}) + \Pi_{s} \ND^{S}_{-s}(\bfA_{x}(\dlt \psi^{n}, \psi^{n}), \psi_{-s}^{n+1}) \label{eq:main-iter-dltpsi-4}
\end{align}
We wish to estimate the $N_{s}^{1/2} \cap L^{2} L^{2} \cap G^{1/2}$ norm of the RHS using Proposition~\ref{prop:dirac} and \ref{prop:trilinear}. More precisely, For the $N^{1/2}_{s}$ norm, we apply \eqref{eq:tri-r} for \eqref{eq:main-iter-dltpsi-1}; \eqref{eq:nre}--\eqref{eq:nrr} for \eqref{eq:main-iter-dltpsi-2}; \eqref{eq:nre}--\eqref{eq:nrr}, \eqref{eq:diffe-opp}--\eqref{eq:diffr-opp} for \eqref{eq:main-iter-dltpsi-3} and \eqref{eq:nrs}, \eqref{eq:diffs} for \eqref{eq:main-iter-dltpsi-4}. For the $L^{2} L^{2} \cap G^{1/2}$ norm, we simply use \eqref{eq:ne-himod}--\eqref{eq:ns-himod} and \eqref{eq:ne-z}--\eqref{eq:ns-z}. Then we obtain
\begin{align*}
\nrm{\Pi_{s} \bb( \alp^{\mu}  \bfA_{\mu}(\dlt \psi^{n}, \psi^{n}) \psi^{n+1} \bb)}_{N_{s}^{1/2} \cap L^{2} L^{2} \cap G^{1/2}} 
\aleq \sup_{s_{1}, s_{2}, s_{3}}\nrm{\dlt \psi^{n}_{s_{1}}}_{\tilde{S}^{1/2}_{s_{1}}} \nrm{\psi^{n}_{s_{2}}}_{\tilde{S}^{1/2}_{s_{2}}} \nrm{\psi^{n+1}_{s_{3}}}_{\tilde{S}^{1/2}_{s_{3}}} 
\end{align*}
Hence by Proposition~\ref{prop:parasys-dirac}, \eqref{eq:main-iter-ind} and \eqref{eq:main-iter-fe} for $\psi^{n+1}$ and $\psi^{n}$, we arrive at
\begin{equation} \label{main-iter-diff-psi-0}
	\sup_{s \in \set{+, -}} \nrm{\Pi_{s} \dlt \psi^{n+1}}_{\tilde{S}^{1/2}_{s}}
	\aleq \eps_{\ast}^{2} (C_{\ast} \eps_{\ast})^{n},
\end{equation}
which is acceptable.

\subsubsection*{Step 2: Uniqueness}
To finish the proof of Statement~(1) of Theorem~\ref{thm:main-iter}, we need to show that the solution $(A, \psi)$ is unique in the iteration space. Let $(A', \psi')$ be another solution to MD-CG with the same data, which obeys $\Pi_{s} \psi' \in \tilde{S}^{1/2}_{s}$, $A_{x}' \in S^{1}$ and $A_{0}' \in Y^{1}$. To prove the desired uniqueness, by a simple continuity argument, it is enough show that $(A, \psi) = (A', \psi')$ on $[0, T]$ for some $T = T(\psi(0), A_{x}[0]) > 0$. Moreover, it is clear from MD-CG that $A_{0}' = \bfA_{0}(\psi', \psi')$ and $A_{x}' = A_{x}^{free} + \bfA_{x}(\psi', \psi')$; hence it suffices to establish 
\begin{equation} \label{eq:main-iter-uni}
	\psi(t) = \psi'(t) \quad \hbox{ for } t \in [0, T].
\end{equation}

Define $\dlt \psi = \psi' - \psi$. By Proposition~\ref{prop:parasys-dirac}, we have
\begin{equation} \label{eq:main-iter-uni-1}
	\sup_{s \in \set{+, -}} \nrm{\Pi_{s} \dlt \psi}_{\tilde{S}^{1/2}_{s}[0, T]} 
	\aleq \sup_{s \in \set{+, -}} \nrm{\Pi_{s} \alp^{\mu} \covD_{\mu}^{A} \dlt \psi}_{N_{s}^{1/2} \cap L^{2} L^{2} \cap G^{1/2}[0, T]} 
\end{equation}
Moreover, writing out the equations for $\alp^{\mu} \covD^{A}_{\mu} \dlt \psi$ and $\alp^{\mu} \rd_{\mu} \dlt \psi$ in terms of $\psi$, $\dlt \psi$ and analyzing it as in the proof of \eqref{eq:main-iter-diff}, we arrive at 
\begin{equation*}
	\hbox{RHS of } \eqref{eq:main-iter-uni-1}
	\aleq \bb( \eps_{\ast} + \sup_{s \in \set{+, -}} \nrm{\Pi_{s} \dlt \psi}_{\tilde{S}^{1/2}_{s}[0, T]} \bb)^{2} \sup_{s \in \set{+, -}} \nrm{\Pi_{s} \dlt \psi}_{\tilde{S}^{1/2}_{s}[0, T]} 
\end{equation*}
In particular, the RHS of \eqref{eq:main-iter-uni-1} is finite; hence the LHS of \eqref{eq:main-iter-uni-1} can be made as small as we want by choosing $T$ sufficiently small (we use \eqref{eq:N-fungibility} for $N_{s}^{1/2}$). Combining 
\eqref{eq:main-iter-uni-1} with the preceding estimate, and taking $\eps_{\ast}$ smaller if necessary, we may conclude that $\dlt \psi = 0$ on $[0, T]$ as desired. 


\subsubsection*{Step 3: Weak Lipschitz dependence}
Here we outline the proof of Statement~(2) of Theorem~\ref{thm:main-iter}. Let $\dlt \psi = \psi - \psi'$ and $\dlt A = A - A'$. It is clear from MD-CG that $A_{0}' = \bfA_{0}(\psi', \psi')$ and $A_{x}' = (A_{x}')^{free} + \bfA_{x}(\psi', \psi')$, where $(A_{x}')^{free}$ is the free wave development of $A_{x}'[0]$. Applying Proposition~\ref{prop:a} with appropriate frequency envelopes, we see that establishing \eqref{eq:weak-lip} reduces to showing
\begin{equation} \label{eq:weak-lip-0}
	\sup_{s \in \set{+, -}} \nrm{\Pi_{s} \dlt \psi}_{\tilde{S}_{s}^{1/2 - \dlt_{2}}} \aleq \nrm{\dlt \psi (0)}_{\dot{H}^{1/2 - \dlt_{2}}} + \nrm{\dlt A_{x}[0]}_{\dot{H}^{1 - \dlt_{2}} \times \dot{H}^{-\dlt_{2}}}.
\end{equation}

For simplicity of exposition, we will assume that $\Pi_{s} \dlt \psi \in \tilde{S}^{1/2-\dlt_{2}}_{s}$ and prove \eqref{eq:weak-lip-0}. This assumption can be bypassed by establishing \eqref{eq:weak-lip-0} for the difference $\dlt \psi = \psi^{n} - (\psi')^{n}$ of Picard iterates in Step 1; we omit the details.

The difference $\dlt \psi$ obeys the covariant equation
\begin{equation*}
	\alp^{\mu} \covD_{\mu}^{A} \dlt \psi
	= - i \alp^{\mu} \bb( \bfA_{\mu}(\dlt \psi, \psi) + \bfA_{\mu}(\psi', \dlt \psi) \bb) \psi' - i \alp^{\ell} \dlt A_{\ell}^{free} \psi' =: \dlt I_{1} + \dlt I_{2}.
\end{equation*}
We claim that
\begin{align}
\sup_{s' \in \set{+, -}} \nrm{\Pi_{s'} \dlt I_{1}}_{N_{s'}^{1/2-\dlt_{2}} \cap L^{2} \dot{H}^{-\dlt_{2}} \cap G^{1/2-\dlt_{2}}}
	\aleq & \eps_{\ast}^{2} \sup_{s \in \set{+, -}} \nrm{\Pi_{s} \dlt \psi}_{\tilde{S}^{1/2-\dlt_{2}}_{s}}, 	\label{eq:weak-lip-1} \\
\sup_{s' \in \set{+, -}} \nrm{\Pi_{s'} \dlt I_{2}}_{N_{s'}^{1/2-\dlt_{2}} \cap L^{2} \dot{H}^{-\dlt_{2}} \cap G^{1/2-\dlt_{2}}}
	\aleq & \eps_{\ast} \nrm{\dlt A_{x}[0]}_{\dot{H}^{1-\dlt_{2}} \times \dot{H}^{-\dlt_{2}}}. \label{eq:weak-lip-2} 
\end{align}
Assuming that \eqref{eq:weak-lip-1}--\eqref{eq:weak-lip-2} hold, we may finish the proof as follows. Applying Proposition~\ref{prop:parasys-dirac} with an appropriate frequency envelope, we obtain
\begin{align*}
	\sup_{s \in \set{+, -}} \nrm{\Pi_{s} \dlt \psi}_{S^{1/2-\dlt_{2}}_{s}} 
	\aleq & \nrm{\dlt \psi (0)}_{\dot{H}^{1/2 - \dlt_{2}}}  \\
	& + \sup_{s \in \set{+, -}} \nrm{\Pi_{s} (\alp^{\mu} \covD_{\mu}^{A} \dlt \psi)}_{N_{s}^{1/2-\dlt_{2}} \cap L^{2} \dot{H}^{-\dlt_{2}} \cap G^{1/2-\dlt_{2}}} 
\end{align*}
The last terms can be estimated using \eqref{eq:weak-lip-1}--\eqref{eq:weak-lip-2}. Taking $\eps_{\ast}$ sufficiently small to absorb the contribution of $\nrm{\Pi_{s} \dlt \psi}_{\tilde{S}^{1/2-\dlt_{2}}_{s}}$ (which is finite by assumption) into the LHS, the desired inequality \eqref{eq:weak-lip-0} follows in a straightforward manner. 

It only remains to establish \eqref{eq:weak-lip-1}--\eqref{eq:weak-lip-2}. The proof of \eqref{eq:weak-lip-1} is very similar to that of \eqref{eq:main-iter-diff} in Step 1; we omit the details. To prove \eqref{eq:weak-lip-2}, we start by writing
\begin{equation*}
	\Pi_{s'}\dlt I_{2} = - i \sum_{s} \Pi_{s'}(\alp^{\ell} \dlt A_{\ell}^{free}) \Pi_{s} \psi_{s} = i \sum_{s} \bb( s \Pi_{s'} \ND^{R}(\dlt A^{free}_{x}, \psi_{s}) - \ND_{s}^{S}(\dlt A^{free}_{x}, \psi_{s}) \bb)
\end{equation*}
where $\psi_{s} = \Pi_{s} \psi$. The $L^{2} \dot{H}^{-\dlt_{2}} \cap G^{1/2-\dlt_{2}}$ norm of both terms can be handled by applying \eqref{eq:ne-himod}--\eqref{eq:ns-himod} and \eqref{eq:ne-z}--\eqref{eq:ns-z} with appropriate frequency envelopes. Henceforth, we focus on the $N^{1/2-\dlt_{2}}_{s}$ norm. The term $\ND_{s}^{S}(\dlt A^{free}_{x}, \psi_{s})$ can be treated using \eqref{eq:nrs} and \eqref{eq:diffs}. For the term $\ND^{R}(\dlt A^{free}_{x}, \psi_{s})$, application of \eqref{eq:nrr} and \eqref{eq:diffr-opp} leaves us only with the term $s' \Pi_{s'} (\Diff^{R}[\dlt A_{x}^{free}] \psi_{s'})$. For this term, we apply \eqref{eq:diffr-free} with frequency envelopes $a$ and $b$ for $\nrm{\dlt A_{x}[0]}_{\dot{H}^{1} \times L^{2}}$ and $\nrm{\psi}_{\tilde{S}^{1/2}_{s'}}$, respectively. Observe that $\sum_{k' < k} a_{k'} \aleq 2^{\dlt_{2}} \nrm{\dlt A_{x}[0]}_{\dot{H}^{1-\dlt_{2}} \times \dot{H}^{-\dlt_{2}}}$, so
\begin{equation*}
	\nrm{\Diff^{R}[A_{x}^{free}] \psi_{s'}}_{N^{1/2-\dlt_{2}}_{s'}}
	\aleq \nrm{\dlt A_{x}[0]}_{\dot{H}^{1-\dlt_{2}} \times \dot{H}^{-\dlt_{2}}} \nrm{\psi_{s'}}_{\tilde{S}^{1/2}_{s'}}
\end{equation*}
which is exactly what we need (it is this point where $\dlt_{2} > 0$ is used). 

\subsubsection*{Step~4: Persistence of regularity}
Finally, we sketch the proof of Statement~(3) of Theorem~\ref{thm:main-iter}. In view of Proposition~\ref{prop:subcrit-lwp}, it suffices to show that 
\begin{equation} \label{eq:hi-reg}
	\sup_{s \in \set{+, -}} \nrm{\Pi_{s} \psi}_{\tilde{S}^{1/2+N}_{s}}
	+ \nrm{A_{x}}_{S^{1+N}} \aleq \nrm{\psi(0)}_{\dot{H}^{1/2+N}} + \nrm{A_{x}[0]}_{\dot{H}^{1+N} \times \dot{H}^{N}}
\end{equation}
for $N = 1, 2$, whenever the RHS is finite. Henceforth, we only consider the case $N = 1$; the case $N=2$ can be handled similarly. Moreover, for simplicity, we will already assume that $\Pi_{s} \psi \in \tilde{S}^{1/2+N}_{s}$ and prove \eqref{eq:hi-reg}. As before, this assumption may be bypassed by repeating the proof of \eqref{eq:hi-reg} for each iterate in Step~1.

The proof closely parallels Step~3; both are essentially analysis of the linearized MD-CG. By Proposition~\ref{prop:a} (for $\Box \nb \bfA_{x}$), it suffices to bound only the contribution of $\psi$ in \eqref{eq:hi-reg}. Observe that $\nb \psi$ obeys 
\begin{equation*}
	\alp^{\mu} \covD_{\mu}^{A} \nb \psi = - i \alp^{\mu} \bb( \bfA_{\mu}(\nb \psi, \psi) \psi +  \bfA_{\mu}(\psi, \nb \psi) \bb) \psi - i \alp^{\ell} \nb A^{free}_{\ell} \psi =: I_{1} + I_{2}.
\end{equation*}
We claim that
\begin{align} 
	& \sup_{s' \in \set{+, -}} \nrm{\Pi_{s'} I_{1}}_{N^{1/2}_{s'} \cap L^{2} L^{2} \cap G^{1/2}} \aleq \eps_{\ast}^{2} \sup_{s \in \set{+, -}} \nrm{\Pi_{s} \psi}_{\tilde{S}^{3/2}_{s}}, \label{eq:hi-reg-1} \\
	& \sup_{s' \in \set{+, -}} \nrm{\Pi_{s'} I_{2}}_{N^{1/2}_{s'} \cap L^{2} L^{2} \cap G^{1/2}} \aleq \eps_{\ast} (\sup_{s \in \set{+, -}} \nrm{\Pi_{s} \psi}_{\tilde{S}^{3/2}_{s}} + \nrm{A_{x}[0]}_{\dot{H}^{2} \times \dot{H}^{1}}), \label{eq:hi-reg-2} 
\end{align}
Then by Proposition~\ref{prop:parasys-dirac} and \eqref{eq:hi-reg-1}--\eqref{eq:hi-reg-2}, we would have 
\begin{equation*}
	\sup_{s \in \set{+, -}} \nrm{\Pi_{s} \psi}_{\tilde{S}^{3/2}_{s}}
	\aleq \nrm{\psi(0)}_{\dot{H}^{3/2}} + \eps_{\ast} \nrm{A_{x}[0]}_{\dot{H}^{2} \times \dot{H}^{1}}+ \eps_{\ast} \sup_{s \in \set{+, -}} \nrm{\Pi_{s} \psi}_{\tilde{S}^{3/2}_{s}}.
\end{equation*}
Taking $\eps_{\ast}$ smaller if necessary, we may absorb the last term into the LHS, which would prove \eqref{eq:hi-reg}.

It remains to justify \eqref{eq:hi-reg-1}--\eqref{eq:hi-reg-2}; below we only discuss \eqref{eq:hi-reg-2}, as the other bounds can be proved in a similar fashion to Step~1 (in parallel with Step~3). By \eqref{eq:nrr}--\eqref{eq:nrs}, \eqref{eq:diffr-opp}--\eqref{eq:diffs}, \eqref{eq:nr-himod}--\eqref{eq:ns-himod} and \eqref{eq:nr-z}--\eqref{eq:ns-z}, it is straightforward to show that
\begin{equation*}
	\nrm{\Pi_{s'}(I_{2} + i s' \Diff^{R} [\nb A_{x}^{free}] \psi_{s'})}_{N^{1/2}_{s'} \cap L^{2} L^{2} \cap G^{1/2}} \aleq \eps_{\ast} \nrm{A_{x}[0]}_{\dot{H}^{2} \times \dot{H}^{1}} .
\end{equation*}
Moreover, the $L^{2} L^{2} \cap G^{1/2}$ norm of $s' \Pi_{s'} (\Diff^{R}[\nb A_{x}^{free}] \psi_{s'})$ can be bounded by the same RHS using \eqref{eq:nr-himod} and \eqref{eq:nr-z}. To handle its $N^{1/2}_{s'}$ norm, we apply \eqref{eq:diffr-free} with frequency envelopes $a$ and $b$ for $\nrm{\nb A_{x}[0]}_{\dot{H}^{1} \times L^{2}}$, $\nrm{\psi}_{\tilde{S}^{1/2}_{s'}}$, respectively. For any $0 < \dlt < \dlt_{1}$, we have 
\begin{equation*}
	\nrm{P_{k} \Pi_{s'} (\Diff^{R}[\nb A_{x}^{free}] \psi_{s'})}_{N^{1/2}_{s'}} \leq (\sum_{k' < k} a_{k'}) b_{k} \leq 2^{\dlt k} \nrm{\nb A_{x}[0]}_{\dot{H}^{1-\dlt} \times \dot{H}^{-\dlt}} b_{k}
\end{equation*}
Square summing over $k$, we see that the $N_{s'}^{1/2}$ norm of $\Pi_{s'} (\Diff^{R}[\nb A_{x}^{free}] \psi_{s'})$ is bounded by $\nrm{\nb A_{x}[0]}_{\dot{H}^{1-\dlt} \times \dot{H}^{-\dlt}} \nrm{\psi_{s'}}_{\tilde{S}^{1/2+\dlt}_{s'}}$. By a simple interpolation, the desired bound \eqref{eq:hi-reg-2} follows.

\subsection{Solvability of the covariant Dirac equation} \label{subsec:iter-dirac}
We now complete the proof of Theorem~\ref{thm:main-iter} by proving Proposition~\ref{prop:parasys-dirac}, employing all estimates stated in Section~\ref{sec:main-est}. 

To solve \eqref{eq:parasys-dirac}, we introduce an auxiliary equation (see \eqref{eq:projsys} below), which on one hand reduces to \eqref{eq:parasys-dirac} after suitable manipulation, and on the other hand possess appropriate structure so that it could be solved via an iteration argument. 
More precisely, we look for a pair $(\varphi_{+}, \varphi_{-})$ of spinor fields which obeys
\begin{equation} \label{eq:projsys}
\begin{aligned}
	(i \rd_{t} + s \abs{D}) \varphi_{s}
	=& \ND^{E}(A_{0}, \Pi_{+} \varphi_{+}) + \ND^{E}(A_{0}, \Pi_{-} \varphi_{-}) + \Pi_{-s} (\pi^{E}[A_{0}] \varphi_{s}) \\
	& - \ND^{R}(A_{x}, \varphi_{+}) + \ND^{R}(A_{x}, \varphi_{-}) \\
	& + \Pi_{s} \ND^{S}_{+}(A_{x}, \varphi_{+}) + \Pi_{s} \ND^{S}_{-}(A_{x}, \varphi_{-}) + i \Pi_{s} F.
\end{aligned}
\end{equation}
with $\varphi_{s} (0) = \Pi_{s} \psi(0)$ for $s \in \set{+, -}$. 

Taking $\Pi_{s}$ of both sides, a computation similar to Lemma~\ref{lem:cov-dirac-HW} shows that $\psi = \Pi_{+} \varphi_{+} + \Pi_{-} \varphi_{-}$ solves the desired covariant Dirac equation; a key observation here is that the last term on the first line vanishes. Therefore, in order to establish the existence statement in Proposition~\ref{prop:parasys-dirac}, it suffices to show that, under the hypotheses of Proposition~\ref{prop:parasys-dirac}, there exists a solution $(\varphi_{+}, \varphi_{-})$ to \eqref{eq:projsys} obeying
\begin{equation} \label{eq:projsys-est}
	\nrm{\varphi_{s}}_{(\tilde{S}^{1/2}[I])_{c}} \aleq \nrm{\Pi_{s} \psi_{0}}_{\dot{H}^{1/2}_{c}} + \nrm{\Pi_{s} F}_{(N_{s}^{1/2} \cap L^{2} L^{2} \cap G^{1/2}[I])_{c}} .
\end{equation}
Our goal in the remainder of this subsection is to prove the preceding statement. The remaining uniqueness statement in Proposition~\ref{prop:parasys-dirac} follows by a similar argument applied to $\Pi_{s} \eqref{eq:projsys}$; we omit the repetitive details.

Before analyzing \eqref{eq:projsys}, we begin with some simple remarks. First, extending $\Pi_{s} F$ by zero outside of $I$ results in an equivalent $N_{s}^{1/2} \cap L^{2} L^{2} \cap G^{1/2}$ norm (see Lemma~\ref{lem:int-loc} and the preceding discussion); therefore, it suffices to focus on the case $I = \bbR$. Next, by Proposition~\ref{prop:a} (note that $\rd_{t} A_{0} = \rd_{t} \NM^{E}(\psi', \psi')$ thanks to the hypothesis $\rd_{\mu} \brk{\psi', \alp^{\mu} \psi'} = 0$), $A$ obeys the following bound: Given an admissible frequency envelope $b$ with $\sup_{s \in \set{+, -}} \nrm{\Pi_{s} \psi'}_{(\tilde{S}^{1/2}_{s})_{b}} \leq 1$, we have
\begin{equation} \label{eq:projsys-a}
	\nrm{A_{0}}_{Y^{1}_{b^{2}}}
	+ \nrm{A_{x} - A^{free}_{x}}_{S^{1}_{b^{2}}} \aleq 1.
\end{equation}
Constructing $b$ as in \eqref{eq:fe-construct}, we have $\nrm{b^{2}}_{\ell^{1}}\leq \nrm{b}_{\ell^{2}}^{2} \aleq \eps_{\ast \ast}^{2}$ by hypothesis.

We are now ready to begin the analysis of \eqref{eq:projsys}. Using the decomposition in Section~\ref{subsec:nonlin} and the identity
\begin{align*}
	\pi^{E}[A_{0}] \Pi_{s} \varphi_{s} + \Pi_{-s} \pi^{E}[A_{0}] \varphi_{s}
	= \pi^{E}[A_{0}] (1 - \Pi_{-s}) \varphi_{s} + \Pi_{-s} \pi^{E}[A_{0}] \varphi_{s},
\end{align*}
the system \eqref{eq:projsys} can be rewritten as $(i \rd_{t} + s \abs{D})_{A^{free}}^{p} \varphi_{s}  = \calE_{s} \varphi + i \Pi_{s} F$, 
where
\begin{align}
\calE_{s} \varphi
= \calE_{s}[A^{free}, \psi'] \varphi
	= & \pi^{E}[\bfA_{0}(\psi', \psi')] \varphi_{s} -s \pi^{R}[\bfA_{x}(\psi', \psi')] \varphi_{s} 		\label{eq:projsys-1} \\
	& + \NR^{E}(A_{0}, \Pi_{s} \varphi_{s}) -s \NR^{R}(A_{x}, \varphi_{s}) 				\label{eq:projsys-2} \\
	& + \ND^{E}(A_{0}, \Pi_{-s} \varphi_{-s}) + s \ND^{R}(A_{x}, \varphi_{-s}) 				\label{eq:projsys-3} \\
	& + \Pi_{s} \ND^{S}_{+}(A_{x}, \varphi_{+}) + \Pi_{s} \ND^{S}_{-}(A_{x}, \varphi_{-}) 		\label{eq:projsys-4} \\
	& + [\Pi_{-s}, \pi^{E}[A_{0}]] \varphi_{s}.										\label{eq:projsys-5}
\end{align}
For any admissible frequency envelope $c$ and $\varphi' = (\varphi_{+}', \varphi_{-}') \in (\tilde{S}^{1/2}_{+} \times \tilde{S}^{1/2}_{-})_{c}$, we claim that
\begin{equation} \label{eq:iter-dirac-key}
	\nrm{\calE_{s} \varphi'}_{(N_{s}^{1/2} \cap L^{2} L^{2} \cap G^{1/2})_{c}} 
	\aleq \eps_{\ast \ast} \sup_{s \in \set{+, -}} \nrm{\varphi_{s}'}_{(\tilde{S}^{1/2}_{s})_{c}}.
\end{equation}
For the moment, we assume the claim and complete the proof. 
Let $\varphi' = (\varphi'_{+}, \varphi'_{-}) \in (\tilde{S}^{1/2}_{+} \times \tilde{S}^{1/2}_{-})_{c}$, and consider a solution $\varphi$ to 
\begin{equation*}
(i \rd_{t} + s \abs{D})_{A^{free}}^{p} \varphi_{s}  = \calE_{s} \varphi' + i \Pi_{s} F
\end{equation*}
given by Theorem~\ref{thm:paradiff}. By the same theorem and \eqref{eq:iter-dirac-key}, we have
\begin{equation*}
	\nrm{\varphi_{s}}_{S^{1/2}_{c}} \aleq \eps_{\ast \ast} \sup_{s \in \set{+, -}} \nrm{\varphi'_{s}}_{(\tilde{S}^{1/2}_{s})_{c}} + \nrm{\varphi(0)}_{\dot{H}^{1/2}} + \nrm{\Pi_{s} F}_{(N^{1/2}_{s} \cap L^{2} L^{2})_{c}}.
\end{equation*}
Combined with the inequality
\begin{equation*}
\nrm{\varphi}_{(\tilde{Z}^{1/2}_{s})_{c}} = \nrm{(i \rd_{t} + s \abs{D}) \varphi}_{G^{1/2}_{c}} \leq \nrm{(i \rd_{t} + s \abs{D})_{A^{free}}^{p} \varphi}_{G^{1/2}_{c}} + \nrm{\Diff^{R}[A^{free}_{x}] \varphi}_{G^{1/2}_{c}}
\end{equation*}
and \eqref{eq:nr-z} (which only involves the $S^{1/2}_{s}$ norm on the RHS), we have
\begin{equation*}
	\nrm{\varphi_{s}}_{\tilde{S}^{1/2}_{c}} \aleq \eps_{\ast \ast} \sup_{s \in \set{+, -}} \nrm{\varphi'_{s}}_{(\tilde{S}^{1/2}_{s})_{c}} + \nrm{\varphi(0)}_{\dot{H}^{1/2}} + \nrm{\Pi_{s} F}_{(N^{1/2}_{s} \cap L^{2} L^{2} \cap G^{1/2})_{c}}.
\end{equation*}
Taking $\eps_{\ast \ast} > 0$ sufficiently small, we may ensure that the map $\varphi' \mapsto \varphi$ is a contraction in $(\tilde{S}_{+}^{1/2} \times \tilde{S}_{-}^{1/2})_{c}$. By iteration (or Banach fixed point theorem), we may then obtain the desired solution $\varphi$ to \eqref{eq:projsys}.

Now it only remains to prove \eqref{eq:iter-dirac-key}. For \eqref{eq:projsys-1}, we use Proposition~\ref{prop:trilinear} with appropriate frequency envelopes. For \eqref{eq:projsys-2}--\eqref{eq:projsys-4}, we apply Proposition~\ref{prop:dirac} and \eqref{eq:projsys-a}.
Finally, \eqref{eq:projsys-5} is handled using \eqref{eq:projsys-a} and the following lemma.
\begin{lemma} \label{lem:diffe-comm}
Let $a$, $b$ be any admissible frequency envelopes, and $s \in \set{+, -}$. Then we have
\begin{equation} \label{eq:diffe-comm}
	\nrm{[\Pi_{-s}, \Diff^{E}[A_{0}]] \psi}_{(N_{s}^{1/2} \cap L^{2} L^{2} \cap G^{1/2})_{ab}} \aleq \nrm{A_{0}}_{Y^{1}_{a}} \nrm{\psi}_{(\tilde{S}^{1/2}_{s})_{b}}
\end{equation}
\end{lemma}
\begin{proof}
By \eqref{eq:ne-himod} and \eqref{eq:ne-z}, \eqref{eq:diffe-comm} holds for the $L^{2} L^{2} \cap G^{1/2}$ norm on the LHS even without the commutator structure; hence it remains to show
\begin{equation} \label{eq:diffe-comm-key}
	\nrm{[\Pi_{-s}, \Diff^{E}[A_{0}]] \psi}_{(N_{s}^{1/2})_{ab}} \aleq \nrm{A_{0}}_{Y^{1}_{a}} \nrm{\psi}_{(\tilde{S}^{1/2}_{s})_{b}}
\end{equation}
Write $A_{k} = P_{k} A_{0}$, $\psi_{k} = P_{k} \psi$ and $\tilde{P}_{k} := \Pi_{-s_{0}} P_{k}$, so that
\begin{equation*}
[\Pi_{-s_{0}}, \Diff^{E}[A_{0}]] \psi = \sum_{k', k_{1}, k : k_{1} < k-5} [\tilde{P}_{k'}, A_{k_{1}}] \psi_{k}.
\end{equation*}
Observe that the summand vanishes unless $k' = k + O(1)$. Moreover, we have the well-known commutator identity
\begin{equation*}
	[\tilde{P}_{k'}, A_{k_{1}}] f = 2^{-k'}\calL(\nb A_{k_{1}}, f)
\end{equation*}
where $\calL$ is a translation-invariant bilinear operator with bounded mass kernel (see \cite[Lemma 2]{Tao2}). Applying Lemma~\ref{lem:ellip} from Section~\ref{subsec:ellip} below, we have
\begin{align*}
\nrm{[\tilde{P}_{k'}, A_{k_{1}}] \psi_{k}}_{N_{s_{0}}^{1/2}}
\aleq & 2^{-k} \nrm{\calL(\nb A_{k_{1}}, \psi_{k})}_{N_{s_{0}}^{1/2}} \\
\aleq & 2^{- \frac{1}{2} k} \nrm{\nb A_{k_{1}}}_{L^{2} L^{2}} \bb(\sum_{\calC_{k_{1}}(0)}\nrm{P_{\calC_{k_{1}}(0)} \psi_{k}}_{L^{2} L^{\infty}}^{2} \bb)^{1/2} \\
\aleq & 2^{\frac{1}{2}(k_{1} - k)} \nrm{A_{k_{1}}}_{Y^{1}} \nrm{\psi_{k}}_{S^{1/2}_{s_{0}}}.
\end{align*}
Thanks to the gain $2^{\frac{1}{2}(k_{1} - k)}$, the frequency envelope bound \eqref{eq:diffe-comm-key} follows. \qedhere
\end{proof}

\subsection{Proof of continuous dependence on data}
Here we prove Statement~(2) of Theorem~\ref{thm:main}. Along the way, we also show that every solution obtained by Theorem~\ref{thm:main-iter} arises as an approximation by smooth solutions.

Let $\psi(0) \in \dot{H}^{1/2}$, $A_{x}[0] \in \dot{H}^{1} \times L^{2}$ be an initial data set for MD-CG. Given $m \in \bbZ$, let $\psi^{(m)}(0), A_{x}^{(m)}[0]$ be the regularization $\psi^{(m)}(0) = P_{\leq m} \psi(0)$, $A_{x}^{(m)}[0] = P_{\leq m} A_{x}[0]$. Denote by $(A, \psi)$ [resp. $(A^{(m)}, \psi^{(m)})$] the solution with the data $\psi(0), A_{x}[0]$ [resp. $\psi^{(m)}(0), A_{x}^{(m)}[0]$ ] given by Theorem~\ref{thm:main-iter}.
\begin{lemma}[Approximation by smooth solutions] \label{lem:smth-data}
Let $c$ be an admissible frequency envelope for $\psi(0)$, $A_{x}[0]$. In the above setting, we have
\begin{equation*}
	\sup_{s \in \set{+, -}} \nrm{\Pi_{s} (\psi - \psi^{(m)})}_{\tilde{S}^{1/2}_{s}} +  \nrm{A_{x} - A_{x}^{(m)}}_{S^{1}} + \nrm{A_{0} - A_{0}^{(m)}}_{Y^{1}}  \aleq \bb( \sum_{k > m} c_{k}^{2} \bb)^{1/2}.
\end{equation*}
\end{lemma}
\begin{proof}
Let $c$ be an admissible frequency envelope for $(\psi(0), A_{x}[0])$; observe that it is also a frequency envelope for $(\psi^{(m)}(0), A_{x}^{(m)}[0])$. Applying the frequency envelope bound \eqref{eq:main-iter-fe} to $(A, \psi)$ and $(A^{(m)}, \psi^{(m)})$ separately, the above estimate follows for $P_{>m} (\psi - \psi^{(m)})$ and $P_{>m} (A - A^{(m)})$. On the other hand, for $P_{\leq m} (\psi - \psi^{(m)})$ and $P_{\leq m} (A - A^{(m)})$ we use weak Lipschitz continuity \eqref{eq:weak-lip}. Observe that
\begin{align*}
	\nrm{P_{\leq m} \Pi_{s} (\psi - \psi^{(m)})}_{\tilde{S}^{1/2}_{s}}  
	\aleq & 2^{\dlt_{2} m} \nrm{P_{\leq m} \Pi_{s} (\psi - \psi^{(m)})}_{\tilde{S}^{1/2-\dlt_{2}}_{s}} \\
	\aleq & 2^{\dlt_{2} m} (\nrm{P_{>m} \psi(0)}_{\dot{H}^{1/2-\dlt_{2}}} + \nrm{P_{>m} A_{x}[0]}_{\dot{H}^{1} \times L^{2}}),
\end{align*}
where the last line is bounded by $(\sum_{k > m} c_{k}^{2})^{1/2}$. Combined with similar observations for $A_{x} - A_{x}^{(m)}$ in $S^{1}$ and $A_{0} - A_{0}^{(m)}$ in $Y^{1}$, the lemma follows. \qedhere
\end{proof}

We are now ready to prove Statement~(2) of Theorem~\ref{thm:main}. Let $\psi^{n}(0), A^{n}_{x}[0]$ be a sequence of initial data sets for MD-CG such that $\psi^{n}(0) \to \psi(0)$ in $\dot{H}^{1/2}$ and $A^{n}_{x}[0] \to A_{x}[0]$ in $\dot{H}^{1} \times L^{2}$. Denote by $(A^{n}, \psi^{n})$ the corresponding solution to MD-CG, which exists for large $n$ by Theorem~\ref{thm:main-iter}. For any $\eps > 0$, we claim that
\begin{equation} \label{eq:cont}
	\sup_{s \in \set{+, -}} \nrm{\Pi_{s} (\psi^{n} - \psi)}_{\tilde{S}^{1/2}_{s}} +  \nrm{A^{n}_{x} - A_{x}}_{S^{1}}
	< \eps
\end{equation}
for sufficiently large $n$. The desired continuity statement is equivalent to this claim.

Let $c$ be an admissible frequency envelope for $(\psi(0), A_{x}[0])$. Applying Lemma~\ref{lem:smth-data}, we may find $m \in \bbZ$ such that for sufficiently large $n$,
\begin{equation} \label{eq:cont-hi}
\begin{aligned}
	\sup_{s \in \set{+, -}} \nrm{\Pi_{s} (\psi - \psi^{(m)})}_{\tilde{S}^{1/2}_{s}} +  \nrm{A_{x} - A_{x}^{(m)}}_{S^{1}}
	< & \frac{1}{4} \eps, \\
	\sup_{s \in \set{+, -}} \nrm{\Pi_{s} (\psi^{n} - \psi^{n(m)})}_{\tilde{S}^{1/2}_{s}} +  \nrm{A^{n}_{x} - A_{x}^{n (m)}}_{S^{1}} 
	< & \frac{1}{4} \eps,
\end{aligned}
\end{equation}
where $(A^{n (m)}, \psi^{n (m)})$ is defined in the obvious manner. By persistence of regularity and Proposition~\ref{prop:subcrit-lwp}, we have (as $n \to \infty$)
\begin{equation*}
	\nrm{(\psi^{n (m)} - \psi^{(m)})(t)}_{C_{t} ([0, T]; H^{1/2, 5/2})}
	+ \nrm{(A_{x}^{n (m)} - A_{x}^{(m)})[t]}_{C_{t} ([0, T]; \calH^{1, 3})} \to 0.
\end{equation*}
Reiterating the preceding bound in MD-CG, we also obtain (as $n \to \infty$)
\begin{equation*}
	\nrm{\alp^{\mu} \rd_{\mu} (\psi^{n (m)} - \psi^{(m)})}_{C_{t} ([0, T]; H^{1/2, 5/2})}
	+ \nrm{\Box (A_{x}^{n (m)} - A_{x}^{(m)})}_{C_{t} ([0, T]; H^{0, 2})} \to 0.
\end{equation*}
In a straightforward manner, the preceding two statements imply
\begin{equation*}
	\sup_{s \in \set{+, -}} \nrm{\Pi_{s}(\psi^{n (m)} - \psi^{(m)})}_{\tilde{S}_{s}^{1/2}[0, T]}
	+ \nrm{A_{x}^{n (m)} - A_{x}^{(m)}}_{S^{1}[0, T]} < \frac{1}{2} \eps
\end{equation*}
for sufficiently large $n$. Combined with \eqref{eq:cont-hi}, the desired conclusion \eqref{eq:cont} follows.

\subsection{Proof of modified scattering}
Here we conclude the proof of Theorem~\ref{thm:main} by sketching the proof of Statement~(3). Without loss of generality, we fix $\pm = +$.

Let $(A, \psi)$ be a solution to MD-CG with data $(\psi(0), A_{x}[0])$ given by Theorem~\ref{thm:main-iter}, and let $A^{free}_{x}$ denote the free wave development of $A_{x}[0]$. To prove modified scattering for $\psi$, we first decompose the covariant Dirac equation into
\begin{equation*}
	\alp^{\mu} \covD_{\mu}^{A^{free}} \psi = - i \alp^{\mu} \bfA_{\mu}(\psi, \psi) \psi.
\end{equation*}
For any $t < t'$, Proposition~\ref{prop:parasys-dirac} implies that
\begin{align*}
\nrm{\psi (t') - S^{A^{free}}(t', t) \psi(t)}_{\dot{H}^{1/2}} 
\aleq \sup_{s \in \set{+, -}} \nrm{\Pi_{s} (\alp^{\mu} \bfA_{\mu}(\psi, \psi) \psi)}_{(N_{s}^{1/2} \cap L^{2} L^{2} \cap G^{1/2})[t, \infty)},
\end{align*}
where $S^{A^{free}}(t', t)$ denotes the propagator from time $t$ to $t'$ for the covariant Dirac equation $\alp^{\mu} \covD_{\mu}^{A^{free}} \varphi = 0$. An analysis as in Section~\ref{subsec:iter-dirac} using Propositions~\ref{prop:a} and \ref{prop:dirac} shows that the RHS is finite for (say) $t = 0$; by \eqref{eq:N-fungibility}, it follows that the RHS vanishes as $t \to \infty$. Using the uniform boundedness of $S^{A^{free}}(0, t')$ on $\dot{H}^{1/2}$ (again by Proposition~\ref{prop:parasys-dirac}), as well as the formula $S^{A^{free}}(t'', t) = S^{A^{free}}(t'', t') S^{A^{free}}(t', t)$, it follows that (as $t \to \infty$)
\begin{equation*}
\nrm{S^{A^{free}}(0, t')\psi (t') - S^{A^{free}}(0, t) \psi(t)}_{\dot{H}^{1/2}}  
\aleq \nrm{\psi (t') - S^{A^{free}}(t', t) \psi(t)}_{\dot{H}^{1/2}}  \to 0.
\end{equation*}
Hence $\lim_{t \to \infty} S^{A^{free}}(0, t)\psi (t)$ tends to some limit $\psi^{\infty}(0)$ in $\dot{H}^{1/2}$, which is precisely the data for $\psi^{\infty}$ in Theorem~\ref{thm:main}.

The proof of scattering for $A_{x}$ is more standard and straightforward. In fact, since $\nrm{\NM_{x}(\psi, \psi)}_{\ell^{1} (N \cap L^{2} \dot{H}^{-1/2})[0, \infty)} < \infty$ by Proposition~\ref{prop:a}, $\lim_{t \to \infty} S[0, t] A_{x}[t]$ tends to a limit $A_{x}^{\infty}[0]$ in $\ell^{1} (\dot{H}^{1} \times L^{2})$; here $S[t', t]$ denotes the propagator for the free wave equation. In particular, we have $A_{x}[0] - A_{x}^{\infty}[0] \in \ell^{1} (\dot{H}^{1} \times L^{2})$; this fact, combined with \eqref{eq:diffr-free} and the preceding argument for $\psi$, allows us to replace $A^{free}$ by $A^{\infty}$ as claimed in Theorem~\ref{thm:main}. We leave the details to the reader.
\begin{remark} \label{rem:hi-d-2}
In a general dimension $d \geq 4$, all arguments in this section apply with substitutions as in Remark~\ref{rem:hi-d-1}.
\end{remark}

\section{Interlude: Bilinear null form estimates} \label{sec:nf}
The remainder of this paper is devoted to the proof of the estimates stated in Section~\ref{sec:main-est}.
In this section, we present a few concepts and basic techniques for carrying out these proofs. 

In Section~\ref{subsec:geom-cone}, we consider a bilinear operator and investigate its vanishing property based on the Fourier supports of the inputs and the output; this is fundamental for orthogonality arguments that we employ later. As an immediate application, in Section~\ref{subsec:ellip} we state and prove a H\"older-type inequality involving box localization. Next, we introduce the notion of an \emph{abstract null form} in Section~\ref{subsec:abs-nf}, which is used in Section~\ref{subsec:MD-nf} to express the null structure of the Maxwell--Dirac system in the Coulomb gauge in a unified fashion. Finally, in Section~\ref{subsec:core-nf}, we state and prove \emph{core bilinear estimates}, which concern the `resonant' case (i.e., when the inputs and the output all have low modulation), when the null form is most useful. 

In this section, we consider the general case of $\bbR^{1+d}$ with any $d \geq 2$, although much of our discussion would be fully useful only in $d \geq 4$. By a \emph{universal} constant, we mean that it depends only on the dimension $d$. We denote by $f, g, h, \ldots$ functions which may take values in a vector space (e.g., $f$ may stand for a spinor field $\psi$, or a real-valued spatial 1-form $A_{x}$ etc.). Accordingly, multilinear operators are assumed to take in and output vector-valued functions. These features would be inessential for the proof of the estimates, and for practical purposes the reader may assume that all functions and multilinear operators are scalar-valued.

\subsection{Orthogonality and geometry of the cone} \label{subsec:geom-cone}
Let $\BL$ be a translation-invariant bilinear operator on $ \bbR^{d}$ or  $\bbR^{1+d}$ with symbol $ m(\xi_1,\xi_2) $, respectively $m(\Xi^{1}, \Xi^{2})$ (which is possibly a distribution), i.e.,
\begin{equation*}
	\BL(f_{1}, f_{2}) (x) = \int e^{i  x \cdot (\xi_{1} + \xi_{2})} m(\xi_1,\xi_2) \hat{f}_{1}(\xi_{1}) \hat{f}_{2}(\xi_{2}) \, \frac{\ud \xi_{1} \, \ud \xi_{2}}{(2 \pi)^{2 d}}.
\end{equation*}
The operator $\calL$ introduced in Section~\ref{subsec:notation} can be written in this form by defining
$$ m(\xi_1,\xi_2)= \hat{K} (\xi_1,\xi_2). $$
Conversely, $ L $ can be written in the form \eqref{transinvop}, if we ensure that $ K \in L^1 $ or that it has bounded mass.
An important example will be provided in Definition~\ref{def:abs-nf} below.

To understand $\BL(f_{1}, f_{2})$, we may consider the `dualized'\footnote{We have chosen to use a pairing different from the standard one $\int f \overline{g}$ to have complete symmetry among $f_{0}, f_{1}, f_{2}$. In the vector-valued case, we multiply $\BL(f_{1}, f_{2})$ by the transpose $f_{0}^{\dagger}$ instead of the hermitian transpose $f_{0}^{\ast} = \overline{f_{0}}^{\dagger}$.} expression
\begin{equation} \label{eq:L-expr}
	\iint f_{0} \BL(f_{1}, f_{2}) \, \ud t \ud x
	= \int_{\set{\Xi^{0} + \Xi^{1} + \Xi^{2} = 0}} m(\Xi^{1}, \Xi^{2}) \hat{f}_{0}(\Xi^{0}) \hat{f}_{1}(\Xi^{1}) \hat{f}_{2}(\Xi^{2}) \, \frac{\ud \Xi^{1} \, \ud \Xi^{2}}{(2 \pi)^{2(d+1)}}.
\end{equation}
In view of performing summation arguments later on, we present below various `orthogonality' statements concerning the vanishing property of the expression \eqref{eq:L-expr} based on the Fourier supports of $f_{i}$ $(i=0,1,2)$. 

Given a triple $k_{0}, k_{1}, k_{2} \in \bbR$, we denote by $k_{\min}$, $k_{\med}$ and $k_{\max}$ the minimum, median and maximum of $k_{0}, k_{1}, k_{2}$. If $f_{i} = P_{k_{i}} f_{i}$, then \eqref{eq:L-expr} vanishes unless the maximum and the median of $k_{0}, k_{1}, k_{2}$ (i.e., the two largest numbers) are apart by at most (say) $5$; this is the standard Littlewood-Paley trichotomy. We furthermore have the following refinement, which is useful when $k_{\min}$ is very small compared to $k_{\max}$:
\begin{lemma} \label{lem:box-orth-0}
Let $k_{0}, k_{1}, k_{2} \in \bbZ$ be such that $\abs{k_{\med} - k_{\max}} \leq 5$. 
For $i=0, 1, 2$, let $\calC^{i}$ be a cube of the form $\calC_{k_{\min}}(0)$ (i.e., of dimension $2^{k_{\min}} \times \cdots \times 2^{k_{\min}}$) situated in $\set{\abs{\xi} \simeq 2^{k_{i}}}$.
\begin{enumerate}[leftmargin=*]
\item Then the expression 
\begin{equation} \label{eq:box-orth-0-expr}
	\iint P_{\calC^{0}} h_{k_{0}} \, \BL(P_{\calC^{1}} f_{k_{1}}, P_{\calC^{2}} g_{k_{2}}) \, \ud t \ud x
\end{equation}
vanishes unless $\calC^{0}+ \calC^{1} + \calC^{2} \ni 0$. 

\item If $\calC^{0}+ \calC^{1} + \calC^{2} \ni 0$, then the cubes situated in the non-minimal frequency annuli are almost diametrically opposite. More precisely, we have
\begin{equation*}
	\abs{\angle(\calC^{i}, -\calC^{i'})} \aleq 2^{k_{\min} - k_{\max}},
\end{equation*}
where $k_{i}, k_{i'}$ $(i \neq i')$ are the median and maximal frequencies.
\item Without loss of generality, assume that $k_{0}$ is non-minimal, i.e., $k_{0} = k_{\med}$ or $k_{\max}$.
For any fixed cube $\calC^{0}$ of the form $\calC_{k_{\min}}(0)$ situated in $\set{\abs{\xi} \simeq 2^{k_{0}}}$, there are only (uniformly) bounded number of cubes $\calC^{1}, \calC^{2}$ of the form $\calC_{k_{\min}}(0)$ in $\set{\abs{\xi} \simeq 2^{k_{1}}}, \set{\abs{\xi} \simeq 2^{k_{2}}}$ such that $\calC^{0}+ \calC^{1} + \calC^{2} \ni 0$. 
\end{enumerate}
\end{lemma}
For the notation $\calC^{0} + \calC^{1} + \calC^{2}$ and $\abs{\angle(\calC, \calC')}$, we refer the reader to Section~\ref{subsec:notation}. 
\begin{proof}
Statement~(1) is obvious from the Fourier space representation of \eqref{eq:box-orth-0-expr}.
For the proof of Statements~(2) and (3), we assume without loss of generality that $k_{2} = k_{\min}$. Since $\calC^{0} + \calC^{1} + \calC^{2} \ni 0$, there exists $\xi^{i} \in \calC^{i}$ $(i = 0, 1, 2)$ forming a triangle, i.e., $\sum_{i} \xi^{i} = 0$. By the law of cosines,
\begin{equation*}
	\abs{\xi^{0}}^{2} + \abs{\xi^{1}}^{2} - 2 \abs{\xi^{0}} \abs{\xi^{1}} \cos \angle(\xi^{0}, - \xi^{1}) = \abs{\xi^{2}}^{2}.
\end{equation*}
Rearranging terms, we see that
\begin{equation*}
	2 \abs{\xi^{0}} \abs{\xi^{1}} (1 - \cos \angle(\xi^{0}, - \xi^{1})) = \abs{\xi^{2}}^{2} - (\abs{\xi^{0}} - \abs{\xi^{1}})^{2}.
\end{equation*}
The LHS is comparable to $2^{2 k_{\max}} \abs{\angle(\xi^{0}, -\xi^{1})}$, whereas the RHS is bounded from above by $\aleq 2^{2 k_{\min}}$. Statement~(2) now follows.

It remains to establish Statement~(3). Since there are only bounded number of cubes $\calC_{k_{\min}}(0)$ in $\set{\abs{\xi} \simeq 2^{k_{\min}}}$, the desired statement for $\calC^{2}$ follows. Observing that $\calC^{0} + \calC^{2}$ is contained in a cube of dimension $\aleq 2^{k_{\min}}$, we see that there are only bounded number of cubes $\calC^{1} = \calC_{k_{\min}}(0)$ such that $\calC^{0} + \calC^{2} \cap (- \calC^{1}) \neq \0$, or equivalently, $\calC^{0} + \calC^{1} + \calC^{2} \ni 0$. \qedhere
\end{proof}

Next, we consider the case when, in addition to frequency, the modulation of $f_{i}$ is localized as well, i.e., $f_{i} = P_{k_{i}} Q_{j_{i}} f_{i}$ for some $k_{i}, j_{i} \in \bbZ$ $(i=0, 1, 2)$. From the triple $j_{0}, j_{1}, j_{2}$, we define $j_{\min}$, $j_{\med}$ and $j_{\max}$ as before. The analogue of the Littlewood-Paley trichotomy does \emph{not} hold for modulations; it is possible that \eqref{eq:L-expr} does not vanish while $j_{\max}$ is much larger than $j_{\med}$. However, modulation localization forces certain angular conditions among the spatial Fourier supports of $f_{i}$. An excellent discussion on this subject can be found in \cite[Section~13]{Tao2}. 
Thanks to the fact that we are in dimension $d \geq 4$, we only need the following simple statement.

\begin{lemma}[Geometry of the cone] \label{lem:geom-cone}
Let $k_{0}$, $k_{1}$, $k_{2}$, $j_{0}, j_{1}, j_{2} \in \bbZ$ be such that $\abs{k_{\med} - k_{\max}} \leq 5$. For $i=0,1,2$, let $\omg_{i} \subseteq \bbS^{d-1}$ be an angular cap of radius $0 < r_{i} < 2^{-5}$ and let $f_{i}$ have Fourier support in the region $\set{\abs{\xi} \simeq 2^{k_{i}}, \, \frac{\xi}{\abs{\xi}} \in \omg_{i}, \, \abs{\tau - s_{i} \abs{\xi}} \simeq 2^{j_{i}}}$. Then there exists a constant $C_{0} > 0$ such that the following statements hold:
\begin{enumerate}[leftmargin=*]
\item Suppose that $j_{\max} \leq k_{\min} + C_{0}$. Define $\ell := \frac{1}{2} (j_{\max} - k_{\min})_{-}$. Then the expression $\iint f_{0} \BL(f_{1}, f_{2}) \, \ud t \ud x$ vanishes unless 
\begin{equation} \label{eq:geom-cone}
	\abs{\angle(s_{i} \omg_{i}, s_{i'} \omg_{i'})} \aleq 2^{k_{\min} - \min\set{k_{i}, k_{i'}}} 2^{\ell} + \max \set{r_{i}, r_{i'}} 
\end{equation}
for every pair $i, i' \in \set{0, 1, 2}$ $(i \neq i')$. 
\item Suppose that $j_{\med} \leq j_{\max} - 5$. Then the expression $\iint f_{0} \BL(f_{1}, f_{2}) \, \ud t \ud x$ vanishes unless either $j_{\max} = k_{\max} + O(1)$ or $j_{\max} \leq k_{\min} + \frac{1}{2}C_{0}$.
\end{enumerate}
\end{lemma}
\begin{proof}
If the expression does not vanish, there exists $\Xi^{i} = (\tau^{i}, \xi^{i}) \in \set{\abs{\xi} \simeq 2^{k_{i}}, \, \frac{\xi}{\abs{\xi}} \in \omg_{i}, \, \abs{\tau - s_{i} \abs{\xi}} \simeq 2^{j_{i}}}$ $(i=0,1,2)$ such that $\sum_{i=0,1,2} \Xi^{i}= 0$. 
Without loss of generality, assume that $\abs{\xi^{2}} \leq \abs{\xi^{0}}, \abs{\xi^{1}}$. Observe that $\abs{\xi^{0}}, \abs{\xi^{1}} = k_{\max} + O(1)$, whereas $\abs{\xi^{2}} = k_{\min} + O(1)$.

\pfstep{Statement~(1)} Consider the quantity
\begin{equation*}
	H := s_{0} \abs{\xi^{0}} + s_{1} \abs{\xi^{1}} + s_{2} \abs{\xi^{2}}.
\end{equation*}
On one hand, using $\sum_{i} \tau^{i} = 0$, observe that
\begin{equation} \label{eq:cone-H-j}
	\abs{H} = \abs{(\tau^{0} - s_{0} \abs{\xi^{0}}) + (\tau^{1} - s_{1} \abs{\xi^{1}}) + (\tau^{2} - s_{2} \abs{\xi^{2}})} \aleq 2^{j_{\max}}. 
\end{equation}

On the other hand, $\abs{H}$ may be related to $k_{0}, k_{1}, k_{2}$ and the angles among $\xi^{0}, \xi^{1}, \xi^{2}$.
When $s_{0}  = s_{1}$, then $\abs{H} \simeq 2^{k_{\max}}$, which implies that $j_{\max} \geq k_{\max} - C$. Combined with the assumption $j_{\max} \leq k_{\min} + C_{0}$, it follows that $\abs{k_{\max} - k_{\min}} \leq C$ and $\abs{\ell} \leq C$. Hence \eqref{eq:geom-cone} trivially holds.

Consider now the case $s_{0} = - s_{1}$. Without loss of generality, we may assume that $(s_{0}, s_{1}, s_{2}) = (-, +, +)$ by swapping $f^{0}$ with $f^{1}$ and replacing $f^{i}$ by $\overline{f}^{i}$ $(i = 0, 1, 2)$ if necessary. We claim that
\begin{align} \label{eq:cone-H-angle}
	\abs{H} \simeq& 2^{k_{\min}} \abs{\angle(s_{i} \xi^{i}, s_{2} \xi^{2})}^{2} \quad \hbox{ for } i = 0, 1.
\end{align}
Assuming \eqref{eq:cone-H-angle}, we may conclude the proof of \eqref{eq:geom-cone}. Combined with \eqref{eq:cone-H-j}, the desired statement \eqref{eq:geom-cone} follows in all cases except $\set{i, i'} = \set{0, 1}$. To prove the remaining case, we first apply the law of sines to obtain
\begin{equation*}
	\sin \angle (-\xi^{0}, \xi^{1}) = \frac{\abs{\xi^{2}}}{\abs{\xi^{0}}} \sin \angle(\xi^{1}, \xi^{2})
\end{equation*}
Since $\xi^{0}, \xi^{1}$ are the two longest vectors, $\angle (-\xi^{0}, \xi^{1})$ must be acute; hence the LHS is comparable to $\abs{\angle(- \xi^{0}, \xi^{1})}$. Combined with \eqref{eq:geom-cone} in the case $\set{i, i'} = \set{1, 2}$, it follows that
\begin{equation*}
	\abs{\angle(s_{0} \xi^{0}, s_{1} \xi^{1})} \simeq 2^{k_{\min} - k_{0}} 2^{\ell}
\end{equation*}
as desired.

It remains to verify \eqref{eq:cone-H-angle}. Using $\sum_{i} \xi^{i} = 0$, we have the identity
\begin{align*}
	H = - \abs{\xi^{0}} + \abs{\xi^{1}} + \abs{\xi^{2}}
	= \frac{- \abs{\xi^{0}}^{2} + (\abs{\xi^{1}} + \abs{\xi^{2}})^{2}}{\abs{\xi^{0}} + \abs{\xi^{1}} + \abs{\xi^{2}}}
	= \frac{- 2 ( \xi^{1} \cdot \xi^{2} - \abs{\xi^{1}} \abs{\xi^{2}})}{\abs{\xi^{0}} + \abs{\xi^{1}} + \abs{\xi^{2}}} 
\end{align*}
Hence
\begin{equation*}
\abs{H}
= \frac{\abs{\xi^{1}} \abs{\xi^{2}}}{\abs{\xi^{0}} + \abs{\xi^{1}} + \abs{\xi^{2}}} (1-\cos \angle(\xi^{1}, \xi^{2}) )
\simeq 2^{k_{\min}} \abs{\angle(s_{1} \xi^{1}, s_{2} \xi^{2})}^{2} 
\end{equation*}
Similarly,
\begin{align*}
	H = - \abs{\xi^{0}} + \abs{\xi^{1}} + \abs{\xi^{2}}
	= \frac{- (\abs{\xi^{0}} - \abs{\xi^{2}})^{2} + \abs{\xi^{1}}^{2}}{\abs{\xi^{0}} + \abs{\xi^{1}} - \abs{\xi^{2}}}
	= \frac{- 2 ( (-\xi^{0}) \cdot \xi^{2} - \abs{-\xi^{0}} \abs{\xi^{2}})}{\abs{\xi^{0}} + \abs{\xi^{1}} - \abs{\xi^{2}}} 
\end{align*}
Since $\abs{\xi^{0}} + \abs{\xi^{1}} - \abs{\xi^{2}} \simeq 2^{k_{\max}}$, we have
\begin{equation*}
	\abs{H} \simeq 2^{k_{\min}} \abs{\angle(s_{0} \xi^{0}, s_{2} \xi^{2})}^{2}.
\end{equation*}

\pfstep{Statement~(2)} From the proof of (1), observe that either $\abs{H} \simeq 2^{k_{\max}}$ or $\abs{H} \aleq 2^{k_{\min}}$. On the other hand, by the assumption $j_{\med} \leq j_{\max} - 5$, we have
$\abs{H} \simeq 2^{j_{\max}}$ instead of \eqref{eq:cone-H-j}. Taking $C_{0} > 0$ large enough, the desired statement now follows. \qedhere
\end{proof}

From Lemma~\ref{lem:geom-cone}, we immediately obtain the following refinement of Lemma~\ref{lem:box-orth-0}.
\begin{lemma} \label{lem:box-orth}
Let $k_{0}, k_{1}, k_{2}, j_{0}, j_{1}, j_{2} \in \bbZ$ be such that $\abs{k_{\med} - k_{\max}} \leq 5$ and $j_{\max} \leq k_{\min} + C_{0}$. Define $\ell := \frac{1}{2} (j_{\max} - k_{\min})_{-}$. 
For $i=0, 1, 2$, let $\calC^{i}$ be a rectangular box of the form $\calC_{k_{\min}}(\ell)$ (i.e., of dimension $2^{k_{\min}} \times 2^{k_{\min} + \ell} \times \cdots \times 2^{k_{\min} + \ell}$, with the longest side aligned in the radial direction) situated in $\set{\abs{\xi} \simeq 2^{k_{i}}}$.
\begin{enumerate}[leftmargin=*]
\item Then the expression
\begin{equation} \label{eq:box-orth-expr} 
	\iint P_{\calC^{0}} Q_{j_{0}}^{s_{0}} h_{k_{0}} \, \BL(P_{\calC^{1}} Q_{j_{1}}^{s_{1}} f_{k_{1}}, P_{\calC^{2}} Q_{j_{2}}^{s_{2}} g_{k_{2}}) \, \ud t \ud x
\end{equation}
vanishes unless 
\begin{equation} \label{eq:box-orth} 
	\calC^{0}+ \calC^{1} + \calC^{2} \ni 0 \quad \hbox{ and } \quad
	\abs{\angle(s_{i} \calC^{i}, s_{i'} \calC^{i'})} \aleq 2^{\ell} 2^{k_{\min} - \min \set{k_{i}, k_{i'}}} 
\end{equation}
for every $i, i' \in \set{0, 1, 2}$ $(i \neq i')$.

\item Let $k_{i} = k_{\med}$ or $k_{\max}$; without loss of generality, assume that $i = 0$. Then for any fixed rectangular box $\calC^{0}$ of the form $\calC_{k_{\min}}(\ell)$ situated in $\set{\abs{\xi} \simeq 2^{k_{0}}}$, there are only (uniformly) bounded number of boxes $\calC^{1}, \calC^{2}$ in $\set{\abs{\xi} \simeq 2^{k_{1}}}, \set{\abs{\xi} \simeq 2^{k_{2}}}$ such that \eqref{eq:box-orth} holds.
\end{enumerate}
\end{lemma}
\begin{proof} 
Statement~(1) follows immediately from Lemma~\ref{lem:geom-cone}. 
Statement~(2) can be proved in a similar fashion as Lemma~\ref{lem:box-orth-0}. We first assume without loss of generality that $k_{2} = k_{\min}$. It is clear that there are only bounded number of $\calC^{2} = \calC_{k_{\min}}(\ell)$ in $\set{\abs{\xi} \simeq 2^{k_{\min}}}$ such that $\abs{\angle(s_{0} \calC^{0}, s_{2} \calC^{2})} \aleq 2^{\ell} $. Moreover, observe that $\calC^{0} + \calC^{2}$ is contained in a cube of sidelength $\aleq 2^{k_{\min}}$. Combined with the angular restriction $\abs{\angle(s_{0} \calC^{0}, s_{1} \calC^{1})} \aleq 2^{k_{\min} - k_{\max}} 2^{\ell}$, it follows that there are only bounded number of $\calC^{1}$ such that \eqref{eq:box-orth} holds. \qedhere
\end{proof}

Finally, we state a simple abstract summation lemma, which will be repeatedly used in conjunction with the orthogonality results in this subsection. Roughly speaking, it is the Cauchy-Schwarz inequality for an `essentially diagonal' sum. 
\begin{lemma} \label{lem:CS}
Let $\set{a_{\alp}}_{\alp \in \calA}$ and $\set{b_{\bt}}_{\bt \in \calB}$ be (countably) indexed sequences of real numbers. Let $\calJ \subseteq \calA \times \calB$ be such that for each fixed $\alp \in \calA$, $\abs{\# \set{\bt: (\alp, \bt) \in \calJ}} \leq M$, and for each fixed $\bt \in \calB$, $\abs{\# \set{\alp: (\alp, \bt) \in \calJ}} \leq M$. Then we have
\begin{equation*}
	\abs{\sum_{\alp, \bt \in \calJ} a_{\alp}  b_{\bt}}
	\leq M \bb( \sum_{\alp \in \calA} a_{\alp}^{2} \bb)^{1/2} \bb( \sum_{\bt \in \calB} b_{\bt}^{2} \bb)^{1/2}.
\end{equation*}
\end{lemma}
We omit the straightforward proof.

\subsection{A H\"older-type estimate} \label{subsec:ellip}
In this short subsection, we prove a H\"older-type estimate which will be useful later for dealing with the high modulation contribution, as well as the elliptic equations of MD-CG. Moreover, the orthogonality argument we employ below will serve as a model for the proof of the core bilinear estimates in Propositions~\ref{prop:no-nf}, ~\ref{prop:nf} and \ref{prop:nfs}. 

\begin{lemma} \label{lem:ellip}
Let $k_{0}, k_{1}, k_{2} \in \bbZ$ be such that $\abs{k_{\med} - k_{\max}} \leq 5$. Let $\calL$ be a translation invariant bilinear operator on $\bbR^{d}$ with bounded mass kernel. Then we have
\begin{align}
	\nrm{P_{k_{0}} \calL(f_{k_{1}}, g_{k_{2}})}_{L^{2} L^{2}} 
	\aleq & \nrm{f_{k_{1}}}_{L^{\infty} L^{2}} \bb( \sum_{\calC_{k_{\min}}} \nrm{P_{\calC_{k_{\min}}} g_{k_{2}}}_{L^{2} L^{\infty}}^{2} \bb)^{1/2}, \label{eq:ellip-0} \\
	\nrm{P_{k_{0}} \calL(f_{k_{1}}, g_{k_{2}})}_{L^{1} L^{2}} 
	\aleq & \nrm{f_{k_{1}}}_{L^{2} L^{2}} \bb( \sum_{\calC_{k_{\min}}} \nrm{P_{\calC_{k_{\min}}} g_{k_{2}}}_{L^{2} L^{\infty}}^{2} \bb)^{1/2}.		\label{eq:ellip-1}
\end{align}
\end{lemma}
\begin{proof}
For $t \in \bbR$ and rectangular boxes $\calC^{0}, \calC^{1}, \calC^{2}$ of the form $\calC_{k_{\min}}(0)$, define
\begin{align*}
	I_{\calC^{0}, \calC^{1}, \calC^{2}}(t)
	= & \int P_{\calC^{0}} h_{k_{0}} \, \calL( P_{\calC^{1}} f_{k_{1}},P_{\calC^{2}} g_{k_{2}}) (t) \, \ud x.
\end{align*}
Also defining $I(t) = \int h_{k_{0}} \calL(f_{k_{1}}, g_{k_{2}}) (t) \, \ud x$, note that $I(t) = \sum_{\calC^{0}, \calC^{1}, \calC^{2}} I_{\calC^{0}, \calC^{1}, \calC^{2}}(t)$,
where the summand vanishes unless $\calC^{0} + \calC^{1} + \calC^{2} \ni 0$. Moreover, for any $1 \leq q_{0}, q_{1}, q_{2} \leq \infty$ such that $q_{0}^{-1} + q_{1}^{-1} + q_{2}^{-1} = 1$, we have
\begin{equation} \label{eq:ellip-atom}
	\abs{I_{\calC^{0}, \calC^{1}, \calC^{2}}(t)}
	\aleq \nrm{P_{\calC^{0}} h_{k_{0}}(t)}_{L^{q_{0}}} \nrm{P_{\calC^{1}} f_{k_{1}}(t)}_{L^{q_{1}}} \nrm{P_{\calC^{2}} g_{k_{2}}(t)}_{L^{q_{2}}}.
\end{equation}
We claim that
\begin{equation} \label{eq:ellip-orth}
	\abs{I(t)}
	\aleq \bb(\sum_{\calC^{0}} \nrm{P_{\calC^{0}} h_{k_{0}}(t)}_{L^{q_{0}}}^{2}\bb)^{1/2} 
		\bb( \sum_{\calC^{1}} \nrm{P_{\calC^{1}} f_{k_{1}}(t)}_{L^{q_{1}}}^{2} \bb)^{1/2} 
		\bb( \sum_{\calC^{2}} \nrm{P_{\calC^{2}} g_{k_{2}}(t)}_{L^{q_{2}}}^{2} \bb)^{1/2}.
\end{equation}
From \eqref{eq:ellip-orth}, the desired estimates follow immediately. Indeed, taking $(q_{0}, q_{1}, q_{2}) = (2, 2, \infty)$ and using orthogonality in $L^{2}$, we obtain
\begin{equation*}
	\abs{I(t)} 
	\aleq 	\nrm{h_{k_{0}}(t)}_{L^{2}} 
			\nrm{f_{k_{1}}(t)}_{L^{2}} 
			\bb( \sum_{\calC^{2}}\nrm{P_{\calC^{2}} g_{k_{2}}(t)}_{L^{\infty}}^{2} \bb)^{1/2}.
\end{equation*}
Integrating and applying H\"older in $t$ appropriately, \eqref{eq:ellip-0} and \eqref{eq:ellip-1} follow by duality.

It remains to prove \eqref{eq:ellip-orth}. The idea is to sum up first over boxes situated in $\set{\abs{\xi} \simeq 2^{k_{\min}}}$, for which there are only (uniformly) bounded summands by Lemma~\ref{lem:box-orth-0}, and then apply Lemma~\ref{lem:CS} to the remaining summation, which is essentially diagonal again by Lemma~\ref{lem:box-orth-0}. More precisely, we split into three cases:
\pfstep{Case 1: (high-high) interaction, $k_{0} = k_{\min}$} 
Summing up first in $\calC^{0}$, for which there are only bounded number of summands for each fixed $(\calC^{1}, \calC^{2})$ by Lemma~\ref{lem:box-orth-0}, we obtain
\begin{align*}
	\abs{\int h_{k_{0}} \calL(f_{k_{1}}, g_{k_{2}}) (t) \, \ud x}
	\aleq & \sup_{\calC^{0}} \nrm{P_{\calC^{0}} h_{k_{0}}(t)}_{L^{q_{0}}} \sum_{\calC^{1}, \calC^{2} : (\star)} \nrm{P_{\calC^{1}} f_{k_{1}}(t)}_{L^{q_{1}}} \nrm{P_{\calC^{2}} g_{k_{2}}(t)}_{L^{q_{2}}}
\end{align*}
where $(\star)$ refers to the condition for $(\calC^{1}, \calC^{2})$ that  $\calC^{0} + \calC^{1} +\calC^{2} \ni 0$ for some $\calC^{0} \subseteq \set{\abs{\xi} \simeq 2^{k_{0}}}$. Again by Lemma~\ref{lem:box-orth-0}, for a fixed box $\calC^{1}$ there are only (uniformly) bounded number of boxes $\calC^{2}$ obeying $(\star)$ and vice versa; hence the $\calC^{1}, \calC^{2}$ summation is essentially diagonal, and an application of Lemma~\ref{lem:CS} gives
\begin{align*}
	\abs{\int h_{k_{0}} \calL(f_{k_{1}}, g_{k_{2}}) (t) \, \ud x}
	\aleq & \sup_{\calC^{0}} \nrm{P_{\calC^{0}} h_{k_{0}}(t)}_{L^{q_{0}}} 		\bb( \sum_{\calC^{1}} \nrm{P_{\calC^{1}} f(t)}_{L^{q_{1}}}^{2} \bb)^{1/2} \\
	& \times 	\bb( \sum_{\calC^{2}} \nrm{P_{\calC^{2}} g(t)}_{L^{q_{2}}}^{2} \bb)^{1/2}.
\end{align*}
The claim \eqref{eq:ellip-orth} now follows.
\pfstep{Cases 2 \& 3: (low-high) or (high-low) interaction, $k_{1} = k_{\min}$ or $k_{2} = k_{\min}$} 
The proof proceeds in exactly the same fashion as Case 1, with the role of $k_{0}$ played by the minimum frequency. \qedhere 
\end{proof}

\subsection{Abstract null forms} \label{subsec:abs-nf}
In the context of nonlinear wave equations, a \emph{null form} is a multilinear operator which gains in the angle between the Fourier supports of inputs (or by duality, an input and the output). According to Lemma~\ref{lem:geom-cone}, this angular gain is helpful in the `resonant' case, when the inputs and the output have low modulation (i.e., close to the characteristic cone).

To unify the treatment of various null forms that arise in MD-CG, we introduce the notion of an \emph{abstract null form}. 
For technical advantage, we formulate this notion for a bilinear operator on the space $\bbR^{d}$, rather than on the spacetime $\bbR^{1+d}$.
\begin{definition} \label{def:abs-nf}
Let $\ANF$ be a translation-invariant bilinear operator corresponding to a symbol $m(\xi, \eta)$ on $\bbR^{d} \times \bbR^{d}$.
We say that $\ANF$ is an \emph{abstract null form} (of order $1$) if the symbol $m$ obeys the following bounds:
\begin{align}
	\abs{S_{\xi}^{n_{1}} S_{\eta}^{n_{2}} m(\xi, \eta)} & \leq A_{n_{1}, n_{2}} \abs{\angle (\xi, \eta)} \label{eq:abs-nf:angle} \\
		\abs{\rd_{\xi}^{\alp_{1}} \rd_{\eta}^{\alp_{2}} m(\xi, \eta)} & \leq A_{\alp_{1}, \alp_{2}} \abs{\xi}^{-\abs{\alp_{1}}} \abs{\eta}^{-\abs{\alp_{2}}} \label{eq:abs-nf:bnd}
\end{align}
where $S_{\xi} = \xi \cdot \rd_{\xi}$ and $S_{\eta} = \eta \cdot \rd_{\eta}$.
\end{definition}

\begin{remark} \label{rem:hom-nf}
If the symbol $m(\xi, \eta)$ of $\ANF$ is homogeneous of degree $0$ in $\xi$, $\eta$ and obeys
\begin{equation*}
	\abs{m(\xi, \eta)} \leq A \abs{\angle(\xi, \eta)}
\end{equation*}
then $\ANF$ is an abstract null form of order $1$.
\end{remark}



We now state a basic estimate for an abstract null form.
\begin{proposition}  \label{prop:nf-basic}
Let $\ANF$ be an abstract null form of order $1$. Let $\omg_{1}, \omg_{2} \subseteq \bbS^{d-1}$ be angular caps (i.e., geodesic balls on $\bbS^{d-1}$) such that the radius $r_{j}$ of $\omg_{j}$ is at most $2^{-10}$ $(j=1,2)$ and define
\begin{equation*}
	\tht := \max \set{\abs{\angle(\omg_{1}, \omg_{2})}, r_{1}, r_{2}}.
\end{equation*}
Let $k_{1}, k_{2} \in \bbZ$ and consider test functions $f_{1}$, $f_{2}$ on $\bbR^{d}$ with Fourier support
\begin{equation*}
\supp \hat{f}_{j} \sbeq E_{j} := \set{\frac{1}{100} 2^{k_{j}} < \abs{\xi} < 100 \cdot 2^{k_{j}}, \ \frac{\xi}{\abs{\xi}} \in \omg_{j}}
\end{equation*}
Then for any $1 \leq p, q_{1}, q_{2} \leq \infty$ such that $p^{-1} = q_{1}^{-1} + q_{2}^{-1}$, we have
\begin{equation} \label{eq:nf-basic}
	\nrm{\ANF(f_{1}, f_{2})}_{L^{p}} 
	\aleq \tht \nrm{f_{1}}_{L^{q_{1}}} \nrm{f_{2}}_{L^{q_{2}}},
\end{equation}
where the implicit constant depends only on $A$.
\end{proposition}

The proof of this proposition is purely technical and disparate from the rest of this section. For the sake of exposition, we defer it until the end of this section.

\subsection{Null structure of Maxwell--Dirac in Coulomb gauge} \label{subsec:MD-nf}
We now recast the null structure of MD-CG in terms of abstract null forms. It is at this point that we may fully explain an important point discussed in the introduction, namely, how the spinorial nonlinearities $\NM^{S}$ and $\ND^{S}$ exhibit \emph{more favorable} null structure compared to the Riesz transform parts $\NM^{R}$ and $\ND^{R}$. We refer the reader to Remark~\ref{rem:spin-nf}.

We begin with some schematic definitions.
\begin{definition}[Symbols $\NF$ and $\NF_{\pm}$] \label{def:nf}
We denote by $\NF_{+}$ an abstract null form (Definition~\ref{def:abs-nf}) of order $1$, and by $\NF_{-}$ a bilinear operator such that $(f, g) \mapsto \NF_{-}(f, \overline{g})$ is an abstract null form of order 1. We call $\NF_{s}$ an \emph{abstract null form} of type $s \in \set{+, -}$. Denoting the symbol of $\NF_{s}$ by $m_{s}$, note that it satisfies
\begin{equation*}
	\abs{S_{\xi}^{k_{1}} S_{\eta}^{k_{2}} m_{s}(\xi, \eta)} \leq A_{s, k_{1}, k_{2}} \abs{\angle (\xi, s \eta)}.
\end{equation*}
We write $ \NF $ for a bilinear operator which is an abstract null form of both types; in short, $\NF = \NF_{+}$ and $\NF_{-}$.
\end{definition}
\begin{definition}[Symbols $\NF^{\ast}$ and $\NF_{\pm}^{\ast}$] \label{def:nfs}
For $s \in \set{+, -}$, we denote by $\NF^{\ast}_{s}$ (called a \emph{dual abstract null form} of type $s$) a bilinear operator such that
\begin{equation} \label{eq:nfs-dual}
	\int h \, \NF^{\ast}_{s}(f, g) \, \ud x = \int f \, \NF_{-s}(h, g) \, \ud x
\end{equation}
for some abstract null form $\NF_{-s}$ of type $-s$. We denote by $\NF^{\ast}$ a bilinear operator which is a dual abstract null form of both types, i.e., $\NF^{\ast} = \NF^{\ast}_{+}$ and $\NF^{\ast}_{-}$. (Note that the second input $ g $ plays a special role  in $\NF^{\ast}$ and $\NF^{\ast}_{s}$.)
\end{definition}

We are now ready to describe the (bilinear) null structure of MD-CG in terms of abstract null forms.
\begin{proposition} \label{prop:MD-CG-nf}
The Maxwell nonlinearities $\NM_{s}^{S}$, $\NM^{R}$ have the null structure
\begin{align} 
\NM_{s_{2}}^{S}(\Pi_{s_{1}} \psi, \varphi) 
=& \calP_{j} \brk{\Pi_{s_{1}}  \psi, \Pi_{-s_{2}} \alp_{x} \varphi} = \NF_{s_{1} s_{2}}(\psi, \varphi), 		\label{eq:ms-nf} \\
\NM^{R}(\psi, \varphi) =& \calP_{j} \brk{\psi, \mR_{x}  \varphi} = \NF^{\ast}(\psi, \varphi)	.		\label{eq:mr-nf}
\end{align}
The Dirac nonlinearities $\ND_{s}^{S}$, $\ND^{R}$ have the null structure
\begin{align} 
\Pi_{s_{0}} \calN^{S}_{s_{2}}(A_{x}, f) =& \Pi_{s_{0}}(a_{j} \Pi_{-s_{2}} (\alp^{j} f)) = \NF^{\ast}_{s_{0} s_{2}}(A_{x}, f),	\label{eq:ns-nf} \\
\calN^{R}(A_{x}, \psi) =& \calP_{j} A_{x} \mR^{j} \psi = \NF(A_{x}, \psi). 				\label{eq:nr-nf} 
\end{align}
\end{proposition}
\begin{proof}
Statements \eqref{eq:ms-nf} and \eqref{eq:ns-nf} follow from Lemma~\ref{lem:spin-nf} and Remark~\ref{rem:hom-nf}. To prove the remaining statements, we use \eqref{eq:Pj} to compute
\begin{align*}
	\calP_{j} A_{x} \mR^{j} \psi
	= & (\dlt^{k \ell} \dlt^{j i} - \dlt^{j \ell} \dlt^{i k})  \mR_{k} \mR_{\ell} A_{i} \mR_{j} \psi \\
	= & \dlt^{k \ell} \dlt^{ij} (\mR_{k} \mR_{\ell} A_{i} \mR_{j} \psi - \mR_{j} \mR_{\ell} A_{i} \mR_{k} \psi)  \\
	= & \dlt^{k \ell} \dlt^{ij} \calN_{k j}(\mR_{\ell} A_{i}, \psi ) 
\end{align*}
where $\calN_{k j}$ is a bilinear operator with symbol $\abs{\xi}^{-1} \abs{\eta}^{-1} (\xi_{k} \eta_{j} - \eta_{j} \xi_{k})$.
By Remark~\ref{rem:hom-nf}, it is clear that each $\calN_{k j}$ is an abstract null form of the form $\calN$, which proves \eqref{eq:mr-nf} and \eqref{eq:nr-nf} (the former follows by duality).
\end{proof}

\begin{remark} \label{rem:spin-nf}
A crucial observation, which is one of the main points of this paper, is that the spinorial nonlinearities have more favorable null structure than the Riesz transform counterparts. To see this, consider the Dirac nonlinearities $\ND^{R}$ and $\ND_{s}^{S}$ in the low-high interaction case, which is the worst frequency balance scenario. According to Section~\ref{subsec:nonlin}, this case corresponds to
\begin{equation*}
	\pi^{R}[A_{x}] \psi = \sum_{k} \sum_{k' < k-10}\ND^{R}(P_{k'} A_{x}, P_{k} \psi), \quad
	\pi^{S}_{s}[A_{x}] \psi = \sum_{k} \sum_{k' < k-10} \ND^{S}_{s}(P_{k'} A_{x}, P_{k} \psi).
\end{equation*}
Proposition~\ref{prop:MD-CG-nf} shows that $\ND^{R}$ gains in the angle $\tht$ between (the Fourier variables of) $A_{x}$ and $\psi$, whereas $\ND_{s}^{S}$ gains in the angle $\tht^{\ast}$ between $\psi$ and the output. In this frequency balance scenario, observe that $\tht^{\ast}$ is smaller than $\tht$! Indeed, for each fixed $k, k'$, the law of sines implies that $\tht^{\ast} \simeq 2^{k' - k} \tht$. This extra exponential high-low gain leads to the improved estimate \eqref{eq:diffs} for $\pi_{s}^{S}[A_{x}]$, which fails for $\pi^{R}[A_{x}]$. Similarly, $\NM_{s}^{S}$ exhibits an extra exponential off-diagonal gain compared to $\NM^{R}$ in the worst frequency balance scenario (high-high, in this case), which leads to the improved $Z^{1}$ norm bound \eqref{eq:axs-Z} below. 

Heuristically, the preceding observation leaves us with only the contribution of the scalar part $\NM^{E}, \NM^{R}, \ND^{E}, \ND^{R}$ to be handled; this is the main point of Proposition~\ref{prop:trilinear} and Theorem~\ref{thm:paradiff}. The redeeming feature of this scalar remainder is that it closely resembles MKG-CG; see Remark~\ref{rem:nonlin}. In particular, exploiting this similarity, we are able to borrow a trilinear null form estimate (Proposition~\ref{prop:tri-MKG}) and parametrix construction (Theorem~\ref{covariantparametrix}) from the MKG-CG case \cite{KST} at key steps in our proof below.
\end{remark}

%
%
%

\subsection{Core bilinear estimates} \label{subsec:core-nf}
We now state \emph{core bilinear estimates}, which are estimates for $\calL$, $\NF_{s}$ and $\NF^{\ast}_{s}$ when the inputs and the output have low modulation (more precisely, less than the minimum frequency). 
\begin{proposition} [Core estimates for $\calL$] \label{prop:no-nf}
Let $k_{0}, k_{1}, k_{2}, j \in \bbZ$ be such that $\abs{k_{\max} - k_{\med}} \leq 5$ and $j \leq k_{\min} + C_{0}$. Define $\ell := \frac{1}{2} (j-k_{\min})_{-}$ and let $\calL$ be a translation invariant bilinear operator on $\bbR^{d}$ with bounded mass. Then for any signs $s_{0}, s_{1}, s_{2} \in \set{+,-}$, the following estimates hold:
\begin{equation} \label{eq:no-nf:est0}
\nrm{P_{k_{0}} Q_{j}^{s_{0}} \calL(Q_{<j}^{s_{1}} f_{k_{1}}, Q_{<j}^{s_{2}} g_{k_{2}})}_{L^{2} L^{2}} 
\aleq \nrm{f_{k_{1}}}_{L^{\infty} L^{2}} \bb(\sum_{\calC_{k_{\min}}(\ell)}\nrm{P_{\calC_{k_{\min}}(\ell)} Q_{<j}^{s_{2}} g_{k_{2}}}_{L^{2} L^{\infty}}^{2} \bb)^{1/2}
\end{equation}
\begin{equation} \label{eq:no-nf:est1}
\nrm{P_{k_{0}} Q_{<j}^{s_{0}} \calL(Q_{j}^{s_{1}} f_{k_{1}}, Q_{<j}^{s_{2}} g_{k_{2}})}_{L^{1} L^{2}} 
\aleq \nrm{Q_{j}^{s_{1}} f_{k_{1}}}_{L^{2} L^{2}} \bb(\sum_{\calC_{k_{\min}}(\ell)}\nrm{P_{\calC_{k_{\min}}(\ell)} Q_{<j}^{s_{2}} g_{k_{2}}}_{L^{2} L^{\infty}}^{2} \bb)^{1/2}
\end{equation}
\end{proposition}
\begin{proposition} [Core estimates for $\NF_{s}$] \label{prop:nf}
Let $k_{0}, k_{1}, k_{2}, j \in \bbZ$ be such that $\abs{k_{\max} - k_{\med}} \leq 5$ and $j \leq k_{\min} + C_{0}$. Define $\ell := \frac{1}{2} (j-k_{\min})_{-}$ and let $ \NF_{s}$ be an abstract null form as in Definition~\ref{def:nf}. Then, for any signs $s_{0}, s_{1}, s_{2} \in \set{+, -}$, the following estimates hold:
\begin{equation} \label{eq:nf:est0}
\begin{aligned} 
& \hskip-2em
\nrm{P_{k_{0}} Q_{j}^{s_{0}} \NF_{s_{1} s_{2}}(Q_{<j}^{s_{1}} f_{k_{1}}, Q_{<j}^{s_{2}} g_{k_{2}})}_{L^{2} L^{2}} \\
\aleq & 2^{\ell} 2^{k_{\min} - \min \set{k_{1}, k_{2}}} \nrm{f_{k_{1}}}_{L^{\infty} L^{2}} \bb(\sum_{\calC_{k_{\min}}(\ell)}\nrm{P_{\calC_{k_{\min}}(\ell)} Q_{<j}^{s_{2}} g_{k_{2}}}_{L^{2} L^{\infty}}^{2} \bb)^{1/2}
\end{aligned}
\end{equation}
\begin{equation} \label{eq:nf:est1}
\begin{aligned} 
& \hskip-2em
\nrm{P_{k_{0}} Q_{<j}^{s_{0}} \NF_{s_{1} s_{2}}(Q_{j}^{s_{1}} f_{k_{1}}, Q_{<j}^{s_{2}} g_{k_{2}})}_{L^{1} L^{2}} \\
\aleq & 2^{\ell} 2^{k_{\min} - \min \set{k_{1}, k_{2}}} \nrm{Q_{j}^{s_{1}} f_{k_{1}}}_{L^{2} L^{2}} \bb(\sum_{\calC_{k_{\min}}(\ell)}\nrm{P_{\calC_{k_{\min}}(\ell)} Q_{<j}^{s_{2}} g_{k_{2}}}_{L^{2} L^{\infty}}^{2} \bb)^{1/2}
\end{aligned}
\end{equation}
\end{proposition}

\begin{proposition} [Core estimates for $\NF^{\ast}_{s}$] \label{prop:nfs}
Let $k_{0}, k_{1}, k_{2}, j \in \bbZ$ be such that $\abs{k_{\max} - k_{\med}} \leq 5$ and $j \leq k_{\min} + C_{0}$. Define $\ell := \frac{1}{2} (j-k_{\min})_{-}$ and let $ \NF^{\ast}_{s}$ be an abstract null form as in Definition~\ref{def:nfs}. Then, for any signs $s_{0}, s_{1}, s_{2} \in \set{+, -}$, the following estimates hold:
\begin{equation} \label{eq:nfs:est0}
\begin{aligned} 
& \hskip-2em
\nrm{P_{k_{0}} Q_{j}^{s_{0}} \NF^{\ast}_{s_{0} s_{2}}(Q_{<j}^{s_{1}} f_{k_{1}}, Q_{<j}^{s_{2}} g_{k_{2}})}_{L^{2} L^{2}} \\
\aleq & 2^{\ell} 2^{k_{\min} - \min \set{k_{0}, k_{2}}} \nrm{f_{k_{1}}}_{L^{\infty} L^{2}} \bb(\sum_{\calC_{k_{\min}}(\ell)}\nrm{P_{\calC_{k_{\min}}(\ell)} Q_{<j}^{s_{2}} g_{k_{2}}}_{L^{2} L^{\infty}}^{2} \bb)^{1/2}
\end{aligned}
\end{equation}
\begin{equation} \label{eq:nfs:est1}
\begin{aligned} 
& \hskip-2em
\nrm{P_{k_{0}} Q_{<j}^{s_{0}} \NF^{\ast}_{s_{0} s_{2}}(Q_{j}^{s_{1}} f_{k_{1}}, Q_{<j}^{s_{2}} g_{k_{2}})}_{L^{1} L^{2}} \\
\aleq & 2^{\ell} 2^{k_{\min} - \min \set{k_{0}, k_{2}}} \nrm{Q_{j}^{s_{1}} f_{k_{1}}}_{L^{2} L^{2}} \bb(\sum_{\calC_{k_{\min}}(\ell)}\nrm{P_{\calC_{k_{\min}}(\ell)} Q_{<j}^{s_{2}} g_{k_{2}}}_{L^{2} L^{\infty}}^{2} \bb)^{1/2}
\end{aligned}
\end{equation}
\begin{equation} \label{eq:nfs:est2}
\begin{aligned} 
& \hskip-2em
\nrm{P_{k_{0}} Q_{<j}^{s_{0}} \NF^{\ast}_{s_{0} s_{2}}(Q_{<j}^{s_{1}} f_{k_{1}}, Q_{j}^{s_{2}} g_{k_{2}})}_{L^{1} L^{2}} \\
\aleq & 2^{\ell} 2^{k_{\min} - \min \set{k_{0}, k_{2}}} \bb(\sum_{\calC_{k_{\min}(\ell)}}\nrm{P_{\calC_{k_{\min}(\ell)}}Q_{<j}^{s_{1}} f_{k_{1}}}_{L^{2} L^{\infty}}^{2} \bb)^{1/2} \nrm{Q_{j}^{s_{2}} g_{k_{2}}}_{L^{2} L^{2}}
\end{aligned}
\end{equation}
\end{proposition}
\begin{remark} It is clear from the proof that each of the inequalities holds (with an adjusted constant) when we replace any of the multipliers $ Q^{s_{i}}_{<j}$ by $Q^{s_{i}}_{\leq j}$ or $Q^{s_{i}}_{< j - C}$ for any fixed $C \geq 0$.
\end{remark}

Although there are numerous cases, all the estimates may be proved in an identical fashion, which combines Lemma~\ref{lem:box-orth} with either \eqref{eq:L-atom} or the following estimate:

\begin{lemma} \label{lem:nf-atom}
Let $k_{0}, k_{1}, k_{2}, j, \ell$ be as in Propositions~\ref{prop:nf} and \ref{prop:nfs}. For $i = 0, 1, 2$, let $s_{i} \in \set{+, -}$ and $\calC^{i}$ be a rectangular box of the form $\calC_{k_{\min}}(\ell)$ situated in $\set{\abs{\xi} \sim 2^{k_{i}}}$ such that \eqref{eq:box-orth} holds. Then for any $1 \leq q_{0}, q_{1}, q_{2} \leq \infty$ such that $q_{0}^{-1} + q_{1}^{-1} + q_{2}^{-1} = 1$, we have
\begin{align*}
& \hskip-2em
	\abs{\int P_{\calC^{0}} h_{k_{0}} \NF_{s_{1} s_{1}} ( P_{\calC^{1}} f_{k_{1}}, P_{\calC^{2}} g_{k_{2}}) \, \ud x} \\
	\aleq & 2^{\ell} 2^{k_{\min} - \min \set{k_{1}, k_{2}}} \nrm{P_{\calC^{0}} h_{k_{0}}}_{L^{q_{0}}} \nrm{P_{\calC^{1}} f_{k_{1}}}_{L^{q_{1}}} \nrm{P_{\calC^{2}} g_{k_{2}}}_{L^{q_{2}}}
\end{align*}
\end{lemma}
\begin{proof}
Upon verifying that the inputs obey the hypothesis of Proposition~\ref{prop:nf-basic}, the lemma follows immediately.
\end{proof}



\begin{proof} [Proof of Propositions~\ref{prop:no-nf}, \ref{prop:nf} and \ref{prop:nfs}]
The proof is similar to Lemma~\ref{lem:ellip}, but we use Lemma~\ref{lem:box-orth} instead of Lemma~\ref{lem:box-orth-0}. We present details in the case of Propositions~\ref{prop:nf} and \ref{prop:nfs}; Proposition~\ref{prop:no-nf} follows from the same proof with Lemma~\ref{lem:nf-atom} replaced by \eqref{eq:L-atom}, which removes $2^{\ell} 2^{k_{\min} - \min \set{k_{1}, k_{2}}}$ in \eqref{eq:nf:pf-atom}.

For $t \in \bbR$ and rectangular boxes $\calC^{0}, \calC^{1}, \calC^{2}$ of the form $\calC_{k_{\min}}(\ell)$, we introduce the expression
\begin{align*}
	I_{\calC^{0}, \calC^{1}, \calC^{2}} (t) 
	= \int P_{\calC^{0}} Q_{j / <j}^{s_{0}} h_{k_{0}} \, \NF_{s_{1}  s_{2}} (P_{\calC^{1}} Q_{j / <j}^{s_{1}} f_{k_{1}}, P_{\calC^{2}} Q_{j / <j}^{s_{2}} g_{k_{2}}) (t)\,  \ud x
\end{align*}
where $Q_{j / <j}^{s_{i}}$ stands for either $Q_{j}^{s_{i}}$ or $Q_{< j}^{s_{i}}$. Note that
\begin{equation} \label{eq:nf:pf-sum}
	\iint Q_{j / <j}^{s_{0}} h_{k_{0}} \, \NF_{s_{1}  s_{2}} (Q_{j / <j}^{s_{1}} f_{k_{1}}, Q_{j / <j}^{s_{2}} g_{k_{2}}) \, \ud t \ud x
	 = \sum_{\calC^{0}, \calC^{1}, \calC^{2}} \int I_{\calC^{0}, \calC^{1}, \calC^{2}} (t)  \, \ud t.
\end{equation}

By Lemma~\ref{lem:box-orth}, the summand on the RHS vanishes unless $\calC^{0}, \calC^{1}, \calC^{2}$ satisfy \eqref{eq:box-orth}.  Using the shorthand $\tilde{h} = Q_{j / <j}^{s_{0}} h_{k_{0}}$, $\tilde{f} = Q_{j / <j}^{s_{1}} f_{k_{1}}$ and $\tilde{g} = Q_{j / <j}^{s_{2}} h_{k_{2}}$, Lemma~\ref{lem:nf-atom} implies
\begin{equation} \label{eq:nf:pf-atom}
	\abs{I_{\calC^{0}, \calC^{1}, \calC^{2}} (t)}
	\aleq 2^{\ell} 2^{k_{\min} - \min \set{k_{1}, k_{2}}} \nrm{P_{\calC^{0}} \tilde{h}(t)}_{L^{q_{0}}} \nrm{P_{\calC^{1}} \tilde{f}(t)}_{L^{q_{1}}} \nrm{P_{\calC^{2}} \tilde{g}(t)}_{L^{q_{2}}}
\end{equation}
for any $1 \leq q_{0}, q_{1}, q_{2} \leq \infty$ such that $q_{0}^{-1} + q_{1}^{-1} + q_{2}^{-1} = 1$. We now sum up the RHS of \eqref{eq:nf:pf-atom} in $(\calC^{0}, \calC^{1}, \calC^{2})$ for which \eqref{eq:box-orth} holds. As in the proof of Lemma~\ref{lem:ellip}, we first sum up the boxes in $\set{\abs{\xi} \simeq 2^{k_{\min}}}$ (for which there are only bounded many summands) and then apply Lemma~\ref{lem:CS} to the remaining (essentially diagonal) summation. We then obtain 
\begin{align*}
\sum_{\calC^{0}, \calC^{1}, \calC^{2} : \eqref{eq:box-orth}} \abs{I_{\calC^{0}, \calC^{1}, \calC^{2}} (t)} 
\aleq & 2^{\ell} 2^{k_{\min} - \min \set{k_{1}, k_{2}}} \bb(\sum_{\calC^{0}} \nrm{P_{\calC^{0}} \tilde{h}(t)}_{L^{q_{0}}}^{2}\bb)^{1/2} \\
& \times \bb( \sum_{\calC^{1}} \nrm{P_{\calC^{1}} \tilde{f}(t)}_{L^{q_{1}}}^{2} \bb)^{1/2} \bb( \sum_{\calC^{2}} \nrm{P_{\calC^{2}} \tilde{g}(t)}_{L^{q_{2}}}^{2} \bb)^{1/2}.
\end{align*}
We are ready to complete the proof in a few strokes. To prove estimates \eqref{eq:nf:est0} and \eqref{eq:nf:est1}, take $(q_{0}, q_{1}, q_{2}) = (2, 2, \infty)$. By orthogonality in $L^{2}$, factors involving $\tilde{h}$ and $\tilde{f}$ can be bounded by $\nrm{\tilde{h}(t)}_{L^{2}}$ and $\nrm{\tilde{f}(t)}_{L^{2}}$, respectively. Integrating and applying H\"older's inequality in $t$, estimates \eqref{eq:nf:est0} and \eqref{eq:nf:est1} follow by duality as in the proof of Lemma~\ref{lem:ellip}. Next, by the definition of $\NF_{s}^{\ast}$ in \eqref{eq:nfs-dual}, \eqref{eq:nfs:est0} and \eqref{eq:nfs:est1} follow from the same method as well (we note that, since we use the pairing $\int f g$, the transpose of $Q_{j / <j}^{s_{0}}$ is $Q_{j / <j}^{-s_{0}}$). Finally, \eqref{eq:nfs:est2} is proved by taking $(q_{0}, q_{1}, q_{2}) = (\infty, 2, 2)$ and proceeding analogously. \qedhere
 \end{proof}

\subsection*{Appendix: Proof of Proposition~\ref{prop:nf-basic}} \label{subsec:nf-basic-pf}
The idea of the proof is to perform a separation of variables to write the symbol $m(\xi, \eta)$ of $\ANF$ in the form
\begin{equation} \label{eq:nf-basic:dcmp}
	m(\xi, \eta) = \sum_{\bfj, \bfk \in \bbZ^{d}} c_{\bfj, \bfk} \, a_{\bfj}(\xi) b_{\bfk}(\eta) \quad \hbox{ for } (\xi, \eta) \in E_{1} \times E_{2}
\end{equation}
where for each integer $n \geq 0$ the coefficient $c_{\bfj, \bfk}$ obeys 
\begin{equation}  \label{eq:nf-basic:c}
	\abs{c_{\bfj, \bfk}} \aleq_{n} \tht (1+\abs{\bfj} + \abs{\bfk})^{-n},
\end{equation}
and for some universal constant $n_{0} > 0$, the quantizations of the symbols $a_{\bfj}$ and $b_{\bfk}$ satisfy
\begin{equation} \label{eq:nf-basic:ab}
	\nrm{a_{\bfj}(D)}_{L^{q} \to L^{q}} \aleq (1+\abs{\bfj})^{n_{0}}, \quad \nrm{b_{\bfk}(D) }_{L^{q} \to L^{q}} \aleq (1+\abs{\bfk})^{n_{0}},
\end{equation}
for every $1 \leq q \leq \infty$.

Assuming \eqref{eq:nf-basic:dcmp}--\eqref{eq:nf-basic:ab}, the desired estimate \eqref{eq:nf-basic} follows immediately. Indeed, \eqref{eq:nf-basic:dcmp} implies that 
\begin{equation*}
	\ANF(f_{1}, f_{2}) = \sum_{\bfj, \bfk \in \bbZ^{d}} c_{\bfj, \bfk} \cdot a_{\bfj}(D) f_{1} \cdot b_{\bfk}(D) f_{2},
\end{equation*}
so \eqref{eq:nf-basic} follows by applying H\"older's inequality and \eqref{eq:nf-basic:ab}, then using \eqref{eq:nf-basic:c} to sum up in $\bfj, \bfk \in \bbZ^{d}$.
 
Without loss of generality, we may assume that $k_{1} \geq k_{2}$ and that $\omg_{1}, \omg_{2}$ are angular caps of an equal diameter, denoted by $r$. Moreover, in view of the scaling invariance of the bounds \eqref{eq:abs-nf:angle} and \eqref{eq:abs-nf:bnd}, we may set $k_{1} = 0$. Let $\hat{E}_{j}$ be an enlargement of $E_{j}$ $(j=1,2)$ with a \emph{fixed} angular dimension and let $ a(\xi), b(\eta) $ be bump function adapted to these sets, which are equal to $ 1 $ on $ E_1 $, respectively $ E_2 $, so that $ \hat{f_1}=a \hat{f_1} $ and $ \hat{f_2}=b \hat{f_2} $. Then 
$$ N(f_1, f_2)=N^{m'} (f_1,f_2), $$
where $ N^{m'} $ is the bilinear operator with symbol $ m'(\xi,\eta)=a(\xi) b(\eta) m(\xi,\eta) $.

The first step is to make an invertible change of variables $\xi \mapsto \tilde{\xi} = \tilde{\xi}(\xi)$, so that $S_{\xi} = \tilde{\xi}_{1} \rd_{\tilde{\xi}_{1}}$ and the Jacobian and its derivatives obey appropriate bounds of all order for $\xi \in \hat{E}_{1}$. We also need to perform a similar change of variables $\eta \mapsto \tilde{\eta}(\eta)$ for $\eta \in \hat{E}_{2}$. 
Essentially, what we need is a polar coordinate system with the radial variable as the first component. 

One concrete way to proceed is as follows. Denote the center of the angular cap $\omg_{1}$ by $\bfp_{1} \in \bbS^{d-1}$. Let $(\zt_{2}, \ldots, \zt_{d}) \in \bbR^{d-1}$ be a smooth positively oriented coordinate system on the hemisphere $\bbS^{d-1} \cap \set{\xi : \bfp_{1} \cdot \xi >0}$, such that $(\zt_{2}, \ldots, \zt_{d}) = (0, \ldots, 0)$ corresponds to $\bfp_{1}$. Define
\begin{equation*}
	\tilde{\xi}(\xi) = \bb( \abs{\xi},\abs{\xi} \zt_{2} \bb(\frac{\xi}{\abs{\xi}} \bb), \ldots, \abs{\xi} \zt_{d} \bb(\frac{\xi}{\abs{\xi}} \bb) \bb) \quad \hbox{ for } \xi \in \set{\xi : \bfp_{1} \cdot \xi > 0}.
\end{equation*}
We define $\tilde{\eta}(\eta)$ for $\eta \in \set{\eta : \bfp_{2} \cdot \eta > 0}$ similarly, with the point $\bfp_{1}$ replaced by the center $\bfp_{2}$ of the cap $\omg_{2}$. Observe that $(\tilde{\xi}, \tilde{\eta})$ are well-defined and invertible on $\tilde{E}_{1} \times \tilde{E}_{2}$, in which $m'$ is supported. Abusing the notation a bit, we write $m(\tilde{\xi}, \tilde{\eta}) = m(\xi(\tilde{\xi}), \eta(\tilde{\eta}))$ and simply $E_{j}$ for the region $\set{\tilde{\xi} : \xi(\tilde{\xi}) \in E_{j}}$ etc.

With such definitions, it is clear that $\tilde{\xi}_{1} \rd_{\tilde{\xi}_{1}} = S_{\xi}$ and $\tilde{\eta}_{1} \rd_{\tilde{\eta}_{1}} = S_{\eta}$. Hence \eqref{eq:abs-nf:angle} translates to
\begin{equation} \label{eq:nf-basic:angle} 
	\abs{\rd_{\tilde{\xi}_{1}}^{n_{1}} \rd_{\tilde{\eta}_{1}}^{n_{2}} m(\tilde{\xi}, \tilde{\eta})} \aleq_{A, n_{1}, n_{2}}  \tht \abs{\tilde{\xi}_{1}}^{-n_{1}} \abs{\tilde{\eta}_{1}}^{-n_{2}}.
\end{equation}
Moreover, since each component of $\tilde{\xi} \in \bbR^{d}$ [resp. $\tilde{\eta}$] is homogeneous of degree $1$ in $\xi$ [resp. in $\eta$], we immediately have the bounds
\begin{equation} \label{eq:nf-basic:jacobian-1}
	\abs{\rd_{\xi}^{\alp} \tilde{\xi} (\xi)} \aleq_{\alp} \abs{\xi}^{1-\abs{\alp}} \quad 
	[\hbox{resp. } \abs{\rd_{\eta}^{\alp} \tilde{\eta} (\eta)} \aleq_{\alp} \abs{\eta}^{1-\abs{\alp}}] \quad 
	\hbox{ for any multi-index } \alp.
\end{equation}
Observe that we have $\abs{\tilde{\xi}(\xi)} \simeq \abs{\xi}$ for $\xi \in \hat{E}_{1}$ [resp. $\abs{\tilde{\eta}(\eta)} \simeq \abs{\eta}$ on $\eta \in \hat{E}_{2}$]. Further straightforward computations using \eqref{eq:nf-basic:jacobian-1} show that 
\begin{equation} \label{eq:nf-basic:jacobian-2}
	\abs{\rd_{\tilde{\xi}}^{\alp} \xi (\tilde{\xi})} \aleq_{\alp} \abs{\tilde{\xi}}^{1-\abs{\alp}} \quad 
	[\hbox{resp. } \abs{\rd_{\tilde{\eta}}^{\alp} \eta (\tilde{\eta})} \aleq_{\alp} \abs{\tilde{\eta}}^{1-\abs{\alp}}] \quad
	\hbox{ for any multi-index } \alp,
\end{equation}
for $(\xi(\tilde{\xi}), \eta(\tilde{\eta})) \in \hat{E}_{1} \times \hat{E}_{2}$. Combined with \eqref{eq:abs-nf:bnd} and the support property of $m$, we have
\begin{equation} \label{eq:nf-basic:bnd}
	\abs{\rd_{\tilde{\xi}}^{\alp_{1}} \rd_{\tilde{\eta}}^{\alp_{2}} m(\tilde{\xi}, \tilde{\eta})} \aleq_{A, \alp_{1}, \alp_{2}}  \abs{\tilde{\xi}}^{-\abs{\alp_{1}}} \abs{\tilde{\eta}}^{-\abs{\alp_{2}}} . 
\end{equation}

We now introduce rectangular boxes $R_{1}$ and $R_{2}$, which are defined as
\begin{equation*}
R_{1}  = \set{\tilde{\xi} : \tilde{\xi}_{1} \simeq 1, \, \sup_{j=2, \ldots, d}\abs{\tilde{\xi}_{j}} \aleq r}, \quad
R_{2}  = \set{\tilde{\eta} : \tilde{\eta}_{1} \simeq 2^{k_{2}}, \, \sup_{j=2, \ldots, d} \abs{\tilde{\eta}_{j}} \aleq 2^{k_{2}} r},
\end{equation*}
where the implicit constants are chosen so that $E_{1} \subseteq R_{1}$ and $E_{2} \subseteq R_{2}$. Let $\tilde{a}(\tilde{\xi})$ and $\tilde{b}(\tilde{\eta})$ be the bump functions adapted to the boxes $R_{1}$ and $R_{2}$, respectively such that $\tilde{a}$ and $\tilde{b}$ are equal to $1$ on $ E_{1}$ and $ E_{2}$, respectively. 

Thus we have the following bounds for $ j =2, \dots d $:
\be \label{cbds} \abs{ (r \rd_{\tilde{\xi}_{j}})^{n}  m(\tilde{\xi}, \tilde{\eta}) \tilde{a}(\tilde{\xi}) \tilde{b}(\tilde{\eta})  } \aleq_{A, n}  \tht, 
	\quad   \abs{ (2^{k_2} r \rd_{\tilde{\eta}_{j}})^{n}  m(\tilde{\xi}, \tilde{\eta}) \tilde{a}(\tilde{\xi}) \tilde{b}(\tilde{\eta})  } \aleq_{A, n}  \tht \ee 

Performing a Fourier series expansion of $m(\tilde{\xi}, \tilde{\eta}) \tilde{a}(\tilde{\xi}) \tilde{b}(\tilde{\eta})  $ in the variables $(\tilde{\xi}, \tilde{\eta})$ by viewing $R_{1} \times R_{2}$ as a torus, we may write
\begin{equation} \label{eq:nf-basic:fs}
	m(\tilde{\xi}, \tilde{\eta}) = \sum_{\bfj, \bfk \in \bbZ^{d}} c_{\bfj, \bfk} \, e^{2 \pi i \bfj \cdot D_{1} \tilde{\xi}} e^{2 \pi i \bfk \cdot D_{2} \tilde{\eta}} \quad \hbox{ for } (\tilde{\xi}, \tilde{\eta}) \in E_{1} \times E_{2},
\end{equation}
where $D_{1}, D_{2}$ are diagonal matrices of the form
\begin{align*}
	D_{1} =& \mathrm{diag} \, (O(1), O(r^{-1}), \ldots, O(r^{-1})), \\
	D_{2} =& \mathrm{diag} \, (O(2^{-k_{2}}), O(2^{-k_{2}} r^{-1}), \ldots, O(2^{-k_{2}} r^{-1})).
\end{align*}
Defining
\begin{equation*}
	a_{\bfj}(\xi) = (\tilde{a}(\tilde{\xi}) e^{2 \pi i \bfj \cdot D_{1} \tilde{\xi}})(\xi), \quad 
	b_{\bfk}(\eta) = (\tilde{b}(\tilde{\eta}) e^{2 \pi i \bfk \cdot D_{2} \tilde{\eta}})(\eta),
\end{equation*}
we obtain the desired decomposition \eqref{eq:nf-basic:dcmp} from \eqref{eq:nf-basic:fs}.

To prove \eqref{eq:nf-basic:c}, we begin with the following formula for the Fourier coefficient $c_{\bfj, \bfk}$:
\begin{equation*}
	c_{\bfj, \bfk} =  \frac{1}{\Vol(R_{1} \times R_{2})}  \int_{R_{1} \times R_{2}} m(\tilde{\xi}, \tilde{\eta}) \tilde{a}(\tilde{\xi}) \tilde{b}(\tilde{\eta}) e^{- 2 \pi i \bfj \cdot D_{1} \tilde{\xi}} e^{- 2 \pi i \bfk \cdot D_{2} \tilde{\eta}} \, \ud \tilde{\xi} \, \ud \tilde{\eta}.
\end{equation*}
Integrating by parts in $\tilde{\xi}_{1}$ [resp. in $\tilde{\eta}_{1}$] and using \eqref{eq:nf-basic:angle}, we obtain 
\begin{equation*}
	\abs{c_{\bfj, \bfk}} \aleq_{n} \tht (1+\abs{\bfj_{1}})^{-n} \quad [\hbox{resp. } \abs{c_{\bfj, \bfk}} \aleq_{n} \tht (1+\abs{\bfk_{1}})^{-n}],
\end{equation*}
for each integer $n \geq 0$. On the other hand, for any $j = 2, \ldots, d$, integration by parts in $\tilde{\xi}_{j}$ [resp. in $\tilde{\eta}_{j}$] and using \eqref{cbds} yields
\begin{equation*}
	\abs{c_{\bfj, \bfk}} \aleq_{n} \tht \abs{\bfj_{j}}^{-n} \quad [\hbox{resp. } \abs{c_{\bfj, \bfk}} \aleq_{n} \tht \abs{\bfk_{j}}^{-n}],
\end{equation*}
The preceding bounds imply \eqref{eq:nf-basic:c} as desired.

Finally, we need to establish \eqref{eq:nf-basic:ab}. We will describe the case of $a_{\bfj}(D)$ in detail, and leave the similar proof for $b_{\bfk}(\eta)$ to the reader. For any multi-index $\alp$, observe that
\begin{equation} \label{eq:nf-basic:ab:sym-bnd-pre}
\abs{\rd_{\tilde{\xi}}^{\alp} (\tilde{a}(\tilde{\xi}) e^{2 \pi i \bfj \cdot D_{1} \tilde{\xi}})} \aleq_{\alp} (1+\abs{\bfj})^{\abs{\alp}} r^{-(\alp_{2} + \cdots +\alp_{d})}.
\end{equation}
By rotation, we may assume that the center of $\omg_{1}$ is aligned with the $\xi_{1}$-axis, i.e., $\bfp_{1} = (1, 0, \ldots, 0)$. Then we claim that
\begin{equation} \label{eq:nf-basic:ab:sym-bnd}
	\abs{\rd_{\xi}^{\alp} a_{\bfj}(\xi)} \aleq_{\alp} (1+\abs{\bfj})^{\abs{\alp}} r^{-(\alp_{2} + \cdots +\alp_{d})}.
\end{equation}
From such a bound, it is straightforward to check that the convolution kernel (i.e., inverse Fourier transform) $\check{a}_{\bfj}(x)$ of $a_{\bfj}(D)$ obeys $\nrm{\check{a}_{\bfj}}_{L^{1}} \aleq (1+\abs{\bfj})^{n_{0}}$ for some universal constant $n_{0}$ (in fact, $n_{0} = d$ would work), which implies the desired $L^{q}$ bounds \eqref{eq:nf-basic:ab} for $a_{\bfj}(D)$.

In order to verify \eqref{eq:nf-basic:ab:sym-bnd}, the key is to ensure that each $\rd_{\xi_{1}}$ derivative does not lose a factor of $r^{-1}$.
Recall that $\tilde{\xi}_{j} = \abs{\xi} \zt_{j}(\xi / \abs{\xi})$ for $j = 2, \ldots, d$. Observe that $\rd_{\xi_{1}}^{n} \zt_{j}(\xi / \abs{\xi}) \restriction_{\xi = \bfp_{1}} = 0$ for every $n \geq 0$ (in fact, $\zt_{j}$ can be chosen to be independent of the first coordinate $\xi_{1}$ everywhere on $\bbS^{d-1} \cap \set{\xi_{1} > 0} \subseteq \bbR^{d}$). Therefore, we have
\begin{equation*}
	\bb| \frac{\rd^{n} \tilde{\xi}_{j}}{\rd \xi_{1}^{n}} \bb| 
	\aleq \sum_{i=0}^{n} \bb| \rd_{\xi_{1}}^{i} \zt_{j} \bb( \frac{\xi}{\abs{\xi}} \bb) \bb| 
	\aleq_{n} \dist \bb( \frac{\xi}{\abs{\xi}}, \bfp_{1} \bb) \aleq r
	\quad \hbox{ for every } n \geq 0, \ \xi \in \supp a_{\bfj}.
\end{equation*}
Let $c(\xi)$ be any smooth function. By an iteration of the chain rule $\rd_{\xi_{1}} = (\rd_{\xi_{1}} \tilde{\xi}_{1}) \rd_{\tilde{\xi}_{1}} + \sum_{j=2}^{d} (\rd_{\xi_{1}} \tilde{\xi}_{j}) \rd_{\tilde{\xi}_{j}}$, it follows that
\begin{equation*}
	\abs{\rd_{\xi_{1}}^{\alp_{1}} c(\xi) } \aleq_{\alp_{1}} \sum_{\abs{\bt} \leq \alp_{1}} r^{\bt_{2} + \cdots +\bt_{d}} \abs{(\rd_{\tilde{\xi}}^{\bt} c)(\xi)}
	\quad \hbox{ for every } \alp_{1} \geq 0, \ \xi \in \supp a_{\bfj}.
\end{equation*}
Substituting $c(\xi) = \rd_{\xi_{2}}^{\alp_{2}} \cdots \rd_{\xi_{d}}^{\alp_{d}} a_{\bfj}(\xi)$ and using \eqref{eq:nf-basic:jacobian-1}, \eqref{eq:nf-basic:ab:sym-bnd-pre}, the desired bound \eqref{eq:nf-basic:ab:sym-bnd} follows after a straightforward computation. 


\section{Proof of the bilinear estimates} \label{sec:bi}
Here we prove Propositions~\ref{prop:a}--\ref{prop:dirac} concerning bilinear estimates.
As a byproduct of our proof, Lemma~\ref{lem:diffr-free} also follows; see Remark~\ref{rem:diffr-free}.
Unless otherwise stated, we restrict to the case $d = 4$; the general case of $d \geq 4$ is discussed in Remark~\ref{rem:hi-d-3} below. 

\subsection{Preliminaries: Conventions and frequency envelope bounds} \label{subsec:freq-env-bnd}
Henceforth, we use the shorthand $A$ to denote any $A_{j}$ $(j = 1, \ldots, 4)$.
Unless otherwise stated, we normalize the frequency envelope norms of the inputs as follows:
\begin{equation} \label{eq:fe-n}
	\nrm{B}_{Y^{1}_{a}} = \nrm{A}_{S^{1}_{a}} = \nrm{\psi}_{(\tilde{S}^{1/2}_{s})_{b}} = \nrm{\varphi}_{(\tilde{S}^{1/2}_{s'})_{c}} = 1.
\end{equation}
Having control of the $S^{1}$ and $S^{1/2}_{\pm}$ norms through the frequency envelopes $a, b$ results in the following estimates, which we will use repeatedly in the proofs of the bilinear and trilinear estimates\footnote{Of course, the same estimates as $\psi$ hold for $\varphi$ with $(s, b_{k})$ replaced by $(s', c_{k})$.}:
\begin{align}
\nrm{A_{k}}_{L^{\infty} L^{2}}
	\aleq & 2^{-k} a_{k}	,	\hskip9em
\nrm{\psi_{k}}_{L^{\infty} L^{2}}
	\aleq 2^{-\frac{1}{2} k} b_{k} ,			\label{eq:fe-L2} \\
\nrm{Q_{j} A_{k}}_{L^{2} L^{2}}
	\aleq & 2^{- \frac{1}{2} \max \set{j, k}}2^{-\frac{1}{2} j} 2^{- \frac{1}{2} k} a_{k} , \ 
\nrm{Q_{j}^{s} \psi_{k}}_{L^{2} L^{2}}
	\aleq  2^{- \frac{1}{2} \max \set{j, k}} 2^{-\frac{1}{2} j} b_{k} . \label{eq:fe-L2L2} 
\end{align}
For $k' < k$, we have
\begin{equation} \label{eq:fe-L2Linfty-0}
\begin{aligned}
\bb(\sum_{\calC_{k'}(0)} \nrm{P_{\calC_{k'}(0)} A_{k}}_{L^{2} L^{\infty}}^{2} \bb)^{1/2}
	\aleq & 2^{k'} 2^{-\frac{1}{2} k} a_{k} ,	\\
\bb(\sum_{\calC_{k'}(0)} \nrm{P_{\calC_{k'}(0)} \psi_{k}}_{L^{2} L^{\infty}}^{2} \bb)^{1/2}
	\aleq & 2^{k'} b_{k}				.
\end{aligned}
\end{equation}
For $k'$ such that $k' \leq k$ and $j \leq k' + C$, define $\ell = \frac{1}{2}(j - k')_{-}$. Then we have
\begin{equation} \label{eq:fe-L2Linfty}
\begin{aligned}
\bb(\sum_{\calC_{k'}(\ell)} \nrm{P_{\calC_{k'}(\ell)} Q_{<j} A_{k}}_{L^{2} L^{\infty}}^{2} \bb)^{1/2}
	\aleq & 2^{k'} 2^{\frac{1}{2} \ell} 2^{-\frac{1}{2} k} a_{k},	\\
\bb(\sum_{\calC_{k'}(\ell)} \nrm{P_{\calC_{k'}(\ell)} Q_{<j}^{s} \psi_{k}}_{L^{2} L^{\infty}}^{2} \bb)^{1/2}
	\aleq & 2^{k'} 2^{\frac{1}{2} \ell} b_{k}.
\end{aligned}
\end{equation}
These bounds follow immediately from the definition of the norms $S^{1}_{a}$ and $S^{1/2}_{b}$.

The $\tilde{Z}_{s}^{1/2}$ component leads to the bound
\begin{equation} \label{eq:fe-L1Linfty}
	\nrm{Q_{j}^{s} \psi_{k}}_{L^{1} L^{\infty}} \aleq 2^{\frac{1}{2} k} 2^{5(k - j)_{+}} b_{k}.
\end{equation}
Indeed, by \eqref{eq:q-disp} we have
\begin{align*}
\nrm{Q_{j}^{s} \psi_{k}}_{L^{1} L^{\infty}} 
\aleq & 2^{-j} 2^{4(k-j)_{+}} \nrm{(i \rd_{t} + s \abs{D}) \psi_{k}}_{L^{1} L^{\infty}} \\
\aleq & 2^{\frac{1}{2} k} 2^{k-j} 2^{4(k-j)_{+}} \nrm{\psi_{k}}_{\tilde{Z}_{s}^{1/2}},
\end{align*}
from which \eqref{eq:fe-L1Linfty} follows.

Finally, the normalization $\nrm{B}_{Y^{1}_{a}} = 1$ implies
\begin{equation} \label{eq:fe-Y}
	\nrm{B_{k}}_{L^{2} L^{2}}  \aleq 2^{-\frac{3}{2} k} a_{k}, \qquad
	\nrm{Q_{j} B_{k}}_{L^{2} L^{2}}  \aleq 2^{-\max\set{j, k}} 2^{-\frac{1}{2} k} a_{k}.
\end{equation}

\subsection{Proof of Proposition~\ref{prop:a}} \label{subsec:bi-a}
Here we prove \eqref{eq:a0}--\eqref{eq:axs}.
\subsubsection*{Step~0:~Reduction to dyadic estimates}
Under the normalization \eqref{eq:fe-n}, we claim:
\begin{align} 
	2^{-\frac{1}{2} k_{0}} \nrm{P_{k_{0}} \calL(\psi_{k_{1}}, \varphi_{k_{2}})}_{L^{2} L^{2}}
	\aleq & 2^{\frac{1}{2} (k_{\max} - k_{\min})} b_{k_{1}} c_{k_{2}}, 		\label{eq:bi-a-est1}\\
	\nrm{P_{k_{0}} \calN^{\ast}(\psi_{k_{1}}, \varphi_{k_{2}})}_{N}
	\aleq & 2^{\dlt_{0} (k_{\max} - k_{\min})} b_{k_{1}} c_{k_{2}}, 		\label{eq:bi-a-est2}\\
	\nrm{P_{k_{0}} \calN_{s s'}(\psi_{k_{1}}, \varphi_{k_{2}})}_{N}
	\aleq & 2^{\dlt_{0} (k_{\max} - k_{\min})} b_{k_{1}} c_{k_{2}}.		\label{eq:bi-a-est3}
\end{align}

Proposition~\ref{prop:a} follows from the above dyadic estimates. We begin with the proof of \eqref{eq:a0}. From \eqref{eq:me-def} and \eqref{eq:d0me-def}, observe that $\NM^{E}(\tilde{P}_{k_{1}} \cdot, \tilde{P}_{k_{2}} \cdot) = \calL$ and $ P_{k_{0}} \rd_{t} \NM^{E} (\tilde{P}_{k_{1}} \cdot, \tilde{P}_{k_{2}} \cdot) = \abs{D} P_{k_{0}} \calL$.
Therefore, \eqref{eq:bi-a-est1} implies
\begin{equation*}
	\nrm{P_{k_{0}} \NM^{E}(\psi_{k_{1}}, \varphi_{k_{2}})}_{L^{2} \dot{H}^{-1/2}}
	+ \nrm{P_{k_{0}} \rd_{t} \NM^{E}(\psi_{k_{1}}, \varphi_{k_{2}})}_{L^{2} \dot{H}^{-3/2}}
	\aleq 2^{\frac{1}{2} (k_{\max} - k_{\min})} b_{k_{1}} c_{k_{2}}.
\end{equation*}
The LHS is non-vanishing only if $\abs{k_{\max} - k_{\med}} \leq 5$ (Littlewood-Paley trichotomy). We now divide into cases $k_{\min} = k_{0}$, $k_{1}$ and $k_{2}$, which roughly correspond to (high-high), (low-high) and (high-low), respectively. In each case, summing up in $k_{1}, k_{2}$ using the exponential gain $2^{\frac{1}{2}(k_{\min} - k_{\max})}$ and the slow variance of $b, c$, we arrive at
\begin{equation*}
	\nrm{P_{k_{0}} \NM^{E}(\psi, \varphi)}_{L^{2} \dot{H}^{-1/2}}
	+ \nrm{P_{k_{0}} \rd_{t} \NM^{E}(\psi, \varphi)}_{L^{2} \dot{H}^{-3/2}}
	\aleq b_{k_{0}} c_{k_{0}},	
\end{equation*}
which is precisely the desired estimate \eqref{eq:a0} under the normalization \eqref{eq:fe-n}.

The proof of \eqref{eq:axr} and \eqref{eq:axs} proceeds similarly.
By Proposition~\ref{prop:MD-CG-nf}, we have $\NM_{x}^{R} = \NF^{\ast}$ and $\NM_{x, s'}^{S} (\Pi_{s} \cdot, \cdot) = \NF_{s s'}(\cdot, \cdot)$. Therefore, \eqref{eq:bi-a-est2} and \eqref{eq:bi-a-est3} imply
\begin{equation*}
	\nrm{P_{k_{0}} \NM^{R}(\psi_{k_{1}}, \varphi_{k_{2}})}_{N}
	+ \nrm{P_{k_{0}} \NM^{S}_{s'}(\Pi_{s} \psi_{k_{1}}, \varphi_{k_{2}})}_{N}
	\aleq 2^{\dlt_{0}(k_{\min} - k_{\max})} b_{k_{1}} c_{k_{2}}.
\end{equation*}
On the other hand, application of \eqref{eq:bi-a-est1} shows that
\begin{equation*}
	\nrm{P_{k_{0}} \NM^{R}(\psi_{k_{1}}, \varphi_{k_{2}})}_{L^{2} \dot{H}^{-1/2}}
	+ \nrm{P_{k_{0}} \NM^{S}_{s'}(\Pi_{s} \psi_{k_{1}}, \varphi_{k_{2}})}_{L^{2} \dot{H}^{-1/2}}
	\aleq 2^{\frac{1}{2} (k_{\min} - k_{\max})} b_{k_{1}} c_{k_{2}}.
\end{equation*}
Proceeding as before using Littlewood-Paley trichotomy, the exponential gain in $k_{\min} - k_{\max}$ and the slow variance of $b, c$, the desired estimates \eqref{eq:axr} and \eqref{eq:axs} follow.

The rest of this subsection is devoted to establishing \eqref{eq:bi-a-est1}--\eqref{eq:bi-a-est3}.
\subsubsection*{Step~1:~Proof of \eqref{eq:bi-a-est1}}
Without loss of generality, assume that $k_{2} \leq k_{1}$. Then \eqref{eq:bi-a-est1} follows from application of \eqref{eq:ellip-0} in Lemma~\ref{lem:ellip} and the frequency envelope bounds \eqref{eq:fe-L2} and \eqref{eq:fe-L2Linfty-0}.

\subsubsection*{Step~2:~Proof of \eqref{eq:bi-a-est2}}
We first treat the high modulation contribution.
\begin{lemma} \label{lem:ax-hi-mod}
Assume the normalization \eqref{eq:fe-n}. For any $k_{0}, k_{1}, k_{2}, j \in \bbZ$, we have
\begin{align*}
	2^{-\frac{1}{2} j} \nrm{P_{k_{0}} Q_{j}\calL(\psi_{k_{1}}, \varphi_{k_{2}})}_{L^{2} L^{2}}
	\aleq & 2^{\frac{1}{2}(k_{\min} - j)} 2^{\frac{1}{2}(k_{\min} - k_{\max})} b_{k_{1}} c_{k_{2}},  \\
	\nrm{P_{k_{0}} \calL(Q_{j}^{s_{1}} \psi_{k_{1}}, \varphi_{k_{2}})}_{L^{1} L^{2}}
	\aleq & 2^{\frac{1}{2}(k_{\min} - j)} 2^{\frac{1}{2}(k_{\min} - k_{\max})} b_{k_{1}} c_{k_{2}},  \\
	\nrm{P_{k_{0}} \calL(\psi_{k_{1}}, Q_{j}^{s_{2}} \varphi_{k_{2}})}_{L^{1} L^{2}} 
	\aleq & 2^{\frac{1}{2}(k_{\min} - j)} 2^{\frac{1}{2}(k_{\min} - k_{\max})} b_{k_{1}} c_{k_{2}}.
\end{align*}
\end{lemma}
\begin{proof}
The lemma is a corollary of Lemma~\ref{lem:ellip}.
Indeed, the first estimate follows from \eqref{eq:ellip-0} and the frequency envelope bounds \eqref{eq:fe-L2} and \eqref{eq:fe-L2Linfty-0}.
Similarly, the second estimate follows from \eqref{eq:ellip-1} and the frequency envelope bounds \eqref{eq:fe-L2L2} and \eqref{eq:fe-L2Linfty-0}. The final estimate follows from the second one by symmetry. \qedhere
\end{proof}

By Lemma~\ref{lem:ax-hi-mod} and \eqref{eq:q-LqL2}, it follows that
\begin{align*}
	\nrm{P_{k_{0}} Q_{\geq k_{\min} - 10} \NF^{\ast}(\psi_{k_{1}}, \varphi_{k_{2}})}_{X^{0, -1/2}_{1}}
	\aleq & 2^{\frac{1}{2}(k_{\min} - k_{\max})} b_{k_{1}} c_{k_{2}}, \\
	\nrm{P_{k_{0}} Q_{< k_{\min} - 10} \NF^{\ast}(Q_{\geq k_{\min} - 10}^{s} \psi_{k_{1}}, \varphi_{k_{2}})}_{L^{1} L^{2}}
	\aleq & 2^{\frac{1}{2}(k_{\min} - k_{\max})} b_{k_{1}} c_{k_{2}}, \\
	\nrm{P_{k_{0}} Q_{< k_{\min} - 10} \NF^{\ast}(Q_{< k_{\min} - 10}^{s} \psi_{k_{1}}, Q_{\geq k_{\min} - 10}^{s'} \varphi_{k_{2}})}_{L^{1} L^{2}}
	\aleq & 2^{\frac{1}{2}(k_{\min} - k_{\max})} b_{k_{1}} c_{k_{2}},
\end{align*}
which are all acceptable.  Using the identity $P_{k_{0}} Q_{<k_{\min} - 10} = \sum_{s_{0}} P_{k_{0}} Q_{< k_{\min} - 10}^{s_{0}}$, the remainder can be written as
\begin{equation*}
	\sum_{s_{0}} P_{k_{0}} Q_{< k_{\min} - 10}^{s_{0}}  \NF^{\ast}(Q_{< k_{\min} - 10}^{s} \psi_{k_{1}}, Q_{< k_{\min} - 10}^{s'} \varphi_{k_{2}})
\end{equation*}

Summing according to the highest modulation, we decompose the remainder into $I_{0} + I_{1} + I_{2}$, where
\begin{align}
	I_{0} =& \sum_{s_{0}} \sum_{j < k_{\min} - 10} P_{k_{0}} Q_{j}^{s_{0}} \NF^{\ast}(Q_{< j}^{s} \psi_{k_{1}}, Q_{< j}^{s'} \varphi_{k_{2}}), \label{eq:nm-I0} \\
	I_{1} =& \sum_{s_{0}} \sum_{j < k_{\min} - 10} P_{k_{0}} Q_{\leq j}^{s_{0}} \NF^{\ast}(Q_{j}^{s} \psi_{k_{1}}, Q_{< j}^{s'} \varphi_{k_{2}}), \label{eq:nm-I1}\\
	I_{2} =& \sum_{s_{0}} \sum_{j < k_{\min} - 10} P_{k_{0}} Q_{\leq j}^{s_{0}} \NF^{\ast}(Q_{\leq j}^{s} \psi_{k_{1}}, Q_{ j}^{s'} \varphi_{k_{2}}). \label{eq:nm-I2}
\end{align}
These sums can be estimated using Proposition~\ref{prop:nfs}. We split into three cases according to Littlewood-Paley trichotomy:
\pfstep{Step~2.1:~(high-high) interaction, $k_{0} = k_{\min}$}
Let $\ell = \frac{1}{2} (j - k_{\min})$. By Proposition~\ref{prop:nfs}, we have
\begin{align*}
	\nrm{I_{0}}_{X^{0, -1/2}_{1}} 
	\aleq & \sum_{j < k_{0} - 10} 2^{-\frac{1}{2} j} 2^{\ell} 
			\nrm{\psi_{k_{1}}}_{L^{\infty} L^{2}} 
			\bb( \sum_{\calC_{k_{0}}(\ell)} \nrm{P_{\calC_{k_{0}}(\ell)} Q_{< j}^{s'} \varphi_{k_{2}}}_{L^{2} L^{\infty}}^{2} \bb)^{1/2}, \\
%
\nrm{I_{1}}_{L^{1} L^{2}} 
	\aleq & \sum_{j < k_{0} - 10} 2^{\ell} 
			\nrm{Q_{j}^{s} \psi_{k_{1}}}_{L^{2} L^{2}} 
			\bb( \sum_{\calC_{k_{0}}(\ell)} \nrm{P_{\calC_{k_{0}}(\ell)} Q_{< j}^{s'} \varphi_{k_{2}}}_{L^{2} L^{\infty}}^{2} \bb)^{1/2}, \\
%
\nrm{I_{2}}_{L^{1} L^{2}} 
	\aleq & \sum_{j < k_{0} - 10} 2^{\ell} 
			\bb( \sum_{\calC_{k_{0}}(\ell)} \nrm{P_{\calC_{k_{0}}(\ell)} Q_{\leq j}^{s} \psi_{k_{1}}}_{L^{2} L^{\infty}}^{2} \bb)^{1/2} 
			\nrm{Q_{j}^{s'} \varphi_{k_{2}}}_{L^{2} L^{2}}.
\end{align*}
Then by the frequency envelope bounds \eqref{eq:fe-L2}, \eqref{eq:fe-L2L2} and \eqref{eq:fe-L2Linfty}, we obtain
\begin{equation*}
	\nrm{I_{0}}_{X^{0, -1/2}_{1}} + \nrm{I_{1}}_{L^{1} L^{2}} + \nrm{I_{2}}_{L^{1} L^{2}} 
	\aleq \sum_{j < k_{0} - 10} 2^{\frac{1}{4} \ell} 2^{\frac{1}{2} (k_{0} - k_{1})} b_{k_{1}} c_{k_{2}},
\end{equation*}
which is bounded by $2^{\frac{1}{2} (k_{0} - k_{1})} b_{k_{1}} c_{k_{2}}$ and thus acceptable.

\pfstep{Step~2.2:~(high-low) interaction, $k_{2} = k_{\min}$}
As before, let $\ell = \frac{1}{2} (j - k_{\min})$. By Proposition~\ref{prop:nfs}, we have
\begin{align*}
	\nrm{I_{0}}_{X^{0, -1/2}_{1}} 
	\aleq & \sum_{j < k_{2} - 10} 2^{\frac{1}{2} j} 2^{\ell} 
			\nrm{\psi_{k_{1}}}_{L^{\infty} L^{2}} 
			\bb( \sum_{\calC_{k_{2}}(\ell)} \nrm{P_{\calC_{k_{2}}(\ell)} Q_{< j}^{s'} \varphi_{k_{2}}}_{L^{2} L^{\infty}}^{2} \bb)^{1/2},	\\
%
\nrm{I_{1}}_{L^{1} L^{2}} 
	\aleq & \sum_{j < k_{2} - 10} 2^{\ell} 
			\nrm{Q_{j}^{s} \psi_{k_{1}}}_{L^{2} L^{2}} 
			\bb( \sum_{\calC_{k_{2}}(\ell)} \nrm{P_{\calC_{k_{2}}(\ell)} Q_{< j}^{s'} \varphi_{k_{2}}}_{L^{2} L^{\infty}}^{2} \bb)^{1/2},
\end{align*}
which are both bounded by $\aleq 2^{\frac{1}{2} (k_{2} - k_{1})} b_{k_{1}} c_{k_{2}}$ by the frequency envelope bounds \eqref{eq:fe-L2}, \eqref{eq:fe-L2L2} and \eqref{eq:fe-L2Linfty}.
However, a naive application of the same strategy to $I_{2}$ only yields 
\begin{equation*}
\nrm{I_{2}}_{N} 
	\aleq \sum_{j < k_{2} - 10} 2^{\ell} 
			\bb( \sum_{\calC_{k_{2}}(\ell)} \nrm{P_{\calC_{k_{2}}(\ell)} Q_{\leq j}^{s} \psi_{k_{1}}}_{L^{2} L^{\infty}}^{2} \bb)^{1/2} 
			\nrm{Q_{j} \varphi_{k_{2}}}_{L^{2} L^{2}}
	\aleq b_{k_{1}} c_{k_{2}}.
\end{equation*}
which lacks the necessary exponential gain in $k_{1} - k_{2}$. 

Here the idea is to use the $\tilde{Z}_{s'}^{1/2}$ bound \eqref{eq:fe-L1Linfty}. We introduce a small number $\dlt_{1} > 0$ to be determined later. We split the $j$-summation in $I_{2}$ to $I_{2}' =\sum_{j < k_{2} - 10 + \dlt_{1} (k_{2} - k_{1})} (\cdots) $ and $I_{2}'' = \sum_{j \in [k_{2} - 10 + \dlt_{1} (k_{2} - k_{1}), k_{2} - 10)}(\cdots)$. For the first sum $I_{2}'$, we use Proposition~\ref{prop:nf}, \eqref{eq:fe-L2L2} and \eqref{eq:fe-L2Linfty} as before to estimate
\begin{align*}
\nrm{I_{2}'}_{L^{1} L^{2}} 
	\aleq & \sum_{j < k_{2} - 10 + \dlt_{1} (k_{2} - k_{1})} 2^{\ell}
			\bb( \sum_{\calC_{k_{2}}(\ell)} \nrm{P_{\calC_{k_{2}}(\ell)} Q_{\leq j}^{s} \psi_{k_{1}}}_{L^{2} L^{\infty}}^{2} \bb)^{1/2} 
			\nrm{Q_{j}^{s'} \varphi_{k_{2}}}_{L^{2} L^{2}} \\
	\aleq & \sum_{j < k_{2} - 10 + \dlt_{1} (k_{2} - k_{1})} 2^{\frac{1}{2} \ell} b_{k_{1}} c_{k_{2}}\aleq 2^{\frac{\dlt_{1}}{4} (k_{2} - k_{1})} b_{k_{1}} c_{k_{2}},
\end{align*}
For the second sum $I_{2}''$, we use \eqref{eq:q-LqL2}, H\"older's inequality and the frequency envelope bounds \eqref{eq:fe-L2} and \eqref{eq:fe-L1Linfty} to bound
\begin{align*}
\nrm{I_{2}''}_{L^{1} L^{2}} 
	\aleq & \sum_{j \in [k_{2} - 10 + \dlt_{1} (k_{2} - k_{1}), k_{2} - 10)} 
			\nrm{\psi_{k_{1}}}_{L^{\infty} L^{2}}
			\nrm{Q_{j}^{s'} \varphi_{k_{2}}}_{L^{1} L^{\infty}} \\
	\aleq & \sum_{j \in [k_{2} - 10 + \dlt_{1} (k_{2} - k_{1}), k_{2} - 10)}  2^{-\frac{1}{2} k_{1}} 2^{\frac{1}{2} k_{2}} 2^{5(k_{2} - j)} b_{k_{1}} c_{k_{2}}
	\aleq 2^{(\frac{1}{2} - 5 \dlt_{1})(k_{2} - k_{1})} b_{k_{1}} c_{k_{2}}.
\end{align*}
In conclusion, we have
\begin{equation*}
	\nrm{I_{2}}_{L^{1} L^{2}} \aleq  2^{\min\set{\frac{\dlt_{1}}{4}, \frac{1}{2} - 5 \dlt_{1}} (k_{2} - k_{1})} b_{k_{1}} c_{k_{2}},
\end{equation*}
which is acceptable once we choose $0 < \dlt_{1} < \frac{1}{10}$.

%
%

\pfstep{Step~2.3:~(low-high) interaction, $k_{1} = k_{\min}$}
This case is strictly easier than Step~2.2, thanks to the additional gain $2^{k_{\min} - \min \set{k_{0}, k_{2}}}\simeq 2^{k_{1} - k_{2}}$ in Proposition~\ref{prop:nfs}; in particular, the use of the $\tilde{Z}_{s}^{1/2}$ bound \eqref{eq:fe-L1Linfty} is not necessary. We omit the details.

\subsubsection*{Step~3:~Proof of \eqref{eq:bi-a-est3}}
We proceed similarly to Step~2, replacing the null form $\NF^{\ast}$ by $\NF_{s s'}$ and thus Proposition~\ref{prop:nfs} by Proposition~\ref{prop:nf}. The proof applies verbatim until reduction to the low modulation case (i.e., before Steps~2.1--2.3). A minor difference now is that the factor $2^{k_{\min} - \min \set{k_{1}, k_{2}}}$ does \emph{not}\footnote{Now this factor gains another $2^{k_{0} - k_{1}}$ in the (high-high) interaction case (i.e., analogue of Step~2.1), which was already fine.} gain $2^{k_{1} - k_{2}}$ in the (low-high) interaction case (i.e., analogue of Step~2.3); however, the same proof as in the (high-low) case applies (Step~2.2).

%

\subsection{Proof of Proposition~\ref{prop:dirac}, part I: $N_{\pm}^{1/2}$-bounds for $\NR$} \label{subsec:bi-d-r}
In this subsection, we prove \eqref{eq:nre}--\eqref{eq:nrs} concerning the remainders $\NR^{E}$, $\NR^{R}$ and $\NR^{S}_{s}$. 
\subsubsection*{Step~0:~Reduction to dyadic estimates}
Recall that $\calN^{E}(\tilde{P}_{k_{1}} \cdot, \tilde{P}_{k_{2}} \cdot) = \calL$, $\calN^{R} = \NF$ and $\Pi_{s'} \calN^{S}_{s} = \NF_{s s'}^{\ast}$, which vanish when applied to inputs $ A_{k_{1}}, \psi_{k_{2}} $ unless (say) $k_{1} \geq k_{2} - 20$.
The condition $k_{1} \geq k_{2} - 20$ effectively eliminates the (low-high) interaction (i.e., $k_{\min} = k_{1}$). More precisely, if $k_{1} = k_{\min}$ and $k_{1} \geq k_{2} - 20$, then all three frequencies must be comparable (i.e., $\abs{k_{\max} - k_{\min}} \leq C$) thanks to the Littlewood-Paley trichotomy $\abs{k_{\max} - k_{\med}} \leq 5$.

Under the normalization \eqref{eq:fe-n} and the condition $k_{1} \geq k_{2} - 20$, we claim:
\begin{align} 
	\nrm{P_{k_{0}} \calL(B_{k_{1}}, \psi_{k_{2}})}_{N_{s'}^{1/2}} 
	\aleq & 2^{\frac{1}{2} (k_{\min} - k_{\max})} a_{k_{1}} b_{k_{2}}, \label{eq:bi-d-r-1} \\
	\nrm{P_{k_{0}} \calN (A_{k_{1}}, \psi_{k_{2}})}_{N_{s'}^{1/2}} 
	\aleq & 2^{\dlt_{0} (k_{\min} - k_{\max})} a_{k_{1}} b_{k_{2}} ,	\label{eq:bi-d-r-2}\\
	\nrm{P_{k_{0}} \calN^{\ast}_{s s'}(A_{k_{1}}, \psi_{k_{2}})}_{N_{s'}^{1/2}} 
	\aleq & 2^{\dlt_{0} (k_{\min} - k_{\max})} a_{k_{1}} b_{k_{2}}.	\label{eq:bi-d-r-3}
\end{align}
From these estimates, \eqref{eq:nre}--\eqref{eq:nrs} follow as in the proof of \eqref{eq:a0}--\eqref{eq:axs} from the dyadic bounds \eqref{eq:bi-a-est1}--\eqref{eq:bi-a-est3} in Section~\ref{subsec:bi-a}; we omit the details.

\subsubsection*{Step~1:~Proof of \eqref{eq:bi-d-r-1}}
By \eqref{eq:ellip-1} in Lemma~\ref{lem:ellip} and the frequency envelope bounds \eqref{eq:fe-Y} and \eqref{eq:fe-L2Linfty-0}, we have
\begin{equation*}
	\nrm{P_{k_{0}} \calL(B_{k_{1}}, \psi_{k_{2}})}_{L^{1} \dot{H}^{1/2}}
	\aleq 2^{k_{\min}} 2^{\frac{1}{2} k_{0}} 2^{-\frac{3}{2} k_{1}} a_{k_{1}} b_{k_{2}},
\end{equation*}
which implies \eqref{eq:bi-d-r-1} under the condition $k_{1} \geq k_{2} - 20$.

\subsubsection*{Step~2:~Proof of \eqref{eq:bi-d-r-2}}
As before, we begin with the high modulation contribution.
\begin{lemma} \label{lem:dr-hi-mod}
Assume the normalization \eqref{eq:fe-n}.
For any $k_{0}, k_{1}, k_{2}, j \in \bbZ$ such that $k_{1} \geq k_{2} - 20$, we have
\begin{align*}
	\nrm{P_{k_{0}} Q_{j}^{s' }\calL(A_{k_{1}}, \psi_{k_{2}})}_{N_{s'}^{1/2}}
	\aleq & 2^{\frac{1}{2}(k_{\min} - j)} 2^{\frac{1}{2}(k_{\min} - k_{\max})} a_{k_{1}} b_{k_{2}}, \\
 	\nrm{P_{k_{0}} \calL(Q_{j} A_{k_{1}}, \psi_{k_{2}})}_{N_{s'}^{1/2}}
	\aleq & 2^{\frac{1}{2}(k_{\min} - j)} 2^{\frac{1}{2}(k_{\min} - k_{\max})} a_{k_{1}} b_{k_{2}} , 	\\
	\nrm{P_{k_{0}} \calL(A_{k_{1}}, Q_{j}^{s} \psi_{k_{2}})}_{N_{s'}^{1/2}} 
	\aleq & 2^{\frac{1}{2}(k_{\min} - j)} 2^{\frac{1}{2}(k_{\min} - k_{\max})} a_{k_{1}} b_{k_{2}}.
\end{align*}
\end{lemma}
\begin{proof}
Like Lemma~\ref{lem:ax-hi-mod}, this lemma is a corollary of Lemma~\ref{lem:ellip}. 
Applying \eqref{eq:ellip-0} with $(f, g) = (A_{k_{1}}, \psi_{k_{2}})$ and the frequency envelope bounds \eqref{eq:fe-L2} and \eqref{eq:fe-L2Linfty-0}, we have
\begin{equation*}
	2^{\frac{1}{2} k_{0}} 2^{-\frac{1}{2} j}\nrm{P_{k_{0}} Q_{j}^{s' }\calL(A_{k_{1}}, \psi_{k_{2}})}_{L^{2} L^{2}}
	\aleq 2^{k_{\min}} 2^{\frac{1}{2} k_{0}} 2^{-k_{1}} 2^{-\frac{1}{2} j} a_{k_{1}} b_{k_{2}},
\end{equation*}
which proves the first estimate under the condition $k_{1} \geq k_{2} - 20$. On the other hand, applying \eqref{eq:ellip-0} in two different ways, then using the frequency envelope bounds \eqref{eq:fe-L2L2} and \eqref{eq:fe-L2Linfty-0}, we have
\begin{align*}
	2^{\frac{1}{2} k_{0}} \nrm{P_{k_{0}} \calL(Q_{j} A_{k_{1}}, \psi_{k_{2}})}_{L^{1} L^{2}}
	\aleq & 2^{k_{\min}} 2^{\frac{1}{2} k_{0}} 2^{-k_{1}} 2^{-\frac{1}{2} j} a_{k_{1}} b_{k_{2}}, \\
	2^{\frac{1}{2} k_{0}} \nrm{P_{k_{0}} \calL(A_{k_{1}}, Q_{j}^{s} \psi_{k_{2}})}_{L^{1} L^{2}}
	\aleq & 2^{k_{\min}} 2^{\frac{1}{2} k_{0}} 2^{-\frac{1}{2} k_{1}} 2^{-\frac{1}{2} k_{2}} 2^{-\frac{1}{2} j} a_{k_{1}} b_{k_{2}}, 
\end{align*}
which imply the other two estimates under the condition $k_{1} \geq k_{2} - 20$. \qedhere
\end{proof}

Proceeding as in Step~2 of Section~\ref{subsec:bi-a}, where we use Lemma~\ref{lem:dr-hi-mod} instead of Lemma~\ref{lem:ax-hi-mod}, the proof of \eqref{eq:bi-d-r-2} is reduced to handling the contribution of
\begin{equation*}
	\sum_{s_{1}} P_{k_{0}} Q_{< k_{\min} - 10}^{s'} \NF(Q_{< k_{\min} - 10}^{s_{1}} A_{k_{1}}, Q_{< k_{\min} - 10}^{s} \psi_{k_{2}}) 
	= I_{0} + I_{1} + I_{2},
\end{equation*}
where
\begin{align}
	I_{0} =& \sum_{s_{1}} \sum_{j < k_{\min} - 10} P_{k_{0}} Q_{j}^{s'} \NF(Q_{\leq j}^{s_{1}} A_{k_{1}}, Q_{\leq j}^{s} \varphi_{k_{2}}), \label{eq:nd-I0} \\
	I_{1} =& \sum_{s_{1}} \sum_{j < k_{\min} - 10} P_{k_{0}} Q_{< j}^{s'} \NF(Q_{j}^{s_{1}} A_{k_{1}}, Q_{< j}^{s} \varphi_{k_{2}}), \label{eq:nd-I1} \\
	I_{2} =& \sum_{s_{1}} \sum_{j < k_{\min} - 10} P_{k_{0}} Q_{< j}^{s'} \NF(Q_{\leq j}^{s_{1}} A_{k_{1}}, Q_{ j}^{s} \varphi_{k_{2}}). \label{eq:nd-I2}
\end{align}
We now split into two (slightly overlapping) cases, which roughly correspond to (high-high) and (high-low) interaction:
\pfstep{Step~2.1:~(high-high) interaction, $k_{0} = k_{\min} + O(1)$}
Let $\ell = \frac{1}{2} (j - k_{\min})$. Using Proposition~\ref{prop:nf} \emph{neglecting} the gain $2^{k_{\min} - \min\set{k_{1}, k_{2}}} \simeq 2^{k_{\min} - k_{\max}}$, we have
\begin{align*}
	\nrm{I_{0}}_{X^{1/2, -1/2}_{s', 1}} 
	\aleq & \sum_{j < k_{\min} - 10} 2^{-\frac{1}{2} j} 2^{\frac{1}{2} k_{0}} 2^{\ell} 
			\nrm{A_{k_{1}}}_{L^{\infty} L^{2}} 
			\bb( \sum_{\calC_{k_{\min}}(\ell)} \nrm{P_{\calC_{k_{\min}}(\ell)} Q_{< j}^{s} \psi_{k_{2}}}_{L^{2} L^{\infty}}^{2} \bb)^{1/2}, \\
	\nrm{I_{1}}_{L^{1} \dot{H}^{1/2}} 
	\aleq & \sum_{s_{1}} \sum_{j < k_{\min} - 10} 2^{\frac{1}{2} k_{0}} 2^{\ell} 
			\nrm{Q_{j}^{s_{1}} A_{k_{1}}}_{L^{2} L^{2}} 
			\bb( \sum_{\calC_{k_{\min}}(\ell)} \nrm{P_{\calC_{k_{\min}}(\ell)} Q_{< j}^{s} \psi_{k_{2}}}_{L^{2} L^{\infty}}^{2} \bb)^{1/2}, \\
	\nrm{I_{2}}_{L^{1} \dot{H}^{1/2}} 
	\aleq & \sum_{s_{1}} \sum_{j < k_{\min} - 10} 2^{\frac{1}{2} k_{0}} 2^{\ell} 
			\bb( \sum_{\calC_{k_{\min}}(\ell)} \nrm{P_{\calC_{k_{\min}}(\ell)} Q_{\leq j}^{s_{1}} A_{k_{1}}}_{L^{2} L^{\infty}}^{2} \bb)^{1/2} 
			\nrm{Q_{j}^{s} \psi_{k_{2}}}_{L^{2} L^{2}}.
\end{align*}
Then by the frequency envelope bounds \eqref{eq:fe-L2}, \eqref{eq:fe-L2L2} and \eqref{eq:fe-L2Linfty}, we obtain
\begin{equation*}
	\nrm{I_{0}}_{N_{s'}^{1/2}} + \nrm{I_{1}}_{N_{s'}^{1/2}} + \nrm{I_{2}}_{N_{s'}^{1/2}} 
	\aleq \sum_{j < k_{\min} - 10} 2^{\frac{1}{4} \ell} 2^{k_{\min} - k_{\max}} a_{k_{1}} b_{k_{2}}
	\aleq 2^{k_{\min} - k_{\max}} a_{k_{1}} b_{k_{2}},
\end{equation*}
which is acceptable.

\pfstep{Step~2.2:~(high-low) interaction, $k_{2} = k_{\min}$}
As in Step~2.2 of Section~\ref{subsec:bi-a}, we need to use the $\tilde{Z}_{s}^{1/2}$ bound \eqref{eq:fe-L1Linfty} in addition to Proposition~\ref{prop:nf}. As before, let $\ell = \frac{1}{2}(j - k_{\min})$ and $\dlt_{1} \in (0, 1/10)$ be the small constant in Step~2.2 of Section~\ref{subsec:bi-a}.  By Proposition~\ref{prop:nf} we have
\begin{align*}
	\nrm{I_{0}}_{X^{1/2, -1/2}_{1}} 
	\aleq & \sum_{j < k_{2} - 10} 2^{\frac{1}{2} j} 2^{\frac{1}{2} k_{0}} 2^{\ell} 
			\nrm{A_{k_{1}}}_{L^{\infty} L^{2}} 
			\bb( \sum_{\calC_{k_{2}}(\ell)} \nrm{P_{\calC_{k_{2}}(\ell)} Q_{< j}^{s} \psi_{k_{2}}}_{L^{2} L^{\infty}}^{2} \bb)^{1/2},	\\
%
\nrm{I_{1}}_{L^{1} L^{2}} 
	\aleq & \sum_{s_{1}} \sum_{j < k_{2} - 10} 2^{\frac{1}{2} k_{0}} 2^{\ell} 
			\nrm{Q_{j}^{s_{1}} A_{k_{1}}}_{L^{2} L^{2}} 
			\bb( \sum_{\calC_{k_{2}}(\ell)} \nrm{P_{\calC_{k_{2}}(\ell)} Q_{< j}^{s} \psi_{k_{2}}}_{L^{2} L^{\infty}}^{2} \bb)^{1/2},
\end{align*}
which are bounded by $\aleq 2^{\frac{1}{2} (k_{2} - k_{1})} a_{k_{1}} b_{k_{2}}$ thanks to the frequency envelope bounds \eqref{eq:fe-L2}, \eqref{eq:fe-L2L2} and \eqref{eq:fe-L2Linfty}. 
For $I_{2}$, we split the $j$-summation and write $I_{2} = I_{2}' + I_{2}''$, where $I_{2}' = \sum_{j < k_{2} - 10 + \dlt_{1}(k_{2} - k_{1})} (\cdots)$ and $I_{2}'' = \sum_{j \in [k_{2} - 10 + \dlt_{1}(k_{2} - k_{1}), k_{2} - 10)} (\cdots)$. For $I_{2}'$, we use Proposition~\ref{prop:nf} to obtain
\begin{align*}
	\nrm{I_{2}'}_{L^{1} \dot{H}^{1/2}}
	\aleq & \sum_{s_{1}} \sum_{j < k_{2} - 10 + \dlt_{1}(k_{2} - k_{1})} 2^{\ell} 2^{\frac{1}{2} k_{0}}
			\bb( \sum_{\calC_{k_{2}}(\ell)} \nrm{P_{\calC_{k_{2}}(\ell)} Q_{< j}^{s_{1}} A_{k_{1}}}_{L^{2} L^{\infty}}^{2} \bb)^{1/2}
			\nrm{Q_{j}^{s} \psi_{k_{2}}}_{L^{2} L^{2}},
\end{align*}
which, in turn, can be bounded by $\aleq 2^{\frac{\dlt_{1}}{4}(k_{2} - k_{1})} a_{k_{1}} b_{k_{2}}$ using 
\eqref{eq:fe-L2L2} and \eqref{eq:fe-L2Linfty}. For $I_{2}''$, we use \eqref{eq:q-LqL2}, H\"older's inequality, \eqref{eq:fe-L2} and \eqref{eq:fe-L1Linfty} to bound
\begin{align*}
\nrm{I_{2}''}_{L^{1} \dot{H}^{1/2}} 
	\aleq & \sum_{j \in [k_{2} - 10 + \dlt_{1} (k_{2} - k_{1}), k_{2} - 10)} 2^{\frac{1}{2} k_{0}}
			\nrm{A_{k_{1}}}_{L^{\infty} L^{2}}
			\nrm{Q_{j}^{s} \psi_{k_{2}}}_{L^{1} L^{\infty}} \\
	\aleq & \sum_{j \in [k_{2} - 10 + \dlt_{1} (k_{2} - k_{1}), k_{2} - 10)}  2^{\frac{1}{2} k_{0}} 2^{-k_{1}} 2^{\frac{1}{2} k_{2}} 2^{5(k_{2} - j)} a_{k_{1}} b_{k_{2}},
\end{align*}
which is bounded by $2^{(\frac{1}{2} - 5 \dlt_{1})(k_{2} - k_{1})} a_{k_{1}} b_{k_{2}}$ and thus acceptable (since $\dlt_{1} < 1/10$).

%

\subsubsection*{Step~3:~Proof of \eqref{eq:bi-d-r-3}}
The argument in Step~2 applies exactly, with $\NF$ and Proposition~\ref{prop:nf} replaced by $\NF^{\ast}_{s s'}$ and Proposition~\ref{prop:nfs}, respectively; note that this is possible since we have \emph{not} used the extra gain $2^{k_{\min} - \min\set{k_{1}, k_{2}}}$ from Proposition~\ref{prop:nf} in Step~2.1 above. We omit the details.

\begin{remark} \label{rem:bi-bal-freq}
In the course of Step~2, we have proved the bound
\begin{equation} \label{eq:bi-bal-freq}
	\nrm{P_{k_{0}} \NF(A_{k_{1}}, \psi_{k_{2}})}_{N_{s'}^{1/2}}
	\aleq \nrm{A_{k_{1}}}_{S^{1}} \nrm{\psi_{k_{2}}}_{S_{s}^{1/2}}
\end{equation}
when $k_{0} = k_{\min} + O(1)$ and $k_{1} > k_{2} - 20$. In fact, the number $20$ does not play any role, and the same bound holds (with an adjusted constant) when all three $k_{0}, k_{1} k_{2}$ are within an $O(1)$-interval of each other. This will be useful in Sections~\ref{sec:tri} and \ref{sec:para}.
\end{remark}

\subsection{Proof of Proposition~\ref{prop:dirac}, part II: $N^{1/2}_{\pm}$-bounds for $\Diff[A]$} \label{subsec:bi-d-diff}
Here we prove \eqref{eq:diffe-opp}, \eqref{eq:diffr-opp} and \eqref{eq:diffs} concerning the paradifferential terms $\Diff^{E}[A_{0}]$, $\Diff^{R}[A_{x}]$ and $\Diff^{S}_{s}[A_{x}]$.
\subsubsection*{Step~0:~Reduction to dyadic estimates}  
As before, note that $\calN^{E}(\tilde{P}_{k_{1}} \cdot, \tilde{P}_{k_{2}} \cdot) = \calL$, $\calN^{R} = \NF$ and $\Pi_{s'} \calN^{S}_{s} = \NF_{s s'}^{\ast}$, 
and $ \Diff^{E} [A_{0}], \Diff^{R} [A_{x}], \Pi_{s'} \Diff^{S}_{s} [A_{x}] $ vanish when applied to $ A_{k_{1}}, \psi_{k_{2}} $ unless (say) $k_{1} < k_{2} - 5$. By Littlewood-Paley trichotomy ($\abs{k_{\max} - k_{\med}} \leq 5$), we only need to consider the (low-high) interaction, i.e., $k_{\min} = k_{1}$ and $k_{0} = k_{2} + O(1)$.

Under the normalization \eqref{eq:fe-n} and the condition $k_{1} < k_{2} - 5$, we claim:
\begin{align} 
	\nrm{P_{k_{0}} \calL(B_{k_{1}}, \psi_{k_{2}})}_{N_{-s}^{1/2}}
	\aleq & 2^{\frac{1}{4} (k_{\min} - k_{\max})} a_{k_{1}} b_{k_{2}},	\label{eq:bi-d-diff-1} \\
	\nrm{P_{k_{0}} \calL(A_{k_{1}}, \psi_{k_{2}})}_{N_{-s}^{1/2}}
	\aleq & 2^{\frac{1}{4} (k_{\min} - k_{\max})} a_{k_{1}} b_{k_{2}},	\label{eq:bi-d-diff-2} \\
	\nrm{P_{k_{0}} \calN^{\ast}_{+}(A_{k_{1}}, \psi_{k_{2}})}_{N_{s}^{1/2}}
	\aleq & 2^{\frac{1}{4} (k_{\min} - k_{\max})} a_{k_{1}} b_{k_{2}}.	\label{eq:bi-d-diff-3}
\end{align}
We remind the reader that $\psi$ is assumed to be normalized in $(S^{1/2}_{s})_{b}$; hence \eqref{eq:bi-d-diff-1} and \eqref{eq:bi-d-diff-2} concern the case when the output is estimated in the opposite-signed $N^{1/2}_{-s}$ space, whereas \eqref{eq:bi-d-diff-3} is the same sign case. 

The estimates \eqref{eq:diffe-opp} and \eqref{eq:diffr-opp} follow from \eqref{eq:bi-d-diff-1} and \eqref{eq:bi-d-diff-2}, respectively, whereas \eqref{eq:diffs} may be proved by combining \eqref{eq:bi-d-diff-2} (opposite sign case) and \eqref{eq:bi-d-diff-3} (same sign case). As the proof is similar to Step~0 of Section~\ref{subsec:bi-a}, we omit the details.

\subsubsection*{Step~1:~Case of opposite waves}
Here we prove \eqref{eq:bi-d-diff-1} and \eqref{eq:bi-d-diff-2}. Henceforth we write $f$ for either $B$ or $A$.
We begin with the case when the output or $\psi$ has high modulation.
\begin{lemma} \label{lem:diff-hi-mod}
Assume the normalization \eqref{eq:fe-n}.
For any $k_{0}, k_{1}, k_{2}, j \in \bbZ$ such that $k_{1} < k_{2} - 5$, we have
\begin{align*}
	2^{\frac{1}{2} k_{0}} 2^{-\frac{1}{2} j}\nrm{P_{k_{0}} Q_{j}^{s' }\calL(f_{k_{1}}, \psi_{k_{2}})}_{L^{2} L^{2}}
	\aleq & 2^{\frac{1}{2}(k_{1} - j)} 2^{-\frac{1}{2} k_{1}} \nrm{f_{k_{1}}}_{L^{2} L^{\infty}} b_{k_{2}} ,	\\
	2^{\frac{1}{2} k_{0}} \nrm{P_{k_{0}} \calL(f_{k_{1}}, Q_{j}^{s} \psi_{k_{2}})}_{L^{1} L^{2}} 
	\aleq & 2^{\frac{1}{2}(k_{1} - j)} 2^{-\frac{1}{2} k_{1}} \nrm{f_{k_{1}}}_{L^{2} L^{\infty}} b_{k_{2}} .
\end{align*}
\end{lemma}
\begin{proof}
The first estimate follows from the H\"older inequality $L^{2} L^{\infty} \times L^{\infty} L^{2} \to L^{2} L^{2}$ and the frequency envelope bound \eqref{eq:fe-L2}. Similarly, the second estimate follows from the H\"older inequality $L^{2} L^{\infty} \times L^{2} L^{2} \to L^{1} L^{2}$ and the frequency envelope bound \eqref{eq:fe-L2L2}. \qedhere
\end{proof}
By the frequency envelope bounds \eqref{eq:fe-L2Linfty-0} and \eqref{eq:fe-Y}, note that $f = B$ and $A$ yield the common bound
\begin{equation} \label{eq:fe-L2Linfty-AB}
	\nrm{B_{k_{1}}}_{L^{2} L^{\infty}} \aleq 2^{2 k_{1}} \nrm{B_{k_{1}}}_{L^{2} L^{2}} \aleq 2^{\frac{1}{2} k_{1}} a_{k_{1}}, \qquad
	\nrm{A_{k_{1}}}_{L^{2} L^{\infty}} \aleq 2^{\frac{1}{2} k_{1}} a_{k_{1}}.
\end{equation}
Since $k_{1} < k_{2} - 5$, we have $k_{\min} = k_{1}$ and $k_{0}, k_{2} = k_{\max} + O(1)$. Then from Lemma~\ref{lem:diff-hi-mod} and \eqref{eq:q-LqL2}, it follows that
\begin{align*}
	\nrm{P_{k_{0}} Q_{\geq k_{0} + \frac{1}{2} (k_{1} - k_{0}) - C'_{1}}^{-s} \calL(f_{k_{1}}, \psi_{k_{2}})}_{N_{-s}^{1/2}} 
	\aleq_{C'_{1}} & 2^{\frac{1}{4}(k_{\min} - k_{\max})} a_{k_{1}} b_{k_{2}}, \\
\nrm{P_{k_{0}} Q_{< k_{0} + \frac{1}{2} (k_{1} - k_{0}) - C'_{1}}^{-s} \calL(f_{k_{1}}, Q_{\geq k_{2} + \frac{1}{2} (k_{1} - k_{2}) - C'_{1}}^{s} \psi_{k_{2}})}_{N_{-s}^{1/2}} 
	\aleq_{C'_{1}} & 2^{\frac{1}{4}(k_{\min} - k_{\max})} a_{k_{1}} b_{k_{2}}, 
\end{align*}
which are acceptable for any $C'_{1} \geq 0$. It remains to treat the contribution of
\begin{equation*}
	I = P_{k_{0}} Q_{< \frac{1}{2} (k_{0} + k_{1}) - C'_{1}}^{-s} \calL(f_{k_{1}}, Q_{< \frac{1}{2} (k_{1} + k_{2}) - C'_{1}}^{s} \psi_{k_{2}})
\end{equation*}
for some $C'_{1} \geq 0$ to be determined. 
We now use the `geometry of the cone' to force modulation localization of $f$. 
\begin{lemma} \label{lem:geom-cone-opp}
Let $k_{0}, k_{1}, k_{2}, j_{0}, j_{1}, j_{2} \in \bbZ$ be such that $\abs{k_{0} - k_{2}} \leq 5$ and $k_{1} \leq \min\set{k_{0}, k_{2}} - 5$. Assume furthermore that $j_{0} \leq k_{0} - C'_{1}$ and  $j_{2} \leq k_{2} - C'_{1}$ for a sufficiently large $C'_{1} > 0$. For any sign $s \in \set{+, -}$, the expression
\begin{equation*}
	P_{k_{0}} Q^{-s}_{j_{0}} \calL (P_{k_{1}} Q_{j_{1}} f, P_{k_{2}} Q_{j_{2}}^{s} g)
\end{equation*}
vanishes unless $j_{1} = k_{\max} + O(1)$.
\end{lemma}
\begin{proof}
By duality, it suffices to consider the expression
\begin{equation*}
	\iint P_{k_{0}} Q^{+}_{j_{0}} h \calL (P_{k_{1}} Q_{j_{1}} f, P_{k_{2}} Q_{j_{2}}^{+} g) \, \ud t \ud x.
\end{equation*}
We proceed as the proof of Lemma~\ref{lem:geom-cone}. If the expression does not vanish, there exists $\Xi^{i}$ $(i=0, 1, 2)$ such that $\sum_{i} \Xi^{i}  = 0$ and $\Xi^{i} \in \set{\abs{\xi} \simeq 2^{k_{i}}, \ \abs{\tau - s_{i} \abs{\xi}} \simeq 2^{j_{i}}}$, where $s_{0} = s_{2} = s$ and $s_{1}$ is the sign of $\tau$. Consider the quantity $H = s_{0} \abs{\xi^{0}} + s_{1} \abs{\xi^{1}} + s_{2} \abs{\xi^{2}}$. Subtracting $\sum_{i} \tau^{i} = 0$ and using the hypothesis on $k_{0}, k_{2}, j_{0}, j_{2}$, we have
\begin{equation*}
	\abs{H} \aleq 2^{j_{1}} + 2^{k_{\max} - C'_{1}}.
\end{equation*}
On the other hand, since $s_{0} = s_{2} = s$ and $k_{1} \leq  \min\set{k_{0}, k_{2}} - 5$, we have
\begin{equation*}
	\abs{H} = \abs{s \abs{\xi^{0}} + s_{1} \abs{\xi^{1}} + s \abs{\xi^{2}}} \simeq 2^{k_{\max}}.
\end{equation*}
Taking $C'_{1}$ sufficiently large, it follows that $j_{\max} \geq k_{\max} - C$ for some constant $C$ independent of $C'_{1}$. Taking $C'_{1}$ even larger so that $j_{1} \geq \max\set{j_{0}, j_{2}} + 5$, we have $\abs{H} \simeq 2^{j_{1}}$ and the claim follows. \qedhere
\end{proof}

Choosing $C'_{1} \geq 0$ to be sufficiently large, Lemma~\ref{lem:geom-cone-opp} is applicable to $I$. Hence
\begin{equation*}
	I = \sum_{j = k_{\max} + O(1)} P_{k_{0}} Q_{< \frac{1}{2} (k_{0} + k_{1}) - C'_{1}}^{-s} \calL(Q_{j} f_{k_{1}}, Q_{< \frac{1}{2} (k_{1} + k_{2}) - C'_{1}}^{s} \psi_{k_{2}}),
\end{equation*}
By \eqref{eq:q-LqL2}, \eqref{eq:ellip-1} and the frequency envelope bound \eqref{eq:fe-L2Linfty-0}, we may estimate
\begin{equation} \label{eq:ellip-1-AB}
\begin{aligned}
	2^{\frac{1}{2} k_{0}} \nrm{I}_{L^{1} L^{2}} 
	\aleq & \sum_{j = k_{\max} + O(1)} 2^{\frac{1}{2} k_{0}} \nrm{Q_{j} f_{k_{1}}}_{L^{2} L^{2}} \bb( \sum_{\calC_{k_{\min}}(0)} \nrm{P_{\calC_{k_{\min}}(0)} \psi_{k_{2}}}_{L^{2} L^{\infty}}^{2} \bb)^{1/2} \\
	\aleq &  \sum_{j = k_{\max} + O(1)} 2^{k_{\min}} 2^{\frac{1}{2} k_{0}} \nrm{Q_{j} f_{k_{1}}}_{L^{2} L^{2}} b_{k_{2}}.
\end{aligned}
\end{equation}
By the frequency envelope bounds \eqref{eq:fe-L2L2} and \eqref{eq:fe-Y}, we have the following common bound for $f = B$ or $A$ when $j > k_{1}$:
\begin{equation} \label{eq:fe-L2L2-AB}
	\nrm{Q_{j} B_{k_{1}}}_{L^{2} L^{2}} \aleq 2^{-j} 2^{-\frac{1}{2} k_{1}} a_{k_{1}}, \quad
	\nrm{Q_{j} A_{k_{1}}}_{L^{2} L^{2}} \aleq 2^{-j} 2^{-\frac{1}{2} k_{1}} a_{k_{1}}.
\end{equation}
Therefore, 
\begin{equation*}
	2^{\frac{1}{2} k_{0}} \nrm{I}_{L^{1} L^{2}} \aleq 2^{\frac{1}{2} (k_{\min} - k_{\max})} a_{k_{1}} b_{k_{2}} ,
\end{equation*}
which completes the proof of \eqref{eq:bi-d-diff-1} and \eqref{eq:bi-d-diff-2}.

\subsubsection*{Step~2:~Proof of \eqref{eq:bi-d-diff-3}}
This is one of the key estimates showing that spinorial nonlinearities have better structure than the Riesz-transform parts. The idea is that the null form $\NF^{\ast}_{+}$ gains an extra factor $2^{k_{\min} - k_{\max}}$ in the low-high case. 

We begin with the high modulation bounds:
\begin{lemma} \label{lem:diffs-hi}
For any $k_{0}, k_{1}, k_{2} \in \bbZ$ such that $\abs{k_{\max} - k_{\med}} \leq 5$ and $k_{1} < k_{2} - 5$, we have
\begin{align}
	\nrm{P_{k_{0}} \calN^{\ast}_{+}(f_{k_{1}}, g_{k_{2}})}_{L^{2} L^{2}}
	\aleq & 2^{k_{1} - k_{2}} \nrm{f_{k_{1}}}_{L^{2} L^{\infty}} \nrm{g_{k_{2}}}_{L^{\infty} L^{2}} ,	\label{eq:diffs-hi-0}\\
	\nrm{P_{k_{0}} \calN^{\ast}_{+}(f_{k_{1}}, g_{k_{2}})}_{L^{1} L^{2}}
	\aleq & 2^{k_{1} - k_{2}} \nrm{f_{k_{1}}}_{L^{2} L^{\infty}} \nrm{g_{k_{2}}}_{L^{2} L^{2}}, \label{eq:diffs-hi-1}\\
	\nrm{P_{k_{0}} \calN^{\ast}_{+}(f_{k_{1}}, g_{k_{2}})}_{L^{1} L^{2}}
	\aleq & 2^{k_{1} - k_{2}} \nrm{f_{k_{1}}}_{L^{2} L^{2}} \bb( \sum_{\calC_{k_{1}}(0)}\nrm{P_{\calC_{k_{1}}(0)} g_{k_{2}}}_{L^{2} L^{\infty}}^{2} \bb)^{1/2} .	\label{eq:diffs-hi-2}
\end{align}
\end{lemma}
\begin{proof}
The idea is to proceed as in the proof of Lemma~\ref{lem:ellip} (where $\calL$ is replaced by $\NF^{\ast}_{+}$) with the following crucial modification: Instead of the bound \eqref{eq:ellip-atom}, use
\begin{equation}
	\abs{I_{\calC^{0}, \calC^{1}, \calC^{2}}(t)}
	\aleq \tht \nrm{P_{\calC^{0}} h_{k_{0}}(t)}_{L^{q_{0}}} \nrm{P_{\calC^{1}} f_{k_{1}}(t)}_{L^{q_{1}}} \nrm{P_{\calC^{2}} g_{k_{2}}(t)}_{L^{q_{2}}},
\end{equation}
where $\tht = \max\set{\abs{\angle(\calC^{0}, -\calC^{2})}, 2^{k_{1} - k_{0}}, 2^{k_{1} - k_{2}}}$. This bound follows from Proposition~\ref{prop:nf-basic}; note that $2^{k_{1} - k_{i}}$ is the angular dimension of $\calC^{i}$ for $i=0, 2$. By Statement~(2) of Lemma~\ref{lem:box-orth-0} and the hypothesis on $k_{0}, k_{1}, k_{2}$, it follows that $\tht \simeq 2^{k_{1} - k_{2}}$. Then proceeding as in the proof of Lemma~\ref{lem:ellip}, we directly obtain \eqref{eq:diffs-hi-2}. The other two estimates \eqref{eq:diffs-hi-0} and \eqref{eq:diffs-hi-1} also follow from the same proof by switching the roles of $f, g$ and using the obvious bound
\begin{equation*}
\bb( \sum_{\calC_{k_{1}}(0)}\nrm{P_{\calC_{k_{1}}(0)} f_{k_{1}}}_{L^{2} L^{\infty}}^{2} \bb)^{1/2}
\simeq \nrm{f_{k_{1}}}_{L^{2} L^{\infty}}. \qedhere
\end{equation*}
\end{proof}
By Lemma~\ref{lem:diffs-hi} and the frequency envelop bounds \eqref{eq:fe-L2}, \eqref{eq:fe-L2L2} and \eqref{eq:fe-L2Linfty-0}, we have
\begin{align*}
	\nrm{P_{k_{0}} Q_{j}^{s}\calN^{\ast}_{+}(A_{k_{1}}, \psi_{k_{2}})}_{N_{s}^{1/2}}
	\aleq & 2^{\frac{1}{2}(k_{1} - j)} 2^{k_{1} - k_{2}} a_{k_{1}} b_{k_{2}} ,	\\
	\nrm{P_{k_{0}} \calN^{\ast}_{+}(A_{k_{1}}, Q_{j}^{s} \psi_{k_{2}})}_{N_{s}^{1/2}} 
	\aleq & 2^{\frac{1}{2}(k_{1} - j)} 2^{k_{1} - k_{2}} a_{k_{1}} b_{k_{2}} , 	\\
	\nrm{P_{k_{0}} \calN^{\ast}_{+}(Q_{j} A_{k_{1}}, \psi_{k_{2}})}_{N_{s}^{1/2}}
	\aleq & 2^{\frac{1}{2}(k_{1} - j)} 2^{\frac{1}{2}(k_{1} - k_{2})} a_{k_{1}} b_{k_{2}} .
\end{align*}
Thanks to the exponential gain in $k_{2} - k_{1}$ (as well as $j - k_{1}$), we may proceed as before (cf. Step~1 of Section~\ref{subsec:bi-a} or \ref{subsec:bi-d-r}) to reduce the proof of \eqref{eq:bi-d-diff-3} to estimating the contribution of
\begin{equation*}
	I = \sum_{s_{1}} P_{k_{0}} Q_{< k_{1} - 10}^{s} \NF^{\ast}_{+}(Q_{< k_{1} - 10}^{s_{1}} A_{k_{1}}, Q_{< k_{1} - 10}^{s} \psi_{k_{2}}).
\end{equation*}
The norm $\nrm{I}_{N_{s}^{1/2}}$ can be bounded by the sum $\sum_{s_{1}} \sum_{j < k_{1} - 10}$ of the terms
\begin{align*}
	\nrm{P_{k_{0}} Q^{s}_{j} \NF^{\ast}_{+}(Q^{s_{1}}_{<j} A_{k_{1}}, Q_{<j}^{s} \psi_{k_{2}})}_{N^{1/2}_{s}} \aleq & 2^{\frac{1}{4} (j - k_{1})} 2^{\frac{1}{2} (k_{1} - k_{2})} a_{k_{1}} b_{k_{2}}, \\ 
	\nrm{P_{k_{0}} Q^{s}_{\leq j} \NF^{\ast}_{+}(Q^{s_{1}}_{j} A_{k_{1}}, Q_{<j}^{s} \psi_{k_{2}})}_{N^{1/2}_{s}} \aleq & 2^{\frac{1}{4} (j - k_{1})} 2^{\frac{1}{2} (k_{1} - k_{2})} a_{k_{1}} b_{k_{2}}, \\
	\nrm{P_{k_{0}} Q^{s}_{\leq j} \NF^{\ast}_{+}(Q^{s_{1}}_{\leq j} A_{k_{1}}, Q_{j}^{s} \psi_{k_{2}})}_{N^{1/2}_{s}} \aleq & 2^{\frac{1}{4} (j - k_{1})} 2^{k_{1} - k_{2}} a_{k_{1}} b_{k_{2}},
\end{align*}
where we used Proposition~\ref{prop:nfs} and the frequency envelope bounds \eqref{eq:fe-L2}, \eqref{eq:fe-L2L2} and \eqref{eq:fe-L2Linfty} to derive the estimates. Observe the crucial exponential gain in $k_{2} - k_{1}$, which arises from the factor $2^{k_{\min} - \min \set{k_{0}, k_{2}}}$ in Proposition~\ref{prop:nfs}. Summing up in $s_{1} \in \set{+, -}$ and $j < k_{1} - 10$, we obtain
\begin{equation*}
	\nrm{I}_{N_{s}^{1/2}} \aleq 2^{\frac{1}{2}(k_{1} - k_{2})} a_{k_{1}} b_{k_{2}} ,
\end{equation*}
which completes the proof of \eqref{eq:bi-d-diff-3}.

\begin{remark} \label{rem:diffr-free}
Repeating Step~2 with $\calN^{\ast}$ replaced by $\calN$ (hence Proposition~\ref{prop:nfs} is replaced by Proposition~\ref{prop:nf}), Lemma~\ref{rem:diffr-free} can be proved. The key differences are the lack of the extra factor $2^{k_{\min} - k_{\max}}$ in Proposition~\ref{prop:nf}, and that $Q_{j} A^{free} = 0$ for any $j \in \bbZ$. We omit the details.
\end{remark}

\subsection{Proof of Proposition~\ref{prop:dirac}, part III: Completion of proof} \label{subsec:bi-d-rest}
We finish the proof of Proposition~\ref{prop:dirac} by establishing the bounds \eqref{eq:ne-himod}--\eqref{eq:ns-z}. Here we do not need to utilize the null structure. Moreover, instead of the normalizing the $(\tilde{S}^{1/2}_{s})_{b}$ norm as in \eqref{eq:fe-n}, we normalize the slightly weaker $(S^{1/2}_{s})_{b}$ norm, i.e., we assume
\begin{equation*}
	\nrm{B}_{Y^{1}_{a}} = \nrm{A}_{S^{1}_{a}} = \nrm{\psi}_{(S^{1/2}_{s})_{b}} = \nrm{\varphi}_{(S^{1/2}_{s'})_{c}} = 1.
\end{equation*}
Note that the bounds \eqref{eq:fe-L2}--\eqref{eq:fe-L2Linfty} and \eqref{eq:fe-Y} still hold.

\subsubsection*{Step~0:~Reduction to dyadic estimates}
Let $f$ denote either $B$ or $A$. Under the normalization \eqref{eq:fe-n}, it clearly suffices to prove the following dyadic bounds:
\begin{align} 
	\nrm{P_{k_{0}} \calL(f_{k_{1}}, \psi_{k_{2}})}_{L^{2} L^{2}} 
	\aleq & 2^{\frac{1}{2} (k_{\min} - k_{\max})} a_{k_{1}} b_{k_{2}},		\label{eq:bi-d-rest-1} \\
	2^{-\frac{3}{2} k_{0}} \nrm{P_{k_{0}} \calL(f_{k_{1}}, \psi_{k_{2}})}_{L^{1} L^{\infty}} 
	\aleq & 2^{\frac{1}{2} (k_{\min} - k_{\max})} a_{k_{1}} b_{k_{2}}.		\label{eq:bi-d-rest-2}
\end{align}

\subsubsection*{Step~1:~Proof of \eqref{eq:bi-d-rest-1}}
We first use Lemma~\ref{lem:ellip} and \eqref{eq:fe-L2} to estimate
\begin{equation*}
	\nrm{P_{k_{0}} \calL(f_{k_{1}}, \psi_{k_{2}})}_{L^{2} L^{2}} 
	\aleq 2^{-\frac{1}{2} k_{2}} \bb( \sum_{\calC_{k_{\min}}(0)} \nrm{P_{\calC_{k_{\min}}(0)} f_{k_{1}}}_{L^{2} L^{\infty}}^{2} \bb)^{1/2} b_{k_{2}}.
\end{equation*}
By Bernstein's inequality, \eqref{eq:fe-L2Linfty-0} and \eqref{eq:fe-Y}, we have
\begin{equation} \label{eq:fe-L2Linfty-AB-0}
\begin{aligned}
	\bb( \sum_{\calC_{k_{\min}}(0)} \nrm{P_{\calC_{k_{\min}}(0)} B_{k_{1}}}_{L^{2} L^{\infty}}^{2} \bb)^{1/2}
	\aleq & 2^{2 k_{\min}} 2^{-\frac{3}{2} k_{1}} a_{k_{1}}, \\
	\bb( \sum_{\calC_{k_{\min}}(0)} \nrm{P_{\calC_{k_{\min}}(0)} A_{k_{1}}}_{L^{2} L^{\infty}}^{2} \bb)^{1/2}
	\aleq & 2^{k_{\min}} 2^{-\frac{1}{2} k_{1}} a_{k_{1}}.
\end{aligned}
\end{equation}
In each case, it can be checked (using Littlewood-Paley trichotomy and dividing into cases $k_{\min} = k_{0}, k_{1}, k_{2}$) that \eqref{eq:bi-d-rest-1} holds.

\subsubsection*{Step~2:~Proof of \eqref{eq:bi-d-rest-2}}
We split into three cases.
\pfstep{Step~2.1:~(high-high) interaction, $k_{0} = k_{\min}$}
Here the factor $2^{-\frac{3}{2} k_{0}}$ on the LHS is detrimental, and we need to perform an orthogonality argument using Lemma~\ref{lem:box-orth-0}. We claim that
\begin{equation} \label{eq:bi-d-rest-2-key}
	\nrm{P_{k_{0}} \calL(f_{k_{1}}, g_{k_{2}})}_{L^{1} L^{\infty}}
	\aleq \bb( \sum_{\calC_{k_{0}}(0)}\nrm{P_{\calC_{k_{0}}(0)} f_{k_{1}}}_{L^{2} L^{\infty}} \bb)^{1/2}
		\bb( \sum_{\calC_{k_{0}}(0)}\nrm{P_{\calC_{k_{0}}(0)} g_{k_{2}}}_{L^{2} L^{\infty}} \bb)^{1/2}
\end{equation}
Once \eqref{eq:bi-d-rest-2-key} is proved, \eqref{eq:bi-d-rest-2} would follow from \eqref{eq:fe-L2Linfty-0} and \eqref{eq:fe-L2Linfty-AB-0}.

To prove the claim, we follow the proof of Lemma~\ref{lem:ellip}. Let $\calC^{0}, \calC^{1}, \calC^{2}$, $I(t)$ and $I_{\calC^{0}, \calC^{1}, \calC^{2}}(t)$ be as in the proof of Lemma~\ref{lem:ellip}, with $g$ replaced by $\psi$. Since there are only finitely many boxes $\calC^{0} = \calC_{k_{0}}(0)$ in $\set{\abs{\xi} \simeq 2^{k_{0}}}$, \eqref{eq:ellip-orth} with $q_{0} = 1$, $q_{1} = q_{2} = \infty$ implies
\begin{equation*}
	\abs{I(t)} \aleq \nrm{h_{k_{0}}(t)}_{L^{1}} \bb( \sum_{\calC^{1}} \nrm{P_{\calC^{1}} f_{k_{1}}(t)}_{L^{\infty}}^{2} \bb)^{1/2} 
		\bb( \sum_{\calC^{2}} \nrm{P_{\calC^{2}} \psi_{k_{2}}(t)}_{L^{\infty}}^{2} \bb)^{1/2}.
\end{equation*}
Then integrating and applying H\"older in $t$ appropriately, the desired claim \eqref{eq:bi-d-rest-2-key} follows by duality.

\pfstep{Steps~2.2~\&~2.3:~(low-high) or (high-low) interaction, $k_{1} = k_{\min}$ or $k_{2} = k_{\min}$}
These cases are easier thanks to the factor $2^{-\frac{3}{2} k_{0}}$ on the LHS, as $k_{0} = k_{\max} + O(1)$ by Littlewood-Paley trichotomy. Indeed, by H\"older's inequality and the frequency envelope bounds \eqref{eq:fe-L2Linfty-0} and \eqref{eq:fe-L2Linfty-AB} we have
\begin{align*}
2^{-\frac{3}{2} k_{0}} \nrm{P_{k_{0}} \calL(f_{k_{1}}, \psi_{k_{2}})}_{L^{1} L^{\infty}} 
\aleq 2^{-\frac{3}{2} k_{0}} \nrm{f_{k_{1}}}_{L^{2} L^{\infty}} \nrm{\psi_{k_{2}}}_{L^{2} L^{\infty}}
\aleq 2^{-\frac{3}{2} k_{\max}} 2^{\frac{1}{2} k_{1}} 2^{k_{2}} a_{k_{1}} b_{k_{2}},
\end{align*}
which is acceptable.

\begin{remark} \label{rem:hi-d-3}
In a general dimension $d \geq 4$, essentially every proof in this section is valid with substitutions as in Remark~\ref{rem:hi-d-1}. The constant $\dlt_{0} > 0$ would change, since \eqref{eq:fe-L1Linfty} must be replaced by
\begin{equation*}
\nrm{Q_{j}^{s} \psi_{k}}_{L^{1} L^{\infty}} \aleq 2^{\frac{5-d}{2} k} 2^{(d+1)(k - j)_{+}} \nrm{\psi_{k}}_{\tilde{Z}^{\frac{d-3}{2}}_{s, k}}.
\end{equation*}

\end{remark}

\section{Proof of the trilinear estimates} \label{sec:tri}
In this section, we establish Proposition~\ref{prop:trilinear}.
In Section~\ref{subsec:tri-bi-ax} and Section~\ref{subsec:tri-bi-a0}, we first decompose the nonlinearity further and treat the part for which the bilinear null structure suffices. We will then be left with a part of the trilinear form
\begin{equation*}
- \lap^{-1} \brk{\Pi_{s_{1}} \varphi^{1}, \mR_{0} \Pi_{s_{2}} \varphi^{2}} \mR_{0} \psi + \Box^{-1} \calP_{i} \brk{\Pi_{s_{1}} \varphi^{1}, \mR_{x} \Pi_{s_{2}} \varphi^{2}} \mR^{i} \psi
\end{equation*}
with certain restriction on the modulation and frequencies of the inputs and the output; for the precise expression, see \eqref{eq:tri-form}. This nonlinearity exhibits the same multilinear null structure as in the case of MKG-CG \cite[Appendix]{KST}. We thus complete the proof of Proposition~\ref{prop:trilinear} in Section~\ref{subsec:tri-tri} by reducing the present case to the multilinear null form estimate in \cite{KST}.

As before, we restrict to the case $d = 4$ for most part of this section. The argument is simpler in the higher dimensional case $d \geq 5$; see Remark~\ref{rem:hi-d-4} below. 
\subsection{Preliminaries: Conventions and definitions}
Fix signs $s_{1}, s_{2}, s \in \set{+, -}$ and let $a, \tilde{a}, b, c, d$ be admissible frequency envelopes. In this section, we normalize the frequency envelope norms of the inputs as follows:
\begin{equation} \label{eq:fe-n-tri}	
\begin{aligned}
	& \nrm{A}_{S^{1}_{a}}
	= \nrm{A}_{Z^{1}_{\tilde{a}}}
	= \nrm{B}_{Y^{1}_{a}}
	= \nrm{B}_{(Z_{ell}^{1})_{\tilde{a}}}
	= 1, \\
	& \nrm{\psi}_{(\tilde{S}^{1/2}_{s})_{b}}
	= \nrm{\varphi^{1}}_{(\tilde{S}^{1/2}_{s_{1}})_{c}} 
	= \nrm{\varphi^{2}}_{(\tilde{S}^{1/2}_{s_{2}})_{d}}
	= 1.
\end{aligned}
\end{equation}
From \eqref{eq:fe-n-tri}, it follows that $A, B, \psi$ obey the frequency envelope bounds \eqref{eq:fe-L2}--\eqref{eq:fe-Y}. Note that also $\psi$ obeys the bound
\begin{equation} \label{eq:fe-L2-ang}
	\sup_{\ell \leq 0} \bb( \sum_{\omg} \nrm{P_{\ell}^{\omg} Q_{<k+2 \ell} \psi}_{L^{\infty} L^{2}}^{2} \bb)^{1/2} \aleq 2^{-\frac{1}{2} k} b_{k}.
\end{equation}
Moreover, $\varphi^{1}, \varphi^{2}$ obey the same estimates with $(s, b_{k})$ replaced by $(s_{1}, c_{k})$ and $(s_{2}, d_{k})$, respectively. The normalizations $\nrm{A}_{Z^{1}_{\tilde{a}}} = 1$ and $\nrm{B}_{(Z^{1}_{ell})_{\tilde{a}}} = 1$ imply
\begin{align}
\sup_{j < k + C} \bb( \sum_{\omg} \nrm{P^{\omg}_{\ell} Q_{j} A_{k}}_{L^{1} L^{\infty}}^{2} \bb)^{\frac{1}{2}} \leq & 2^{- \frac{1}{4} (j - k)} \tilde{a}_{k},  \label{eq:fe-Z} \\
\sup_{j < k + C} \bb( \sum_{\omg} \nrm{P^{\omg}_{\ell} Q_{j} B_{k}}_{L^{1} L^{\infty}}^{2} \bb)^{\frac{1}{2}} \leq & 2^{\frac{1}{4}(j - k)}\tilde{a}_{k}.  \label{eq:fe-Z-ell}
\end{align}

To identify the part that we cannot handle with only bilinear estimates, we borrow some definitions (with minor modification) from \cite{KST}. Given $k \in \bbZ$ and a translation-invariant bilinear operator $\BL$, define 
\begin{align}
\calH_{k} \BL(f, g)
= & \sum_{j < k + C_{2}}  P_{k} Q_{j} \BL(Q_{< j} f, Q_{< j} g), \label{eq:Hk-def} \\
\calH^{\ast}_{k} \BL(f, g)
= & \sum_{j < k + C^{\ast}_{2}} Q_{< j} \BL(P_{k} Q_{j} f, Q_{< j} g). \label{eq:Hsk-def}
\end{align}
Here $C_{2}, C^{\ast}_{2} > 0$ are universal constants such that
\begin{equation} \label{eq:C012}
	\frac{1}{2} C_{0} < C^{\ast}_{2} < C_{1} < C_{2} < C_{0},
\end{equation}
where $C_{0}$ is the constant in Lemma~\ref{lem:geom-cone} and $C_{1}$ is the constant in the definitions \eqref{eq:Z-def}--\eqref{eq:Z-ell-def} of $Z^{r}_{k}$ and $Z^{r}_{ell, k}$ (which is, in fact, chosen at this point to satisfy \eqref{eq:C012}). 

Given signs $s_{1}, s_{2}, s \in \set{+, -}$, we also define  
\begin{align*}
\calH_{s_{1}, s_{2}} \BL(f, g)
= & \sum_{k_{0}, k_{1}, k_{2} : \, k_{0} < k_{2} - C_{2} - 10} \calH_{k_{0}} \BL(T_{s_{1}} f_{k_{1}}, P_{k_{2}} T_{s_{2}} g), \\ 
\calH^{\ast}_{s', s} \BL(f, g)
= & \sum_{k_{0}, k_{1}, k_{2} : \, k_{1} < k_{2} - C^{\ast}_{2} - 10} P_{k_{0}} T_{s'} \calH^{\ast}_{k_{1}} \BL(f, T_{s} g_{k_{2}}).
\end{align*}

\subsection{Further decomposition of $\bfA_{x}$ and $\Diff^{R}$} \label{subsec:tri-bi-ax}
Consider the trilinear operator
\begin{equation}
	\calT^{R}_{s_{1}, s_{2}, s} (\varphi^{1}, \varphi^{2}, \psi) = s s_{2} \calH^{\ast}_{s, s} \big( \calH_{s_{1}, s_{2}} \Box^{-1} \calP_{i} \brk{\Pi_{s_{1}} \varphi^{1}, \mR_{x} \Pi_{s_{2}} \varphi^{2}} \mR^{i} \psi \big),
\end{equation}
where $\Box^{-1}$ denotes the Fourier multiplier\footnote{In general, this `multiplier' is problematic near $\set{ \tau^{2} - \abs{\xi}^{2} = 0}$; however, thanks to the modulation projection $Q_{j}$ in the definition of $\calH_{s_{1}, s_{2}}$, the expression $\calH_{s_{1}, s_{2}} \Box^{-1}$ is well-defined and coincides with $\calH_{s_{1}, s_{2}} K$, where $K f$ denotes the solution $\phi$ to $\Box \phi = f$ with $\phi[0] = 0$.} with symbol $(\tau^{2} - \abs{\xi}^{2})^{-1}$.
Our goal is to show that all of $\Diff^{R}[\bfA(\Pi_{s_{1}} \varphi^{1}, \Pi_{s_{2}} \varphi^{2})] \psi$ except $\calT^{R}_{s_{1}, s_{2}, s}$ can be handled by applying bilinear estimates in tandem. We use the auxiliary $Z^{1}$ norm as an intermediary. 

More precisely, under the normalization \eqref{eq:fe-n-tri} and $f$ as in \eqref{eq:trilinear-fe}, we claim that
\begin{equation} \label{eq:tri-bi-ax}
\nrm{- s \Diff^{R}[\bfA(\Pi_{s_{1}} \varphi^{1}, \Pi_{s_{2}} \varphi^{2})] \psi - \calT^{R}_{s_{1}, s_{2}, s}(\varphi^{1}, \varphi^{2}, \psi)}_{(N_{s}^{1/2})_{f}} 
	\aleq 1.
\end{equation}
\subsubsection*{Step~0:~Reduction to bilinear estimates}
Let $a, b, c, d$ be admissible frequency envelopes. Define $e_{k} = (\sum_{k' < k} a_{k'}) b_{k}$ and $\tilde{e}_{k} = (\sum_{k' < k} \tilde{a}_{k'}) b_{k}$.
We claim that
\begin{align}  
\nrm{(\Id - \calH^{\ast}_{s, s}) \Diff^{R}[A] \psi}_{(N^{1/2}_{s})_{e}}
	\aleq & \nrm{A}_{S^{1}_{a}} \nrm{\psi}_{(\tilde{S}^{1/2}_{s})_{b}} ,  \label{eq:diffr-Hs}\\
\nrm{\calH^{\ast}_{s, s} \Diff^{R}[A] \psi}_{(N^{1/2}_{s})_{\tilde{e}}}
	\aleq & \nrm{A}_{Z^{1}_{\tilde{a}}} \nrm{\psi}_{(\tilde{S}^{1/2}_{s})_{b}},  \label{eq:diffr-Z} \\
\nrm{(I - \calH_{s_{1}, s_{2}}) \bfA^{R}(\varphi^{1}, \varphi^{2})}_{Z^{1}_{cd}} 
	\aleq & \nrm{\varphi^{1}}_{(\tilde{S}^{1/2}_{s_{1}})_{c}} \nrm{\varphi^{2}}_{(\tilde{S}^{1/2}_{s_{2}})_{d}}, \label{eq:axr-Z}\\
\nrm{\bfA_{s_{2}}^{S}(\Pi_{s_{1}} \varphi^{1}, \varphi^{2})}_{Z^{1}_{cd}} 
	\aleq & \nrm{\varphi^{1}}_{(\tilde{S}^{1/2}_{s_{1}})_{c}} \nrm{\varphi^{2}}_{(\tilde{S}^{1/2}_{s_{2}})_{d}}.	\label{eq:axs-Z} 
\end{align}
Assuming these estimates, we first conclude the proof of \eqref{eq:tri-bi-ax}. Assume the normalization \eqref{eq:fe-n-tri}. Note that $P_{k} \Pi_{s}$ is disposable for any $k \in \bbZ$ and $s \in \set{+, -}$. Hence, from the bilinear estimates \eqref{eq:axr}--\eqref{eq:axs} and \eqref{eq:axr-Z}--\eqref{eq:axs-Z}, we obtain
\begin{gather*}
	\nrm{\bfA^{R}(\Pi_{s_{1}} \varphi^{1}, \Pi_{s_{2}} \varphi^{2})}_{S^{1}_{cd}}
	+ \nrm{(1 - \calH_{s_{1}, s_{2}})\bfA^{R}(\Pi_{s_{1}} \varphi^{1}, \Pi_{s_{2}} \varphi^{2})}_{Z^{1}_{cd}} \aleq  1, \\
	\nrm{\bfA_{s_{2}}^{S}(\Pi_{s_{1}} \varphi^{1}, \Pi_{s_{2}} \varphi^{2})}_{(S^{1} \cap Z^{1})_{cd}}
	\aleq 1.
\end{gather*}
Applying \eqref{eq:diffr-Hs} and \eqref{eq:diffr-Z} with $a = \tilde{a} = cd$, $e = \tilde{e} = (\sum_{k' < k} c_{k'} d_{k'}) b_{k}$ and
\begin{equation*}
A = \bfA(\Pi_{s_{1}} \varphi^{1}, \Pi_{s_{2}} \varphi^{2}) = - s_{2} \bfA^{R}(\Pi_{s_{1}} \varphi^{1}, \Pi_{s_{2}} \varphi^{2}) + \bfA^{S}_{s_{2}}(\Pi_{s_{1}} \varphi^{1}, \Pi_{s_{2}} \varphi^{2}),
\end{equation*}
we arrive at
\begin{equation*}
	\nrm{s \Diff^{R}[\bfA(\Pi_{s_{1}} \varphi^{1}, \Pi_{s_{2}} \varphi^{2})] \psi - s \calH^{\ast}_{s, s} \pi^{R}[\calH_{s_{1}, s_{2}} (- s_{2} \bfA^{R})(\Pi_{s_{1}} \varphi^{1}, \Pi_{s_{2}} \varphi^{2})] \psi}_{(N_{s}^{1/2})_{e}} \aleq 1.
\end{equation*}
Recalling the definitions of $\bfA^{R}$ and $\Diff^{R}$ from Section~\ref{subsec:nonlin}, observe that
\begin{equation*}
s_{2} s \calH^{\ast}_{s, s} \pi^{R}[\calH_{s_{1}, s_{2}} \bfA^{R}(\Pi_{s_{1}} \varphi^{1}, \Pi_{s_{2}} \varphi^{2})] \psi
= \calT^{R}_{s_{1}, s_{2}, s}(\varphi^{1}, \varphi^{2}, \psi).
\end{equation*}
Moreover, by Cauchy-Schwarz, the frequency envelope $e$ is dominated by $f$ as in \eqref{eq:trilinear-fe}. The desired estimate \eqref{eq:tri-bi-ax} follows.

\subsubsection*{Step~1:~Proof of \eqref{eq:diffr-Hs}}
Under the normalization \eqref{eq:fe-n-tri} and the condition $k_{1} < k_{2} - C^{\ast}_{2} - 5$, we claim that:
\begin{equation} \label{eq:diffr-Hs-dyadic}
\nrm{P_{k_{0}} \NF(A_{k_{1}}, \psi_{k_{2}}) - P_{k_{0}} T_{s} \calH^{\ast}_{k_{1}} \NF(A, T_{s}\psi_{k_{2}})}_{N^{1/2}_{s}}
	\aleq a_{k_{1}} b_{k_{2}}.
\end{equation}
Since $\Diff^{R}[A] \psi = \sum_{k} \NF(P_{< k - 10} A, \psi_{k})$ by Proposition~\ref{prop:MD-CG-nf}, \eqref{eq:diffr-Hs} clearly follows from summing up \eqref{eq:diffr-Hs-dyadic} for $k_{1} < k_{2} - C^{\ast}_{2} -10$ and \eqref{eq:bi-bal-freq} in Remark~\ref{rem:bi-bal-freq} for $k_{1} \in [k_{2} - C^{\ast}_{2} - 10, k_{2} - 10)$.

To prove \eqref{eq:diffr-Hs-dyadic}, we first split $A = \sum_{j \geq k_{1} + C^{\ast}_{2}} Q_{j} A + Q_{< k_{1} + C^{\ast}_{2}} A$ in the expression $P_{k_{0}} \NF(A_{k_{1}}, \psi_{k_{2}})$.

\pfstep{Step~1.1:~Contribution of $Q_{\geq k_{1} + C^{\ast}_{2}} A$}
Fix $j \geq k_{1} + C^{\ast}_{2}$. By H\"older's inequality, \eqref{eq:q-disp} (for $Q_{j} A_{k_{1}}$) and the frequency envelope bounds \eqref{eq:fe-L2} and \eqref{eq:fe-L2L2}, we have
\begin{align*}
\nrm{P_{k_{0}} Q^{s}_{\geq j - 5} \NF(Q_{j} A_{k_{1}}, \psi_{k_{2}})}_{N_{s}^{1/2}} 
\aleq & 2^{-\frac{1}{2} j} \nrm{A_{k_{1}}}_{L^{2} L^{\infty}} b_{k_{2}}, \\ 
\nrm{P_{k_{0}} Q^{s}_{< j - 5} \NF(Q_{j} A_{k_{1}}, Q_{\geq j - 5} \psi_{k_{2}})}_{N_{s}^{1/2}} 
\aleq & 2^{- \frac{1}{2} j} \nrm{A_{k_{1}}}_{L^{2} L^{\infty}} b_{k_{2}}.
\end{align*}
Using \eqref{eq:fe-L2Linfty-AB} and summing up in $j \geq k_{1} + C^{\ast}_{2}$, it follows that $P_{k_{0}} \NF(Q_{\geq k_{1} + C^{\ast}_{2}} A_{k_{1}}, \psi_{k_{2}})$ is acceptable except:
\begin{equation*}
\sum_{j \geq k_{1} + C^{\ast}_{2}} P_{k_{0}} Q^{s}_{< j - 5} \NF(Q_{j} A_{k_{1}}, Q_{< j - 5}^{s} \psi_{k_{2}}).
\end{equation*}
Splitting $Q_{j} = \sum_{s_{1}} Q_{j}^{s_{1}} T_{s_{1}}$ and applying Lemma~\ref{lem:geom-cone} (where we remind the reader that $C^{\ast}_{2} > \frac{1}{2} C_{0}$), we see that the summand vanishes unless $j = k_{\max} + O(1)$. Proceeding as in \eqref{eq:ellip-1-AB}, we obtain
\begin{equation*}
\sum_{j = k_{\max} + O(1)} \nrm{P_{k_{0}} Q^{s}_{< j - 5} \NF(Q_{j} A_{k_{1}}, Q_{< j - 5} \psi_{k_{2}})}_{N_{s}^{1/2}} 
\aleq 2^{k_{\min}} 2^{\frac{1}{2} k_{0}} \nrm{Q_{j} A_{k_{1}}}_{L^{2} L^{2}} b_{k_{2}}
\end{equation*}
By \eqref{eq:fe-L2L2-AB}, the RHS is bounded by $\aleq 2^{\frac{1}{2} (k_{\min} - k_{\max})} a_{k_{1}} b_{k_{2}}$, which is better than needed.

\pfstep{Step~1.2:~Contribution of $Q_{< k_{1} + C^{\ast}_{2}} A$} 
By Lemma~\ref{lem:diff-hi-mod}, as well as \eqref{eq:q-disp} and \eqref{eq:fe-L2Linfty-AB} for $Q_{< k_{1} + C^{\ast}_{2}} A_{k_{1}}$, we have
\begin{align*}
\nrm{P_{k_{0}} Q^{s}_{\geq k_{1} + C^{\ast}_{2}} \NF(Q_{< k_{1} + C^{\ast}_{2} } A_{k_{1}}, \psi_{k_{2}})}_{N_{s}^{1/2}} 
\aleq & a_{k_{1}} b_{k_{2}} , \\ 
\nrm{P_{k_{0}} Q^{s}_{< k_{1} + C^{\ast}_{2}} \NF(Q_{< k_{1} + C^{\ast}_{2}} A_{k_{1}}, Q_{\geq k_{1} + C^{\ast}_{2}}^{s} \psi_{k_{2}})}_{N_{s}^{1/2}} 
\aleq & a_{k_{1}} b_{k_{2}} .
\end{align*}
It remains to bound the expression
\begin{equation*}
P_{k_{0}} Q^{s}_{< k_{1} + C^{\ast}_{2}} \NF(Q_{< k_{1} + C^{\ast}_{2}} A_{k_{1}}, Q_{< k_{1} + C^{\ast}_{2}}^{s} \psi_{k_{2}})
- P_{k_{0}} T_{s} \calH^{\ast}_{k_{1}} \NF(A, T_{s}\psi_{k_{2}})
= I_{0} + I_{2} + R
\end{equation*}
where $I_{0}$, $I_{2}$ are exactly as in \eqref{eq:nd-I0}, \eqref{eq:nd-I2}, and
\begin{align*}
R = & \sum_{j \in [k_{1} - 10, k_{1} + C^{\ast}_{2})} 
		\bb( P_{k_{0}} Q^{s}_{j} \NF(Q_{\leq j} A_{k_{1}}, Q_{\leq j}^{s} \psi_{k_{2}}) 
			+ P_{k_{0}} Q^{s}_{< j} \NF(Q_{\leq j} A_{k_{1}}, Q_{j}^{s} \psi_{k_{2}}) \bb).
\end{align*}
By Lemma~\ref{lem:diff-hi-mod}, along with \eqref{eq:q-disp} and \eqref{eq:fe-L2Linfty-AB} for $Q_{\leq j} A_{k_{1}}$, we have $\nrm{R}_{N_{s}^{1/2}} \aleq a_{k_{1}} b_{k_{2}}$, which is acceptable.
To complete the proof of \eqref{eq:diffr-Hs-dyadic}, it remains to show that
\begin{equation} \label{eq:diffr-Hs-low-mod}
	\nrm{I_{0}}_{N_{s}^{1/2}}
	+ \nrm{I_{2}}_{N_{s}^{1/2}} \aleq a_{k_{1}} b_{k_{2}}.
\end{equation}
Applying Proposition~\ref{prop:nf} with $(f, g) = (\psi, A)$, as well as the frequency envelope bounds \eqref{eq:fe-L2}, \eqref{eq:fe-L2L2} and \eqref{eq:fe-L2Linfty}, the desired estimate \eqref{eq:diffr-Hs-low-mod} follows.

\subsubsection*{Step~2:~Proof of \eqref{eq:diffr-Z}}
Assuming \eqref{eq:fe-n-tri} and $k_{1} < k_{2} - C^{\ast}_{2} - 5$, we claim:
\begin{equation} \label{eq:diffr-Z-dyadic}
\nrm{P_{k_{0}} T_{s} \calH^{\ast}_{k_{1}} \NF(A, T_{s} \psi_{k_{2}})}_{L^{1} \dot{H}^{1/2}}
	\aleq  \tilde{a}_{k_{1}} b_{k_{2}} . 
\end{equation}
As before, \eqref{eq:diffr-Z} clearly follows from \eqref{eq:diffr-Z-dyadic}.

To prove \eqref{eq:diffr-Z-dyadic}, let $\ell = \frac{1}{2} (j - k_{1})_{-}$ and expand
\begin{align*}
P_{k_{0}} T_{s} \calH^{\ast}_{k_{1}} \NF(A, T_{s} \psi_{k_{2}})
= \sum_{j < k_{1} + C^{\ast}_{2}} \sum_{\omg_{1}, \omg_{2}} P_{k_{0}} Q_{<j}^{s} \NF(P_{\ell}^{\omg_{1}} Q_{j} A_{k_{1}}, P_{\ell}^{\omg_{2}} Q_{<j}^{s} \psi_{k_{2}}).
\end{align*}
Splitting $Q_{j} = Q_{j}^{+} T_{+} + Q_{j}^{-} T_{-}$ and applying Lemma~\ref{lem:geom-cone}, we see that the summand on the RHS vanishes unless
\begin{equation*}
	\abs{\angle(s_{1} \omg_{1}, s \omg_{2})} \aleq 2^{\ell} \qquad \hbox{ for } s_{1} = + \hbox{ or } -.
\end{equation*}
By \eqref{eq:q-LqL2} and Proposition~\ref{prop:nf-basic}, we have
\begin{equation} \label{eq:diffr-Z-atom}
\nrm{P_{k_{0}} Q_{<j}^{s} \NF(P_{\ell}^{\omg_{1}} Q_{j} A_{k_{1}}, P_{\ell}^{\omg_{2}} Q_{<j}^{s} \psi_{k_{2}})}_{L^{1} L^{2}}
\aleq 2^{\ell} \nrm{P_{\ell}^{\omg_{1}} Q_{j} A_{k_{1}}}_{L^{1} L^{\infty}} \nrm{P_{\ell}^{\omg_{2}} Q_{<j}^{s} \psi_{k_{2}}}_{L^{\infty} L^{2}}
\end{equation}
Note that for a fixed $\omg_{1}$, there are only (uniformly) bounded number of $\omg_{2}$ such that the expression is nonvanishing, and vice versa. We may therefore apply Lemma~\ref{lem:CS} and the fact that $k_{0} = k_{2} + O(1)$, to estimate
\begin{align*}
& \hskip-2em
\nrm{P_{k_{0}} T_{s} \calH^{\ast}_{k_{1}} \NF(A, T_{s} \psi_{k_{2}})}_{L^{1} \dot{H}^{1/2}} \\
\aleq & \sum_{j < k_{1} + C_{2}^{\ast}} 2^{\ell} 2^{\frac{1}{2} k_{2}} \bb( \sum_{\omg_{1}} \nrm{P^{\omg_{1}}_{\ell} Q_{j} A_{k_{1}}}_{L^{1} L^{\infty}}^{2} \bb)^{1/2} \bb( \sum_{\omg_{2}} \nrm{P^{\omg_{2}}_{\ell} Q_{<j}^{s} \psi_{k_{2}}}_{L^{\infty} L^{2}}^{2} \bb)^{1/2}.
\end{align*}
By the frequency envelope bounds \eqref{eq:fe-L2-ang} and \eqref{eq:fe-Z}, the summand on the RHS is bounded by $2^{\frac{1}{2} \ell} \tilde{a}_{k_{1}} b_{k_{2}}$. Summing up in $j < k_{1} + C_{2}^{\ast}$, \eqref{eq:diffr-Z-dyadic} follows.

\subsubsection*{Step~3:~Proof of \eqref{eq:axr-Z}}
For $k_{0} \geq k_{2} - C_{2} - 20$, we claim that
\begin{equation} \label{eq:axr-Z-1}
\nrm{P_{k_{0}} \NF^{\ast}(\varphi^{1}_{k_{1}}, \varphi^{2}_{k_{2}})}_{\Box Z^{1}} 
	\aleq 2^{\dlt_{0} (k_{\max} - k_{\min})} c_{k_{1}} d_{k_{2}}.
\end{equation}
Moreover, for $k_{0} < k_{2} - C_{2} - 5$, we claim that
\begin{equation} \label{eq:axr-Z-2}
\nrm{P_{k_{0}} \NF^{\ast}(\varphi^{1}_{k_{1}}, \varphi^{2}_{k_{2}})
- \calH_{k_{0}} \NF^{\ast}(T_{s_{1}} \varphi^{1}_{k_{1}}, T_{s_{2}} \varphi^{2}_{k_{2}})}_{\Box Z^{1}} 
	\aleq 2^{\dlt_{0} (k_{0} - k_{1})} c_{k_{1}} d_{k_{2}}.
\end{equation}
Since $\Box \bfA^{R} = \NM^{R} =  \NF^{\ast}$ by Proposition~\ref{prop:MD-CG-nf}, \eqref{eq:axr-Z} clearly follows from \eqref{eq:axr-Z-1} and \eqref{eq:axr-Z-2}.

We will simultaneously prove \eqref{eq:axr-Z-1} and \eqref{eq:axr-Z-2}. As in Step~1 above, we start by splitting the output modulation to $\sum_{j \geq k_{\min} + C_{2}} Q_{j} \NF^{\ast}$ and $Q_{<k_{\min} + C_{2}} \NF^{\ast}$.
\pfstep{Step~3.1:~Contribution of $Q_{\geq k_{\min} + C_{2}} \NF^{\ast}$} 
Fix $j \geq k_{\min} + C_{2}$. Applying \eqref{eq:q-LqL2} and the $L^{1} L^{2}$ estimates in Lemma~\ref{lem:ax-hi-mod}, we have
\begin{align*}
	\nrm{P_{k_{0}} Q_{j} \NF^{\ast}(Q_{\geq j - 5}^{s_{1}} \varphi^{1}_{k_{1}}, \varphi^{2}_{k_{2}})}_{L^{1} L^{2}} \aleq & 2^{\frac{1}{2}(k_{\min} - j)} 2^{\frac{1}{2}(k_{\min} - k_{\max})} c_{k_{1}} d_{k_{2}}, \\
	\nrm{P_{k_{0}} Q_{j} \NF^{\ast}(Q_{< j - 5}^{s_{1}} \varphi^{1}_{k_{1}}, Q_{\geq j - 5}^{s_{2}} \varphi^{2}_{k_{2}})}_{L^{1} L^{2}}
	\aleq & 2^{\frac{1}{2}(k_{\min} - j)} 2^{\frac{1}{2}(k_{\min} - k_{\max})} c_{k_{1}} d_{k_{2}}.
\end{align*}
Using \eqref{eq:Z-L1L2} and summing up in $j \geq k_{\min} + C_{2}$, it remains to bound the $\Box Z^{1}$ norm of
\begin{equation*}
	\sum_{j \geq k_{\min} + C_{2}} P_{k_{0}} Q_{j} \NF^{\ast}(Q_{<j-5}^{s_{1}} \varphi^{1}_{k_{1}}, Q_{<j-5}^{s_{2}} \varphi^{2}_{k_{2}}).
\end{equation*}
Since $C_{2} > C_{1}$, the $\Box Z^{1}$ norm of this expression vanishes unless $k_{0} \neq k_{\min}$ (hence $k_{0}= k_{\max} + O(1)$). Moreover, since $C_{2} > \frac{1}{2} C_{0}$, the summand vanishes unless $j = k_{\max} + O(1)$; to see this, split $Q_{j} = Q_{j}^{+} T_{+} + Q_{j}^{-} T_{-}$ and apply Lemma~\ref{lem:geom-cone}. For the nonvanishing terms, we apply \eqref{eq:q-disp} (which is applicable since $j = k_{\max} + O(1)$), H\"older's inequality $L^{2} L^{\infty} \times L^{2} L^{\infty} \to L^{1} L^{\infty}$ and \eqref{eq:fe-L2Linfty-0} to estimate
\begin{equation*}
	\nrm{P_{k_{0}} Q_{j} \NF^{\ast}(Q_{<j-5}^{s_{1}} \varphi^{s_{1}}_{k_{1}}, Q_{<j-5}^{s_{2}} \varphi^{s_{2}}_{k_{2}})}_{L^{1} L^{\infty}}
	\aleq 2^{k_{\min} + k_{\max}} c_{k_{1}} d_{k_{2}}.
\end{equation*}
By \eqref{eq:Z-L1Linfty}, we conclude that
\begin{equation*}
	\nrm{P_{k_{0}} Q_{j} \NF^{\ast}(Q_{<j-5}^{s_{1}} \varphi^{1}_{k_{1}}, Q_{<j-5}^{s_{2}} \varphi^{2}_{k_{2}})}_{\Box Z^{1}}
	\aleq 2^{k_{\min} - k_{\max}} c_{k_{1}} d_{k_{2}},
\end{equation*}
which is acceptable for both \eqref{eq:axr-Z-1} and \eqref{eq:axr-Z-2}.

\pfstep{Step~3.2:~Contribution of $Q_{< k_{\min} + C_{2}} \NF^{\ast}$, dominant input modulation}
We claim that
\begin{equation} \label{eq:axr-Z-I12}
	\nrm{P_{k_{0}} Q_{<k_{\min} + C_{2}} \NF^{\ast}(\varphi^{1}_{k_{1}}, \varphi^{2}_{k_{2}}) - \tilde{I}_{0}}_{\Box Z^{1}} \aleq 2^{\dlt_{0} (k_{\min} - k_{\max})} c_{k_{1}} d_{k_{2}},
\end{equation}
where
\begin{equation*}
	\tilde{I}_{0} = \sum_{j < k_{\min} + C_{2}} P_{k_{0}} Q_{j} \NF^{\ast}(Q_{<j}^{s_{1}} \varphi^{1}_{k_{1}}, Q_{<j}^{s_{2}} \varphi^{2}_{k_{2}}).
\end{equation*}

By \eqref{eq:q-LqL2} and Lemma~\ref{lem:ax-hi-mod}, we have
\begin{align*}
	\nrm{P_{k_{0}} Q_{< k_{\min} + C_{2}} \NF^{\ast}(Q_{\geq k_{\min} + C_{2}}^{s_{1}} \varphi^{1}_{k_{1}}, \varphi^{2})}_{L^{1} L^{2}} \aleq 2^{\frac{1}{2} (k_{\min} - k_{\max})} c_{k_{1}} d_{k_{2}}, \\
	\nrm{P_{k_{0}} Q_{< k_{\min} + C_{2}} \NF^{\ast}(Q_{< k_{\min} + C_{2}}^{s_{1}} \varphi^{1}_{k_{1}}, Q_{\geq k_{\min} + C_{2}}^{s_{2}} \varphi^{2})}_{L^{1} L^{2}} \aleq 2^{\frac{1}{2} (k_{\min} - k_{\max})} c_{k_{1}} d_{k_{2}}.
\end{align*}
It remains to consider the expression
\begin{equation*}
	P_{k_{0}} Q_{< k_{\min} + C_{2}} \NF^{\ast}(Q_{< k_{\min} + C_{2}}^{s_{1}} \varphi^{1}_{k_{1}}, Q_{< k_{\min} + C_{2}}^{s_{2}} \varphi^{2}_{k_{2}}) - \tilde{I}_{0}
	= I_{1} + I_{2} + R_{1} + R_{2},
\end{equation*}
where
\begin{align*}
	R_{1} = & \sum_{j \in [k_{\min} - 10, k_{\min} + C_{2})}  P_{k_{0}} Q_{\leq j} \NF^{\ast}(Q_{j}^{s_{1}} \varphi^{1}_{k_{1}}, Q_{<j}^{s_{2}} \varphi^{2}_{k_{2}}) , \\
	R_{2} = & \sum_{j \in [k_{\min} - 10, k_{\min} + C_{2})} P_{k_{0}} Q_{\leq j} \NF^{\ast}(Q_{\leq j}^{s_{1}} \varphi^{1}_{k_{1}}, Q_{j}^{s_{2}} \varphi^{2}_{k_{2}}) ,
\end{align*}
and $I_{1}$, $I_{2}$ are as in \eqref{eq:nm-I1}, \eqref{eq:nm-I2}, respectively\footnote{Of course, with $(\psi, s) \to (\varphi^{1}, s_{1})$ and $(\varphi, s') \to (\varphi^{2}, s_{2})$.}. 
By \eqref{eq:q-LqL2} and Lemma~\ref{lem:ax-hi-mod}, we have 
\begin{equation*}
\nrm{R_{1}}_{L^{1} L^{2}} + \nrm{R_{2}}_{L^{1} L^{2}} \aleq 2^{\frac{1}{2} (k_{\min} - k_{\max})} c_{k_{1}} d_{k_{2}},
\end{equation*}
which is acceptable for both \eqref{eq:axr-Z-1} and \eqref{eq:axr-Z-2}. Moreover, by \eqref{eq:Z-L1L2} and the argument in Step~2.1--2.2 in Section~\ref{subsec:bi-a} (which makes use of Proposition~\ref{prop:nfs} and the $\tilde{Z}^{1/2}$ norm), we have
\begin{equation} \label{eq:axr-Z-low-mod}
	\nrm{I_{1}}_{\Box Z^{1}} + \nrm{I_{2}}_{\Box Z^{1}} 
	\aleq \nrm{I_{1}}_{L^{1} L^{2}} + \nrm{I_{2}}_{L^{1} L^{2}} \aleq 2^{\dlt_{0} (k_{\min} - k_{\max})} c_{k_{1}} d_{k_{2}},
\end{equation}
which proves \eqref{eq:axr-Z-I12}.

\pfstep{Step~3.3:~Contribution of $Q_{< k_{\min} + C_{2}} \NF^{\ast}$, dominant output modulation}
It remains to handle $\tilde{I}_{0}$. When $k_{0} < k_{2} - C_{2} - 5$, we have
\begin{equation*}
\tilde{I}_{0} = \calH_{k_{0}} \NF^{\ast}(T_{s_{1}} \varphi^{1}_{k_{1}}, T_{s_{2}} \varphi^{2}_{k_{2}}). 
\end{equation*}
Hence \eqref{eq:axr-Z-2} follows from the estimates we have so far. 

To prove \eqref{eq:axr-Z-1}, we need to estimate $\nrm{\tilde{I}_{0}}_{\Box Z^{1}}$. For $k_{0} \geq k_{2} - C_{2} - 20$, we claim:
\begin{equation} \label{eq:diffr-Z-I0}
	\sum_{j < k_{\min} + C_{2}} \nrm{P_{k_{0}} Q_{j} \NF^{\ast}(Q_{<j}^{s_{1}} \varphi^{1}_{k_{1}}, Q_{<j}^{s_{2}} \varphi^{2}_{k_{2}})}_{\Box Z^{1}} 
	\aleq  2^{\frac{1}{2} (k_{\min} - k_{\max})} c_{k_{1}} d_{k_{2}}. 
\end{equation}
Clearly, \eqref{eq:axr-Z-1} would follow from \eqref{eq:diffr-Z-I0}.

Note that $k_{0} \geq k_{2} - C_{2} - 20$ implies $k_{0} = k_{\max} + O(1)$. For concreteness, assume that $k_{2} \leq k_{1}$, so that $k_{1} = k_{\max} + O(1)$ and $k_{2} = k_{\min} + O(1)$; the opposite case can be handled similarly. Define $\ell = \frac{1}{2} (j - k_{\min})_{-}$ and $\ell_{0} = \frac{1}{2}(j-k_{0})_{-}$. We decompose
\begin{align*}
P_{k_{0}} Q_{j} \NF^{\ast}(Q_{<j}^{s_{1}} \varphi^{1}_{k_{1}}, Q_{<j}^{s_{2}} \varphi^{2}_{k_{2}})  
=  \sum_{\omg_{0}, \omg_{1}, \omg_{2}} P_{k_{0}} P_{\ell_{0}}^{-\omg_{0}} Q_{j} \NF^{\ast}(P_{\ell_{0}}^{\omg_{1}} Q_{<j}^{s_{1}} \varphi^{1}_{k_{1}}, P_{\ell}^{\omg_{2}} Q_{<j}^{s_{2}} \varphi^{2}_{k_{2}}).
\end{align*}
Splitting $Q_{j} = Q_{j}^{+} T_{+} + Q_{j}^{-} T_{-}$ and applying Lemma~\ref{lem:geom-cone}, we see that the summand on the RHS vanishes unless
\begin{equation*}
\begin{aligned}
	\abs{\angle(s_{0} \omg_{0}, s_{1} \omg_{1})} \aleq & 2^{\ell} 2^{k_{\min} - \min\set{k_{0}, k_{1}}} + 2^{\ell_{0}} \aleq 2^{\ell_{0}} \\
	\abs{\angle(s_{0} \omg_{0}, s_{2} \omg_{2})} \aleq & 2^{\ell} 2^{k_{\min} - \min \set{k_{0}, k_{2}}} + \max\set{2^{\ell_{0}}, 2^{\ell}} \aleq 2^{\ell}
\end{aligned}
\quad \hbox{ for } s_{0} = + \hbox{ or } -.
\end{equation*}
In this case, Proposition~\ref{prop:nf-basic} implies
\begin{equation}  \label{eq:axr-Z-atom}
\begin{aligned}
& \hskip -4em
\nrm{P_{k_{0}} P_{\ell_{0}}^{-\omg_{0}} Q_{j} \NF^{\ast}(P_{\ell_{0}}^{\omg_{1}} Q_{<j}^{s_{1}} \varphi^{1}_{k_{1}}, P_{\ell}^{\omg_{2}} Q_{<j}^{s_{2}} \varphi^{2}_{k_{2}})}_{L^{1} L^{\infty}}  \\
\aleq & 2^{\ell} \nrm{P_{\ell_{0}}^{\omg_{1}} Q_{<j}^{s_{1}} \varphi^{1}_{k_{1}}}_{L^{2} L^{\infty}} \nrm{P_{\ell}^{\omg_{2}} Q_{<j}^{s_{2}} \varphi^{2}_{k_{2}}}_{L^{2} L^{\infty}}
\end{aligned}
\end{equation}
For a fixed $\omg_{0}$ [resp. $\omg_{1}$], there are only (uniformly) bounded number of $\omg_{1}, \omg_{2}$ [resp. $\omg_{0}, \omg_{2}$] such that the expression is nonvanishing. Summing up first in $\omg_{2}$ (for which there are only finitely many terms) and then applying Lemma~\ref{lem:CS} to the summation in $\omg_{0}, \omg_{1}$ (which is essentially diagonal), we obtain
\begin{align*}
& \hskip-4em
\bb( \sum_{\omg_{0}} \nrm{P_{k_{0}} P_{\ell_{0}}^{-\omg_{0}} Q_{j} \NF^{\ast}(Q_{<j}^{s_{1}} \varphi^{1}_{k_{1}}, Q_{<j}^{s_{2}} \varphi^{2}_{k_{2}})}_{L^{1} L^{\infty}}^{2} \bb)^{1/2} \\
\aleq & 2^{\ell} \bb( \sum_{\omg_{1}} \nrm{P_{\ell_{1}}^{\omg_{1}} Q_{<j}^{s_{1}} \varphi^{1}_{k_{1}}}_{L^{2} L^{\infty}}^{2} \bb)^{1/2} \sup_{\omg_{2}} \nrm{P_{\ell_{2}}^{\omg_{2}} Q_{<j}^{s_{i}} \varphi^{2}_{k_{2}}}_{L^{2} L^{\infty}} .
\end{align*}
By \eqref{eq:Z-L1Linfty} and the frequency envelope bound \eqref{eq:fe-L2Linfty}, it follows that
\begin{align*}
\nrm{P_{k_{0}} Q_{j} \NF^{\ast}(Q_{<j}^{s_{1}} \varphi^{1}_{k_{1}}, Q_{<j}^{s_{2}} \varphi^{s}_{k_{2}})}_{\Box Z^{1}}
	\aleq & 2^{\frac{1}{4}(j - k_{\min})} 2^{\frac{1}{2}(k_{\min} - k_{\max})} c_{k_{1}} d_{k_{2}}.
\end{align*}
Summing up in $j < k_{\min} + C_{2}$, we have proved \eqref{eq:diffr-Z-I0}.

\subsubsection*{Step~4:~Proof of \eqref{eq:axs-Z}}
In this case, recall that $\Box \bfA^{S}_{s_{2}}(\Pi_{s_{1}} \cdot, \cdot) = \NM^{S}_{s_{2}}(\Pi_{s_{1}} \cdot, \cdot) = \NF_{s_{1} s_{2}} (\cdot, \cdot)$ by Proposition~\ref{prop:MD-CG-nf}. Repeating the argument in Step~3, the following analogues of \eqref{eq:axr-Z-1} and \eqref{eq:axr-Z-2} can be proved:
For $k_{0} \geq k_{2} - C_{2} - 20$, we have
\begin{equation} \label{eq:axs-Z-1}
\nrm{P_{k_{0}} \NF_{s_{1}s_{2}}(\varphi^{1}_{k_{1}}, \varphi^{2}_{k_{2}})}_{\Box Z^{1}} 
	\aleq 2^{\dlt_{0} (k_{\min} - k_{\max})} c_{k_{1}} d_{k_{2}}, 
\end{equation}
and for $k_{0} < k_{2} - C_{2} - 5$, we have
\begin{equation} \label{eq:axs-Z-2}
\nrm{P_{k_{0}} \NF_{s_{1}s_{2}}(\varphi^{1}_{k_{1}}, \varphi^{2}_{k_{2}})
- \calH_{k_{0}} \NF_{s_{1}s_{2}}(T_{s_{1}} \varphi^{1}_{k_{1}}, T_{s_{2}} \varphi^{2}_{k_{2}})}_{\Box Z^{1}} 
	\aleq 2^{\dlt_{0} (k_{0} - k_{2})} c_{k_{1}} d_{k_{2}}.
\end{equation}
We omit the straightforward details. 

Under the condition $k_{0} < k_{2} - C_{2} - 5$, we claim furthermore that
\begin{equation} \label{eq:axs-Z-3}
\nrm{\calH_{k_{0}} \NF_{s_{1}s_{2}}(\varphi^{1}_{k_{1}}, \varphi^{2}_{k_{2}})}_{\Box Z^{1}} 
	\aleq 2^{k_{0} - k_{2}} c_{k_{1}} d_{k_{2}}.
\end{equation}
Clearly, \eqref{eq:axs-Z} would follow from \eqref{eq:axs-Z-1}--\eqref{eq:axs-Z-3}.

To prove \eqref{eq:axs-Z-3}, we need to estimate
\begin{equation*}
I = P_{k_{0}} Q_{j} \NF_{s_{1} s_{2}}(Q_{<j}^{s_{1}} \varphi^{1}_{k_{1}}, Q_{<j}^{s_{2}} \varphi^{2}_{k_{2}})
\end{equation*}
in $\Box Z^{1}$. We proceed similarly to the proof of Proposition~\ref{prop:nf} and perform an orthogonality argument using Lemma~\ref{lem:box-orth}. 

Let $j \leq k_{0} +C_{2}$ and $\ell = \frac{1}{2} (j - k_{0})_{-}$. For $i= 0, 1,2$, let $\calC^{i}$ be a rectangular box of the form $\calC_{k_{0}}(\ell)$. We split
\begin{equation*}
I = \sum_{\calC^{0}, \calC^{1}, \calC^{2}} P_{k_{0}} P_{- \calC^{0}} Q_{j} \NF_{s_{1} s_{2}}(P_{\calC^{1}} Q_{<j}^{s_{1}} \varphi^{1}_{k_{1}}, P_{\calC^{2}} Q_{<j}^{s_{2}} \varphi^{2}_{k_{2}}).
\end{equation*}
Splitting $Q_{j} = Q_{j}^{+} T_{+} + Q_{j}^{-} T_{-}$ and applying Lemma~\ref{lem:box-orth}, we see that the summand on the RHS vanishes unless \eqref{eq:box-orth} is satisfied for $s_{0} = +$ or $-$. In particular, by disposability of $P_{k} P_{\calC_{k_{0}}(\ell)} Q_{j} = P_{k} P_{\ell}^{\omg} Q_{j}$ and Proposition~\ref{prop:nf-basic}, it follows that
\begin{equation} \label{eq:axs-Z-atom}
\begin{aligned}
& \hskip-4em
	\nrm{P_{k_{0}} P_{-\calC^{0}} Q_{j} \NF_{s_{1} s_{2}}(P_{\calC^{1}} Q_{<j}^{s_{1}} \varphi^{1}_{k_{1}}, P_{\calC^{2}} Q_{<j}^{s_{2}} \varphi^{2}_{k_{2}})}_{L^{1} L^{\infty}} \\
	\aleq & 2^{\ell} 2^{k_{0} - k_{2}} \nrm{P_{\calC^{1}} Q_{<j}^{s_{1}} \varphi^{1}}_{L^{2} L^{\infty}} \nrm{P_{\calC^{2}} Q_{<j}^{s_{2}} \varphi^{2}}_{L^{2} L^{\infty}} .
\end{aligned}
\end{equation}
Moreover, by Lemma~\ref{lem:box-orth}, note that for a fixed $\calC^{1}$ [resp. $\calC^{2}$], there are only (uniformly) bounded number of $\calC^{0}, \calC^{2}$ [resp. $\calC^{0}, \calC^{1}$] such that \eqref{eq:box-orth} is satisfied with $s_{0} = +$ or $-$. Summing up first in $\calC_{0}$ (for which there are only finitely many terms) and then applying Lemma~\ref{lem:CS} to the summation in $\calC^{1}, \calC^{2}$ (which is essentially diagonal), we obtain
\begin{align*}
& \hskip-2em
\sum_{\calC^{0}} \nrm{P_{k_{0}} P_{-\calC^{0}} Q_{j} \NF_{s_{1}s_{2}}(Q_{<j}^{s_{1}} \varphi^{1}_{k_{1}}, Q_{<j}^{s_{2}} \varphi^{2}_{k_{2}})}_{L^{1} L^{\infty}} \\
\aleq & 2^{\ell} 2^{k_{0} - k_{2}}
		\bb( \sum_{\calC^{1}} \nrm{P_{\calC^{1}} Q_{<j}^{s_{1}} \varphi^{1}_{k_{1}}}_{L^{2} L^{\infty}}^{2} \bb)^{1/2}
		\bb( \sum_{\calC^{2}} \nrm{P_{\calC^{2}} Q_{<j}^{s_{2}} \varphi^{2}_{k_{2}}}_{L^{2} L^{\infty}}^{2} \bb)^{1/2}.
\end{align*}
Recall the convention $P_{k} P_{\calC_{k}(\ell)} = P_{k} P_{\ell}^{\omg}$. By \eqref{eq:Z-L1Linfty} and \eqref{eq:fe-L2Linfty}, we have
\begin{equation*}
	\nrm{P_{k_{0}} Q_{j} \NF_{s_{1}s_{2}}(P_{\calC^{1}} Q_{<j}^{s_{1}} \varphi^{1}_{k_{1}}, P_{\calC^{2}} Q_{<j}^{s_{2}} \varphi^{2}_{k_{2}})}_{\Box Z^{1}}
	\aleq 2^{\frac{1}{4}(j - k_{0})} 2^{k_{0} - k_{2}} c_{k_{1}} d_{k_{2}}.
\end{equation*}
Summing up in $j < k_{0} + C_{2}$, \eqref{eq:axs-Z-3} follows.

\begin{remark}  \label{rem:diff-Hs}
In this subsection, the key places where the null structure is used are \eqref{eq:diffr-Hs-low-mod}, \eqref{eq:diffr-Z-atom}, \eqref{eq:axr-Z-low-mod}, \eqref{eq:axr-Z-atom} and \eqref{eq:axs-Z-atom}.
Many estimates for $\Diff^{E}[A_{0}]$ and $\bfA_{0}$ in the next subsection will be proved by similar arguments, but with modification at the above places.
\end{remark}

\subsection{Further decomposition of $\bfA_{0}$ and $\Diff^{E}$} \label{subsec:tri-bi-a0}

We now deal with the term involving $A_{0} = \bfA_{0}(\varphi^{1}, \varphi^{2})$ in $\Diff^{E}[A_{0}] \psi$. 
Consider the trilinear operator
\begin{equation*}
	\calT^{E}_{s_{1}, s_{2}, s}(\varphi^{1}, \varphi^{2}, \psi) = s_{2} s \calH^{\ast} \big( \calH_{s_{1}, s_{2}} \lap^{-1}  \brk{\Pi_{s_{1}} \varphi^{1}, \mR_{0} \Pi_{s_{2}} \varphi^{2}} \mR^{0} \psi \big).
\end{equation*}
We will show that all of $\Diff^{E}[\bfA_{0}(\Pi_{s_{1}} \varphi^{1}, \Pi_{s_{2}} \varphi^{2})] \psi$ except $\calT^{E}_{s_{1}, s_{2}, s}$ can be handled by bilinear estimates.  The $Z^{1}_{ell}$ norm will be used as an intermediary.

Under the normalization \eqref{eq:fe-n-tri}, we claim that
\begin{equation} \label{eq:tri-bi-a0}
\nrm{\Diff^{E}[\bfA_{0}(\Pi_{s_{1}} \varphi^{1}, \Pi_{s_{2}} \varphi^{2})] \psi - \calT^{E}_{s_{1}, s_{2}, s}(\varphi^{1}, \varphi^{2}, \psi)}_{(N_{s}^{1/2})_{e}} 
	\aleq 1.
\end{equation}

\subsubsection*{Step~0: Reduction to bilinear estimates}
Let $a, b, c, d$ be admissible frequency envelopes. Define $e_{k} = (\sum_{k' < k} a_{k'}) c_{k}$ and $\tilde{e}_{k} = (\sum_{k' < k} \tilde{a}_{k'}) c_{k}$. We claim that
\begin{align} 
	\nrm{(\Id - \calH^{\ast}_{s, s}) \Diff^{E} [A_{0}] \psi}_{(N^{1/2}_{s})_{e}}
	\aleq & \nrm{A_{0}}_{Y^{1}_{a}} \nrm{\psi}_{(\tilde{S}^{1/2}_{s})_{c}} ,  \label{eq:diffe-Hs}\\
	\nrm{\calH^{\ast}_{s, s} \Diff^{E}[A_{0}] (\Id + s \mR^{0}) \psi}_{(N^{1/2}_{s})_{e}}
	\aleq & \nrm{A_{0}}_{Y^{1}_{a}} \nrm{\psi}_{(\tilde{S}^{1/2}_{s})_{c}} ,		\label{eq:diffe-HsR} \\
	\nrm{\calH^{\ast}_{s, s} \Diff^{E}[A_{0}] \mR^{0} \psi}_{(L^{1} \dot{H}^{1/2})_{\tilde{e}}}
	\aleq & \nrm{A_{0}}_{(Z_{ell}^{1})_{\tilde{a}}} \nrm{\psi}_{(\tilde{S}_{s}^{1/2})_{c}} , \label{eq:diffe-Z} \\
	\nrm{(1 - \calH_{s_{1}, s_{2}})\bfA_{0}(\varphi^{1}, \varphi^{2})}_{(Z^{1}_{ell})_{cd}} 
	\aleq & \nrm{\varphi^{1}}_{(\tilde{S}^{1/2}_{s_{1}})_{b}} \nrm{\varphi^{2}}_{(\tilde{S}^{1/2}_{s_{2}})_{c}}, \label{eq:a0-Z} \\
	\nrm{\calH_{s_{1}, s_{2}} (\bfA_{0} + s_{2} \bfA^{R}_{0}) (\varphi^{1}, \varphi^{2})}_{(Z^{1}_{ell})_{cd}} 
	\aleq & \nrm{\varphi^{1}}_{(\tilde{S}^{1/2}_{s_{1}})_{b}} \nrm{\varphi^{2}}_{(\tilde{S}^{1/2}_{s_{2}})_{c}}  . \label{eq:a0-r}  
\end{align}
where
\begin{equation*}
	\bfA_{0}^{R}(\varphi^{1}, \varphi^{2})
	:= \lap^{-1} \brk{\varphi^{1}, \mR_{0} \varphi^{2}} = - \lap^{-1} \brk{\varphi^{1}, \mR^{0} \varphi^{2}}.
\end{equation*}
Assuming these estimates, we now prove \eqref{eq:tri-bi-a0}. Assume the normalization \eqref{eq:fe-n-tri}. By \eqref{eq:a0}, \eqref{eq:a0-Z} and \eqref{eq:a0-r}, as well as disposability of $P_{k} \Pi_{s}$, we have
\begin{equation*}
	\nrm{\bfA_{0}(\Pi_{s_{1}} \varphi^{1}, \Pi_{s_{2}} \varphi^{2})}_{Y^{1}_{cd}}
	+ \nrm{(\bfA_{0} - \calH_{s_{1}, s_{2}} (- s_{2} \bfA_{0}^{R}))(\Pi_{s_{1}} \varphi^{1}, \Pi_{s_{2}} \varphi^{2})}_{(Z^{1}_{ell})_{cd}} \aleq 1.
\end{equation*}
Applying \eqref{eq:diffe-Hs}--\eqref{eq:diffe-Z} with $a = \tilde{a} = cd$, $e = \tilde{e} = (\sum_{k' < k} c_{k'} d_{k'}) b_{k}$ and $A_{0} = \bfA_{0}(\Pi_{s_{1}} \varphi^{1}, \Pi_{s_{2}} \varphi^{2})$, we obtain
\begin{equation*}
	\nrm{\Diff^{E}[\bfA_{0}(\Pi_{s_{1}} \varphi^{1}, \Pi_{s_{2}} \varphi^{2})] \psi + s \calH^{\ast}_{s, s} \Diff^{E}[\calH_{s_{1}, s_{2}} (- s_{2} \bfA_{0}^{R})(\Pi_{s_{1}} \varphi^{1}, \Pi_{s_{2}} \varphi^{2})] \mR^{0} \psi}_{(N_{s}^{1/2})_{e}} \aleq 1.
\end{equation*}
By definition, observe that
\begin{equation*}
s_{2} s \calH^{\ast}_{s, s} \Diff^{E}[\calH_{s_{1}, s_{2}} \bfA^{R}_{0}(\Pi_{s_{1}} \varphi^{1}, \Pi_{s_{2}} \varphi^{2})] \mR^{0} \psi
= \calT^{E}_{s_{1}, s_{2}, s}(\varphi^{1}, \varphi^{2}, \psi).
\end{equation*}
As before, the frequency envelope $e$ is dominated by $f$ as in \eqref{eq:trilinear-fe} by Cauchy-Schwarz; this completes the proof of \eqref{eq:tri-bi-a0}.

\subsubsection*{Step~1:~Proof of \eqref{eq:diffe-Hs}}
Assuming \eqref{eq:fe-n-tri} and $k_{1} < k_{2} - C^{\ast}_{2} - 5$, it suffices to prove:
\begin{equation} \label{eq:diffe-Hs-dyadic}
	\nrm{P_{k_{0}} \calL(B_{k_{1}}, \psi_{k_{2}}) - P_{k_{0}} T_{s} \calH^{\ast}_{k_{1}} \calL(B, T_{s} \psi_{k_{2}})}_{N_{s}^{1/2}} \aleq a_{k_{1}} b_{k_{2}}.
\end{equation}

All of Step~1 in Section~\ref{subsec:tri-bi-ax} applies in this case except for \eqref{eq:diffr-Hs-low-mod} (see Remark~\ref{rem:diff-Hs}), and it only remains to establish the analogue of \eqref{eq:diffr-Hs-low-mod}, where $I_{0}$, $I_{2}$ are as in \eqref{eq:nd-I0}, \eqref{eq:nd-I2}, but with $\NF$ replaced by $\calL$.

We provide a detailed proof for $I_{0}$, leaving the similar proof for $I_{2}$ to the reader. Instead of Proposition~\ref{prop:nf}, we simply apply Proposition~\ref{prop:no-nf} without the gain $2^{\ell} 2^{k_{\min} - \min \set{k_{1}, k_{2}}}$. By \eqref{eq:fe-L2}, we have
\begin{equation} \label{eq:diff-Hs-low-mod-0}
\nrm{P_{k_{0}} Q_{j}^{s} \calL(Q_{\leq j}^{s_{1}} B_{k_{1}}, Q_{\leq j}^{s} \psi)}_{X^{1/2, -1/2}_{s, 1}}
\aleq 2^{-\frac{1}{2} j} \bb( \sum_{\calC_{k_{1}}(\ell)} \nrm{P_{\calC_{k_{1}}(\ell)} B_{k_{1}}}_{L^{2} L^{\infty}}^{2} \bb)^{1/2} b_{k_{2}},
\end{equation}
where $\ell = \frac{1}{2} (j - k_{1})$. By Bernstein's inequality and \eqref{eq:fe-Y}, we have
\begin{align*}
\bb( \sum_{\calC_{k_{1}}(\ell)} \nrm{P_{\calC_{k_{1}}(\ell)} B_{k_{1}}}_{L^{2} L^{\infty}}^{2} \bb)^{1/2}
\aleq 2^{2 k_{1}} 2^{\frac{3}{2} \ell} \bb( \sum_{\calC_{k_{1}}(\ell)} \nrm{P_{\calC_{k_{1}}(\ell)} B_{k_{1}}}_{L^{2} L^{2}}^{2} \bb)^{1/2}
\aleq 2^{\frac{1}{2} k_{1}} 2^{\frac{3}{2} \ell} a_{k_{1}}.
\end{align*}
Hence it follows that
\begin{equation*}
	\sum_{s_{1}} \sum_{j < k_{1} - 10} 
	\nrm{P_{k_{0}} Q_{j}^{s} \calL(Q_{\leq j}^{s_{1}} B_{k_{1}}, Q_{\leq j}^{s} \psi)}_{N_{s}^{1/2}}
	\aleq \sum_{j < k_{1} - 10} 2^{\frac{1}{4}(j - k_{1})} a_{k_{1}} b_{k_{2}} \aleq a_{k_{1}} b_{k_{2}}
\end{equation*}
which is acceptable.

\subsubsection*{Step~2:~Proof of \eqref{eq:diffe-HsR}}
Assuming \eqref{eq:fe-n-tri} and $k_{1} < k_{2} - C^{\ast}_{2} - 5$, we claim:
\begin{equation} \label{eq:diffe-HsR-dyadic}
	\nrm{P_{k_{0}} T_{s} \calH^{\ast}_{k_{1}} \calL(B, T_{s} (\Id + s \mR^{0}) \psi_{k_{2}}))}_{N_{s}^{1/2}}
	\aleq 2^{k_{1} - k_{2}} a_{k_{1}} b_{k_{2}}.
\end{equation}
Note that \eqref{eq:diffe-HsR-dyadic} is more than enough to prove \eqref{eq:diffe-HsR} (i.e., the gain $2^{k_{1} - k_{2}}$ is unnecessary).

Fix $j < k_{1} + C^{\ast}_{2}$ and introduce the shorthand $\tilde{\psi} = (\Id + s \mR^{0}) \psi$. By \eqref{eq:q-LqL2}, H\"older's inequality $L^{2} L^{\infty} \times L^{2} L^{2} \to L^{1} L^{2}$, Bernstein's inequality and \eqref{eq:fe-Y}, we have
\begin{align*}
	\nrm{P_{k_{0}} Q_{<j}^{s} \calL(Q_{j} B_{k_{1}}, Q_{<j}^{s} \tilde{\psi}_{k_{2}}))}_{L^{1} \dot{H}^{1/2}}
	\aleq 2^{\frac{1}{2} k_{0}+\frac{1}{2} k_{1}} a_{k_{1}} \nrm{Q_{<j}^{s} \tilde{\psi}_{k_{2}}}_{L^{2} L^{2}}
\end{align*}
By \eqref{eq:commAlpPi:0} and \eqref{eq:fe-L2L2}, we have
\begin{equation} \label{eq:fe-L2L2-comm0}
\begin{aligned}
\nrm{Q_{<j} \tilde{\psi}_{k_{2}}}_{L^{2}L^{2}}
= & \nrm{Q_{<j}\frac{i \rd_{t} + s \abs{D}}{\abs{D}} \psi_{k_{2}}}_{L^{2}L^{2}}  \\
\aleq & \sum_{j' < j} 2^{j' - k_{2}} \nrm{Q_{j'} \psi_{k_{2}}}_{L^{2}L^{2}} \aleq 2^{\frac{1}{2} j} 2^{-\frac{3}{2} k_{2}} b_{k_{2}}.
\end{aligned}
\end{equation}
It follows that
\begin{align*}
	\nrm{P_{k_{0}} Q_{<j}^{s} \calL(Q_{j} B_{k_{1}}, Q_{<j}^{s} (\Id + s \mR^{0}) \psi_{k_{2}}))}_{L^{1} \dot{H}^{1/2}}
	\aleq 2^{\frac{1}{2}(j - k_{1})} 2^{k_{1} - k_{2}} a_{k_{1}} b_{k_{2}}.
\end{align*}
Summing up in $j < k_{1} + C^{\ast}_{2}$, we obtain \eqref{eq:diffe-HsR-dyadic} as desired.

\subsubsection*{Step~3:~Proof of \eqref{eq:diffe-Z}}
Assuming \eqref{eq:fe-n-tri} and $k_{1} < k_{2} - C^{\ast}_{2} - 5$, it suffices to prove:
\begin{equation} \label{eq:diffe-Z-dyadic}
	\nrm{P_{k_{0}} T_{s} \calH^{\ast}_{k_{1}} \calL(B, T_{s} \psi_{k_{2}})}_{L^{1} \dot{H}^{1/2}} \aleq \tilde{a}_{k_{1}} b_{k_{2}}.
\end{equation}
We proceed exactly as in Step~2 in Section~\ref{subsec:tri-bi-ax} with $\NF$ and $A$ replaced by $\calL$ and $B$, respectively, and the null form estimate \eqref{eq:diffr-Z-atom} replaced by
\begin{equation} \label{eq:diffe-Z-atom} 
\nrm{P_{k_{0}} Q_{<j}^{s} \calL(P_{\ell}^{\omg_{1}} Q_{j} B_{k_{1}}, P_{\ell}^{\omg_{2}} Q_{<j}^{s} \psi_{k_{2}})}_{L^{1} L^{2}}
\aleq \nrm{P_{\ell}^{\omg_{1}} Q_{j} B_{k_{1}}}_{L^{1} L^{\infty}} \nrm{P_{\ell}^{\omg_{2}} Q_{<j}^{s} \psi_{k_{2}}}_{L^{\infty} L^{2}}.
\end{equation}
We then arrive at
\begin{align*}
& \hskip-2em
\nrm{P_{k_{0}} T_{s} \calH^{\ast}_{k_{1}} \calL(B, T_{s} \psi_{k_{2}})}_{L^{1} \dot{H}^{1/2}} \\
\aleq & \sum_{j < k_{1} + C_{2}^{\ast}} 2^{\frac{1}{2} k_{2}} \bb( \sum_{\omg_{1}} \nrm{P^{\omg_{1}}_{\ell} Q_{j} B_{k_{1}}}_{L^{1} L^{\infty}}^{2} \bb)^{1/2} \bb( \sum_{\omg_{2}} \nrm{P^{\omg_{2}}_{\ell} Q_{<j}^{s} \psi_{k_{2}}}_{L^{\infty} L^{2}}^{2} \bb)^{1/2}.
\end{align*}
Although we lost the factor $2^{\ell}$ in \eqref{eq:diffe-Z-atom}, we gain it back from the $Z^{1}_{ell}$ norm in \eqref{eq:fe-Z}. By \eqref{eq:fe-L2-ang} and \eqref{eq:fe-Z}, the summand on the RHS is bounded by $2^{\frac{1}{2} \ell} \tilde{a}_{k_{1}} b_{k_{2}}$. Summing up in $j < k_{1} + C_{2}^{\ast}$, \eqref{eq:diffe-Z-dyadic} follows.

\subsubsection*{Step~4:~Proof of \eqref{eq:a0-Z}}
Under the normalization \eqref{eq:fe-n-tri}, it suffices to prove the following dyadic bounds:
For $k_{0} \geq k_{2} - C_{2} - 20$, we claim that
\begin{equation} \label{eq:a0-Z-1}
\nrm{P_{k_{0}} \NF^{\ast}(\varphi^{1}_{k_{1}}, \varphi^{2}_{k_{2}})}_{\lap Z^{1}_{ell}} 
	\aleq 2^{\dlt_{0} (k_{0} - k_{1})} c_{k_{1}} d_{k_{2}},
\end{equation}
and for $k_{0} < k_{2} - 5$, we claim that
\begin{equation} \label{eq:a0-Z-2}
\nrm{P_{k_{0}} \NF^{\ast}(\varphi^{1}_{k_{1}}, \varphi^{2}_{k_{2}})
- \calH_{k_{0}} \NF^{\ast}(T_{s_{1}} \varphi^{1}_{k_{1}}, T_{s_{2}} \varphi^{2}_{k_{2}})}_{\lap Z^{1}_{ell}} 
	\aleq 2^{\dlt_{0} (k_{0} - k_{1})} c_{k_{1}} d_{k_{2}}.
\end{equation}

We proceed as in Step~3 in Section~\ref{subsec:tri-bi-ax} with $\NF$ and $A$ replaced by $\calL$ and $B$. Since there is no use of null structure, the argument in Step~3.1 also applies in this case with \eqref{eq:Z-L1Linfty} and \eqref{eq:Z-L1L2} replaced by \eqref{eq:Zell-L1Linfty} and \eqref{eq:Zell-L1L2}, respectively. For the argument in Step~3.2, we replace \eqref{eq:Z-L1L2} by \eqref{eq:Zell-L1L2}, and \eqref{eq:axr-Z-low-mod} by the estimate: 
\begin{equation} \label{eq:a0-Z-low-mod}
	\nrm{I_{1}}_{\lap Z^{1}_{ell}} + \nrm{I_{2}}_{\lap Z^{1}_{ell}}
	\aleq 2^{\dlt_{0} (k_{\min} - k_{\max})} c_{k_{1}} d_{k_{2}},
\end{equation}
where $I_{1}$, $I_{2}$ are as in \eqref{eq:nm-I1}, \eqref{eq:nm-I2}, but with $\NF, \psi, \varphi$ replaced by $\calL, \varphi^{1}, \varphi^{2}$, respectively. We defer the proof of \eqref{eq:a0-Z-low-mod} for the moment. It can be checked that the rest of the argument in Step~3.2 goes through in the present case. Finally, for the argument in Step~3.3, we replace \eqref{eq:Z-L1Linfty} by \eqref{eq:Zell-L1Linfty}, and \eqref{eq:axr-Z-atom} by
\begin{equation} \label{eq:a0-Z-atom}
\begin{aligned}
& \hskip-4em
	\nrm{P_{k_{0}} P_{\ell_{0}}^{-\omg_{0}} Q_{j} \calL(P_{\ell_{0}}^{\omg_{1}} Q_{<j}^{s_{1}} \varphi^{1}_{k_{1}}, P_{\ell}^{\omg_{2}} Q_{<j}^{s_{2}} \varphi^{2}_{k_{2}})}_{L^{1} L^{\infty}}   \\
\aleq & \nrm{P_{\ell_{0}}^{\omg_{1}} Q_{<j}^{s_{1}} \varphi^{1}_{k_{1}}}_{L^{2} L^{\infty}} \nrm{P_{\ell}^{\omg_{2}} Q_{<j}^{s_{2}} \varphi^{2}_{k_{2}}}_{L^{2} L^{\infty}}
\end{aligned}
\end{equation}
which is an easy consequence of \eqref{eq:L-atom}. Although \eqref{eq:a0-Z-atom} loses $2^{\ell}$, \eqref{eq:Zell-L1Linfty} gains it back, and thus the rest of the argument in Step~3.3 applies. Then \eqref{eq:a0-Z-1} and \eqref{eq:a0-Z-2} follow.

It remains to establish \eqref{eq:a0-Z-low-mod}. We only consider $I_{1}$, since the proof for $I_{2}$ is entirely symmetric. Repeating the argument in Steps~2.1--2.2 in Section~\ref{subsec:bi-a}, but with Proposition~\ref{prop:nfs} replaced by Proposition~\ref{prop:no-nf}, we obtain
\begin{equation} \label{eq:a0-Z-low-mod-key}
	\nrm{P_{k_{0}} Q_{\leq j}^{s_{0}} \calL(Q_{j}^{s_{1}} \varphi^{1}_{k_{1}}, Q_{< j}^{s_{2}} \varphi^{2}_{k_{1}})}_{L^{1} L^{2}} \aleq 2^{-\frac{1}{2}\ell} 2^{\dlt_{0}(k_{\min} - k_{\max})}c_{k_{1}} d_{k_{2}}
\end{equation}
where $\ell = \frac{1}{2} (j - k_{\min})_{-}$. By \eqref{eq:Zell-L1L2-l}, we have
\begin{align*}
	\nrm{P_{k_{0}} Q_{\leq j}^{s_{0}} \calL(Q_{j}^{s_{1}} \varphi^{1}_{k_{1}}, Q_{< j}^{s_{2}} \varphi^{2}_{k_{1}})}_{\lap Z^{1}_{ell}} 
	\aleq & \sum_{j' \leq j} 2^{\frac{1}{2}(j' - k_{0})}	\nrm{P_{k_{0}} Q_{j'}^{s_{0}} \calL(Q_{j}^{s_{1}} \varphi^{1}_{k_{1}}, Q_{< j}^{s_{2}} \varphi^{2}_{k_{1}})}_{L^{1} L^{2}} \\
	\aleq & 2^{\frac{1}{2}\ell} 2^{\dlt_{0}(k_{\min} - k_{\max})}c_{k_{1}} d_{k_{2}}
\end{align*}
Summing up in $s_{0} \in \set{+, -}$ and $j < k_{\min} - 10$, \eqref{eq:a0-Z-low-mod} for $I_{1}$ follows. 

\subsubsection*{Step~5:~Proof of \eqref{eq:a0-r}}
Assuming \eqref{eq:fe-n-tri} and $k_{0} < k_{2} - C_{2} - 5$, it suffices to prove
\begin{equation} \label{eq:a0-r-dyadic}
	\nrm{\calH_{k_{0}} \calL(T_{s_{1}} \varphi^{1}_{k_{1}}, T_{s_{2}}(\Id + s_{2} \mR^{0}) \varphi^{2}_{k_{2}}}_{\lap Z^{1}_{ell}} \aleq 2^{\frac{3}{2} (k_{0} - k_{2})} c_{k_{1}} d_{k_{2}}.
\end{equation}

In order to ensure that the projections $Q_{<j}$ in $\calH_{k_{0}}$ are disposable, we perform an orthogonality argument as in Step~4 in Section~\ref{subsec:tri-bi-ax}. Fix $j < k_{0} +C_{2}$ and introduce the shorthands $\ell = \frac{1}{2} (j - k_{0})_{-}$ and $\tilde{\varphi} := (\Id + s_{2} \mR^{0}) \varphi^{2}$. For $i = 0, 1, 2$, let $\calC^{i}$ be a rectangular box of the form $\calC_{k_{0}}(\ell)$. We expand
\begin{equation*}
P_{k_{0}} Q_{j} \calL(Q_{<j}^{s_{1}} \varphi^{1}_{k_{1}}, Q_{<j}^{s_{2}} \tilde{\varphi}_{k_{2}})
= \sum_{\calC^{0}, \calC^{1}, \calC^{2}} P_{k_{0}} P_{- \calC^{0}} Q_{j} \calL(P_{\calC^{1}} Q_{<j}^{s_{1}} \varphi^{1}_{k_{1}}, P_{\calC^{2}} Q_{<j}^{s_{2}} \tilde{\varphi}_{k_{2}}).
\end{equation*}
By \eqref{eq:L-atom}, we have
\begin{equation*}
\nrm{P_{k_{0}} P_{- \calC^{0}} Q_{j} \calL(P_{\calC^{1}} Q_{<j}^{s_{1}} \varphi^{1}_{k_{1}}, P_{\calC^{2}} Q_{<j}^{s_{2}} \tilde{\varphi}_{k_{2}})}_{L^{1} L^{2}} \aleq \nrm{P_{\calC^{1}} Q_{<j}^{s_{1}} \varphi^{1}_{k_{1}}}_{L^{2} L^{\infty}} \nrm{Q_{<j}^{s_{2}} \tilde{\varphi}}_{L^{2} L^{2}}
\end{equation*}
Moreover, splitting $Q_{j} = Q_{j}^{+} T_{+} + Q_{j}^{-} T_{-}$ and applying Lemma~\ref{lem:box-orth}, we see that the LHS vanishes unless \eqref{eq:box-orth} holds with $s_{0} = +$ or $-$. 
Thus for a fixed $\calC^{1}$ [resp. $\calC^{2}$], there are only (uniformly) bounded number of $\calC^{0}, \calC^{2}$ [resp. $\calC^{0}, \calC^{1}$] such that LHS does not vanish. Summing up first in $\calC_{0}$ and then applying Lemma~\ref{lem:CS} to the (essentially diagonal) summation in $\calC^{1}, \calC^{2}$, we obtain
\begin{align*}
\nrm{P_{k_{0}} Q_{j} \calL(Q_{<j}^{s_{1}} \varphi^{1}_{k_{1}}, Q_{<j}^{s_{2}} \tilde{\varphi}_{k_{2}})}_{L^{1} L^{2}} 
\aleq \bb( \sum_{\calC^{1}} \nrm{P_{\calC^{1}} Q_{<j}^{s_{1}} \varphi^{1}_{k_{1}}}_{L^{2} L^{\infty}}^{2} \bb)^{1/2}
		\nrm{Q_{<j}^{s_{2}} \tilde{\varphi}_{k_{2}}}_{L^{2} L^{2}}
\end{align*}
By \eqref{eq:Zell-L1L2-l}, \eqref{eq:fe-L2Linfty} and \eqref{eq:fe-L2L2-comm0}, we have
\begin{equation*}
	\nrm{P_{k_{0}} Q_{j} \calL(Q_{<j}^{s_{1}} \varphi^{1}_{k_{1}}, Q_{<j}^{s_{2}} \tilde{\varphi}_{k_{2}})}_{\lap Z^{1}_{ell}}
	\aleq 2^{\frac{5}{2}(j - k_{0})} 2^{\frac{3}{2}(k_{0} - k_{2})} c_{k_{1}} d_{k_{2}}
\end{equation*}
Summing up in $j < k_{0} + C_{2}$, the desired estimate \eqref{eq:a0-r-dyadic} follows.

\subsection{Genuinely multilinear null form estimate} \label{subsec:tri-tri}
To complete the proof of Proposition~\ref{prop:trilinear}, it remains to estimate
\begin{equation} \label{eq:tri-form}
\begin{aligned}
	\calT_{s_{1}, s_{2}, s}(\varphi^{1}, \varphi^{2}, \psi)
	=& \calT^{E}_{s_{1}, s_{2}, s} (\varphi^{1}, \varphi^{2}, \psi) + \calT^{R}_{s_{1}, s_{2}, s} (\varphi^{1}, \varphi^{2}, \psi) \\
	=& s_{2} s \bb( - \calH^{\ast}_{s, s} \big( \calH_{s_{1}, s_{2}} \lap^{-1} \brk{\Pi_{s_{1}} \varphi^{1}, \mR_{0} \Pi_{s_{2}} \varphi^{2}} \mR_{0} \psi \big) \\
	& \phantom{s_{2} s \bb( }
		+ \calH^{\ast}_{s, s} \big( \calH_{s_{1}, s_{2}} \Box^{-1} \calP_{i} \brk{\Pi_{s_{1}} \varphi^{1}, \mR_{x} \Pi_{s_{2}} \varphi^{2}} \mR^{i} \psi \big) \bb).
\end{aligned}
\end{equation}
This part has a multilinear null structure akin to MKG-CG described in \cite[Appendix]{KST}.
In fact, thanks to the way we have set things up, it is possible to directly borrow the relevant estimates in \cite{KST}. We introduce the trilinear operator
\begin{equation*}
\begin{aligned}
\calT^{MKG}_{k, k'}(f^{1}_{k_{1}}, f^{2}_{k_{2}}, f^{3}_{k_{3}})
= &  -\calH^{\ast}_{k} \bb( \calH_{k'} \lap^{-1} \calL(f^{1}_{k_{1}}, \rd_{t} f^{2}_{k_{2}}) \rd_{t} f^{3}_{k_{3}} \bb)  \\
& 	+ \calH^{\ast}_{k} \bb( \calH_{k'} \Box^{-1} \calP_{i} \calL(f^{1}_{k_{1}}, \rd_{x} f^{2}_{k_{2}}) \rd^{i} f^{3}_{k_{3}} \bb),
\end{aligned}
\end{equation*}
where $\calL$ on both lines represent a single bilinear operator. Note that $\calT^{MKG}_{k, k'}$ vanishes unless $\abs{k - k'} < 3$. Moreover, in \cite[{Eq.~(136), (137) and (138); Appendix}]{KST}, the following estimate was proved\footnote{We remark that in \cite{KST}, this estimate is stated with the exponential factor $2^{\dlt(k - k_{\min})}$ instead of $2^{\dlt(k - k_{1})}$. A closer inspection of the proofs of \cite[Eq.~(136), (137) and (138)]{KST}, however, reveals that \eqref{eq:tri-MKG} holds.}:
\begin{proposition} \label{prop:tri-MKG}
For $k < \min\set{k_{0}, k_{1}, k_{2}, k_{3}} - C$ and $\abs{k' -k} < 3$, we have
\begin{equation} \label{eq:tri-MKG}
\begin{aligned}
	\nrm{P_{k_{0}} \calT^{MKG}_{k, k'}(f^{1}, f^{2}, f^{3})}_{N_{k_{0}}} 
	\aleq 2^{\dlt_{0} (k - k_{1})} 2^{\frac{1}{2} k_{0}} 2^{k_{1}} 2^{k_{2}} 2^{k_{3}} \prod_{i=1}^{3}  \nrm{f^{i}_{k_{i}}}_{S_{k_{i}}}.
\end{aligned}
\end{equation}
\end{proposition}
\begin{remark}  \label{rem:tri-null}
The proof of \eqref{eq:tri-MKG} exploits a trilinear null structure originally uncovered by Machedon--Sterbenz \cite{MachedonSterbenz}, which is sometimes referred to as the secondary null structure of Maxwell--Klein--Gordon. Roughly speaking, it says that the sum of the elliptic component $-\lap^{-1} \calL(f^{1}, \rd_{t} f^{2}) \rd_{t} f^{3}$ and hyperbolic components $\Box^{-1} \calP_{i} \calL(f^{1}, \rd_{x} f^{2}) \rd^{i} f^{3}$ in $\calT^{MKG}_{k, k'}(f^{1}, f^{2}, f^{3})$ can be rewritten as a combination of  the following three $Q_{0}$-type null forms (we refer to \cite[Appendix]{KST} for details):
\be \label{Q:dec2}
\begin{aligned}
	\mathcal{Q}_1(f^1,f^2,f^3) =& - \Box^{-1}  \calL  (f^1 \pt_{\al} f^2)\cdot \partial^{\al}f^3 , \\
	\mathcal{Q}_2(f^1,f^2,f^3) =& \Delta^{-1} \Box^{-1} \pt_t \pt_{\al} \calL  (f^1\pt_{\al} f^2)\cdot \partial_{t}f^3 , \\
	\mathcal{Q}_3(f^1,f^2,f^3)  =&  \Delta^{-1} \Box^{-1} \pt_{\al} \pt^i  \calL  (f^1 \pt_{i} f^2)\cdot \partial^{\al}f^3 .
\end{aligned}
\ee
The term $ \mathcal{Q}_1(f^1,f^2,f^3) $ is the more delicate one that requires the bilinear $ L^2 $ estimate \eqref{bilL2est}.

\

\end{remark}
Plugging in
\begin{equation*}
	f^{1}_{k_{1}} = \Pi_{s_{1}} Q_{<k_{1} - 3} T_{s_{1}} \varphi^{1}_{k_{1}}, \quad
	f^{2}_{k_{2}} = \frac{1}{i \abs{D}} \Pi_{s_{1}} Q_{<k_{2} - 3} T_{s_{2}} \varphi^{2}_{k_{2}}, \quad
	f^{3}_{k_{3}} = \frac{1}{i \abs{D}} Q_{<k_{3} - 3} T_{s} \psi_{k_{3}},
\end{equation*}
observe that
\begin{equation*}
	P_{k_{0}} \calT_{s_{1}, s_{2}, s}(\varphi^{1}_{k_{1}}, \varphi^{2}_{k_{2}}, \psi_{k_{3}})
	= - s_{2} s \sum_{\substack{k < k_{3} - C^{\ast}_{2} - 10 \\ k' < k_{2} - C_{2} - 10}} P_{k_{0}} \calT^{MKG}_{k, k'}(f^{1}_{k_{1}}, f^{2}_{k_{2}}, f^{3}_{k_{3}}).
\end{equation*}
By Proposition~\ref{prop:tri-MKG} and the facts that $k_{1} = k_{2} + O(1)$, $k_{3} = k_{0} + O(1)$, we have
\begin{equation} \label{eq:}
\nrm{P_{k_{0}} \calT_{s_{1}, s_{2}, s}(\varphi^{1}_{k_{1}}, \varphi^{2}_{k_{2}}, \psi_{k_{3}})}_{N_{s}^{1/2}}
\aleq 2^{\dlt_{0}(\min\set{k_{1}, k_{3}} - k_{1})} c_{k_{1}} d_{k_{2}} b_{k_{3}}.
\end{equation}
Keeping $k_{0}$ fixed and summing up in $k_{1}, k_{2}, k_{3}$, we obtain
\begin{equation*}
\nrm{P_{k_{0}} \calT_{s_{1}, s_{2}, s}(\varphi^{1}, \varphi^{2}, \psi)}_{N_{s}^{1/2}}
\aleq (\sum_{k' < k_{0}} c_{k'}^{2} )^{1/2} (\sum_{k' < k_{0}} d_{k'}^{2} )^{1/2} b_{k_{0}}
\end{equation*}
which completes our proof.


\begin{remark} \label{rem:hi-d-4}
In the higher dimensional case $d \geq 5$, all proofs in Section~\ref{subsec:tri-bi-ax} and \ref{subsec:tri-bi-a0} are valid with the substitutions as in Remark~\ref{rem:hi-d-3}, as well as $Z^{1} \to Z^{\frac{d-2}{2}}$ and $Z^{1}_{ell} \to Z^{\frac{d-2}{2}}_{ell}$.
Moreover, the multilinear null form estimate in Proposition~\ref{prop:tri-MKG} is unnecessary. We claim that the following additional estimates hold:
\begin{align}
	\nrm{\calH_{s_{1}, s_{2}} \bfA_{x}^{R}(\varphi^{1}, \varphi^{2})}_{(Z^{\frac{d-2}{2}})_{bc}} 
	\aleq & \nrm{\varphi^{1}}_{(\tilde{S}^{\frac{d-3}{2}}_{s_{1}})_{b}} \nrm{\varphi^{2}}_{(\tilde{S}^{\frac{d-3}{2}}_{s_{2}})_{c}},	\label{eq:hi-d-axr}  \\
	\nrm{\calH_{s_{1}, s_{2}} \bfA_{0}(\varphi^{1}, \varphi^{2})}_{(Z^{\frac{d-2}{2}}_{ell})_{bc}} 
	\aleq & \nrm{\varphi^{1}}_{(\tilde{S}^{\frac{d-3}{2}}_{s_{1}})_{b}} \nrm{\varphi^{2}}_{(\tilde{S}^{\frac{d-3}{2}}_{s_{2}})_{c}},	\label{eq:hi-d-a0} 
\end{align}

where the space $\tilde{S}^{\frac{d-3}{2}}_{s}$ does \emph{not} involve the null frame spaces $PW^{\mp}_{\omg}(l)$ and $NE$; see Section~\ref{sec:fs} for the precise definition. Combined with (the higher dimensional analogues of) \eqref{eq:diffr-Z} and \eqref{eq:diffe-Z}, we obtain an analogue of Proposition~\ref{prop:tri-MKG} without relying on the null structure of $\calT_{s_{1}, s_{2}, s}$ discussed in Remark~\ref{rem:tri-null}.

We provide proofs of \eqref{eq:hi-d-axr} and \eqref{eq:hi-d-a0} in the appendix below.

\end{remark}

\subsection*{Appendix: Trilinear estimates in the case $d \geq 5$}

We exploit the improved gain in the angular dimensions available in this case, captured by the $S^{box}_{k}$ component of the $\tilde{S}^{\frac{d-3}{2}}_{s}$ norm. At the technical level, the arguments below are analogous to that of \eqref{eq:axs-Z-3}.

In the proof, we assume that
\begin{equation} \label{eq:fe-hid}
	\nrm{\varphi^{1}}_{(\tilde{S}_{s_{1}}^{\frac{d-3}{2}})_{b}}
	= \nrm{\varphi^{2}}_{(\tilde{S}_{s_{2}}^{\frac{d-3}{2}})_{c}} 
	= 1
\end{equation}
Then the following analogue of \eqref{eq:fe-L2Linfty} holds for $\varphi^{1}$:
\begin{equation} \label{eq:fe-L2Linfty-hid}
	\bb( \sum_{\calC_{k'}(\ell)} \nrm{P_{\calC_{k'}(\ell)} Q^{s}_{<j} \varphi^{1}}_{L^{2} L^{\infty}}^{2} \bb)^{1/2}
	\aleq 2^{\frac{d-2}{2} k'} 2^{\frac{d-3}{2} \ell} 2^{-\frac{d-4}{2} k} b_{k},
\end{equation}
and the same estimate holds for $\varphi^{2}$ with $(s_{1}, b_{k})$ replaced by $(s_{2}, c_{k})$.

In $d \geq 5$, it turns out that neither the null structure in $\bfA^{R}_{x}$ nor the $2^{\ell}$ improvement in $Z^{\frac{d-2}{2}}_{ell}$ is necessary. To give a unified proof, we begin by estimating
\begin{equation*}
P_{k_{0}} Q_{j} \calL (Q_{<j}^{s_{1}} \varphi^{1}_{k_{1}}, Q_{<j}^{s_{2}} \varphi^{2}_{k_{2}})
\end{equation*}
in $L^{1} L^{\infty}$ for a fixed $j \leq k_{0} + C_{2}$. Let $\ell = \frac{1}{2} (j - k_{0})_{-}$, and for $i= 0, 1,2$, let $\calC^{i}$ be a rectangular box of the form $\calC_{k_{0}}(\ell)$. As before, we split
\begin{equation*}
P_{k_{0}} Q_{j} \calL (Q_{<j}^{s_{1}} \varphi^{1}_{k_{1}}, Q_{<j}^{s_{2}} \varphi^{2}_{k_{2}}) = \sum_{\calC^{0}, \calC^{1}, \calC^{2}} P_{k_{0}} P_{- \calC^{0}} Q_{j} \calL(P_{\calC^{1}} Q_{<j}^{s_{1}} \varphi^{1}_{k_{1}}, P_{\calC^{2}} Q_{<j}^{s_{2}} \varphi^{2}_{k_{2}}).
\end{equation*}
Further splitting $Q_{j} = Q_{j}^{+} T_{+} + Q_{j}^{-} T_{-}$ and applying Lemma~\ref{lem:box-orth}, we see that the summand on the RHS vanishes unless \eqref{eq:box-orth} is satisfied for $s_{0} = +$ or $-$. Moreover, by the same lemma, note that for a fixed $\calC^{1}$ [resp. $\calC^{2}$], there are only (uniformly) bounded number of $\calC^{0}, \calC^{2}$ [resp. $\calC^{0}, \calC^{1}$] such that \eqref{eq:box-orth} is satisfied with $s_{0} = +$ or $-$. Summing up first in $\calC_{0}$ (for which there are only finitely many terms) and then applying Lemma~\ref{lem:CS} to the summation in $\calC^{1}, \calC^{2}$ (which is essentially diagonal), we obtain
\begin{equation*} 
\begin{aligned}
& \hskip-2em
\sum_{\calC^{0}} \nrm{P_{k_{0}} P_{-\calC^{0}} Q_{j} \calL(Q_{<j}^{s_{1}} \varphi^{1}_{k_{1}}, Q_{<j}^{s_{2}} \varphi^{2}_{k_{2}})}_{L^{1} L^{\infty}} \\
\aleq & 	\bb( \sum_{\calC^{1}} \nrm{P_{\calC^{1}} Q_{<j}^{s_{1}} \varphi^{1}_{k_{1}}}_{L^{2} L^{\infty}}^{2} \bb)^{1/2}
		\bb( \sum_{\calC^{2}} \nrm{P_{\calC^{2}} Q_{<j}^{s_{2}} \varphi^{2}_{k_{2}}}_{L^{2} L^{\infty}}^{2} \bb)^{1/2},
\end{aligned}
\end{equation*}
where we used disposability of $P_{k} P_{\calC_{k_{0}}(\ell)} Q_{j} = P_{k} P_{\ell}^{\omg} Q_{j}$ and H\"older's inequality to estimate each summand. Recall the convention $P_{k} P_{\calC_{k}(\ell)} = P_{k} P_{\ell}^{\omg}$. By \eqref{eq:Z-L1Linfty} and \eqref{eq:fe-L2Linfty-hid}, it follows that
\begin{align*}
& \hskip-2em
	\nrm{P_{k_{0}} Q_{j} \calL(P_{\calC^{1}} Q_{<j}^{s_{1}} \varphi^{1}_{k_{1}}, P_{\calC^{2}} Q_{<j}^{s_{2}} \varphi^{2}_{k_{2}})}_{\Box Z^{\frac{d-2}{2}}} \\
	\aleq & 2^{-\frac{3}{2} \ell} 2^{-2 k_{0}} \bb( \sum_{\calC^{1}} \nrm{P_{\calC^{1}} Q_{<j}^{s_{1}} \varphi^{1}_{k_{1}}}_{L^{2} L^{\infty}}^{2} \bb)^{1/2}
		\bb( \sum_{\calC^{2}} \nrm{P_{\calC^{2}} Q_{<j}^{s_{2}} \varphi^{2}_{k_{2}}}_{L^{2} L^{\infty}}^{2} \bb)^{1/2} \\
	\leq & 2^{\frac{2d - 9}{4} (j - k_{0})} 2^{(d-4)(k_{0} - k_{1})} b_{k_{1}} c_{k_{2}}.
\end{align*}
Since $d \geq 5$, note that both $2^{\frac{2d - 9}{4} (j - k_{0})}$ and $2^{(d-4)(k_{0} - k_{1})}$ give exponential gains. Summing up in $j < k_{0} + C_{2}$, \eqref{eq:hi-d-axr} follows. Moreover, since $\Box Z^{\frac{d-2}{2}} \subseteq \lap Z^{\frac{d-2}{2}}_{ell}$ (alternatively, compare \eqref{eq:Zell-L1Linfty} with \eqref{eq:Z-L1Linfty}), \eqref{eq:hi-d-a0} follows as well.

\section{Solvability of paradifferential covariant half-wave equations} \label{sec:para}

The goal of this section is to prove Theorem~\ref{thm:paradiff}. The disclaimer ``for $ \ep>0 $ small enough'' applies to all the statements. Below, we denote by $\dlt > 0$ a small constant which depends only on $\ep$ and $\dlt \to 0$ as $\ep \to 0$.

\subsection{Parametrix}

Suppose $ \Box A^{free}=0 $ with $ \vn{ A^{free}}_{\Hc} \leq \ep $ together with the Coulomb condition $\rd^{\ell} A^{free}_{\ell} =0 $. Without loss of generality we assume $ s=+ $. Define the paradifferential half-wave operators by
\be \phw = i \pt_t + \vm{D} - i  \sum_{k \in \mb{Z}} P_{<k-C} A^{free,j} \frac{\pt_j}{ \vm{D}} P_k  \ee
\be \label{paracovop2} \hw_{A_{<k}}^p = i \pt_t + \vm{D} - i  P_{<k-C} A^{free,j} \frac{\pt_j}{\vm{D}} P_k   \ee
and the paradifferential covariant $ \Box $ operator by 
\be \label{paracovop3} \Box_{A_{<k}}^p=\Box -2i  P_{<k-C} A^{free,j} P_k \pt_j, \ee

Consider the problem
\begin{equation}\label{problem}
\left\{ 
\begin{array}{l}
 \phw \psi=F  \\
 \psi(0)=f.
\end{array} 
\right.
\end{equation}
  
The proof of Theorem~\ref{thm:paradiff} will reduce to the following proposition, whose proof we postpone to Section~\ref{subsecproof}:

\begin{proposition} \label{papp}
For any $ F \in N_{+}^{1/2}\cap L^{2} L^{2} $ and any $ f \in \Hcr $ there exists $ \psi^a \in S_{+}^{1/2} $ such that for any admissible frequency envelope $c$, we have
\begin{equation} 
\begin{aligned}
\vn{\psi^a(0)-f}_{\Hcr_c}+ \vn{\phw \psi^a - F}_{(N_{+}^{1/2}\cap L^{2} L^{2})_c} \\
\leq \delta \lpr  \vn{f}_{\Hcr_c}+ \vn{F}_{(N_{+}^{1/2}\cap L^{2} L^{2})_c} \rpr,
\end{aligned}\end{equation}
\begin{equation} 
	\nrm{\psi^a}_{(S^{1/2}_{+})_{c}} \aleq \nrm{f}_{\dot{H}^{1/2}_{c}} + \nrm{F}_{(N^{1/2}_{+} \cap L^{2} L^{2})_{c}} \ .
\end{equation}
\end{proposition}

\begin{proof}[Proof of Theorem~\ref{thm:paradiff}] 

Denote by $ \psi^a[f,F] $ the approximate solution obtained in Proposition~\ref{papp}.
We define $ \psi \defeq \lim \psi^{\leq n} $ where 
$$ \psi^{\leq n} \defeq \psi^1+\dots +\psi^n $$ 
and $ \psi^n $ are defined inductively by $ \psi^1 \defeq \psi^a[\psi_0,F] $ and
\be \psi^n \defeq \psi^a [\psi_0- \psi^{\leq n-1}(0), F- \phw \psi^{\leq n-1}], \qquad n\geq 2. \ee
Normalizing $ \vn{\psi_0}_{\Hcr_c}+ \vn{F}_{(N_{+}^{1/2}\cap L^{2} L^{2})_c}=1 $ it follows by induction using Proposition~\ref{papp} that we have
\be  \vn{\psi^{\leq n}(0)-\psi_0}_{\Hcr_c}+ \vn{\phw \psi^{\leq n} - F}_{(N_{+}^{1/2}\cap L^{2} L^{2})_c} \leq \delta^n  , \ee
\be \vn{\psi^n}_{(S_{+}^{1/2})_c} \lesssim \delta^{n-1}  . \ee
Applying these bounds for a frequency envelope $ \tilde{c} $ satisfying
$$ \vn{\tilde{c}}_{\ell^2} \simeq \vn{\psi_0}_{\Hcr}+ \vn{F}_{N_{+}^{1/2}\cap L^{2} L^{2}} $$
yields
\be \label{partialeq} 
\begin{aligned}
\vn{\psi^{\leq n}(0)-\psi_0}_{\Hcr}+ \vn{\phw \psi^{\leq n} - F}_{N_{+}^{1/2}\cap L^{2} L^{2}}  \\
\lesssim \delta^n (\vn{\psi_0}_{\Hcr}+ \vn{F}_{N_{+}^{1/2}\cap L^{2} L^{2}}) 
\end{aligned}\ee
\be \vn{\psi^n}_{S_{+}^{1/2}} \lesssim \delta^{n-1} (\vn{\psi_0}_{\Hcr}+ \vn{F}_{N_{+}^{1/2}\cap L^{2} L^{2}}) . \ee

Thus $ \psi^{\leq n} $ is a Cauchy sequence in $ S_{+}^{1/2} $ and $ \psi $ is well-defined, satisfying \eqref{spsestim} and
\be \vn{\psi}_{(S_{+}^{1/2})_c} \lesssim \sum \delta^k  \lesssim 1.\ee
Furthermore, passing to the limit in \eqref{partialeq} we get that $ \psi $ exactly solves \eqref{problem}. \qedhere
\end{proof}

\subsection{Renormalization and parametrix construction for $\Box^{p}_{A_{<0}}$} \label{subsec:rn-box}

The purpose of this subsection is to review the parametrix construction from \cite[Sections 6--10]{KST}, which proceeds by conjugating the paradifferential covariant d'Alembertian operator $\Box^{p}_{A_{<0}}$ to $\Box$ by (pseudodifferential) suitable renormalization operators
\begin{equation} \label{eq:rn-op}
 e^{-i \Psi_{\pm}}_{<0} (t,x,D) ,\qquad e^{i \Psi_{\pm}}_{<0} (D,y,s) 
\end{equation}
with a real-valued phase $\Psi_{\pm} = \Psi_{\pm}(t,x,\xi)$. Here, $ P(x,D) $ denotes the left quantization, while $ P(D,y) $ denotes the right quantization (only in the space variables). The $ <0 $ subscript represents space-time frequency localization to the region $\set{\abs{(\tau, \xi)} \ll 1 }$.


For some $\sgm > 0$ to be chosen as small as necessary below, the phase $\Psi_{\pm} = \Psi_{\pm}(t, x, \xi)$ is defined as 
\begin{equation} \label{eq:phasefc}
	\Psi_{\pm}(t, x, \xi) = \sum_{k < - C} L^{\omg}_{\pm} \lap^{-1}_{\omg^{\perp}} (\Pi^{\omg}_{> \sgm k} P_{k} (\omg \cdot A_{x})),
\end{equation}
where $\omg = \frac{\xi}{\abs{\xi}}$, $L^{\omg}_{\pm} = \pm \rd_{t} + \omg \cdot \nb_{x}$, $\lap_{\omg} = \lap - (\omg \cdot \nb_{x})^{2}$ and $\Pi^{\omg}_{> \sgm k}$ is a smooth restriction in (spatial) Fourier space to the region $\set{\eta: \abs{\angle(\eta, \pm \omg)} \ageq 2^{\sgm k}}$. Note that $\lap_{\omg^{\perp}}^{-1}$ is singular along the line parallel to $\omg$ in Fourier space, but $\Pi^{\omg}_{> \sgm k}$ vanishes there; hence \eqref{eq:phasefc} is well-defined.

\subsubsection*{Motivation for \eqref{eq:phasefc}} 
Before we proceed any further, some motivation for the formula \eqref{eq:phasefc} is in order. For a more detailed account, see \cite[Sections~7--8]{RT} or \cite[Section~6]{KST}. 

The aim of the renormalization procedure is to conjugate $\Box^{p}_{A_{<0}}$ to $\Box$ up to manageable errors; i.e., by choosing an appropriate real-valued symbol $\Psi$, we wish to achieve
\begin{equation*}
	e^{i \Psi}_{<0} (D, y, s) \Box^{p}_{A_{<0}} e^{- i \Psi}_{<0} (t, x, D)
	=  \Box  + \cdots
\end{equation*}
for inputs frequency-localized in $\set{\abs{\xi} \simeq 1}$, where $(\cdots)$ represent some manageable errors. Then the idea is to construct a parametrix (i.e., an approximate solution) for $\Box^{p}_{A_{<0}}$ using the solution operator for $\Box$.

The main term in the conjugation error takes the form
\begin{align*}
& \hskip-2em
	\left( e^{i \Psi}_{<0} (D, y, s) \Box^{p}_{A_{<0}} e^{- i \Psi}_{<0} (t, x, D)
		-  \Box \right) \phi \\
	= & 2 e^{i \Psi}_{<0} (D, y, s) (P_{<-C} A^{free, j} \xi_{j} - \rd_{t} \Psi \tau + \rd^{j} \Psi \xi_{j}) e^{- i \Psi}_{<0} (t, x, D) \phi + \cdots.
\end{align*}
To proceed with the heuristic discussion, we focus on the task of constructing a parametrix for the homogeneous paradifferential covariant wave equation. Then in order to make $\Box \phi$ vanish, it is reasonable to let $\phi$ be a $\pm$-half wave, so that $\tau = \pm \abs{\xi}$ on the frequency support of $\phi$. Thus, in order to cancel the first term on the RHS, we are motivated to take
\begin{equation*}
	- L^{\omg}_{\mp} \Psi_{\pm} = (\pm \rd_{t} - \omg \cdot \nb_{x}) \Psi_{\pm} ``=" P_{<-C} (\omg \cdot A^{free}_{x}),
\end{equation*}
where $\omg \in \bbS^{3}$ is given by $\omg = \frac{\xi}{\abs{\xi}}$. In particular, we choose a different phase $\Psi_{\pm}$ depending on the adapted characteristic cone of the input.

In general, the symbol of a transport operator, such as $L^{\omg}_{\mp}$, vanishes in a hyperplane (co)normal to the direction of transport. Thus inverting $L^{\omg}_{\mp}$ does not gain any derivative except in the direction of $L^{\omg}_{\mp}$ itself, which is unsatisfactory for estimating $\Psi_{\pm}$. However, since $A^{free}_{x}$ solves the free wave equation, it is possible to invert $L^{\omg}_{\mp}$ in a more advantageous way. Indeed, observe that on the frequency support of $A^{free}_{x}$, which is contained in $\set{(\sgm, \eta) : \abs{\sgm} = \abs{\eta}}$, the symbol of $L^{\omg}_{\mp}$ (which equals $\mp \sgm + \omg \cdot \eta$) vanishes only along the line $\set{(\sgm, \eta) : \sgm = \pm (\omg \cdot \eta), \ \eta \parallel \omg}$. Therefore, $L^{\omg}_{\mp}$ is elliptic on the support of $\Pi_{away}^{\omg} A^{free}$, where $\Pi^{\omg}_{away}$ is a smooth cutoff in spatial Fourier space supported just away from the line $\set{\eta : \eta \parallel \omg}$.

To obtain a concrete formula for $\Psi_{\pm}$ along the ideas just discussed, we make use of the following null frame decomposition of $\Box$, which holds for any fixed $\omg \in \bbS^{3}$:
\begin{equation*}
	\Box = L^{\omg}_{\mp} L^{\omg}_{\pm} + \lap_{\omg^{\perp}}.
\end{equation*}
Thus $\Box A_{x}^{free} = 0$ implies
\begin{equation*}
	- L^{\omg}_{\mp} \left( L^{\omg}_{\pm} \lap_{\omg^{\perp}}^{-1} \Pi^{\omg}_{away} P_{<-C} (\omg \cdot A_{x}^{free}) \right) = P_{<-C} (\omg \cdot A_{x}^{free}),
\end{equation*}
where we note that the symbol of $\lap_{\omg^{\perp}}$ is nonvanishing in the frequency support of $\Pi^{\omg}_{away}$. This formula motivates the definition
\begin{equation*}
	\Psi_{\pm} ``=" L^{\omg}_{\pm} \lap_{\omg^{\perp}}^{-1} \Pi^{\omg}_{away} P_{<-C} (\omg \cdot A_{x}^{free}).
\end{equation*}

It remains to pin down the angular cutoff $\Pi^{\omg}_{away}$. There are two factors to balance here: On the one hand, in order to make the remaining conjugation error $\Pi^{\omg}_{near} P_{<-C} (\omg \cdot A_{x}) = (1 - \Pi^{\omg}_{far}) P_{<-C} (\omg \cdot A_{x})$ small, we wish to cut away only a small angle with $\Pi^{\omg}_{far}$. On the other hand, the smaller the angle cut away, the rougher the $\xi$-dependence of the symbol $\Psi_{\pm}$ and the worse the mapping properties of the renormalization operator $e^{-i\Psi_{\pm}}(t, x, D)$.

In order to ensure summability (in dyadic frequency) of the conjugation error, it seems necessary to cut away smaller angles for lower frequency parts of $A^{free}_{x}$. Thus we are led to the choice
\begin{equation*}
	\Psi_{\pm} = \sum_{k < -C} L^{\omg}_{\pm} \lap_{\omg^{\perp}}^{-1} \left( \Pi^{\omg}_{> \sgm k} P_{k} (\omg \cdot A_{x}) \right)
\end{equation*}
for some $\sgm > 0$, which coincides with the desired formula \eqref{eq:phasefc}. At this point, only the exponent $\sgm > 0$ is left to be determined. Note that a larger $\sgm$ corresponds to a smaller angular cutoff.

Here lies the key difference between the constructions in \cite{RT} and in \cite{KST}. In \cite{RT}, $\sgm$ was chosen after careful balancing of availability of Strichartz estimates for the parametrix (favoring larger angular cutoff) and smallness of the remaining conjugation error (favoring smaller angular cutoff); this limited the validity of the construction to $d \geq 6$. On the other hand, in \cite{KST} it was observed that, for $\sgm$ small enough, $e^{-i \Psi_{\pm}}(t, x, D)$ obeys nice mapping properties in $X^{s, b}$-type spaces. This observation led to improved estimate for the remaining conjugation error (essentially, one then has access to bilinear estimates as in Section~\ref{sec:nf} to treat the conjugation error), and allowed taking $\sgm > 0$ arbitrarily small for every $d \geq 4$, as claimed in \eqref{eq:phasefc}.

\subsubsection*{Properties of the renormalization operator}
The following theorem summarizes the key properties of the operators \eqref{eq:rn-op}:
\begin{theorem} \label{opsummary} 
If $\sgm > 0$ is chosen sufficiently small, the frequency localized renormalization operators have the following properties with $ Z \in \{ N_0,L^2,N^{*}_0 \} $:

\be \label{renbd} e_{<0}^{\pm i \Psi_{\pm}} (t,x,D) : Z \to Z \ee
\be \label{renbdt} \pt_t e_{<0}^{\pm i \Psi_{\pm}} (t,x,D) : Z \to  \ep Z, \qquad \pt_x   e_{<0}^{\pm i \Psi_{\pm}} (t,x,D) : L^2 \to \ep L^2 \ee
\be e_{<0}^{-i \Psi_{\pm}} (t,x,D) e_{<0}^{i \Psi_{\pm}} (D,y,s)-I : Z \to Z \ee
\be \label{conj} e_{<0}^{-i \Psi_{\pm}} (t,x,D) \Box - \Box_{A_{<0}}^p e_{<0}^{-i \Psi_{\pm}} (t,x,D) : (N_{0}^{\pm})^{\ast} \to \ep N_{0}^{\pm}  \ee
\be \label{solnorbd} e_{<0}^{-i \Psi_{\pm}} (t,x,D) : S_{0}^{\#} \to S_0  \ee
\end{theorem}
The second mapping from \eqref{renbdt} is (116) from \cite{KST}, and the rest are from Theorem 3 in \cite{KST}. From now on, we fix the choice of $\sgm >0$ so that Theorem~\ref{opsummary} is valid.

\begin{remark} 
As discussed above, note that the above mapping properties hold (in particular) in $Z = N_{0}, N_{0}^{\ast}$, which contain $X^{s,b}$-type spaces. We remark that \eqref{conj} is the conjugation error estimate, and \eqref{solnorbd} is precisely the dispersive estimates for the parametrix.
\end{remark}

\subsubsection*{Construction of the parametrix}
The parametrix (or approximate solution) constructed in \cite{KST} for the equation\footnote{We use the definition \eqref{paracovop3} for $ \Box_{A_{<0}}^p $ as in  \cite{KST}, but note that we use a different convention for the Minkowski metric.}
\begin{equation} \label{approxcovsol}
\left\{ 
\begin{array}{l}
  \Box_{A_{<0}}^p \phi=F  \\
 \phi[0]=(g,h)
\end{array} 
\right.
\end{equation}
takes the form
\be  \label{phiapp} \phi_{app}=\frac{1}{2} \lpr T^{+}+T^{-}+ S^{+}+S^{-} \rpr  \ee
 where\footnote{Note that if the $ e_{<0}^{\pm i \Psi_{\pm}} $ terms are removed one obtains the solution of the ordinary wave equation $ \Box \phi=F, \ \phi[0]=(g,h) $.    }
\be T^{\pm}=e_{<0}^{-i \Psi_{\pm}} (t,x,D) \frac{1}{\vm{D}} e^{\pm i t \vm{D}} e^{i \Psi_{\pm}}_{<0}(D,y,0)(\vm{D}g\pm i^{-1} h) \ee
\be \label{spm} S^{\pm}=\mp e_{<0}^{-i \Psi_{\pm}} (t,x,D) \frac{1}{\vm{D}} K^{\pm} e^{i \Psi_{\pm}}_{<0}(D,y,s) i^{-1} F \ee
where $ K^{\pm}F $ denotes the solution $ u $ of the equation
\be (\pt_t \mp i \vm{D})u=F, \qquad u(0)=0 \ee
given by the Duhamel formula
$$ K^{\pm}F (t)= \int_0^t e^{\pm i (t-s) \vm{D}} F(s) \dd s. $$ 
More precisely, the result in  \cite{KST} states  

\begin{theorem} \label{covariantparametrix}
Assume that $ F,g,h $ are localized at frequency $ 1 $, and also that $ F $ is localized at modulation $ \lesssim 1 $. Then $  \phi_{app} $ is an approximate solution for \eqref{approxcovsol}, in the sense that
\be \vn{\phi_{app}}_{S_0} \lesssim \vn{g}_{L^2}+\vn{h}_{L^2}+\vn{F}_{N_0}  \ee
and
\be  \vn{\phi_{app}[0]-(g,h)}_{L^2}+ \vn{\Box_{A_{<0}}^p \phi_{app}-F}_{N_0} \leq \dlt (\vn{g}_{L^2}+\vn{h}_{L^2}+\vn{F}_{N_0} ) \ee
\end{theorem}
The spaces $ S_0 $ and $ N_0 $ are defined in Section \ref{sec:fs} (see also also Remark \ref{sboxrk}).

\subsection{Renormalization for $(i \rd_{t} + \abs{D})^{p}_{A}$} \label{subsec:rn-hw}
Henceforth, our goal is to similarly obtain a parametrix (or approximate solution) for \eqref{problem} in order to prove Proposition~\ref{papp}, using the results in Section~\ref{subsec:rn-box} to renormalize $(i \rd_{t} + \abs{D})^{p}_{A_{<k}}$.

Suppose $ F,g,h $ are localized at frequency $ 1 $, and consider $S^{\pm}, T^{\pm}$ defined in Section~\ref{subsec:rn-box}. If $ F $ has small $ Q^+ $ -modulation, then so do $ S^{+} $ and $ T^{+} $. This also applies to $ S^{-} $, except for a part with Fourier support in the lower characteristic cone. Therefore we decompose 
\be \label{decoms}  S^{-}=Q^{+}_{\leq -2}  S^{-} + S^{-}_0, \qquad S^{-}_0 \defeq e_{<0}^{-i \Psi_{-}} (t,x,D) Q^{-}_{\leq -1} \lpr \frac{1}{\vm{D}} u  \rpr,   \ee
according to the following definitions
\be u  \defeq \frac{1}{i} K^{-}\tilde{F}, \qquad  \  \tilde{F} \defeq e^{i \Psi_{\pm}}_{<0}(D,y,s) F,  \ee
so that $ \hwm u=\tilde{F}, \ u(0)=0 $.
Let us define the function $ v $ such that
\be \label{vfct} \mathcal{F} v (\tau,\xi) \defeq \frac{-1}{\tau+\vm{\xi}} \mathcal{F}(\tilde{F})(\tau,\xi),\hbox{ so } \quad \hwm v=\tilde{F}. \ee
The term $S_{0}^{-}$ can be controlled by $\nrm{F}_{N^{0}}$ as follows.
\begin{lemma}
Suppose $ F $ is localized at frequency $ 1 $ and at $Q^{+}$-modulation $ \leq 1 $. Then for $ S^{-}_0 $ and $ v $ defined by \eqref{decoms} and \eqref{vfct} we have:
\be   \label{vzr} \vn{v(0)}_{L^2}  \lesssim \vn{F}_{N_0} \ee
\be  \label{seq} S^{-}_0  =-e_{<0}^{-i \Psi_{-}} (t,x,D) \frac{1}{\vm{D}}e^{- i t \vm{D}} (v(0)) \ee
 \be  \label{szr}  \vn{\hwm S^{-}_0(0)}_{L^2}  \lesssim \ep \vn{F}_{N_0}. \ee
\end{lemma}

\begin{proof} 
The proof is divided into three steps.
\pfstep{Step~1: Proof of \eqref{vzr}} Since $ F $ and  $ \tilde{F} $ are localized at $ Q^{-}$-modulation $ \gtrsim 1$   from \eqref{embeasy}  and \eqref{renbd} we have
\be \vn{v(0)}_{L^2} \lesssim \vn{v}_{L^{\infty} L^2} \lesssim \vn{\tilde{F}}_{N_0^{-}} \lesssim \vn{\tilde{F}}_{N_0} \lesssim \vn{F}_{N_0}. \ee

\pfstep{Step~2: Proof of \eqref{seq}} Subtracting $ v $ from $ u $ we get
\be \hwm(u-v)=0, \qquad (u-v)(0)=-v(0). \ee
Thus $ Q^{-}_{\leq -1} u= Q^{-}_{\leq -1} (u-v)=e^{- i t \vm{D}} (-v(0)) $ from which \eqref{seq} follows.

\pfstep{Step~3: Proof of \eqref{szr}}  Using \eqref{vzr}, it suffices to show $ \vn{\hwm S^{-}_0(0)}_{L^2} \lesssim \ep \vn{v(0)}_{L^2}. $
\be \hwm S^{-}_0(0)=i [\pt_t e_{<0}^{- i \Psi_{-}}] (0,x,D)  \lpr \frac{v(0)}{\vm{D}} \rpr + \lpp e_{<0}^{- i \Psi_{-}},\vm{D} \rpp \lpr \frac{v(0)}{\vm{D}} \rpr. \ee

The first term is estimated by \eqref{renbdt}. For the second, we use the dual of Lemma 7.2 in \cite{KST} and \eqref{renbdt} to obtain
$$ \vn{\vm{D} e_{<0}^{-i \Psi_{-}} (0,x,D) -   e_{<0}^{-i \Psi_{-}} (0,x,D) \vm{D} }_{L^2 \to L^2} \lesssim \vn{\pt_x e_{<0}^{-i \Psi_{-}} (0,x,D)}_{L^2\to L^2} \lesssim \ep   \qedhere $$ \end{proof}

The following proposition is essentially a restatement of Theorem~\ref{covariantparametrix} in a convenient form for our application.
\begin{proposition} \label{propmod}
Suppose $ F $ and $ f $ are localized at frequency $ \{ \vm{\xi} \in  [2^{-2},2^{2}] \} $ and $ F $ is also localized at $ Q^+ $-modulation $ \{ \vm{\tau -\vm{\xi}} \leq 2^{-4} \} $.  Then there exists $ \phi $ localized at $ \{ \vm{\xi} \in [2^{-3},2^{+3}], \ \vm{\tau -\vm{\xi}} \leq 2^{-3} \} $     such that
\be \label{einq1} \vn{(i \pt_t -\vm{D})  \phi(0)-f}_{L^2}+ \vn{\Box_{A_{<0}}^p \phi-F}_{N_0}   \leq \delta \lpr \vn{f}_{L^2}+\vn{F}_{N_0} \rpr \ee
\be \label{einq2} \vn{\phi}_{S_0} \lesssim \vn{f}_{L^2}+\vn{F}_{N_0}. \ee
\end{proposition}

\begin{proof}
Let us choose $ g $ and $ h $  such that 
\be i h+\vm{D}g=0, \qquad ih-\vm{D}g=f \ee
and apply Theorem \ref{covariantparametrix} to $ (F,g,h) $. Then $ T^{-}=0 $ in the definition of $ \phi_{app} $ from  \eqref{phiapp}--\eqref{spm}. From Theorem \ref{covariantparametrix} we have
\be \label{estt1} \vn{\phi_{app}[0]-(g,h)}_{L^2}+ \vn{\Box_{A_{<0}}^p \phi_{app}-F}_{N_0} \ll B, \quad \vn{\phi_{app}}_{S_0} \lesssim B. \ee
where $ B=\vn{g}_{L^2}+\vn{h}_{L^2}+\vn{F}_{N_0} $. Observe that it suffices to bound the LHS of \eqref{einq1} by $ \delta B $.
We define  
\be \phi \defeq \frac{1}{2} \lpr T^{+}+ S^{+}+ Q^{+}_{\leq -2}  S^{-} \rpr \ee
and observe that $ \phi $ has the stated $ Q^+ $-modulation. Furthermore, 
$$  \phi_{app}=\phi+\frac{1}{2} S^{-}_0 $$
where $ S^{-}_0 $ is given by \eqref{decoms}, \eqref{seq}. We write
$$ \Box_{A_{<0}}^p \phi-F=\lpr \Box_{A_{<0}}^p \phi_{app}-F \rpr + \lpr \Box_{A_{<0}}^p e_{<0}^{-i \Psi_{-}} (t,x,D) -e_{<0}^{-i \Psi_{-}} (t,x,D) \Box \rpr \frac{e^{\pm i t \vm{D}}}{\vm{D}} (v(0))  $$
The first term is estimated by \eqref{estt1}, while for the second use \eqref{conj} and \eqref{vzr}. Moreover, 
\be 
\begin{aligned}
\hwm \phi(0)-f= & \lpp \hwm \phi_{app}-(ih-\vm{D}g) \rpp  \\
&+ \lpp (ih-\vm{D}g)-f \rpp - \frac{1}{2} \hwm S^{-}_0(0) 
\end{aligned}
\ee
The first term is estimated by \eqref{estt1}, the second term is zero, and the third term follows from \eqref{szr}. This proves \eqref{einq1}.

The bound \eqref{einq2} follows from \eqref{estt1}, \eqref{solnorbd}, \eqref{seq} and \eqref{vzr}. \qedhere
\end{proof}

We are now ready to construct the key part of our parametrix for \eqref{problem}.
\begin{proposition} \label{phyb}
Suppose $ F $ and $ f $ are localized at frequency $ \{ \vm{\xi} \in [2^{k-2},2^{k+2}] \} $ and $ F $ is also localized at $ Q^+ $ -modulation $ \{ \vm{\tau -\vm{\xi}} \leq 2^{k-4} \} $. Then there exists $ \psi_k^1 $ localized at $ \{ \vm{\xi} \in [2^{k-3},2^{k+3}], \ \vm{\tau -\vm{\xi}} \leq 2^{k-3} \} $ such that
\be \vn{\psi_k^1(0)-f}_{L^2}+ \vn{\hw_{A_{<k}}^p \psi_k^1-F}_{N_k^+}   \leq \delta \lpr \vn{f}_{L^2}+\vn{F}_{N_k^+} \rpr  \ee
\be \label{splusest}  \vn{\psi_k^1}_{S_k^+} \lesssim \vn{f}_{L^2}+\vn{F}_{N_k^+}. \ee
\end{proposition}

\begin{proof}
By scaling invariance, we may assume $ k=0 $.
Define 
$$  \psi_0^1 \defeq \hwm \phi $$
where $ \phi $ is obtained by applying Proposition~\ref{propmod} to $ F, f $ and $ -A^{free} $. At this low $ Q^{+} $-modulation, the norms of $ N_0 $ and $ N_0^{+} $ coincide. Observe that on that space-time frequency region, the symbol of $ \hwm $ is $ \sim 1 $ and behaves as a bump function. Moreover, 
$$ \vn{\psi_0^1}_{S_0^+} \lesssim \vn{\psi_0^1}_{S_0} \lesssim \vn{\phi}_{S_0}   $$
which implies \eqref{splusest}. We write
\be \label{covopreduction}
\begin{aligned}
\hw_{A_{<0}}^p \psi^1_0 &= \Box \phi - i  A^{free, \ell}_{<-C} \frac{\pt_{\ell}}{\vm{D}} P_0 (i\pt_t +\vm{D}-2\vm{D})\phi \\
&=  \Box_{-A_{<0}}^p \phi - i  A^{free, \ell}_{<-C} \frac{\pt_{\ell}}{\vm{D}} \hw P_0 \phi. 
\end{aligned} \ee
Since $ \vn{A^{free}_{<-C}}_{L^2 L^{\infty}} \lesssim \ep $, we estimate
\be
\begin{aligned}
\vn{A^{free, \ell}_{<-C} \frac{\pt_{\ell}}{\vm{D}} \hw P_0 \phi}_{L^1 L^2 } &  \lesssim \ep \sum_{j \leq 0} 2^j \vn{Q^{+}_j P_0 \phi}_{L^{2} L^{2}} \lesssim \ep \vn{\phi}_{X_{\infty}^{0,\frac{1}{2}}} \\
& \lesssim \ep \vn{\phi}_{S_0}  \lesssim \ep ( \vn{f}_{L^2}+\vn{F}_{N_0})
\end{aligned}  \ee
where the last inequality comes from Proposition~\ref{propmod}, which completes the proof. \qedhere
\end{proof}

\subsection{Proof of Proposition \ref{papp}} \label{subsecproof}
We are now ready to prove Proposition~\ref{papp}.

\subsubsection*{The approximate solution $ \psi^a $.}

We define $ \psi^a \defeq \sum_k \psi_k^a $ from its frequency-localized versions 
$$ \psi_k^a \defeq \psi_k^1+\psi_k^2 $$ 
which remain to be defined. 

We decompose $ F=\sum_k P_k F $ and $ P_k F=Q_{<k-6}^+ P_k F + Q_{>k-6}^+ P_k F $. We first define $ \psi_k^2 $ by  
\be \mathcal{F} \psi_k^2 (\tau,\xi) \defeq \frac{1}{-\tau+\vm{\xi}} \mathcal{F}(Q_{>k-6}^+ P_k F)(\tau,\xi) \ee
so that $ \hw \psi_k^2=Q_{>k-6}^+ P_k F $.

Then we apply Proposition~\ref{phyb} to $ Q_{<k-6}^+ P_k F $ and $ P_k f-\psi_k^2(0) $ which defines the function $ \psi_k^1 $.
\newline

\subsubsection*{Reduction to the frequency-localized case.}
By redefining $ \delta $ (taking $ \ep $ smaller), it suffices to show
\begin{equation} 
\begin{aligned}
\vn{P_{k} [\psi^a(0)-f]}_{\dot{H}^{1/2}}+ \vn{P_k [\phw \psi^a - F]}_{N_{+}^{1/2} \cap L^{2} L^{2}}  \\
\lesssim \delta \sum_{k'=k+O(1)} \lpr  \vn{P_{k'}f}_{\dot{H}^{1/2}}+ \vn{P_{k'} F}_{N_{+}^{1/2} \cap L^{2} L^{2}} \rpr ,
\end{aligned}
\end{equation}
\be \label{ltwobd} \vn{P_k \psi^a}_{S_{+}^{1/2}} \lesssim \sum_{k'=k+O(1)} \vn{P_{k'} f}_{\dot{H}^{1/2} }+ \vn{P_{k'} F}_{N_{+}^{1/2} \cap L^{2} L^{2}}. \ee 

Notice that
\be P_{k}[ \phw \psi^a-F]=\sum_{k'=k+O(1)} P_{k} [\phw \psi_{k'}^a-P_{k'}F], \ee
and the analogous summation for $ P_k \psi^a $ and $ P_k [\psi^a(0)-f] $.
By disposing of $ P_k $ it suffices to show the following estimates:
\be \label{apres1} 
\begin{aligned}
\vn{\hw_{A_{<k'}}^p \psi_{k'}^a- P_{k'} F}_{N_{k'}^+} + \vn{\psi_{k'}^a(0)-P_{k'} f}_{L^2} \\
\lesssim \delta \lpr \vn{P_{k'} f}_{L^2}+ \vn{P_{k'}F}_{N_{k'}^{+} \cap L^2 \dot{H}^{-1/2}} \rpr 
\end{aligned}
\ee
\be \label{apres2} 
\begin{aligned}
2^{-k'/2} \vn{\hw_{A_{<k'}}^p \psi_{k'}^a- P_{k'} F}_{L^{2} L^{2}} \lesssim \delta \lpr \vn{P_{k'} f}_{L^2}+ \vn{P_{k'}F}_{N_{k'}^{+} \cap L^2 \dot{H}^{-1/2} } \rpr 
\end{aligned}
\ee
\be \label{apres3} \vn{\psi_{k'}^a}_{S_{k'}^+} \lesssim \vn{P_{k'} f}_{L^2}+ \vn{P_{k'} F}_{N_{k'}^{+} \cap L^2 \dot{H}^{-1/2}} \ee
\be \label{ltwobd}   2^{-\frac{k}{2}} \vn{\hw \psi_{k'}^a}_{L^{2} L^{2}} \lesssim \vn{P_{k'} f}_{L^2}+\vn{P_{k'} F}_{N_{k'}^+ \cap L^{2} \dot{H}^{-1/2}} \ee 
and the following error term, where $ k'',k'''=k'\pm O(1) $:
\be \label{aferr} \vn{A^{free, j}_{k''-c} \frac{\rd_{j}}{\abs{D}} P_{k'''} \psi_{k'}^a}_{N_{k'}^{+} \cap L^2 \dot{H}^{-1/2}} \lesssim \ep  \vn{\psi_{k'}^a}_{S_{k'}^+} \ee

\subsubsection*{Proof of claims \eqref{apres1}--\eqref{aferr}}
It only remains to prove \eqref{apres1}--\eqref{aferr}.
\pfstep{Step~1: Proof of \eqref{apres3}}
For $ \psi_k^2 $ we have, by Lemma \ref{lemmadivsymb}
\be \label{eqbd1} \vn{\psi_k^2}_{S_k^+} \lesssim \vn{Q_{>k-6}^+ P_k F}_{N_k^+ \cap L^2 \dot{H}^{-1/2}}. \ee 
For the function $ \psi_k^1 $, by Proposition~\ref{phyb}, we have
\be \label{eqbd2} \vn{\psi_k^1}_{S_k^+} \lesssim \vn{P_k f-\psi_k^2(0)}_{L^2}+\vn{Q_{<k-6}^+ P_k F}_{N_k^+} \lesssim  \vn{P_k f}_{L^2}+ \vn{P_k F}_{N_k^+ \cap L^2 \dot{H}^{-1/2}}. \ee
We have used \eqref{eqbd1} to bound $ \vn{\psi_k^2(0)}_{L^2}$.

\pfstep{Step~2: Proof of \eqref{apres1}}  By Proposition~\ref{phyb}, we have
\begin{equation} \label{intermedest}
\begin{aligned}
& \hskip-2em \vn{\psi_k^1(0)-[P_k f-\psi_k^2(0)]}_{L^2}+ \vn{\hw_{A_{<k}}^p \psi_k^1-Q_{<k-6}^+ P_k F}_{N_k^+}  \\
\leq & \delta \lpr \vn{P_k f-\psi_k^2(0)}_{L^2}+\vn{Q_{<k-6}^+ P_k F}_{N_k^+}  \rpr \\
\lesssim & \delta ( \vn{P_{k} f}_{L^2}+ \vn{P_{k}F}_{N_{k}^+ \cap L^2 \dot{H}^{-1/2}} ) 
\end{aligned}
\end{equation}
It remains to estimate
\begin{align*}
& \hskip-2em \vn{\hw_{A_{<k}}^p \psi_k^2-Q_{>k-6}^+ P_k F}_{N_k^+} \leq \vn{A_{<k-C}^{free,j} \frac{\pt_j}{\vm{D}} P_k \psi_k^2}_{N_k^+} \\
\lesssim & \vn{A_{<k-C}^{free}}_{L^2 L^{\infty}} \vn{\psi_k^2}_{L^{2} L^{2}} \lesssim (\ep 2^{k/2}) 2^{-k/2} \vn{Q_{>k-6}^+ P_k F}_{N_k^+\cap L^2 \dot{H}^{-1/2}}  \\\lesssim & \ep \vn{P_k F}_{N_k^+\cap L^2 \dot{H}^{-1/2}}
\end{align*}
The first inequality follows from the definition \eqref{paracovop2}. The third inequality follows from \eqref{eqbd1}.

\pfstep{Step~3: Proof of \eqref{apres2}} We estimate 
\be  2^{-\frac{k}{2}} \vn{\hw_{A_{<k}}^p \psi_{k}^1- Q_{<k-6}^+ P_k F}_{L^{2} L^{2}} \lesssim \delta ( \vn{P_{k} f}_{L^2}+ \vn{P_{k}F}_{N_{k}^+ \cap L^2 \dot{H}^{-1/2}} )  \label{ltwoestim} \ee 
using \eqref{simpleembedding} and \eqref{intermedest}. For $  \hw_{A_{<k}}^p \psi_k^2 - Q_{>k-6}^+ P_k F $, using \eqref{eqbd1} we estimate 
$$ \vn{A_{<k-C}^{free,j} \frac{\pt_j}{\vm{D}} P_k \psi_k^2}_{L^{2} L^{2}} \lesssim  \vn{A_{<k-C}^{free}}_{L^2 L^{\infty}} \vn{\psi_k^2}_{L^{\infty} L^2} \lesssim 2^{k/2} \ep \vn{Q_{>k-6}^+ P_k F}_{N_k^+\cap L^2 \dot{H}^{-1/2} }.   $$

\pfstep{Step~4: Proof of \eqref{ltwobd}} We write
\begin{align*}
\hw \psi_k^a =&Q_{>k-6}^+ P_k F+Q_{<k-6}^+ P_k F+ \\
& + (\hw_{A_{<k}}^p \psi_k^1-Q_{<k-6}^+ P_k F)+ A_{<k-C}^{free,j} \frac{\pt_j}{\vm{D}} P_k \psi_k^1.
\end{align*}
We use \eqref{ltwoestim}  and it remains to estimate
\
\begin{align} 
2^{-k/2} \vn{A_{<k-C}^{free,j} \frac{\pt_j}{\vm{D}} P_k \psi_k^1}_{L^2 L^2} 
\lesssim &  2^{-k/2} \vn{A_{<k-C}^{free}}_{L^2 L^{\infty}} \vn{\psi_k^1}_{L^{\infty} L^2} \label{almostfinished} \\ 
\lesssim & \ep (\vn{P_k f}_{L^2}+ \vn{P_k F}_{N_k^+\cap L^2 \dot{H}^{-1/2}}). 
\end{align} 

\pfstep{Step~5: Proof of \eqref{aferr}} The $ N_{k'}^{+} $ bound follows from \eqref{eq:bi-bal-freq}, while the $L^{2} \dot{H}^{-1/2}$ bound follows from the estimate \eqref{almostfinished} with $ k $ replaced by $ k',k'',k'''$.

\begin{remark} \label{rem:hi-d-5}
The construction in \cite[Sections~6--11]{KST} may be generalized to $\bbR^{1+d}$ with $d \geq 5$ without much difficulty. The rest of the argument in this section then goes through with the substitutions as in Remark~\ref{rem:hi-d-1}.
\end{remark}
\bibliographystyle{abbrv}
\bibliography{4dMD}


\end{document}